\documentclass[a4paper,11pt]{article}

\usepackage[utf8]{inputenc}
\usepackage{amssymb,amsmath,mathrsfs,amsthm,bbm,xcolor}
\usepackage{verbatim}
\usepackage{indentfirst}
\usepackage{geometry}
\usepackage{enumerate} \usepackage{graphicx}

\usepackage{booktabs,caption}
\usepackage[flushleft]{threeparttable}
\usepackage{hyperref} 
\hypersetup{colorlinks=true,allcolors=blue}
\usepackage{hypcap}
\usepackage{longtable}

\usepackage[capitalize]{cleveref} 

\newtheorem{thm}{Theorem}[section]
\newtheorem{prop}[thm]{Proposition}
\newtheorem{cor}[thm]{Corollary}

\newtheorem{lem}[thm]{Lemma}
\newtheorem*{lem*}{Lemma}

\newtheorem{ques}[thm]{Question}

\newtheorem*{claim*}{Claim}
\newtheorem{claim}{Claim}

\theoremstyle{remark}
\newtheorem{rem}[thm]{Remark}
\newtheorem*{rem*}{Remark}

\theoremstyle{definition}
\newtheorem{defi}[thm]{Definition}

\makeatletter
\newtheorem*{rep@theorem}{\rep@title}
\newcommand{\newreptheorem}[2]{%
\newenvironment{rep#1}[1]{%
 \def\rep@title{#2 \ref{##1}}%
 \begin{rep@theorem}}%
 {\end{rep@theorem}}}
\makeatother
\newreptheorem{thm}{Theorem}
\newreptheorem{lem}{Lemma}

\input{macros}

\begin{document}
	
\bibliographystyle{alpha}
\title{\textbf{On the dimension of limit sets on $\P(\R^3)$ via stationary measures: the theory and applications}}
\author{Jialun Li, Wenyu Pan and Disheng Xu}
\date{}
\maketitle
\begin{abstract}
    This paper investigates the (semi)group action of $\SL_3(\R)$ on $\P(\R^3)$, a primary example of non-conformal, non-linear, and non-strictly contracting action. We study the Hausdorff dimension of a dynamically defined limit set in $\P(\R^3)$ and generalize the classical Patterson-Sullivan formula using the approach of stationary measures.

The two main examples are Anosov representations in $\SL_3(\R)$ and the Rauzy gasket.
\begin{enumerate}
    \item For Anosov representations in $\SL_3(\R)$, we establish a sharp lower bound for the dimension of
    their limit sets in $\P(\R^3)$. Coupled with the upper bound in \cite{pozzetti_anosov_nodate}, it shows that their Hausdorff dimensions equal the affinity exponents. The merit of
    our approach is that it works uniformly for all the components of irreducible Anosov representations in $\SL_3(\R)$.
    As an application, it reveals a surprising dimension jump phenomenon in the Barbot component, which is a local generalization of Bowen's dimension rigidity result \cite{bowen_hausdorff_1979}.
\item For the Rauzy gasket, we confirm a folklore conjecture about the Hausdorff dimension of the gasket and improve the numerical lower bound to $3/2$.
\end{enumerate}

These results originate from a dimension formula of stationary measures on $\P(\R^3)$.
Let $\nu$ be a probability measure on $\SL_3(\R)$ whose support is finite and spans a Zariski dense subgroup. Let $\mu$ be the associated stationary measure for the action on $\P(\R^3)$. Under the exponential separation condition on $\nu$, we prove that the Hausdorff dimension of $\mu$ equals its Lyapunov dimension, which extends \cite{hochman_dimension_2017} and \cite{barany_hausdorff_2017} to non-conformal and projective settings respectively.

\end{abstract}

\tableofcontents

\section{Introduction}

This paper investigates the (semi)group action of $\SL_3(\R)$ on $\P(\R^3)$, a primary example of non-conformal, non-linear, and non-strictly contracting action. We focus on the following two important and closely related problems:
\begin{enumerate}
\item (\textbf{dimension of the limit sets}) Let $\Gamma$ be a discrete subgroup (subsemigroup) of $\SL_3(\R)$ acting on $\P(\R^3)$. We study how large the limit set $L(\Gamma)$ is by determining 
the Hausdorff dimension of $L(\Gamma)$. 

\item (\textbf{dimension of stationary measures}) Let $\nu$ be a finitely supported probability measure on $\SL_3(\R)$. 
Let $\mu$ be a $\nu$-stationary measure for the action on $\P(\R^3)$. We study how large $\mu$ is by determining the Hausdorff dimension of $\mu$. 
\end{enumerate}

\subsection{Hausdorff dimension of limit sets}
Studying dynamically defined limit sets is a central question in various areas of mathematics, for example, dynamical systems, fractal geometry, and hyperbolic geometry.

In his seminal papers \cite{sullivan_density_1979}, \cite{sullivan_entropy_1984}, Sullivan considered a Kleinian group $\Gamma$ in $\Isom_+(\H^3)\cong \PSL_2(\C)$ and studied its limit set in the visual boundary $\partial\H^3$. The Poincar\'e series of the group $\Gamma$ is defined by 
$$P_\Gamma(s)=\sum_{\gamma\in\Gamma}\Bsb{\frac{\sigma_2}{\sigma_1}}^s(\gamma),$$
where $\sigma_i(\gamma)$ are the singular values of $\rho(\gamma)$ in decreasing order.
The classical critical exponent $\delta(\Gamma)$ of $\Gamma$ is defined to be the critical exponent of series $P_\Gamma(s)$, i.e. $$\delta(\Gamma)={ \sup}_{P_\Gamma(s)=\infty}\ s={ \inf}_{P_\Gamma(s)<\infty}\ s.$$ Sullivan \cite{sullivan_density_1979} proved the beautiful equality between the critical exponent $\delta(\Gamma)$, a dynamical invariant, and the Hausdorff dimension of its conic limit set, a geometric invariant. It generalized the earlier work by Patterson for Fuchsian groups \cite{patterson_limit_1976}. Furthermore, Sullivan established the deep relation between the Hausdorff dimension of the limit set and
the bottom of the eigenvalue of the Laplace operator and the entropy 
of the geodesic flow on $\Gamma\backslash \H^3$.

In this paper, we consider a discrete subgroup (subsemigroup) $\Gamma$ of $\SL_3(\R)$ acting on $\P(\R^3)$. We 
investigate whether an analogous equality is available for the Hausdorff dimension of its limit set. Two main examples are the Anosov representations in $\SL_3(\R)$ and the Rauzy gasket. The group $\PSL_2(\C)$ acts conformally on $\partial \H^3$, which plays an essential role in Sullivan's proof. In contrast, $\SL_3(\R)$ acts on $\P(\R^3)$ non-conformally: circles are mapped to ellipses, which poses an immediate challenge in the estimation of the Hausdorff dimension.

In the following, we usually denote by $\dim \cdot$ for the Hausdorff dimension of a set and by $L(\cdot)$ for the limit set (see \cref{defi: Anosov rep}) of the projective action of a discrete linear subgroup (subsemigroup).

\subsubsection{Anosov representation}


The concept of \textit{Anosov representation} was first introduced by Labourie in \cite{labourie_anosov_2006} to study Hitchin components of the representations of surface groups, which is a natural generalization of convex-cocompact representations in rank one semisimple Lie groups to the higher rank setting. See also \cite{guichard_anosov_2012}, \cite{kapovich_morse_2018}, \cite{gueritaud_anosov_2017}, \cite{bochi_anosov_2019} for further study on Anosov representations. 
\begin{defi}\label{defi: Anosov rep}
Let $\Gamma$ be a hyperbolic group \footnote{In this paper we always consider finitely generated hyperbolic groups.}
and $\rho$ a representation from $\Gamma$ to $\SL_3(\R)$. Then $\rho$ is called an Anosov representation if there exist $C,c>0$ such that for any $\gamma\in\Gamma $, we have
\[ \frac{\sigma_1(\rho(\gamma))}{\sigma_2(\rho(\gamma))}\geq \frac{1}{C}e^{c|\gamma|}, \]
where $\sigma_i(\rho(\gamma))$ are the singular values of $\rho(\gamma)$ in decreasing order, and $|\gamma|$ is the word length with respect to a fixed symmetric generating set.  The limit set $L(\rho(\Gamma))$ of $\rho$ in $\P(\R^3)$ is defined to be the closure of the set of attracting fixed points of elements in $\rho(\Gamma)$. Let $\Hom(\Gamma,\SL_3(\R))$ be the set of $\SL_3(\R)$-representations of the group $\Gamma$, and $\HA(\Gamma,\SL_3(\R))$ be the subset of $\Hom(\Gamma,\SL_3(\R))$ consisting of Anosov representations. 
\end{defi}

Here are some important properties of Anosov representations. If $\rho\in \HA(\Gamma,\SL_3(\R))$, then $\rho(\Gamma)$ is discrete in $\SL_3(\R)$. Any Anosov representation is stable under small perturbation in $\Hom(\Gamma,\SL_3(\R))$, hence  $\HA(\Gamma,\SL_3(\R))$ is open in $\Hom(\Gamma,\SL_3(\R))$. 
\\
\\
\textbf{Affinity exponent.}
The concept of \textit{affinity exponent} (or \textit{affinity dimension}) of Anosov representations
 was introduced in \cite{pozzetti_anosov_nodate}, which is essentially a natural extension of the affinity exponent of self-affine fractals due to Falconer \cite{falconer_hausdorff_1988}.

For a representation $\rho:\Gamma\to \SL_3(\R)$, we can similarly define a Poincar\'e series of the group $\rho(\Gamma)$ by 
\begin{equation}\label{eqn: def aff exp}
P_\rho(s)=\case{&\sum_{\gamma\in\Gamma}\Bsb{\frac{\sigma_2}{\sigma_1}}^s(\rho(\gamma)),& 0<s\leq 1;\\
&\sum_{\gamma\in\Gamma}\Bsb{\frac{\sigma_2}{\sigma_1}}(\rho(\gamma))\Bsb{\frac{\sigma_3}{\sigma_1}}^{s-1}(\rho(\gamma)), & 1<s\leq 2. }  
\end{equation}

%
We denote the critical exponent of $P_{\rho}(s)$ by $s_{\mathrm{A}}(\rho)$ and call it the affinity exponent (or the affinity dimension). The affinity exponent $s_{\mathrm{A}}(\rho)$ is always an upper bound of $\dim L(\rho(\Gamma))\subset \P(\R^3)$ if $\rho:\Gamma\to \SL_3(\R)$ is an Anosov representation \cite{pozzetti_anosov_nodate}. In this paper, we show $s_{\mathrm{A}}(\rho)$ is also a lower bound of $\dim L(\rho(\Gamma))$, which is a first general result in this non-conformal, non-linear, non-strictly contracting setting. Hence, we  extend Sullivan's theorem to 
Anosov representations in $\SL_3(\R)$.

\begin{thm}\label{thm:hausdorff}
    Let $\Gamma$ be a hyperbolic group and $\rho:\Gamma\rightarrow\SL_3(\R)$ be an irreducible Anosov representation, then the Hausdorff dimension of the limit set $L(\rho(\Gamma))$ in $\P(\R^3)$ equals the affinity exponent $s_{\mathrm{A}}(\rho)$.
     Moreover, the exponent $s_{\mathrm{A}}(\rho)$ is continuous on $\HA(\Gamma,\SL_3(\R))$.
\end{thm}

When $\Gamma$ is a surface group, 
if $\rho\in \Hom(\Gamma,\SL_3(\R))$ is in the Hitchin component, 
then the limit set $L(\rho(\Gamma))$ is a $C^1$-circle \cite{labourie_anosov_2006}, and hence its Hausdorff dimension is $1$. 
The Hausdorff dimension of limit sets of Anosov representations has been extensively studied, see for example \cite{pozzetti_anosov_nodate}, \cite{pozzetti_conformality_2019}, \cite{glorieux_hausdorff_2019}, \cite{dey_patterson-sullivan_2022}, \cite{kimHausdorffDimensionDirectional2023}. In \cite{dufloux_hausdorff_2017}, Dufloux studied this problem for the non-conformal action of Schottky groups in $\operatorname{PU}(1,n)$ on $\partial \mathbb{H}^n_{\C}$, the boundary of the complex hyperbolic $n$-space. 

There is a rich literature on studying the dimension theory of non-conformal action. In \cite{barany_hausdorff_2017}, B\'{a}r\'{a}ny-Hochman-Rapaport provided a relatively complete study on self-affine contracting Iterated-Functional-Systems (IFS) on $\R^2$. Our proof draws inspiration from their method. See also \cite{hochman_hausdorff_2021} and \cite{rapaport_self-affine_2022} for further development, as well as the classical results, for example, \cite{falconer_hausdorff_1988}, \cite{hueter_falconers_1995} in this setting. See \cite{ren_dichotomy_2021} for the graphs of Weierstrass-type functions,  \cite{chen_dimension_2010} and \cite{cao_dimension_2019} for non-conformal repellers.

\subsubsection{Dimension jump}

One surprising corollary of \cref{thm:hausdorff} is a local generalization of Bowen's dimension rigidity result \cite{bowen_hausdorff_1979}.
Specifically, let $\Gamma$ be a surface group, 
and let $\eta_0:\Gamma\rightarrow \PSL_2(\R)$ be 
a faithful representation onto a cocompact Fuchsian group of $\PSL_2(\R)$.
We view $\PSL_2(\R)$ as a subgroup of $\PSL_2(\C)$ and let $\iota_0$ be the natural embedding of $\PSL_2(\R)$ in $\PSL_2(\C)$, then 
$\iota_0\circ\eta_0$ is a representation 
of $\Gamma$ to $\PSL_2(\C)$. The limit set of $\iota_0\circ\eta_0(\Gamma)$ is the  smooth circle $\P_\R^1$ in the sphere $\P_\C^1$. 
Let $\rho:\Gamma\to \PSL_2(\C)$ be a representation, which is a small perturbation of $\iota_0\circ \eta_0.$ 
The limit set $L(\rho(\Gamma))$ is known to be 
homeomorphic to a circle, which implies $\dim L(\rho(\Gamma))\geq 1$. Bowen proved that 
$\dim L(\rho(\Gamma))=1$ 
if and only if $\rho(\Gamma)$ is contained in a conjugate of $\SL_2(\R)$.\footnote{The actual result of Bowen is global, 
not only for small perturbation. The version stated here is a weaker local version.}

We extend Bowen's result in a non-conformal setting: consider the action of $\SL_3(\R)$ on $\P(\R^3)$. 
One of our main results is a dimension jump phenomenon, which leads to a local dimension rigidity as follows. 

We consider $\Hom(\Gamma,\SL_3(\R))$ and deform representations here in a way similar to Bowen's setting.
More precisely, 
take an arbitrary lift  $\rho_0:\Gamma\rightarrow\SL_2(\R)$ of $\eta_0$.\footnote{One can do it by choosing the images of generators in the double cover $\SL_2(\R)$ of $\PSL_2(\R)$.} Denote by $\iota$ the embedding from $\SL_2(\R)$ to the upper left corner of $\SL_3(\R)$, then the composition $\rho_1:=\iota\circ\rho_0$ is an element in $\Hom(\Gamma,\SL_3(\RR)).$ 
The action of $\rho_1(\Gamma)$ preserves $\P(\R^2)$ in $\P(\R^3)$, which is a smooth circle. 
 Due to \cite{sullivan_quasiconformal_1985}, if $\rho$ in $\Hom(\Gamma,\SL_3(\R))$ is sufficiently close to $\rho_1$, i.e. its images on generators are close to that of $\rho_1$, then $\rho(\Gamma)$ preserves a unique topological circle close to $\P(\R^2)$ in $\P(\R^3)$, which turns out to be the limit set $L(\rho(\Gamma))$ (see \cref{defi: Anosov rep}).
\begin{thm}\label{thm:dimension_jump}
    For every $\epsilon>0$, there exists a small neighborhood $O$ of $\rho_1$ in $\mathrm{Hom}(\Gamma,\SL_3(\R))$ such that for any $\rho$ in $O$ we have
    \begin{itemize}
    \item either $\rho(\Gamma)$ acts reducibly, i.e. fixing a point or a projective line in $\P(\R^3)$; 
    \item or $\rho(\Gamma)$ is irreducible and 
    \[ |\dim L(\rho(\Gamma))- \frac{3}{2}|\leq \epsilon. \]
    \end{itemize}
\end{thm}

In other words, once we perturb $\rho_1$ to any irreducible $\rho$, there is  
a dimension jump with size about $1/2$, which does not occur 
in the conformal case. A generic perturbation $\rho$ is irreducible because reducible representations form a proper subvariety in $\Hom(\Gamma,\SL_3(\R))$ (cf. \cite{bridgeman_pressure_2015}).
The representation $\rho_1$ in Theorem \ref{thm:dimension_jump} is in the Barbot component of $\Hom(\Gamma, \SL_3(\R))$. For an irreducible Anosov representation in Barbot's component, Barbot \cite{barbot_three-dimensional_2005} proved that the limit set is not Lipschitz. Our result provides new knowledge for Anosov representations in  Barbot's component. 

There are non-algebraic settings where a similar dimension jump occurs. For example, see \cite{ren_dichotomy_2021} and references therein for results on graphs of Weierstrass-type functions, and \cite{bonatti_discontinuity95} for discontinuity of Hausdorff dimension of hyperbolic sets.

\subsubsection{Rauzy gasket}

The results (\cref{thm:lyapunov}) in this paper allow us to study other fractal sets in $\P(\R^3)$. In the second part of our work with Yuxiang Jiao \cite{JLPX},
we study the Rauzy gasket, a self-projective fractal set in $\P(\R^3)$, and establish the identity between its Hausdorff dimension and its affinity exponent. 

The gasket has an interesting history, appearing for the first time in 1991 in the work
of Arnoux and Rauzy \cite{arnoux_geometric_1991} in the context of interval exchange transformations. The gasket reemerged in the work of De Leo and Dynnikov \cite{deleo2009geometry} in the context of Novikov’s theory
of magnetic fields on monocrystals. It also appeared in the work of Gamburd, Magee and
Ronan \cite{gamburdAsymptoticFormulaInteger2019} about the asymptotic estimates for integer solutions of the Markov-Hurwitz equations. For more background and history of the Rauzy gasket, please see \cite{dynnikov_dynamical_2020} and \cite{JLPX}.

Let $\Delta$ be the projectivization of $\{(x,y,z): x,y,z\geq 0\}$ in $\PP(\RR^3).$
Let $\Gamma$ be the semigroup generated by 
    \[A_1=\begin{pmatrix} 1 & 1 & 1 \\ 0 & 1 & 0\\ 0& 0& 1\end{pmatrix},\  A_2=\begin{pmatrix} 1 & 0 & 0 \\ 1 & 1 & 1\\ 0& 0& 1\end{pmatrix},\  A_3=\begin{pmatrix} 1 & 0 & 0 \\ 0 & 1 & 0\\ 1& 1& 1\end{pmatrix},\] 
    and we call it the Rauzy semigroup.
    Then as $\Gamma\subset\SL_3(\R)$, the semigroup $\Gamma$ acts on $\P(\R^3)$. Due to the choice of $\Delta$, the semigroup $\Gamma$ preserves $\Delta$. The Rauzy gasket $\mathscr{R}$ is the unique attractor of the Rauzy semigroup, which can be defined formally as
    \[\bigcap_{n\rightarrow\infty}\bigcup_{i_j\in\{1,2,3\}} (A_{i_1}\cdots A_{i_{n}}\Delta). \]

\begin{figure}[!ht]
\centering
\includegraphics[width=0.43\textwidth]{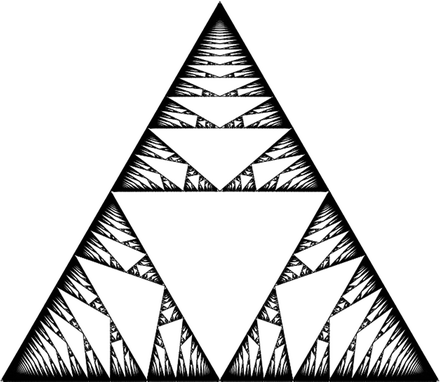}
\caption{Figure in \cite{arnoux_rauzy_2013}}
\end{figure}

It has received a lot of attention to study the Hausdorff dimension of the Rauzy Gasket $\mathscr{R}$, and the best-known bound is $1.19< \dim \mathscr R\leq 1.7407$ due to \cite{gutierrez-romo_lower_2020} and \cite{pollicott_upper_2021} respectively.
See also the previous works, for example, \cite{levitt1993dynamique}, \cite{avila_hausdorff_2016}, \cite{fougeron2020dynamical}.


The affinity exponent $s_{\mathrm{A}}(\Gamma)$ of the Rauzy gasket is the critical exponent of the Poincar\'e series $P_{\Gamma}(s)$ defined in \cref{eqn: def aff exp} using the Rauzy semigroup $\Gamma$. 
The following theorem confirms a folk-lore conjecture of the Hausdorff dimension of the Rauzy gasket. 

\begin{thm}\label{thm:rauzy}Let  $\mathscr{R},\Gamma$ be the Rauzy gasket and Rauzy semigroup respectively, $s_{\mathrm{A}}(\Gamma)$ be the associated affinity exponent, then $\dim \mathscr{R}=s_{\mathrm{A}}(\Gamma)$. 

\end{thm}

We improve the numerical lower bound as an immediate corollary of the result and the proof of Theorem \ref{thm:rauzy}.
\begin{cor}\label{cor: 3/2} We have $\dim \mathscr{R}\geq 3/2$.
\end{cor}
The argument is versatile and allows us to deal with 
a (semi)group that contains a large reducible subsemigroup, which is also the secret sauce of the dimension jump phenomenon in Theorem \ref{thm:dimension_jump}.
\color{black}

\begin{rem} Recently
 Natalia Jurga also obtained \cref{thm:rauzy} independently. 
\end{rem}

\subsection{Dimension of stationary measures}

Stationary measures 
are a fundamental tool and a central topic in the study of random walks arising from (semi)group actions on metric spaces. Its Hausdorff dimension encodes 
other important dynamically defined quantities of random walks, such as entropies and Lyapunov exponents. 
It provides a powerful tool to estimate the lower bound of the Hausdorff dimension of a fractal set, allowing us to access the rich toolbox of random walks.

\subsubsection{Hausdorff dimension and Lyapunov dimension of stationary measures}

One key ingredient to prove \cref{thm:hausdorff} and \cref{thm:rauzy} is to study the Hausdorff dimension and \textit{Lyapunov dimension} of stationary measures coming from random walks on $\P(\R^3)$, which is partially inspired by the methods in \cite{barany_hausdorff_2017} 
and \cite{hochman_dimension_2017}, which are developed from \cite{hochman_self-similar_2014}. 

Recall that the Hausdorff dimension of a Borel probability measure $\mu$ on a metric space $X$ is defined as \[\dim \mu:=\inf_{A\subset X, \mu(A)=1}\dim A.\]

The Lyapunov dimension of measures is 
an analog of the affinity exponent of sets. For our purpose, we only define it for stationary measures on $\P(\R^3)$ as follows.

In this paper, we call a measure supported on $\SL_3(\R)$  \textit{Zariski dense} if the group generated by $\supp\nu$ is Zariski dense in $\SL_3(\R)$. Let $\nu$ be a finitely supported Zariski dense probability measure on $\SL_3(\R)$, and $\mu$ be the unique $\nu$-stationary measure (Furstenberg measure) on $\P(\R^3)$, that is, $\mu$ satisfies $\nu*\mu=\mu$, where the convolution is induced by the projective action of $\SL_3(\R)$ on $\P(\R^3)$. Let $\lambda(\nu)=(\lambda_1,\lambda_2,\lambda_3)$ be the Lyapunov vector of $\nu$ in decreasing order and $\chi_i(\nu)=\lambda_1-\lambda_{i+1}$. As $\nu$ is Zariski dense, we have $\chi_2(\nu)>\chi_1(\nu)>0$ \cite{goldsheid1989lyapunov} and \cite{guivarch_frontiere_1985}. Then the Lyapunov dimension of $\mu$ can be defined as (following \cite{KY79} and \cite{DO80})
\[ \dim_{\rm{LY}}\mu=\case{
   & \frac{h_{\rm F}(\mu,\nu)}{\chi_1(\nu)}, &\text{ if } h_{\rm F}(\mu,\nu)\leq \chi_1(\nu); \\
    & 1+\frac{h_{\rm F}(\mu,\nu)-\chi_1(\nu)}{\chi_2(\nu)},&\text{ otherwise,}}  \]
where $h_{\rm F}(\mu,\nu)$ is the Furstenberg entropy (see \cref{equ:f entropy} for definition). 

In \cite{ledrappier_exact_2021} and \cite{rapaport_exact_2021}, the authors proved that $\dim_{\rm{LY}}\mu\geq \dim\mu$. 
Obtaining $\dim_{\rm{LY}}\mu\leq \dim\mu$ is much more challenging. In the one-dimensional setting, the dimension theory and absolute continuity of Bernoulli convolution have attracted a lot of attention since Erd\"{o}s. Hochman \cite{hochman_self-similar_2014} made a breakthrough in the dimension theory. See also \cite{varju2022self} for the series of progress of Varj\'{u} on the absolute continuity. 

In this paper, we take a step further and study the action of $\SL_3(\R)$ on $\P(\R^3)$, which extends \cite{hochman_dimension_2017} and \cite{barany_hausdorff_2017} to non-conformal and projective settings respectively.
Our result is as follows. We pose a mild separation assumption on random walks as in \cite{hochman_dimension_2017},  \cite{barany_hausdorff_2017}, \cite{hochman_self-similar_2014}, etc.: a finitely supported measure $\nu$ on $\SL_3(\R)$ is called satisfying the exponential separation condition, if 
there exist $C>0$ and $N>0$, for all $n\geq N$ and $(g_1,\cdots, g_n),\ (g_1',\cdots, g_n')$ in $(\supp\nu)^{\times n}$ with $g_1\cdots g_n\neq g_1'\cdots g_n'$, then 
\[d(g_1\cdots g_n,g_1'\cdots g_n' )>C^{-n},  \]
where $d$ is a left $\SL_3(\R)$-invariant and right $\SO_3(\R)$-invariant Riemannian metric on $\SL_3(\R)$. In particular, if $\supp\nu$ generates a discrete subgroup (for example, $\nu$ is supported on an Anosov representation), then $\nu$ satisfies the exponential separation condition.

\begin{thm}\label{thm:lyapunov}
    Let $\nu$ be a Zariski dense, finitely supported probability measure on $\SL_3(\R)$ that satisfies the exponential separation condition, and $\mu$ be its stationary measure on $\P(\R^3)$. Then we have
    \[\dim\mu=\dim_{\rm{LY}}\mu.  \]
\end{thm}

For the action of $\SL_2(\R)$ on $\P(\R^2)$, Ledrappier obtained this result in 1983 \cite{ledrappier-1983}. Later, Hochman-Solomyak advanced the knowledge in this setting \cite{hochman_dimension_2017}. Recall that the random walk entropy $h_{\rm{RW}}(\nu)$ of $\nu$ is defined by 
\begin{equation}\label{equ:random walk entropy}
h_{\rm{RW}}(\nu):=\lim_{n\rightarrow\infty}H(\nu^{*n})/n  
\end{equation}
with $H(\nu^{*n})$ the Shannon entropy of a finite set. Hochman-Solomyak obtained \cref{thm:lyapunov} by replacing the Furstenberg entropy $h_{\rm{F}}(\mu,\nu)$ by the random walk entropy $h_{\rm{RW}}(\nu)$. 

Ledrappier-Lessa \cite{ledrappier_exact_2021} proved \cref{thm:lyapunov} for the case when the measure $\nu$ is supported on a Hitchin representation in $\PSL_n(\R)$ for any $n\geq 3$. For a Hitchin representation, the associated limit set and hence the support of $\mu$ lies in a $C^{1+\beta}$-circle with $\beta>0$, which they heavily rely on. \cref{thm:lyapunov} is the first general result in a higher dimensional projective space.

\subsubsection{Dimension of the projection of stationary measures}
Our strategy to show \cref{thm:lyapunov} is to combine a Ledrappier-Young formula in \cite{rapaport_exact_2021} and \cite{ledrappier_exact_2021}, and a central result of dimension of projection measures which is a projective version of that in \cite{barany_hausdorff_2017}. 

For any one-dimensional subspace $V\subset \R^3$, we consider the orthogonal projection from $\R^3-V$ to $V^\perp$, which  
induces a smooth projection (along projective lines passing through $\P(V)$) from $\P(\R^3)-\{V\}$ to $  \P(V^\perp)$. We denote the smooth projection by $\pi_{V^\perp}$. For any probability measure $\nu$ on $\SL_3(\R)$, its inverse $\nu^{-}$ is defined by $\nu^-(A):=\nu(\{g^{-1}:g\in A\})$ with a Borel set $A$. In particular, $\nu$ is Zariski dense if and only if so is $\nu^-$.

Given $\nu$ a Zariski dense, finitely supported probability measure on $\SL_3(\R)$, let 
$\mu^-$ be the stationary measure of $\nu^-$ on $\P(\R^3)$. In \cite{rapaport_exact_2021} and \cite{ledrappier_exact_2021}, they obtain 
a Ledrappier-Young formula, i.e. for $\mu^-$-a.e. $V\in\P(\R^3)$,
\begin{equation}\label{equ:ledrappieryoung}
    \dim\mu=
     \dim\pi_{V^\perp}\mu+\frac{h_{\rm F}(\mu,\nu)-\chi_1(\nu)\dim\pi_{V^\perp}\mu}{\chi_2(\nu)};
\end{equation} 
moreover, they established an upper bound: $\dim\pi_{V^\perp}\mu\leq \min\{1,h_{\rm F}(\mu,\nu)/\chi_1 \}$ (see \cref{sec:ledrappier-young} for more details).

In the following core theorem, we show that under the assumptions of Theorem \ref{thm:lyapunov}, the upper bound above is actually also a lower bound. Then Theorem \ref{thm:lyapunov} is a direct corollary of the following theorem, \eqref{equ:ledrappieryoung} and the general fact that $h_{\rm F}(\mu,\nu)\leq h_{\rm{RW}}(\nu)$.
\begin{thm}\label{thm:projection}
    Let $\nu$ be a Zariski dense, finitely supported probability measure on $\SL_3(\R)$ that satisfies the exponential separation condition, and $\mu$ be its stationary measure on $\P(\R^3)$. Then for $\mu^-$-a.e. $V\in \P(\R^3)$, we have
    \[\dim\pi_{V^\perp}\mu=\min\{1,\frac{h_{\rm{RW}}(\nu)}{\chi_1}\}.  \]
\end{thm}

\subsection{Further questions}
For the dimension jump phenomenon in Theorem \ref{thm:dimension_jump}, if $\rho\in \Hom(\Gamma,\SL_3(\R))$ is reducible, the analysis becomes more subtle, and we have several cases. If $\rho$ preserves 
a projective line in $\P(\R^3)$, 
  then $L(\rho(\Gamma))$ is this projective line. Hence $\dim L(\rho(\Gamma))=1$. 
If $\rho$ does not preserve a projective line, then $L(\rho(\Gamma))$ is not Lipschitz \cite[Theorem 4.4]{barbot_three-dimensional_2005}. And $L(\rho(\Gamma))$ is a graph-like set, and its dual representation $\rho^\ast$ 
preserves a projective line which is actually $L(\rho^\ast(\Gamma))$. 

\begin{ques}
    In the following table, are the entries with question marks true?
\end{ques}
If we can have a positive answer to this question, this table shows that a similar dimension jump phenomenon holds. Then one can distinguish different types of $\rho$ in $O$ using the Hausdorff dimension of limit sets.

\begin{table}[!ht]
    \centering 
    \caption{Hausdorff dimension of different types of representations}
    \begin{tabular}{|c|c|c|}
    \hline
     Reducibility type of $\rho\in O$ & $\dim L(\rho(\Gamma))$ & $\dim L(\rho^\ast(\Gamma))$\\
    \hline
    Semisimple 
    & 1 &  1\\
    \hline
  point-irreducible & 1 &
    $\approx 3/2 $ (?) \\
    \hline 
    line-irreducible & $\approx 3/2$ (?) & $1$\\
    \hline 
    Irreducible (in this paper) & $\approx 3/2$&$\approx 3/2$\\
    \hline
\end{tabular}

\begin{tablenotes}
\footnotesize
    \item We follow the definition in \cite{lee_anosov_2021}. A reducible representation $\rho:\Gamma\to \SL_3(\R)$ is called \textit{point-irreducible} if it does not preserve any point in $\P(\R^3)$;  \textit{line-irreducible} if it does not preserve any projective line in $\P(\R^3)$; \textit{semisimple} if it is a product of irreducible representations.
\end{tablenotes}

\end{table}

A natural further study is to generalize the results in this paper to higher dimensional Lie groups, and among them, a relatively approachable question is to generalize the dimension jump phenomenon in \cref{thm:dimension_jump}.
To some degree, this phenomenon is attributed to the fact that the Hausdorff dimension of the limit set is not equal to the affinity exponent when the representation {preserves a plane}. For higher dimensional Lie groups, we can similarly define Anosov representations (cf. \cite{canary_note}, etc.), associated limit sets, and affinity exponents. We pose the following general question.

\begin{ques}  Let $\Gamma$ be a hyperbolic group. Classify for which $\rho\in \HA(\Gamma, \SL_n(\R))$ satisfies 
   $$\dim L(\rho(\Gamma))<s_{\mathrm{A}}(\rho).$$ 
\end{ques}
This question is indeed quite delicate even when one seeks a complete answer for $\SL_4(\R)$ analogous to the above table for $\SL_3(\R)$.

The continuity of the affinity exponent $s_{\mathrm{A}}(\rho)$ is essentially due to \cite{potrie_eigenvalues_2017} and \cite{bridgeman_pressure_2015}. For completeness, we provide a proof in \cref{sec:jump}.
In the setting of Bowen's result \cite{bowen_hausdorff_1979}, it is known that classical critical exponents depend on representations analytically \cite{ruelle_repellers_1982}. It is natural to ask the following question. 
\begin{ques}Does the affinity exponent $s_{\mathrm{A}}(\rho)$ depend on $\rho\in \HA(\Gamma, \SL_3(\R))$ analytically?
\end{ques}
Bowen's dimension rigidity result \cite{bowen_hausdorff_1979} is a global result for all quasi-Fuchsian representations. It is interesting to investigate the following question.

\begin{ques}
 Do all irreducible Anosov representations in Barbot's component have an affinity exponent greater than $1$?
\end{ques}

In works like \cite{sullivan_density_1979}, Patterson-Sullivan measures are central in computing the Hausdorff dimension of limit sets.  For Anosov representations, we also have Patterson-Sullivan measures on the limit sets \cite{quint_mesures_2002}, \cite{pozzetti_anosov_nodate}. Our current computation of the Hausdorff dimension of limit sets does not use these nice measures. It is interesting to explore the following question.
\begin{ques}
 Is there a Patterson-Sullivan measure that has the same Hausdorff dimension as the limit set?
\end{ques}

\subsection*{Convention}
Throughout the paper, we use base-$q$ logarithm, where $q$ is a fixed integer large enough that will be defined in \cref{sec:partition}.
\\
\\
\textbf{Organization of the paper.}
After some preparations,
the proof of \cref{thm:projection} will occupy the most part the paper, \cref{sec:non concentration}-\cref{sec:main argument}. 
In \cref{sec:proofs}, we will give the proof of \cref{thm:hausdorff} 
from \cref{thm:lyapunov} by admitting a variational principle proved in the second paper \cite{JLPX} joint with Yuxiang Jiao.
\\
\\
\textbf{Acknowledgement.}
The authors would like to thank in particular Andres Sambarino for helping discussions about the Barbot's component, which motivates this work and the observation of dimension jump. The authors would also like to thank Fran\c{c}ois Ledrappier for his courses on the Ledrappier-Young formula.
We would like to thank Cagri Sert, Federico Rodriguez-Hertz, Pablo Lessa for helpful discussions. We would like to thank Rafael Potrie for pointing out the reference \cite{bonatti_discontinuity95}.
Part of the work was done in the conference  ``Beyond uniform hyperbolicity'' at the Banach Center in B\k{e}dlewo, Poland, in 2023. We thank the organizers and the hospitality of the center.


\subsection*{Notation}

We summarize our main notation and conventions here.

\vspace{0.5cm}

\begin{small}
	
\noindent 
\begin{longtable}{|p{4cm}p{13cm}|}
\hline 
& \tabularnewline
$\log$ & Logarithm with base $q$. 
\tabularnewline
$q$ & a fixed large integer  defined in \cref{defi: q}. 
\tabularnewline
$V_g^+$ & an attracting point in the projective space of $g\in \SL_n(\R)$ defined in Lemma \ref{lem:action g}.
\tabularnewline
$H_g^-$ & a repelling hyperplane in the projective space of $g\in \SL_n(\R)$ defined in Lemma \ref{lem:action g}.
\tabularnewline
$b(g^-,\epsilon), B(g^+,\epsilon)$ & repelling and attracting basins of $\P(\R^n)$ defined in Lemma \ref{lem:action g}.
\tabularnewline
$E_i,1\leq i\leq 3$ & projectifications of standard basis of $\R^3$.
\tabularnewline
$\Pi_{V^\perp}, \Pi_{V,W}$ & linear projections defined in Definition \ref{defn:projections}.
\tabularnewline
$\Pi(V,W,W')=\Pi_{V,W}|_{W'}$ & the restriction to $W'$ of a linear projection $\Pi_{V,W}$ in Definition \ref{defn:projections}.
\tabularnewline
$\pi_{V^\perp}, \pi_{V,W}, \pi(V,W,W')$ & projective transformations induced by $\Pi_{V^\perp}, \Pi_{V,W}, \Pi(V,W,W')$ in Definition \ref{defn:projections}.
\tabularnewline
$h_{V,g}$ & the composition map $\pi_{V^\perp}\circ g|_{V^\perp}$ where $g\in \SL_3(\R)$.
\tabularnewline
$U_V$, $L_V$ decomposition & a decomposition of $\SL_3(\R)$, where the parameter $V$ is an element in $\P(\R^3)$. 
\tabularnewline
$U$, $L$ decomposition & $U_V$, $L_V$ decomposition with $V=E_1$.
\tabularnewline $\SL_2^\pm(\R)$
&  two by two matrices with determinant equal to $1$ or $-1$.
\tabularnewline
$\pi_{L_V}$ & the projection from $U_{V}L_{V}$ to $L_V$.
\tabularnewline $b(f_\ell,r)$
& the complement of the $r$-neighborhood of a hyperplane in $\P(\R^3)$ \cref{defi:psi good region}. 
\tabularnewline
$\cal Q_n$ & $q$-adic decomposition of metric spaces as $\R$ and Lie groups.
\tabularnewline
$\bU(n), \bI(n)$ & random words defined in \cref{sec:partition}.
\tabularnewline
$g\mu$  & the pushforward of $\mu$ under the action of $g$.
\tabularnewline
$\nu*\mu$ & the convolution of $\nu$ on $\SL_n(\R)$ and $\mu$ on $\P(\R^n)$
\tabularnewline
$\theta*\tau$ & the convolution of two measures on $\R$.
\tabularnewline
$[\theta.\mu]$ & the projection $\pi_{E_1^\perp}$ of convolution of $\theta$ on $\SL_3(\R)$ and $\mu$ on $\P(\R^3)$. \cref{sec:pre convolution}
\tabularnewline
$h_{\rm{F}}(\mu,\nu)$ & the Furstenberg entropy defined in \cref{equ:f entropy}
\tabularnewline
$h_{\rm {RW}}(\nu)$ & the random walk entropy of $\nu$ defined in \cref{equ:random walk entropy}.
\tabularnewline
$\sigma_i(g)$ & Singular values of $g\in \SL_n(\R)$.
\tabularnewline
$\chi_i(\nu)$ & $\lambda_1(\nu)-\lambda_{i+1}(\nu)$, where $(\lambda_1,\dots, \lambda_n)$ are the Lyapunov exponents of $\nu$.
\tabularnewline
$\mathcal F=\mathcal F(\R^3)$ & Flag variety in $\R^3$, cf. Definition \ref{defi: flag}.
\tabularnewline
$C_L$ & constant defined in \cref{equ:C_L}.
\tabularnewline
$C_p$ & constant defined in \cref{eqn: partition size}.
\tabularnewline
$\mathcal{C}$ & subset of $\P(\R^3)$ defined by $\P(\R^3)-\P(E_1^\perp)$ .

\tabularnewline
\hline 
\end{longtable}
\end{small}

\begin{itemize}
    \item If $A$ and $B$ are two quantities, we write $A\ll B$ or $A=O(B)$ means that there exists some constant $C>0$ (possibly depending on the ambient group and the random walk $\nu$) 
     such that $A\leq CB$. 
 
    \item We write $A\ll_a B$ and $A=O_a(B)$, if the constant $C$ depending on an extra parameter $a$.

    \item  We $A\simeq B$, if $A\ll B$ and $B\ll A$. We write $A\simeq_a B$, if $A\ll_a B$ and $B\ll_a A$.
    
\end{itemize}


\section{Preliminaries}

\subsection{Action of \texorpdfstring{$\GL_n(\R)$}{GL\_n(R)} on \texorpdfstring{$\P(\R^n)$}{P(R\^{}n)}}
Consider the $n$-dimensional Euclidean space $\R^n$ and denote by $\|\cdot\|$ the Euclidean norm. By abuse of the notation, we denote by $\|\cdot\|$ the norm on $\wedge^2\R^n$ induced by the one in $\R^n$. Any element $g\in \GL_n(\R)$ acts on $\R^n$ and we denote the operator norm by $\|g\|$. Let $e_1,\cdots,e_n$ be the standard orthonormal basis of $\R^n$ and $e_1^*,\cdots, e_n^*$ be the dual basis of $(\R^n)^*$. Let $E_i=\R e_i$ be the corresponding element in $\P(\R^n)$ with $i=1,\cdots, n$ . Throughout the paper,  we always consider the following metric on the projective space $\P(\R^n)$ unless otherwise stated.
\begin{defi}\label{def: prj metr} The distance $d$ on  $\P(\R^n)$ is defined by
\begin{equation}
\label{eqn:metric}
d(\R v,\R w):=\frac{\|v\wedge w\|}{\|v\|\|w\|}\,\,\,\text{for any}\,\,\, \R v,\R w\in \P(\R^n).
\end{equation}

When $A, B$ are two hyperplanes of $\P(\R^n)$, we write $d(A,B)$ for their Hausdorff distance.
\end{defi}
\begin{rem}
\label{two metrics} The metric $d$ is bi-Lipschitz equivalent to the standard $\mathrm{SO}(n)$-invariant Riemannian metric $d_R$ on $\P(\R^n)$ (coming from the double cover of $\P(\R^n)$ by the unit sphere). More precisely, we have $d(\R v,\R w)=\sin (d_R(\R v,\R w))$ for any $\R v, \R w\in \P(\R^n)$. 
\end{rem}

We will frequently consider the following subgroups  of $\GL_n(\R)$:
\begin{align*}
K:=&\O_n(\R),\\
A:=&\{a=\operatorname{diag}(a_1,\ldots,a_n): a_i\neq 0\},\\
A^+:=&\{a=\operatorname{diag}(a_1,\ldots,a_n)\in A: a_1\geq \ldots \geq a_n\}.
\end{align*}
\begin{defi}
\label{def:Cartan decomposition}
For $g\in \GL_n(\R)$, let $\sigma_1(g)\geq \cdots \geq \sigma_n(g)$ be the singular values of $g$ and $g=\tilde{k}_ga_gk_g\in KA^+K$ be the Cartan decomposition of $g$ (singular value decomposition). Let $\epsilon>0$. Set
 \begin{itemize}
 \item $\chi_i(g):=\log (\sigma_1(g)/\sigma_{i+1}(g)), 1\leq i\leq n-1$;

  \item  $V_g^+:=\tilde{k}_gE_1 \in \P(\R^n)$, which is an attracting point of $g$.
     \item $H_g^-:=k_g^{-1}(E_2\oplus\cdots\oplus E_n)\subset \P(\R^d)$, which is a repelling hyperplane of $g$
 
 (if $\sigma_1>\sigma_2$, then $V_g^+$ and $H_g^-$ are uniquely defined);

\item $b(g^-,\epsilon):=\{x\in\P(\R^n):\ d(x,H_g^-)>\epsilon \}$ for any $\epsilon>0$;

\item $B(g^+,\epsilon):=\{x\in\P(\R^n):\ d(x,V_g^+)\leq \epsilon \}$ for any $\epsilon>0$.
 \end{itemize}
 The Cartan projection of $g$ is defined to be $\kappa(g):=(\log\sigma_1(g),\cdots ,\log\sigma_n(g))$.
\end{defi}

We now state some basic contracting properties of the action of $\GL_n(\R)$ on $\P(\R^n)$.

\begin{lem}\label{lem:action g}

    For any $g\in \GL_n(\R)$ and any $\epsilon>0$, we have
    \begin{itemize}
    \item $g$ acts on $b(g^-,\epsilon)$ by contraction:  for any $x\neq y$ in $b(g^-,\epsilon)$, we have $\frac{d(g(x),g(y))}{d(x,y)}\leq\frac{\sigma_2}{\epsilon^2\sigma_1}$;
    \item $g(b(g^-,\epsilon))\subset B(g^+,\frac{\sigma_2}{\epsilon^2\sigma_1})$;
    \item  when $g\in \SL_2(\R)$, we have $\frac{\sigma_1(g)}{\sigma_2(g)}=\|g\|^2$.
    \end{itemize}
\end{lem}

Please see \cite{benoist_random_2016} or \cite[Lemma 2.11]{li_fourier_2018} for the proof.

We state another useful lemma \cite[Lemma 14.2 ]{benoist_random_2016} for later use.
\begin{lem}\label{lem:gv d v g-}
For any $g\in \GL_n(\R)$ and $V=\R v\in \P(\R^n)$, we have
\begin{align*}
  d(V, H_g^-)\leq \frac{\|gv\|}{\|g\|\|v\|}\leq  d(V, H_g^-)+q^{-\chi_1(g)},\quad
 d(gV,V_g^+)d(V, H_g^-)\leq q^{-\chi_1(g)}.
\end{align*}
\end{lem}

\begin{defi}\label{def: scale distor}
For a bi-Lipschitz map $f$ between two metric spaces $(X,d_X)$ and $(Y,d_Y)$, we say $f$ scales by $u>0$ with distortion $C>1$ if for any $x\neq x'\in X$,
\begin{equation}
\label{equ:bounded distortion}
 \frac{1}{C}\leq \frac{d_Y(f(x),f(x'))}{ud_X(x,x')}\leq C.
 \end{equation}
\end{defi}
We fix an identification between
$\P(\R^2)$ and $\R/\Z$: 
\begin{align}
\label{eqn:identification}
\iota: \P(\R^2)&\to \R/\Z \nonumber\\
\R(\cos \theta,\sin \theta)&\mapsto \frac{\theta}{\pi} . 
\end{align}
When an element $g\in \GL_2(\R)$ acts on $\P(\R^2)$, we view $g$ as a diffeomorphism of $\R/\Z$ by conjugating it by $\iota$ and  denote by $|g' x|$ and $|g''x|$ the $1$st and $2$nd derivatives of $\iota g\iota^{-1}$ at $\iota x$. The following are estimates of $|g'x|$ and $|g'' x|$, which can be found in \cite[Page 826]{hochman_dimension_2017}. For later usage, we introduce $\SL_2^\pm(\R)$, the group of two by two matrices with determinant equal to $1$ or $-1$.

\begin{lem}\label{lem: sl2 basic}
For $g\in\SL_2^\pm(\R)$ and $x\in\P(\R^2)$, we have
\[ \|g\|^{-2}\leq |g'x|\leq \|g\|^{2},\ |g''x|\leq 4 \|g\|^{2}. \]
\end{lem}

\begin{lem}\label{lem: sl2 bounded distortion}
For all $0<\epsilon<1/3$, the following holds. Let $g$ be an element in $\SL_2^\pm(\R)$. 
Then for action of $g$ on $b(g^-,\epsilon)\subset \P(\R^2)$,  it scales by $\|g\|^{-2}$ with distortion $10\epsilon^{-2}$. Actually, we have a more precise estimate: for any $x\in b(g^-,\epsilon)$,
\begin{align*}
  -2\log\|g\|\leq \log |g'(x)|
  &\leq -2\log\|g\| -2\log(\epsilon), \\|g''(x)|&\leq \frac{10}{\epsilon^3 \|g\|^2}.
  \end{align*}
\end{lem}

\subsection{Projections in \texorpdfstring{$\R^3$}{R\^{}3} and \texorpdfstring{$\P(\R^3)$}{P(R\^{}3)}}\label{sec:projections}
\begin{defi}
\label{defn:projections}
For any line $V$ and any hyperplanes $W,W'$ of $\R^3$ such that $V\not\subset W, W'$, and any $g\in \SL_3(\R)$ such that $g^{-1}V\not\subset V^\perp$, we denote 
\begin{itemize}
\item the linear projection on  $\R^3$ with kernel $V$ and image $W$ by $\Pi_{V,W}$; 
\item the orthogonal projection with kernel $V$ by $\Pi_{V^\perp}(=\Pi_{V,V^\perp})$;
\item the linear projection from $W'$ to $W$ along $V$ (i.e. with kernel $V$) by $\Pi(V,W,W')$, i.e. $\Pi(V,W,W')=\Pi_{V,W}|_{W'}$;
\item the projective transformations associated to  $\Pi_{V,W}$, $\Pi_{V^\perp}$, $\Pi(V,W,W')$ by $\pi_{V,W}$, $\pi_{V^\perp}$, $\pi(V,W,W')$ respectively;
\item the composition map $\pi_{V^\perp}\circ g|_{V^\perp}$ by $h_{V,g}: \P(V^\perp) \to \P(V^\perp)$.
\end{itemize}
\end{defi}

One important observation is the following geometric lemma on a decomposition of the map $\pi_{V^\perp}\circ g$, which enables us to change the direction of projection and apply the ergodic theory on projections.
\begin{lem}\label{lem: dec pi V A} For any $g\in \SL_3(\mathbb R)$, $V\in \P(\R^3)$ such that $g^{-1}V\not\subset V^\perp$, we have:
\begin{align}
\label{eqn: dec pi V A 1}
\pi_{V^\perp}\circ g&{}=h_{V,g} \circ \pi_{{g^{-1}V},V^\perp},\\
\label{eqn: dec pi V A}
 \pi_{V^\perp}\circ g&{}=h_{V,g} \circ \pi({g^{-1}V},V^{\perp},(g^{-1}V)^{\perp})\circ \pi_{(g^{-1}V)^\perp}.   
\end{align}
\end{lem}
\begin{proof} 
It suffices to show the corresponding equations for linear maps in $\R^3$. We have
\begin{eqnarray*}
&&\Pi_{V^\perp}\circ g\\
&=&(\Pi_{V^\perp}\circ g)|_{V^\perp}\circ \Pi_{g^{-1}V,V^\perp} \\
&=&(\Pi_{V^\perp}\circ g)|_{V^\perp}\circ \Pi({g^{-1}V}, V^{\perp},(g^{-1}V)^\perp)\circ \Pi_{g^{-1}V, (g^{-1}V)^\perp}.
\end{eqnarray*}
The first equality holds because both $\Pi_{V^\perp}\circ g$ and  $(\Pi_{V^\perp}\circ g)|_{V^\perp}\circ \Pi_{g^{-1}V,V^\perp} $ are linear maps from $\R^3$ to $V^\perp$. They have the same kernel which is $g^{-1}V$, and their restrictions to  $V^\perp$, a complement to $g^{-1}V$, are the same.

The second equality is obtained similarly
by analyzing the kernels and images of the linear maps.
\end{proof}

\subsection{Decomposition of \texorpdfstring{$\SL_3(\R)$}{SL\_3(R)} for the composition maps \texorpdfstring{$\pi_{V^{\perp}}g$}{π\_\{V⊥\}g}}\label{sec:decomposition}
In this subsection, we introduce a decomposition of the Lie group $\SL_3(\mathbb{R})$, which describes the structure of the composition maps $\pi_{V^{\perp}}g$ with $V\in \P(\R^3)$ and $g\in \SL_3(\R)$.

We start with the decomposition for $E_1\in \P(\R^3)$. Let $U,L$ be two closed Lie subgroups of $\SL_3(\R)$ defined by 
\begin{align*}
U&{}:=\left\{\begin{pmatrix}
    \lambda^2 & x \ \  y\\
    0 & \lambda^{-1}\Id_2
\end{pmatrix}: \lambda\in \R^+,x,y\in\R\right\},\\
L&{}:=\left\{\begin{pmatrix}
    \det h & 0 \\
    n & h
\end{pmatrix}:n\in\R^2, h\in \SL_2^\pm(\R)\right\}.
\end{align*}
The group $U$ is a solvable subgroup, and the group $L$ is isomorphic to $\SL_2^\pm(\R)\ltimes \R^2$.

For a general $V\in \P(\R^3)$, we define $U_V:=k^{-1}Uk$ and $ L_V:=k^{-1}Lk$, where $k$ is any matrix in $\SO_3(\R)$ satisfying $kV=E_1$. We also define the projection
\begin{equation}\label{lem:pi L V}
\begin{split}
\pi_{L_V}:U_VL_V &\to L_V\\
u\ell&\mapsto \ell.
\end{split}
\end{equation}

The following lemma shows that 
 $U_V,L_V,\pi_{L_V}$ are well-defined and explains the roles that $U_V$ and $L_V$ play in the composition maps $\pi_{V^{\perp}}g$ with $g\in \SL_3(\R)$.
 
\begin{lem}\label{lem: U V propty}
For any $V\in \P(\R^3)$ and any $k\in \SO_3(\R)$ such that $kV=E_1$. The followings hold.
\begin{enumerate}
\item 
The groups $U_V$ and $L_V$ are well-defined, i.e. independent of the choice of $k\in \SO_3(\R)$.

\item The product $U_VL_V$ equals the set
\begin{equation*}
\{g\in \SL_3(\R): g^{-1}V\notin V^\perp\},
\end{equation*}
and it is Zariski open and Zariski dense in $\SL_3(\R)$.

\item The map $\pi_{L_V}$ is well-defined, i.e., if $u\ell=u'\ell'$ in $U_VL_V$, then $\pi_{L_V}(u\ell)=\pi_{L_V}(u'\ell')$.

\item For the group $U_V$, we have
\begin{equation}
\label{projection identity}
\{g\in \SL_3(\R): \pi_{V^\perp}g=\pi_{V^\perp}id\}=U_V\cup U_Vk^{-1}\begin{pmatrix}
    1 & 0  \\
    0 & -\Id_2
\end{pmatrix}k.
\end{equation}

\item 
For every $g\in U_VL_V$, write $g=k^{-1}u\ell k=k^{-1}u\begin{pmatrix} \det h & 0 \\
    n & h\end{pmatrix}k$ with $u\in U$, $\ell\in L$,  $n\in \R^2$ and $h\in\SL_2^\pm(\R)$. We have
    \begin{equation*}
    h_{V,g}=h
    \end{equation*}
    as elements in $\mathrm{PGL}_2(\R)$, where $h_{V,g}$ is given as in \cref{defn:projections}. 

\item Consider the map 
\begin{align*}
\Phi:L_V&\to\{\text{maps from}\,\,\P(\R^3)\,\,\text{to}\,\,\P(V^{\perp})\}\\
 \ell&\mapsto \pi_{V^{\perp}}\ell.
\end{align*}
If $\ell\neq \ell'\in L$ satisfy $\Phi(\ell)=\Phi(\ell')$, then $\ell'\ell^{-1}=k^{-1}\begin{pmatrix} 1 & 0\\ 0 & -\Id_2 \end{pmatrix} k$.

\end{enumerate}

\end{lem}

\begin{proof}
\begin{enumerate}[1.]
\item For any $V\in \P(\R^3)$, let $k_0$ be a matrix in $\SO_3(\R)$ satisfying $kV=E_1$. We have
\begin{equation*}
\{k\in \SO_3(\R):kV=E_1\}=\left\{\begin{pmatrix}1 & 0\\ 0 & k_1\end{pmatrix}k_0:k_1\in \SO_2(\R)\right\}.
\end{equation*}

Any matrix of the form $\begin{pmatrix}1 & 0\\ 0 & k_1\end{pmatrix}$ with $k_1\in \operatorname{SO}_2(\R)$ normalizes the groups $U$ and $L$. This shows that 
 $U_V$ and $L_V$ are well-defined. 

With the first statement available, it suffices to prove the remaining statements for the case $V=E_1$.

\item We prove
\begin{equation}
\label{UV and containment}
UL=\{g\in \SL_3(\R): g^{-1}E_1\notin E_1^\perp\}.
\end{equation} 
The direction $\subset$ can be checked by a straightforward computation using the definitions of $U$ and $L$. 

For the  direction $\supset$, given any $g\in \SL_3(\R)$ such that $g^{-1}E_1\notin E_1^\perp$, the entry in the first row and first column of $g^{-1}$ is non-zero. 
So the 2 by 2 submatrix $g'$ of $g$ in the lower right corner is non-degenerate. We can write $g'=\lambda^{-1}h$ with $\lambda>0$ and $h\in \SL_2^\pm(\R)$, and hence we can find $u\in U$ and $\ell\in L$ such that $g=u\ell$. 

$UL$ is Zariski open and Zariski dense because \eqref{UV and containment} shows that it is the complement of the proper Zariski closed subvariety $\{g\in \SL_3(\R):g^{-1}E_1\in E_1^\perp\}$.

\item Suppose $u\ell=u'\ell'$ in $UL$. 
Then $\ell'\ell^{-1}=u'^{-1}u\in U\cap L=\{\Id_3\}$.

\item We can check that the group $U$ is contained in the group \eqref{projection identity} by a straightforward computation using the definition of $U$. 

Given any $g\in \SL_3(\R)$ such that  $\pi_{E_1^\perp}g=\pi_{E_1^\perp}\id$, we have $gE_1=E_1$. So $g\in UL$ by \eqref{UV and containment} and we can write $g=u\ell$ with $u\in U$ and $\ell\in L$. The relation $E_1=gE_1=u\ell E_1$ gives $\ell E_1=E_1$. As  ha result, we can write $\ell=\begin{pmatrix}\det h& 0\\0&h\end{pmatrix}$ with $h\in \SL_2^\pm(\R)$. So
\[\pi_{E_1^{\perp}} \id=\pi_{E_1^\perp}g =\pi_{{E_1}^\perp}\begin{pmatrix}
	\det h & 0\\ 0 & h
\end{pmatrix}=h.\]
This implies $h=\Id_2$ or $-\Id_2$.

\item The fifth statement 
holds because
 \begin{align*}
     \pi_{E_1^\perp}g|_{E_1^\perp}=\pi_{E_1^\perp}\begin{pmatrix}
    \det h & 0 \\
    n & h
\end{pmatrix}\bigg|_{E_1^\perp}=h.
 \end{align*}

\item
Since we have
\[ \pi_{E_1^\perp}(\ell'\ell^{-1})=\pi_{E_1^\perp}(\ell')\circ\ell^{-1}=\pi_{E_1^\perp}(\ell)\circ\ell^{-1}=\pi_{E_1^\perp}id, \]
the sixth statement follows from the fourth statement. 

\end{enumerate}
\end{proof}

\cref{lem: good continious V} fixes an identification between $\P(V^\perp)$ with $\P(\R^2)$ for $V\in \mathcal{C}=\P(\R^3)-\P(E_1^\perp)$. This identification provides a unique choice of $k\in \SO_3(\R)$.

\subsection{Linearize the projection of \texorpdfstring{$L$}{L}-action}
In this part, we consider the point $E_1\in \P(\R^3)$ and the corresponding $UL$-decomposition. We fix a left $L$-invariant and right $\operatorname{SO}(2)$-invariant Riemannian metric $d$ on $L$. 
For $\ell\in L$ and $x\in\P(\R^3)$, we introduce the following notation
\[[\ell(x)]:=\pi_{E_1^\perp}\ell(x)\in \P(E_1^\perp)\cong\P(\R^2).  \]
We will always write $\ell= \begin{pmatrix}
    \det h & 0 \\
    n & h
\end{pmatrix}$ with $n\in \R^2$ and $h\in \SL_2^\pm(\R)$ as the matrix representation of $\ell$. Then 
\begin{equation}
\label{projection psi formula}
[\ell(x)]=\R(h(b,c)^t+an), 
\end{equation}
for any $x=\R(a,b,c)^t\in \P(\R^3)$.

The matrix $h$ acts on $\P(E_1^\perp)\cong\P(\R^2)$. Let $h=\tilde{k}_ha_hk_h\in KA^+K$ be the Cartan decomposition of $h$, and let $H_{h}^-\in \P(E_1^\perp)$ be its repelling point given as in \cref{def:Cartan decomposition}. 

Observing that by \cref{eqn: dec pi V A 1} and \cref{lem: U V propty}, for $\ell=\begin{pmatrix}
	\det h & 0 \\
	n & h
\end{pmatrix}$, we have 
\begin{equation}
\label{eqn:fundamental equation}
 \pi_{{E_1}^\perp}\ell = h  \circ \pi_{\ell^{-1}E_1,E_1^{\perp}}.    
 \end{equation}
(As in Figure \ref{fig:fundamental equation}, for any $x$, the projection $\pi_{\ell^{-1}E_1,E_1^\perp}x$ can
be viewed as the intersection of the projective line joint by $x$ and $\pi_{(\ell^{-1}E_1)^\perp}x$ with the projective line in red $\P(E_1^\perp)$, and we decompose $\pi_{E_1^\perp}\ell$ as the composition of $h$ with the projection $\pi_{\ell^{-1}E_1,E_1^\perp}$.)
\begin{figure}[!ht]
\begin{minipage}{0.48\textwidth}
   \centering
    \includegraphics{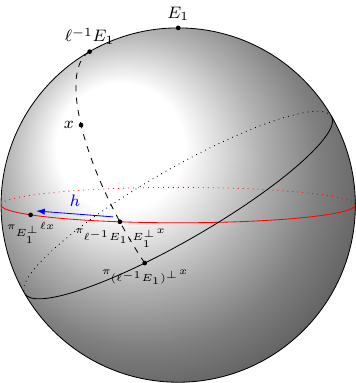}
    \caption{$\pi_{{E_1}^\perp}\ell = h  \circ \pi_{\ell^{-1}E_1,E_1^{\perp}}$}
    \label{fig:fundamental equation}
\end{minipage}\hfill
\begin{minipage}{0.48\textwidth}
    \centering
    \includegraphics{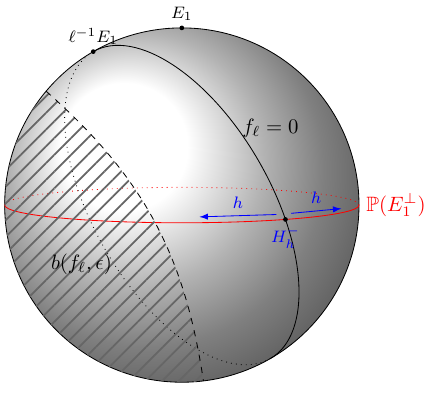}
\caption{The $[\ell]$-attracting region}
    \label{fig:attracting region}
\end{minipage}
\end{figure}
The contracting region of $\pi_{E_1^{\perp}}\ell$ should be away from the projective plane spanned by $\ell^{-1}{E_1}$ and $H_{h}^-$. At the same time, note that we have the formula for the linear map 
\begin{align*}
\Pi_{\ell^{-1}E_1,E_1^{\perp}}:\R^3&\to \R^2\\
(a,b,c)^t&\mapsto (b,c)^t+ah^{-1}n. 
\end{align*}
This motivates the following construction.

Let $f_h$ be a linear form on $E_1^\perp
\cong\R^2$ with kernel $H_{h}^-$ of norm $1$. 
For any non-zero vector $v\in\R^2$, we have
\begin{equation}
\label{fh formula}
 |f_h(v)|=\|v\|d(\R v,H_{h}^-). 
\end{equation}

\begin{defi}\label{defi:psi good region}
    Let $\ell=\begin{pmatrix}
    \det h & 0 \\
    n & h
\end{pmatrix}\in L$. 
We define the linear form $f_\ell$ on $\R^3$ by 
\begin{equation*}
    f_{\ell}(v):=
    f_h(\Pi_{\ell^{-1}E_1,E_1^{\perp}}(v)) \,\,\,\text{for}\,\,\,v\in \R^3.
    \end{equation*}
    The kernel of $f_\ell$ is the plane spanned by $\ell^{-1}{E_1}=\R (1,-h^{-1}n)^t$ \footnote{We abuse the notation here and $(1,-h^{-1}n)$ is actually the row vector $(1,-(h^{-1}n)^t)$.} and $H_{h}^-\subset E_1^\perp\subset \R^3$.
    
For $\epsilon>0$, we define an $[\ell]$-attracting region 
by
   \begin{equation}
   \label{psi attracting region}
    b(f_\ell,\epsilon):=\left\{
    x=\R v\in\P(\R^3): |f_{\ell}(v)|\geq \epsilon\|f_\ell\|\|v\|
    \right\}.
    \end{equation}
    Figure \ref{fig:attracting region} shows the region $b(f_l,\epsilon)$, here the arrows in blue means $h$ is repelling near $H_h^-\in \P(E_1^\perp)$. 
\end{defi}
We summarize the structure of the rest of this subsection: in \cref{lem:bad psi region}, we prove the contracting properties of $\pi_{E_1^\perp}\ell(\cdot)=[\ell(\cdot) ]$ on the region $b(f_\ell,\epsilon)$; \cref{lem:distance psi psi'}, \cref{lem:bad psi region continuity} and \cref{lem:psi near psi} collect the continuity properties of $\ell$, $f_\ell$ and $b(f_\ell,\epsilon)$; we also need an estimate of the diameter of a set acted by the projection of $L$-action (\cref{lem:diameter action}); \cref{lem:linearization inequality} is the linearization of the map $\pi_{E_1^\perp}\ell$.

We start with some basic estimates, which will be frequently used. Note that for any $\ell\in L$,
\begin{equation}\label{equ:h-1n}
d(\ell^{-1}E_1,E_1^\perp)= d(\R(1,-h^{-1}n)^t,E_1^\perp)= 1/\|(1,-h^{-1}n)^t \|.
\end{equation}
Hence
\begin{equation}
\label{coordinate and distance psi}
(1+\|h^{-1}n\|^2)^{1/2}= 1/d(\ell^{-1}E_1,E_1^\perp).
\end{equation}
For $x=\R v$ with $v=(a,b,c)^t$ and $\|v\|=1$, we have
\begin{itemize}
\item 
\begin{equation}
\label{eqn: projection l-1E1}
\|\Pi_{\ell^{-1}E_1,E_1^{\perp}}(v)\|\leq 1/d(\ell^{-1}E_1,E_1^{\perp});
\end{equation}

\item 
\begin{equation}
\label{norm of linear form}
\|f_{\ell}\|\in [1,1/d(\ell^{-1}E_1,E_1^{\perp})];
\end{equation}

    \item if $x\in b(f_{\ell},\epsilon)$, then \begin{align}\label{equ:abc norm lower bound}
\|\Pi_{\ell^{-1}E_1,E_1^{\perp}}(v) \|\geq \epsilon\|v\|,\\
\label{equ: x psi-1 {E_1}}
    d(x,\ell^{-1}{E_1})\geq \epsilon d(\ell^{-1}{E_1},{E_1}^\perp).
\end{align}
\cref{eqn: projection l-1E1} holds because 
\begin{equation*}
\|\Pi_{\ell^{-1}E_1,E_1^{\perp}}(v)\|=\|(b,c)^t+ah^{-1}n\|\leq \|(1,-h^{-1}n)^t\|.
\end{equation*}

\cref{equ:abc norm lower bound} holds because
\[ \|\Pi_{\ell^{-1}E_1,E_1^{\perp}}(v)\|\geq |f_h(\Pi_{\ell^{-1}E_1,E_1^{\perp}}(v))|\geq \epsilon\|f_\ell\|\|v\|\geq \epsilon\|v\|. \]

\cref{norm of linear form} is obtained as follows. For any $v\in E_1^\perp$, $f_{\ell}(v)=f_h(v)$. As the norm of $f_h$ is $1$, $\|f_{\ell}\|\geq 1$. Meanwhile, by definition, we have $\|f_{\ell}\|\leq \|f_h\|\cdot \|\Pi_{\ell^{-1}E_1,E_1^{\perp}}\|$. So we use \cref{eqn: projection l-1E1} to get the upper bound.

We obtain \cref{equ: x psi-1 {E_1}} as follows: 
\begin{align*}
d(x,\ell^{-1}{E_1})&=\frac{\|(a,b,c)^t\wedge (1,-h^{-1}n)^t\|}{\|(1,-h^{-1}n)^t \|}\\
&\geq d(\ell^{-1}{E_1},E_1^\perp)\|(b,c)^t+ah^{-1}n \|=d(\ell^{-1}{E_1},E_1^\perp)\|\Pi_{\ell^{-1}E_1,E_1^{\perp}}(v) \|,
\end{align*}
where for the inequality we project the vector $(a,b,c)^t\wedge (1,-h^{-1}n)^t$ to the subspace generated by $E_1\wedge E_2, E_1\wedge E_3$ in $\wedge^2\R^3$ and due to \cref{equ:h-1n}.

\end{itemize}

\begin{lem}\label{lem:bad psi region}
    Let $C_1,C_2>2 $. Let $\ell=\begin{pmatrix}
    \det h & 0 \\
    n & h
\end{pmatrix}\in L$ be such that
   $d(\ell^{-1}{E_1},E_1^\perp)>1/C_1$. 
   Then for $x=\R v$ with $\|v\|=1$ in the $[\ell]$-attracting region $ b(f_\ell,1/C_2)$,
   we have
  \begin{align}
        &\|\Pi_{E_1^\perp}\ell( v)\|=\|h\circ\Pi_{\ell^{-1}E_1,E_1^{\perp}}(v) \|\geq \|h\|/C_2, \label{hv formula}\\
       &d(\pi_{\ell^{-1}E_1,E_1^{\perp}}(x),H_{h}^-)\geq 1/C_1C_2.
        \label{away from repelling region}
       \end{align}
Moreover, for the action of $\pi_{E_1^{\perp}}\ell$ on $b(f_{\ell},1/C_2)$, we have
   \begin{equation}
   \label{projection psi contraction rate}
   \frac{d([\ell(x)],[\ell(x')])}{d(x,x')}\leq 2C_1C_2^2/\|h\|^2
   \end{equation}
    for any two distinct points $x, x'$ in $b(f_\ell,1/C_2)$, and the inclusion
   \begin{equation}
   \label{psi inclusion}
   [\ell( b(f_\ell,1/C_2))]\subset B(h^+, C_1C_2/\|h\|^2). 
   \end{equation}
   \end{lem}

\begin{proof}



    We write $w=\Pi_{\ell^{-1}E_1,E_1^{\perp}}(v)$. Hence $\R w=\pi_{\ell^{-1}E_1,E_1^{\perp}}(x)$.
    
    \cref{hv formula}  holds because 
    \begin{equation*}
    \|h w\|\geq \|h\| \|w\| d(\R w, H_h^-)=
    \|h\| |f_h(w) |=\|h\| |f_\ell(v) |\geq \|h\|/C_2.
    \end{equation*}
    Here the first inequality follows from \cref{lem:gv d v g-}; the second one is by \cref{fh formula}; the last inequality is due to $\|f_{\ell}\|\geq 1$ (\cref{norm of linear form}) and the definition of $b(f_{\ell},1/C_2)$.
    
     As $\|v\|=1$,  by \cref{eqn: projection l-1E1}, we have
     \begin{equation}
     \label{h-1n norm}
     \|w\|=\|\Pi_{\ell^{-1}E_1,E_1^{\perp}}(v)\|\leq 1/d(\ell^{-1}E_1,E_1^\perp)< C_1.
     \end{equation}
     \cref{away from repelling region} holds because
    \begin{equation*}
    d(\R w,H_{h}^-)= |f_h(w)|/\|w\|\geq 1/C_1C_2.
    \end{equation*}

    For \cref{psi inclusion}, note that \cref{eqn:fundamental equation} gives
    \begin{equation*}
    [\ell(x)]=h(\R w).
    \end{equation*}
    With \cref{away from repelling region} available, we apply \cref{lem:action g} to the action of $h$ and obtain the inclusion relation.


    It remains to prove \cref{projection psi contraction rate}.  Take any point $x'\neq x$ in $b(f_\ell,1/C_2)$. Write $x'=\R(a',b',c')^t$ with $\|(a',b',c')^t\|=1$. We have
    \begin{align}
   d([\ell x],[\ell x'])=&\frac{\|(h(b,c)^t+an)\wedge(h(b',c')^t+a'n) \|}{ \|h(b,c)^t+an\|\|h(b',c')^t+a'n\|} \nonumber\\
     \label{projection psi contraction}
     \leq & \left(\frac{C_2}{\|h\|}\right)^2 \|(h(b,c)^t+an)\wedge(h(b',c')^t+a'n) \|,
    \end{align}
    where we use \cref{hv formula} to obtain the inequality.
    
    Note that
    \begin{equation}
    \label{eqn:distance x x'}
    d(x,x')=\|(a,b,c)^t\wedge(a',b',c')^t \|\geq \max\{\|(b,c)^t\wedge(b',c')^t\|,\| a(b',c')^t-a'(b,c)^t\|\}.
    \end{equation}
    Then the numerator of \cref{projection psi contraction} is equal to
    \begin{align}
     &\|h(b,c)^t\wedge h(b',c')^t+n\wedge h(a(b',c')^t-a'(b,c)^t)\|\nonumber\\
     =&\|(b,c)^t\wedge(b',c')^t+h^{-1}n\wedge (a(b',c')^t-a'(b,c)^t)\|\nonumber\\
     \label{use d x x as an upper bound}
      \leq &\|(b,c)^t\wedge(b',c')^t\|+\|h^{-1}n\|\cdot\| a(b',c')^t-a'(b,c)^t\|
    \leq 2C_1d(x,x'), 
    \end{align}
    where the second line is due to $h\in \SL_2^\pm(\R)$ and the third line is due to $\|h^{-1}n\|\leq 1/d(\ell^{-1}E_1,E_1^\perp)\leq C_1$ ( see \cref{coordinate and distance psi}).
    This verifies \cref{projection psi contraction rate}.
\end{proof}

\begin{lem}\label{lem:distance psi psi'}
There exist $\epsilon_1 >0$ and $C_3>1$ such that for $\ell=\begin{pmatrix}
  \det h&0\\
  n&h
\end{pmatrix}\,\,\,\text{and}\,\,\, \ell'=\begin{pmatrix}
  \det h'&0\\
  n'&h'
\end{pmatrix}$ in $L$ with $d(\ell,\ell')\leq \epsilon_1$, we have
\begin{align}
    & \|id-h^{-1}h'\|\leq C_3d(\ell,\ell'),\,\,\|h^{-1}n'-h^{-1}n \|\leq C_3 d(\ell,\ell'),  \nonumber\\
    \label{equ:h-1nh'-1n}
     &\|h^{-1}n-h'^{-1}n' \|\leq C_3 d(\ell,\ell')/d(\ell^{-1}{E_1},E_1^\perp),\\
     &\left|\log\|h'\|-\log\|h\| \right|\leq C_3d(\ell,\ell'),\nonumber\\
    &d(\ell^{-1}{E_1},(\ell')^{-1}{E_1})\leq C_3 d(\ell,\ell')/d(\ell^{-1}{E_1},E_1^\perp).\nonumber
    \end{align}
    
    Moreover, if $\ell,\ell'$ satisfy $d(\ell',\ell)\leq \min\{\epsilon_1, 1/C_3 \}$, then 
    \begin{align*}
    d((\ell')^{-1}E_1,E_1^\perp)\geq d(\ell^{-1}E_1,E_1^\perp)/2.
    \end{align*}
\end{lem}

\begin{proof}
    Since $\ell$ is near $\ell'$, we have $\det h=\det h'$.
    For any $\ell,\ell'\in L$, as $d$ is a left $L$-invariant metric on $L$, we have
\[ d(\ell,\ell')=d(id,\ell^{-1}\ell')=d\left(id,\begin{pmatrix}
    1 & 0 \\
    h^{-1}n'-h^{-1}n & h^{-1}h'
\end{pmatrix}\right). \]
 The metric $d$ and the distance induced by the norm are locally bi-Lipschitz. Hence, we obtain the two inequalities of the first line. 

We also have
\begin{align*}
 &\|h^{-1}n-h'^{-1}n'\|=\|h^{-1}n-h^{-1}n'+h^{-1}n'-h'^{-1}n' \|\\
 \ll & d(\ell,\ell')+\|(h^{-1}h'-id)h'^{-1}n' \|\ll d(\ell,\ell')(1+\|h'^{-1}n' \|).     
\end{align*}
Using \cref{coordinate and distance psi}, we obtain \cref{equ:h-1nh'-1n}.

    For $\log\|h\|$, we may suppose $\|h\|\geq \|h'\|$. Then we have
    \[\log\|h\|-\log\|h'\|=\log\left(1+\frac{\|h\|-\|h'\|}{\|h'\|}\right)\leq \frac{\|h\|-\|h'\|}{\|h'\|}\leq \frac{\|h-h'\|}{\|h'\|}\leq {\|h'^{-1}h-id\|}. \]

For $d(\ell^{-1}{E_1},\ell'^{-1}{E_1})$, by definition and \cref{equ:h-1nh'-1n}, we have
\begin{align*}
    &d(\ell^{-1}{E_1},\ell'^{-1}{E_1})= \frac{\|(1,-h^{-1}n)^t\wedge (1,-h'^{-1}n')^t\|}{\|(1,-h^{-1}n)^t\|\|(1,-h'^{-1}n')^t\|}
\end{align*}
The numerator is bounded above by
\begin{align*}
&\|h^{-1}n-h'^{-1}n'\|+\|h^{-1}n\wedge h'^{-1}n' \|
\\
=&\|h^{-1}n-h'^{-1}n'\|+\| h^{-1}n\wedge(h^{-1}n-h'^{-1}n')\|
\leq (1+\|h^{-1}n\|)\|h^{-1}n-h'^{-1}n'\|,
\end{align*}
which yields the estimate of $d(\ell^{-1}E_1,\ell'^{-1}E_1)$.

For $d((\ell')^{-1}E_1,E_1^\perp)$, 
using \cref{coordinate and distance psi} and \cref{equ:h-1nh'-1n}, we have 
\begin{align*}
     1/d((\ell')^{-1}E_1,E_1^\perp)&=\|(1,-h'^{-1}n')\|\leq \|(1,-h^{-1}n) \|+\| h'^{-1}n'-h^{-1}n\|\\
     &\leq 1/d(\ell^{-1}E_1,E_1^\perp) +C_3d(\ell,\ell')/d(\ell^{-1}E_1,E_1^\perp)\leq 2/d(\ell^{-1}E_1,E_1^\perp).
    \end{align*}
\end{proof}

We need a lemma about the continuity of Cartan decomposition.
For any $\epsilon>0$, let $\calO_\epsilon$ be the neighborhood of the identity in $\SL_2(\R)$ with radius $\epsilon$. Recall that $A$ is the diagonal subgroup and $A^+=\{\diag(a_1,a_1^{-1}):a_1\geq 1\}$. 
Let $\widetilde{A}^\delta=\{\diag(a_1,a_1^{-1}):a_1>1+\delta \}$. 
\begin{lem}[Effective Cartan decomposition, 
\cite{gorodnikStrongWavefrontLemma2010}, Theorem 1.6]\label{lem_effective_cartan}
Given $\delta>0$, there exist $l_0,\epsilon_2>0$ such that the following holds. For any $\epsilon<\epsilon_2$ and any $g=k_1 ak_2\in K \widetilde{A}^\delta K$, we have
\[\calO_\epsilon g\calO_\epsilon\subset (\calO_{l_0\epsilon}\cap K)k_1(\calO_{l_0\epsilon}\cap A^+)ak_2(\calO_{l_0\epsilon}\cap K). \]
\end{lem}
Since we will mostly deal with $h\in \SL_2^\pm(\R)$ with $\|h\|$ large, we fix $\delta=\log 2$ and the corresponding constants $l_0,\epsilon_2$ in the last lemma.

\begin{lem}\label{lem:bad psi region continuity}
    There exists $C_4>1$ such that the following holds. Let $\ell=\begin{pmatrix}\det h & 0\\ n& h\end{pmatrix}\in L$ such that $\|h\|>2$. The linear form $f_\ell$ is continuous with respect to $\ell$: if $d(\ell,\ell')\leq \min\{\epsilon_1,\epsilon_2/C_3 \}$, then
    \[ \|f_\ell-f_{\ell'}\|\leq C_4 d(\ell,\ell')/ d(\ell^{-1}{E_1},E_1^\perp). \]
\end{lem}

\begin{proof}
  Write $\ell'=\begin{pmatrix}\det h' & 0\\ n' & h'\end{pmatrix}\in L$.  When $d(\ell,\ell')\leq \epsilon_1$, \cref{lem:distance psi psi'} gives that
  \begin{equation*}
  \|\id-h^{-1}h'\|< C_3 d(\ell,\ell')\leq \epsilon_2.
  \end{equation*}
  Let $h=\tilde k_ha_hk_h$ and $h'=\tilde k_{h'}a_{h'}k_{h'}$ be the Cartan decomposition of $h$ and $h'$ respectively. Then 
\cref{lem_effective_cartan} gives
\begin{equation*}
d(k_h,k_{h'})\leq l_0 d(h,h').
\end{equation*}
Combining these two inequalities, we obtain
\begin{equation}
\label{difference fh fh'}
\|f_h-f_{h'} \|=\|e_1^*\circ k_h-e_1^*\circ k_{h'}\| \leq l_0d(h,h')\ll d(\ell,\ell').
\end{equation}
For any $x=\R (a,b,c)^t$ with $\|(a,b,c)^t\|=1$, by
\cref{defi:psi good region}, we obtain
\begin{align*}
&|f_\ell((a,b,c)^t)-f_{\ell'}((a,b,c)^t)|=|f_h((b,c)^t+ah^{-1}n)-f_{h'}((b,c)^t+a(h')^{-1}n')|\\
\leq & \|f_h-f_{h'}\|\cdot \|(b,c)^t+ah^{-1}n\|+|f_{h'}(ah^{-1}n-a(h')^{-1}n')|.
\end{align*}
To get the stated inequality, we use \cref{difference fh fh'} and \cref{eqn: projection l-1E1} to estimate the first term, and \cref{equ:h-1nh'-1n} for the second term. 
\end{proof}

We fix a large constant $C_L>0$ such that 
\begin{equation}\label{equ:C_L}
C_L>\max\{1/\epsilon_1, C_3, C_3/\epsilon_2, 4C_4\}
\end{equation}
where the constants $\epsilon_1,\epsilon_2,C_3$ and $C_4$ are given as in \cref{lem:distance psi psi'}, \cref{lem_effective_cartan} and \cref{lem:bad psi region continuity}.

\begin{lem}\label{lem:psi near psi}
     Fix any  $C_1,C_2>2$. Let $\ell=\begin{pmatrix}\det h & 0\\ n& h\end{pmatrix}\in L$ be such that $\|h\|>2$. 
   If $d(\ell^{-1}{E_1},{E_1}^\perp)>1/C_1$, then for any $\ell'=\begin{pmatrix}\det h' & 0\\ n' & h'\end{pmatrix} \in B(\ell,1/C_LC_1C_2)\subset L$, we have 
   \begin{align*}
   \frac{1}{2}\|h\|\leq \|h'\|&\leq 2\|h\|, \\
     b(f_\ell,1/C_2)&\subset b(f_{\ell'},1/2C_2). \\
     d((\ell')^{-1}E_1,E_1^\perp)&>1/(2C_1) 
    \end{align*}
\end{lem}
\begin{proof}
    The first line follows from \cref{lem:distance psi psi'}.
    
    For the second line, due to \cref{lem:bad psi region continuity} and $\|f_{\ell}\|\geq 1$ (\cref{norm of linear form}), we have $$\left|\|f_\ell\|-\|f_{\ell'}\|\right|\leq 1/4C_2\leq \|f_{\ell}\|/4C_2.$$ Hence $\|f_\ell\|\geq \|f_{\ell'}\|/(1+1/4C_2)$. Again, due to \cref{lem:bad psi region continuity}, for any $\R v\in b(f_\ell,1/C_2)$ with $\|v\|=1$, we have 
    \[|f_{\ell'}(v)|\geq |f_\ell (v)|-1/4C_2\geq \|f_\ell\|/C_2-1/4C_2\geq \|f_{\ell'}\|/(C_2+1/4)-\|f_{\ell'}\|/4C_2\geq \|f_{\ell'}\|/2C_2. \]
     
     The third line follows from \cref{lem:distance psi psi'}.
\end{proof}

For sets $E\subset L$ and $F\subset \P(\R^3)$, let $[E.F]=\{[\ell(x)],\ \ell\in E,\ x\in F \}$ be a subset in $\P(\R^2)$.
\begin{lem}\label{lem:diameter action}
    Fix any $C_1,C_2>2$. 
    Let $E$ be a subset in $L$ such that any $\ell=\begin{pmatrix}\det h & 0\\ n&h\end{pmatrix}\in E$ satifies $\|h\|>2$ and $\diam E<1/C_LC_1C_2$. Suppose there exists $\ell\in E$ such that $d(\ell^{-1}{E_1},E_1^\perp)>1/C_1$. Let $F$ be a subset in $[\ell]$-attracting region $b(f_\ell,1/C_2)\subset \P(\R^3)$. Then
    \[ \diam([E.F])\leq 16C_LC_1C_2^2(\diam E+\diam F)\|h\|^{-2}. \]
\end{lem}

\begin{proof}
For any $\ell\in E$ and $x,x'\in F$, \cref{lem:bad psi region} gives $d([\ell x], [\ell x'])\leq 2C_1C_2^2 d(x,x')\|h\|^{-2}$.
    
    Let $\ell,\ell'\in E$ and $x\in F$. Write $\ell'=\begin{pmatrix}\det h' & 0\\ n' & h'\end{pmatrix}$ and $x=\R(a,b,c)^t$ with $\|(a,b,c)^t\|=1$. We have
    \begin{equation}
    \label{psi x psi x}
    d([\ell x],[\ell'x])=\frac{\|(h(b,c)^t+an)\wedge(h'(b,c)^t+an') \|}{\|h(b,c)^t+an\|\|h'(b,c)^t+an'\|}. 
    \end{equation}
    Due to $x\in b(f_{\ell},1/C_2)$, we obtain from \cref{lem:bad psi region} that 
    \[  \|h(b,c)^t+an\|\geq\|h\|/C_2.  \]
   \cref{lem:psi near psi} gives that $x\in b(f_{\ell'},1/2C_2)$ and
     \[  \|h'(b,c)^t+an'\|\geq\|h'\|/2C_2\geq \|h\|/4C_2. \]
    For the numerator of \cref{psi x psi x}, we have
   \begin{align*}
       &\|(h(b,c)^t+an)\wedge(h'(b,c)^t+an') \|\\
       =&\|(b,c)^t\wedge h^{-1}h'(b,c)^t+a^2h^{-1}n\wedge h^{-1}n'+a h^{-1}n\wedge h^{-1}h'(b,c)^t-a h^{-1}n'\wedge (b,c)^t  \|\\
       \leq&\|(b,c)^t\wedge h^{-1}h'(b,c)^t\|+\|a^2h^{-1}n\wedge h^{-1}n'\|+\|a h^{-1}n\wedge h^{-1}h'(b,c)^t-a h^{-1}n'\wedge (b,c)^t  \|.
   \end{align*}
   We use \cref{lem:distance psi psi'} to estimate the first two terms and get an upper bound $2C_Ld(\ell,\ell')$. For the last term, we also use \cref{lem:distance psi psi'}
   \begin{align*}
       &\|ah^{-1}n\wedge h^{-1}h'(b,c)^t-ah^{-1}n\wedge(b,c)^t \|\\
       &\leq \|h^{-1}n\wedge(h^{-1}h'-id)(b,c)^t+(h^{-1}n-h^{-1}n')\wedge(b,c)^t \|\\
       &\leq C_1C_Ld(\ell',\ell)+C_Ld(\ell',\ell).
   \end{align*}
   Collecting all the terms, we obtain the lemma.
\end{proof}

Let $\ell=\begin{pmatrix}\det h & 0 \\ n& h\end{pmatrix}\in L$. We are interested in the map
\begin{equation}\label{equ:pi e1 psi}
[\ell(\cdot)]=\pi_{E_1^{\perp}}\ell(\cdot)=h\circ \pi(\ell^{-1}E_1,E_1^{\perp},(\ell^{-1}E_1)^{\perp})\circ \pi_{(\ell^{-1}E_1)^{\perp}}(\cdot).
\end{equation}
Now we are ready to explain how to linearize the map $[\ell(\cdot)]$ with $\ell\in L$. We will approximate it by a composition of translation and scaling map.

\begin{defi}\label{defi:st}

For any $t\in \R$, we define the scaling map $S_t:\R\to \R$ by $S_tx=q^tx$. 
\end{defi}

We have the identifications $\P(V^\perp)\cong \P(\R^2)$ for $V\in \P(\R^3)$ in \cref{lem: good continious V} and $\iota:\P(\R^2)\cong \R/\Z$ (\cref{eqn:identification}). 
In the followings, for any $\star,\bullet\in \P(V^\perp)$, $d(\star,\bullet)$ means to view $\star,\bullet$ as elements in $\P(\R^2)$ and $d$ is the distance given as in \cref{def: prj metr}; $\star-\bullet$ means to view $\star,\bullet$ as elements in $\R/\Z$ and $
|\star-\bullet|$ is to compute their distance using the rotational invariant distance on $\R/\Z$. We have $d(\star,\bullet)=\sin (\pi|\star-\bullet|)$.

\begin{lem}\label{lem:linearization inequality}
     Fix any $C_1,C_2>20$.
Let $\ell_0=\begin{pmatrix} \det h_0 & 0\\ n_0 & h_0\end{pmatrix}\in L$ be such that $\|h_0\|>2$ and $d(\ell_0^{-1}E_1,E_1^{\perp})>1/C_1$. Then for any $\ell\in B(\ell_0,1/C_LC_1^2C_2^2)\subset L$, and $x,x_0$ in the $[\ell_0]$-attracting region $b(f_{\ell_0},1/C_2)$ with $d(x,x_0)\leq (1/C_1C_2)^{10}$, there exists a scaling map $S_{-t(\ell_0,x_0)}:\R\to \R$ with 
\begin{equation}
\label{eqn:estimate of scaling map}
\big|t(\ell_0,x_0)-2\log\|h_0\|\big|\leq 4\log(C_1C_2)
\end{equation}
such that
\begin{align}
\label{eqn: linear approximation}
    &\big|[\ell(x)]-[\ell(x_0)]-S_{-t(\ell_0,x_0)}\left(\pi_{(\ell_0^{-1}{E_1})^\perp}x- \pi_{(\ell_0^{-1}{E_1})^\perp}x_0\right)\big|\\
    \leq & C_L(C_2C_1)^9(d(\ell,\ell_0)+d(x,x_0))d(x,x_0)/\|h_0\|^2.\nonumber
   \end{align} 
\end{lem}

\begin{rem}\label{rem:con distortion}

Let us explain the notions in \cref{eqn: linear approximation}. 
Let $\iota'$ be the identification between $\R/\Z$ and $[-\frac{1}{2},\frac{1}{2})$. Let $\pi_{\R/\Z}:\R\to \R/\Z$ be the quotient map. It will be clear from the proof that $\left|\pi_{(\ell_0^{-1}{E_1})^\perp}x- \pi_{(\ell_0^{-1}{E_1})^\perp}x_0\right|<\frac{1}{2}$.  
By abusing the notation, we denote $\pi_{\R/\Z}\circ S_{-t(\ell_0,x_0)}\circ \iota'\left(\pi_{(\ell_0^{-1}E_1)^\perp} x-\pi_{(\ell_0^{-1}E_1)^\perp}x_0\right)$
by $S_{-t(\ell_0,x_0)}\left(\pi_{(\ell_0^{-1}E_1)^\perp} x-\pi_{(\ell_0^{-1}E_1)^\perp}x_0\right)$. \cref{eqn: linear approximation} is to subtract elements in $\R/\Z$.
\end{rem}
\begin{proof}

\textbf{Step 1:} Approximate $[\ell_0(x)]-[\ell_0(x_0)]$. 
We have
\begin{align*}
[\ell_0(x)]-[\ell_0(x_0)]=&h_0\circ \pi (\ell_0^{-1}E_1,E_1^{\perp},(\ell_0^{-1}E_1)^{\perp})\circ\pi_{(\ell_0^{-1}{E_1})^\perp}x\\
&-h_0\circ \pi (\ell_0^{-1}E_1,E_1^{\perp},(\ell_0^{-1}E_1)^{\perp})\circ \pi_{(\ell_0^{-1}{E_1})^\perp}x_0.
\end{align*}
We will use a basic inequality in analysis to approximate the difference gradually. For a $C^2$-map $g$ on $\R$ and two points $z<y$ in $\R$, we have
\begin{equation}\label{equ:newton-lebnize}
 |gz-gy-g'y(z-y)|\leq |z-y|^2\sup\{|g''w|: z\leq w\leq y\}. 
\end{equation}

We want to apply \cref{equ:newton-lebnize} to
\begin{align*}
g:=&h_0,\\
z:=&\pi (\ell_0^{-1}E_1,E_1^{\perp},(\ell_0^{-1}E_1)^{\perp})\circ\pi_{(\ell_0^{-1}{E_1})^\perp}(x)=\pi_{\ell_0^{-1}E_1,E_1^\perp}(x),\\
y:=&\pi (\ell_0^{-1}E_1,E_1^{\perp},(\ell_0^{-1}E_1)^{\perp})\circ \pi_{(\ell_0^{-1}{E_1})^\perp}(x_0)=\pi_{\ell_0^{-1}E_1,E_1^\perp}(x_0). 
\end{align*}
Due to $x,x_0\in b(f_{\ell_0},1/C_2)$,  \cref{away from repelling region} in \cref{lem:bad psi region} implies that
\begin{equation}
\label{y x in the attracting region}
y,z\in b(h_0^-,1/C_1C_2).
\end{equation}
Due to \cref{lem: sl2 bounded distortion}, we know
\begin{equation}\label{equ: derivative h0}
 |\log h_0'(y)+2\log\|h_0\||\leq 2\log(C_1C_2).
\end{equation}
Using \cref{lem: sl2 basic} and \cref{projection psi contraction rate}, we obtain
\begin{equation}\label{d z y}
\begin{split}
    d(z,y)&=\frac{d(h_0^{-1}[\ell_0(x_0)], h_0^{-1}[\ell_0(x)]) }{d([\ell_0(x_0)], [\ell_0(x)])}\frac{d([\ell_0(x_0)], [\ell_0(x)])}{d(x_0, x)}d(x,x_0)\\
    &\leq \|h_0\|^2\frac{2C_1C_2^2}{\|h_0\|^2}d(x,x_0)=2C_1C_2^2d(x,x_0).
\end{split}
\end{equation}
We have $[z,y]\subset b(h_0^-,1/C_1C_2)$ due to \cref{y x in the attracting region}, \cref{d z y} and the hypothesis that $d(x,x_0)\leq 1/(C_1C_2)^{10}$.
Hence, by locally identifying $\P(\ell_0^{-1}E_1)$ with $\R$, we can use \cref{equ:newton-lebnize} and the estimate for $h_0''$ (\cref{lem: sl2 bounded distortion}) to obtain
\begin{align}\label{equ:psi0xzy}
&\big|[\ell_0(x)]-[\ell_0(x_0)]-h_0'(y)(z-y)\big|\leq |z-y|^2\sup\{|h_0''w|:z\leq w\leq y\}\nonumber\\
\leq& |z-y|^2 10(C_1C_2)^3/\|h_0\|^2
 \leq 40C_1^5C_2^7d(x,x_0)^2/\|h_0\|^2. 
\end{align}

Then we let
\begin{align*}
\tilde g:=&\pi (\ell_0^{-1}E_1,E_1^{\perp},(\ell_0^{-1}E_1)^{\perp}),\\
z':=&\pi_{(\ell_0^{-1}{E_1})^\perp}x,\\
y':=&\pi_{(\ell_0^{-1}{E_1})^\perp}x_0.
\end{align*}
We use \cref{equ:newton-lebnize} to approximate $\tilde g(z')-\tilde g(y')$.
Due to \cref{lem:g-1 V lip} and
\cref{lem: sl2 basic}, we have for any $w\in \P(E_1^\perp)$, 
\begin{equation}\label{equ: g' g''}
|\tilde g''(w)|\leq 4/d(\ell_0^{-1}E_1,E_1^\perp)\leq 4C_1,\ 1/C_1\leq |\tilde g'(w)|\leq C_1.
\end{equation}
We apply \cref{equ: x psi-1 {E_1}} to $x,x_0$ and then \cref{lem:projection}(2) yields
\begin{equation}\label{equ:pi ell x}
 d(z',y')=d(\pi_{(\ell_0^{-1}{E_1})^\perp}x ,\pi_{(\ell_0^{-1}{E_1})^\perp}x_0)\leq C_1C_2d(x,x_0). 
 \end{equation}
Therefore, we use \cref{equ:newton-lebnize} to obtain
\begin{equation}
\label{eqn:z and y}
|z-y-\tilde g'(y')(z'-y')|\leq |z'-y'|^2(4C_1)\leq 4C_1^3C_2^2d(x,x_0)^2. 
\end{equation}
Combining with \cref{equ:psi0xzy} and \cref{equ: derivative h0}, we obtain
\begin{equation}\label{equ:Netwon-Lebnize}
    |[\ell_0(x)]-[\ell_0(x_0)]-S_{-t(\ell_0,x_0)}(\pi_{(\ell_0^{-1}{E_1})^\perp}x- \pi_{(\ell_0^{-1}{E_1})^\perp}x_0)| \leq 44 (C_1C_2)^7d(x,x_0)^2/\|h\|^2.
\end{equation}
Here $S_{-t(\ell_0,x_0)}=h_0'(y) \tilde g'(y')$ is the contracting action with contracting rate of $h_0\circ \tilde g$ at $y_0:=\pi_{(\ell_0^{-1}{E_1})^\perp}x_0$, and we obtain the estimate of $t(\ell_0,x_0)$ (\cref{eqn:estimate of scaling map}) using \cref{equ: derivative h0} and \eqref{equ: g' g''}.

Before proceeding to the general case, we would like to point out 
a corollary of Step 1:  
the sign of $[\ell_0(x)]-[\ell_0(x_0)]$ is the same with the sign of $z'-y'$ as elements in $[-\frac{1}{2},\frac{1}{2})$. The reason is as follows. Note that $\tilde g$ scales $\P(\R^2)$ 
by $1$ with distortion $C_1$ (\cref{lem:action g}). As $d(z',y')$ is small (\cref{equ:pi ell x}), we can deduce that the sign of $z-y$  is the same as that of $z'-y'$. As $[\ell_0(x)]=h_0(z)$, $[\ell_0(x_0)]=h_0(y)$, and $h_0$ contracts the minor arc $[z,y]$, we obtain the corollary. 
\\


\textbf{Step 2:} 
We need to distinguish two cases.

Case 1: Suppose that $\ell^{-1}{E_1}$ and $\ell_0^{-1}{E_1}$ are in the same side of the great circle passing through $x,x_0$,\footnote{Take representatives $u,u_0\in\S^2$ of  $\ell^{-1}{E_1}$ and $\ell_0^{-1}{E_1}$ with $d_R(u_0,u)<\pi/2$. The two points $\ell^{-1}{E_1}$ and $\ell_0^{-1}{E_1}$ are in the same side of a great circle actually means $u,u_0$ are in the same side. In the figure, we write $\ell^{-1}{E_1},\ell_0^{-1}E_1$ instead of $u,u_0$.} see \cref{fig:case1}. We write $\ell_0=\begin{pmatrix} \det h_0 &0 \\ n_0&h_0\end{pmatrix}$, $\ell=\begin{pmatrix} \det h &0\\ n&h\end{pmatrix}$, $x_0=\R (a_0,b_0,c_0)^t$, $x=\R(a,b,c)^t$. We have
\begin{align*}
d([\ell(x)],[\ell(x_0)])-d([\ell_0(x)],[\ell_0(x_0)])
=& \frac{\|((b,c)^t\wedge(b_0,c_0)^t+h^{-1}n\wedge (a(b_0,c_0)^t-a_0(b,c)^t) \|}{ \|h(b,c)^t+an\|\|h(b_0,c_0)^t+a_0n\| }\\
&-\frac{\|((b,c)^t\wedge(b_0,c_0)^t+h_0^{-1}n_0\wedge (a(b_0,c_0)^t-a_0(b,c)^t) \|}{ \|h_0(b,c)^t+an_0\|\|h_0(b_0,c_0)^t+a_0n_0\| }. 
\end{align*}
The idea is to compare the difference term by term using the formula
\begin{align*}
    \frac{A}{BC}-\frac{A_0}{B_0C_0}=\frac{AB_0C_0-A_0BC}{BCB_0C_0}=\frac{(A-A_0)B_0C_0+A_0(B_0-B)C_0+A_0B(C_0-C) }{BCB_0C_0},
\end{align*} 
with
\begin{align*}
    &A=\|((b,c)^t\wedge(b_0,c_0)^t+h^{-1}n\wedge (a(b_0,c_0)^t-a_0(b,c)^t) \|, \\
    &B=\|h(b,c)^t+an \|,\,\,\,C=\|h(b_0,c_0)^t+a_0n\|
\end{align*}
and $A_0,B_0,C_0$ the corresponding one with $h,n$ replaced by $h_0,n_0$. From \cref{coordinate and distance psi}, we obtain
\[A_0\leq (1+\|h_0^{-1}n_0\|)d(x,x_0)\leq 2d(x,x_0)/d(\ell_0^{-1}E_1,E_1^\perp)\leq 2C_1 d(x,x_0). \]
Similarly, we have $A_0\leq 4C_1d(x,x_0)$.
From \cref{lem:psi near psi}, we can apply \cref{lem:bad psi region} to obtain all possible pairs between $\ell,\ell_0$ and $x,x_0$.
Due to \cref{hv formula}, we obtain
\[ B,B_0,C,C_0\geq \frac{1}{4C_2}\|h_0\|. \]
Then due to $\|h(b,c)^t+an \|\leq \|h\| (1+\|h^{-1}n\|)$, using \cref{coordinate and distance psi} we obtain
\[ B,B_0,C,C_0\leq 4C_1\|h_0\|. \]
For the differences, by triangle inequality, \cref{equ:h-1nh'-1n} and \cref{eqn:distance x x'}, we have
\[ |A-A_0|\leq \|(h^{-1}n-h_0^{-1}n_0)\wedge (a(b_0,c_0)^t-a_0(b,c)^t) \|\leq \frac{C_Ld(\ell,\ell_0)d(x,x_0)}{d(\ell_0^{-1}E_1,E_1^\perp)}\leq  C_LC_1 d(\ell,\ell_0)d(x,x_0). \]
Similarly, using triangle inequality and \cref{lem:distance psi psi'}
\[ |B-B_0|,|C-C_0|\leq \|h_0\|(\|id-h_0^{-1}h\|+\|h^{-1}_0n-h^{-1}_0n_0\|)\leq 2C_L\|h_0\|d(\ell,\ell_0). \]

In the end, we obtain from these computations that
\begin{equation}\label{equ:d psi x}
 |d([\ell(x)],[\ell(x_0)])-d([\ell_0(x)],[\ell_0(x_0)])|\leq 2^{16} C_L(C_1C_2)^4 d(x,x_0)d(\ell,\ell_0)/\|h_0\|^2.
\end{equation}
By the position of $\ell^{-1}{E_1}, \ell_0^{-1}{E_1},x,x_0$ on $\P(\R^3)$ (see \cref{fig:case1}) and \cref{equ:pi ell x}, $\pi_{\ell^{-1}E_1^\perp}x- \pi_{\ell^{-1}E_1^\perp}x_0$ and $\pi_{\ell_0^{-1}E_1^\perp}x-\pi_{\ell_0^{-1}E_1^\perp}x_0$ have the same sign as elements in $[-1/2,1/2)$ . 
The corollary of Step 1 holds for both $[\ell(x)]-[\ell(x_0)]$ and $ [\ell_0(x)]-[\ell_0(x_0)]$. Then they have the same sign. We have the relation 
\begin{equation}
\label{eqn:comparing l l0}
 \left|([\ell(x)]-[\ell(x_0)])-([\ell_0(x)]-[\ell_0(x_0)])\right|\leq \big||[\ell(x)]-[\ell(x_0)]|-|[\ell_0(x)]-[\ell_0(x_0)]|\big|. 
\end{equation}
Due to $d([\ell(x)],[\ell(x_0)])\leq 1/2$ and $|a-b|\leq 2|\sin(a)-\sin(b)|$ for $a,b\in[0,\pi/3]$, combined with \cref{equ:d psi x}, we obtain
\begin{equation}
 |([\ell(x)]-[\ell(x_0)])-([\ell_0(x)]-[\ell_0(x_0)])|\leq 2^{17}C_L (C_1C_2)^4 d(x,x_0)d(\ell,\ell_0)/\|h_0\|^2. 
 \end{equation}
Combined with \cref{equ:Netwon-Lebnize} for both $\ell$ and $\ell_0$, we obtain the lemma due to $C_1,C_2$ large.

\begin{figure}[!ht]
\begin{minipage}{0.48\textwidth}
    \centering
    \includegraphics{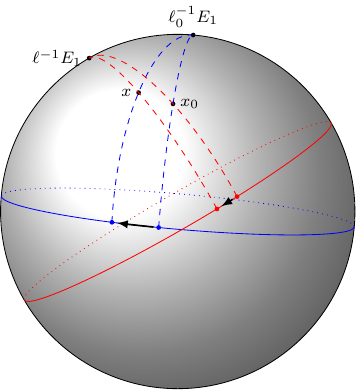}
    \caption{Case 1}
    \label{fig:case1}
\end{minipage}\hfill
\begin{minipage}{0.48\textwidth}
     \centering
     \includegraphics{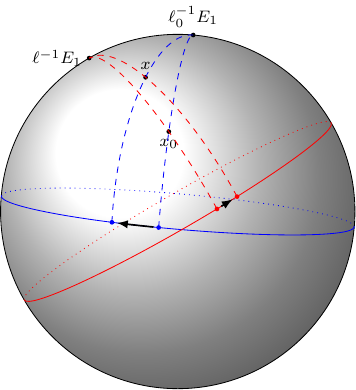}
     \caption{Case 2}
     \label{fig:case2}
\end{minipage}
\end{figure}

Case 2: Suppose $\ell^{-1}{E_1}$ and $\ell_0^{-1}{E_1}$ are in different sides of the great circle passing through $x,x_0$. 
By spherical trigonometry of the triangle $(x,\ell^{-1}E_1,\ell_0^{-1}E_1)$, we have
\[\sin \angle_x(\ell^{-1}{E_1},\ell_0^{-1}{E_1})=\frac{\sin\angle_{\ell^{-1}{E_1}}(x_,\ell_0^{-1}{E_1})}{ d(x,\ell_0^{-1}{E_1})} d(\ell^{-1}{E_1},\ell_0^{-1}{E_1})\leq \frac{d(\ell^{-1}{E_1},\ell_0^{-1}{E_1})}{d(x,\ell_0^{-1}{E_1})}.\]
Due to  \cref{lem:distance psi psi'} and
 \cref{equ: x psi-1 {E_1}}, we have
\begin{equation}\label{equ:angle le1}
\sin \angle_x(\ell^{-1}{E_1},\ell_0^{-1}{E_1})\leq \frac{C_Ld(\ell,\ell_0)/d(\ell_0^{-1}{E_1},{E_1}^\perp)}{d(\ell_0^{-1}{E_1},{E_1}^\perp)/C_2}\leq C_LC_1^2C_2d(\ell,\ell_0)\leq 1/C_2.
\end{equation} 
Again by spherical trigonometry of the triangle $(x,x_0,\ell^{-1}E_1)$, we have
\[d(\pi_{(\ell^{-1}{E_1})^\perp}x,\pi_{(\ell^{-1}{E_1})^\perp}x_0)=\sin\angle_{\ell^{-1}E_1}(x,x_0)= d(x,x_0)\frac{\sin \angle_x(\ell^{-1}{E_1},x_0) }{d(x_0,\ell^{-1}{E_1})}.\]
Due to the hypothesis of case 2,
we obtain $\sin \angle_x(\ell^{-1}{E_1},x_0)\leq \sin \angle_x(\ell^{-1}{E_1},\ell_0^{-1}{E_1}) $ (the angle $\angle_x(\ell^{-1}{E_1},\ell_0^{-1}{E_1})$ must be acute due to $d(x,\ell^{-1}E_1)\geq d(\ell^{-1}E_1,\ell_0^{-1}E_1)$). Combined with \cref{equ:angle le1} and \cref{equ: x psi-1 {E_1}}, we obtain
\[ d(\pi_{(\ell^{-1}{E_1})^\perp}x,\pi_{(\ell^{-1}{E_1})^\perp}x_0)\leq d(x,x_0)\frac{\sin \angle_x(\ell^{-1}{E_1},\ell_0^{-1}E_1) }{d(x_0,\ell^{-1}{E_1})}\leq C_LC_1^3C_2^2d(\ell,\ell_0)d(x,x_0). \]
The same inequality also holds with $\ell$ replaced by $\ell_0$.
Combining with \cref{equ:Netwon-Lebnize} for $\ell$ and the triangle inequality, using the fact that $\theta\leq 2\sin\theta$ for any $\theta\in[0,\pi/2]$, we obtain
\begin{align*}
&\left|[\ell(x)]-[\ell(x_0)]-S_{-t(\ell_0,x_0)}\left(\pi_{(\ell_0^{-1}{E_1})^\perp}x- \pi_{(\ell_0^{-1}{E_1})^\perp}x_0\right)\right|\\
\leq &\left|[\ell(x)]-[\ell(x_0)]-S_{-t(\ell,x_0)}\left(\pi_{(\ell^{-1}{E_1})^\perp}x- \pi_{(\ell^{-1}{E_1})^\perp}x_0\right)\right|+2q^{-t(\ell,x_0)}d(\pi_{(\ell^{-1}{E_1})^\perp}x,\pi_{(\ell^{-1}{E_1})^\perp}x_0)\\
+&2q^{-t(\ell_0,x_0)}d(\pi_{(\ell_0^{-1}{E_1})^\perp}x,\pi_{(\ell_0^{-1}{E_1})^\perp}x_0)\\
\leq & 44(C_1C_2)^7d(x,x_0)^2/\|h\|^2+(q^{-t(\ell,x_0)}+q^{-t(\ell_0,x_0)})C_LC_1^3C_2^2d(\ell,\ell_0)d(x,x_0)\\
\leq &C_L(C_1C_2)^7(2d(\ell,\ell_0)+88d(x,x_0))d(x,x_0)/\|h_0\|^2.\qedhere
\end{align*}
\end{proof}

\subsection{\texorpdfstring{$q$}{q}-Adic partitions}
The $q$-adic level-$n$ partition for $n\in\N$ of $[0,1)$ is defined by
\begin{equation*}
\mathcal{Q}_n=\left\{[\frac{k}{q^n},\frac{k+1}{q^n}):0\leq k<q^n,\ k\in\Z\right\}.
\end{equation*}
Using the canonical identification between $[0,1)$ and the circle $\P(\R^2)$, we can regard $\mathcal{Q}_n$ as a partition of $\P(\R^2)$. We write $\mathcal Q_t = \mathcal Q_{[t]}$ when $t\in \mathbb{R}^+$ is not an integer. 

We also need similar partitions of the group $L$ defined in \cref{sec:decomposition}. By Theorem 2.1 of \cite{kaenmaki_existence_2012}, for $q$ large enough, there exists a
collection of Borel sets $\{Q_{n,i}\subset L: n\in \N, i\in \N\}$, having the following properties:
\begin{enumerate}
\item $L= \cup_{i\in \N} Q_{n,i}$ for every $n\in\N$;
\item $Q_{n,i}\cap Q_{m, j}=\emptyset$ or $Q_{n,i}\subset Q_{m, j}$ for $n, m\in \N$ with $n\geq m$ and any $i, j\in  \N$;
\item There exists a constant $C_p > 1$ (independent of large $q$) such that for every $n\in \N$ and $i\in \N$ there
exists $\ell\in Q_{n,i}$ with
\begin{equation}\label{eqn: partition size}
 B(\ell, C_p^{-1}q^{-n})\subset Q_{n,i}\subset B(\ell, q^{-n}), 
\end{equation}
where $B(\ell,q^{-n})$ is measured using the left $L$-invariant metric on $L$. 
\end{enumerate}

For each $n\in \N$, denote by $\cal Q_n^L$ the partition $\{Q_{n,i} : i\in \N\}$ of $L$.  We will usually omit the superscript when there is no ambiguity about which space to partition. Fix a left-invariant Haar measure on $L$.  In view of \eqref{eqn: partition size}, there exists a constant $C>1$ independent of $n$ such that $$1/C\leq q^{5n}\times\,\,\,\text{measure of any atom in} \,\,\,\cal Q_n^L \leq C.$$ 
Hence there exists a constant $C'>0$ such that for every $n$ and any atom $Q\in \cal Q_n^L$,
$$\#\{Q'\in \cal Q_{n+1}^L : Q'\subset Q \}\leq C'.$$

\subsection{Component measures and random measures}

For a probability measure $\theta$ on a measure space $X$ and a measurable subset $Y$ of $X$, define the normalized restriction measure by 
\begin{equation*}
\theta_Y=\begin{cases}
\frac{1}{\theta(Y)}{\theta|_{Y}}, \,\,\,&\text{if}\,\,\,\theta(Y)>0,\\
0,\,\,\,&\text{otherwise}.
\end{cases}
\end{equation*}


For a probability measure $\theta$ on $\mathbb{P}(\mathbb{R}^2)$ (or $L$) equipped with a partition $\mathcal{Q}_n$, we write $\mathcal{Q}_n(x)$ for the unique partition element containing $x$. We define measure valued random variables $\theta_{x,n}$ such that $\theta_{x,n}=\theta_{\mathcal{Q}_n(x)}$ with probability $\theta(\mathcal{Q}_n(x))$. 
For integers $n_2\geq n_1$ and an event $\mathcal{U}$, we write
\begin{equation*}
\mathbb{P}_{n_1\leq i\leq n_2}(\theta_{x,i}\in \mathcal{U})=\frac{1}{n_2-n_1+1} \sum_{n=n_1}^{n_2} \mathbb{P}(\theta_{x,n}\in \mathcal{U}).
\end{equation*}

\subsection{Regularity of stationary measures}
\label{sec:regularity of stationary measures}
Let $\Lambda$ be a finite set and let $\nu=\sum_{i\in \Lambda}p_i\delta_{g_i}$ be a finitely supported probability measure on $\SL_3(\R)$ such that the group generated by the support of $\nu$, $\langle\supp \nu\rangle$, is Zariski dense in $\SL_3(\R)$.
We write $\nu^{*n}$ for the convolution, that is, $\nu^{*n}$ is the law of the product $g_1\cdots g_n$, with i.i.d. $g_i$ following the law of $\nu$. 

Let $\mu$ be the unique associated stationary measure on $\P(\R^3)$ (\cite{furstenberg_boundary_1973}), that is to say, we have
\begin{equation*}
\mu=\nu*\mu :=\sum_{i\in \Lambda} p_i (g_i)_* \mu,
\end{equation*}
where $(g_i)_*\mu$ is the pushforward measure defined by $(g_i)_*\mu(E)=\mu (g_i^{-1}E)$ for any measurable set $E\subset \P(\R^3)$. 

There are $3$ 
{\it{Lyapunov exponents}} $\lambda_1\geq\lambda_2\geq \lambda_3$ of $\nu$ such that 
for any $i=1,2,3$,
\begin{equation*}
\lambda_i=\lim_{n\to \infty} \int \frac{1}{n} \log \sigma_i(g_{i_1}\cdots g_{i_n})d\nu(g_{i_1})\cdots d\nu(g_{i_n}).
\end{equation*}
Since $\langle \supp \nu\rangle$ is Zariski dense, we have $\lambda_1>\lambda_2>\lambda_3$ (see for example \cite{guivarch_frontiere_1985}, \cite{goldsheid1989lyapunov} ). Set $\chi_i=\lambda_1-\lambda_{i+1}$ for $i=1,2$, and $\lambda(\nu)=(\lambda_1,\lambda_2,\lambda_3)$ to be the Lyapunov vector.

We will also consider the measure $\nu^-$ on $\SL_3(\R)$ defined by $\nu^-(g)=\nu(g^{-1})$. 
Let $\mu^-$ be the unique associated stationary measure.

The results stated in this section and \cref{sec:ldp} hold under a more general assumption that $\supp \nu$ generates a Zariski dense subgroup and that $\nu$ is of finite exponential moment. But for the interest of the current paper, we will focus on the case when $\nu$ is finitely supported.

Classical references of products of random matrices include \cite{furstenberg_noncommuting_1963} and \cite{bougerol_products_1985}.
We refer the reader to \cite{benoist_random_2016}. We recall an equidistribution result for $\mu$, $\mu^{-}$ (\cite{guivarch_frontiere_1985}, \cite[Chapter III, Theorem 4.3]{bougerol_products_1985})

\begin{lem}\label{lem:equidistribution x}
For any $x\in \P(\R^3)$, as $n\rightarrow \infty$ we have
\[ \nu^{*n}*\delta_x\rightarrow \mu,\ (\nu^-)^{*n}*\delta_x\rightarrow \mu^-, \]
and the convergence is uniform with respect to $x$.
\end{lem}

We have the Guivarc'h regularity of stationary measures ( \cite{guivarch_produits_1990},
\cite[Prop 14.1]{benoist_random_2016})
\begin{lem}\label{lem: Hold reg proj mes}
    There exist $C>0,\beta>0$ such that for any $n\geq 1$ and hyperplane $W$ in $\P(\R^3)$
    \[ \mu \{ x\in\P(\R^3):\ d(W,x)\leq q^{-n}\}\leq Cq^{-\beta n}. \]
    In particular, for any $V\in\P(\R^3)$ and $x\in \P(\R^2)$
    \[ (\pi_{V^\perp}\mu) \left(B(x,q^{-n})\right)\leq \mu \{y:d((x,V),y)\leq q^{-n}\}\leq Cq^{-\beta n}, \]
    where $(x,V)$ is the projective plane spanned by $x$ and $V$.
\end{lem}

As a corollary, we have the following result. 
\begin{lem}\label{lem: Hold regul general projc mesre}
Let $C_1>1$. For any $g\in \SL_3(\R)$ and $V\in \P(\R^3)$ satisfy 
$
d(g^{-1}V, V^\perp)>1/C_1,
$
we have for any 
for any $x\in \P(V^\perp)$ and $r>0$, 
\begin{equation*}
(\pi_{g^{-1}V, V^\perp}\mu)(B(x,r))\leq CC_1^\beta  r^\beta.
\,\,\,
\end{equation*}
where $C,\beta$ are defined in Lemma \ref{lem: Hold reg proj mes}.
\end{lem}
\begin{proof}
It follows from the proof of Lemma \ref{lem: dec pi V A} that 
\begin{equation*}
\pi_{g^{-1}V,V^\perp}=\pi(g^{-1}V, V^\perp,(g^{-1}V)^\perp)\circ \pi_{(g^{-1}V)^\perp}.
\end{equation*}
\cref{lem: Hold reg proj mes} holds for $\pi_{(g^{-1}V)^\perp}\mu$. Moreover, \cref{lem:g-1 V lip} states that $\pi(g^{-1}V,V^\perp,(g^{-1}V)^\perp)$ scales by $1$ with distortion $C_1$. These imply the lemma. 
\end{proof}

More generally, the H\"older regularity of stationary measures also holds for irreducible representations
(\cite{guivarch_produits_1990}, \cite[Proposition 10.1]{benoist_random_2016}). Let $p_2$ be the map from $\P(\R^3)$ to the projective space of the second symmetric power $\P(Sym^2(\R^3))$, defined by sending $\R v$ to $\R v\otimes v$. It is $\SL_3(\R)$-equivariant. Let $(p_2)_*\mu$ be the pushfoward measure on $\P(Sym^2(\R^3))$. Then $(p_2)_*\mu$ is $\nu$-stationary and we have
\begin{lem}\label{lem:subvariety}
There exist $C>0,\beta>0$ such that for any $n\geq 1$ and hyperplane $W$ in $\P(Sym^2(\R^3))$, we have
\[ (p_2)_*\mu \{ x\in\P(Sym^2\R^3):\ d(W,x)\leq q^{-n}\}\leq Cq^{-\beta n}. \]
\end{lem}

\subsection{Large deviation estimates}\label{sec:ldp}
We first introduce several LDP (Large Deviation Principle) estimates for random walks on $\SL_3(\R)$. Set $G=\SL_3(\R)$ in this subsection. The following lemma can be found in \cite{le_page_theoremes_1982} and Proposition 14.3, Theorem 13.17 (iii), Theorem 13.11 (iii) in \cite{benoist_random_2016}. Recall that $\nu$ is finitely supported and Zariski dense and $\mu$ is the $\nu$-stationary measure on $\P(\R^3)$.

\begin{lem}\label{lem: hold points}
    For any $c>0$, there exist $C>0$ and $\epsilon >0$ such that for any $n\in \N$, we have for any hyperplane $W$ in $\P(\R^3)$ and any $V=\R v\in\P(\R^3)$ with $\|v\|=1$,
    \begin{align}
       \label{eq:hol gV} \nu^{*n}\{g\in G:\ d(gV,W)\leq q^{-cn} \}\leq Cq^{-\epsilon n},    \\
        \label{eq:hol g+}\nu^{*n}\{g\in G:\ d(V_g^+,W)\leq q^{-cn} \}\leq Cq^{-\epsilon n},\\
        \label{eq:hol g-}\nu^{*n}\{g\in G:\ d(H_g^-,V)\leq q^{-cn} \}\leq Cq^{-\epsilon n},\\
        \label{eq:close g+ gv}\nu^{*n}\{g\in G:\ d(V_g^+,gV)\geq q^{-(\chi_1-c)n} \}\leq Cq^{-\epsilon n},\\
        \label{eq:ldp sigma g} \nu^{*n}\{g\in G:\ \|\frac{1}{n}\kappa(g)-\lambda(\nu)\|\geq c \}\leq Cq^{-\epsilon n},\\
        \label{eq:ldp cocycle g} \nu^{*n}\{g\in G:\ \frac{1}{n} |\log\|g v\|-\log\sigma_1(g)|\geq c \}\leq Cq^{-\epsilon n}.
    \end{align}
\end{lem}

As a corollary, we have the following lemma as in \cite[Lemma 14.11]{benoist_random_2016}.
\begin{lem}\label{lem:hol g m}
        For any $c>0$, there exist $C>0$ and $\epsilon >0$ such that for any $m\leq n$, we have for any hyperplane $W$ in $\P(\R^3)$ and any $V\in \P(\R^3)$,  
    \begin{align}
        \nu^{*n}\{g\in G:\ d(gV,W)\leq q^{-cm} \}\leq Cq^{-\epsilon m}\label{eq:gvw n>m},
        \\
        \label{eq:hol g+ n>m}\nu^{*n}\{g\in G:\ d(V_g^+,W)\leq q^{-cm} \}\leq Cq^{-\epsilon m},\\
        \label{eq:hol g- n>m} \nu^{*n}\{g\in G:\ d(V,H_g^-)\leq q^{-cm} \}\leq Cq^{-\epsilon m}.
    \end{align}
\end{lem}
\begin{proof}

    For \cref{eq:gvw n>m}, when $m=n$, this is \cref{eq:hol gV}. Since 
    \begin{equation*}
    \nu^{*n}\{g\in G: d(gV,W)\leq q^{-cm}\}=\int_{G} \nu^{*m}\{g\in G: d(ghV,W)\leq q^{-cm}\}d\nu^{*(n-m)}(h),
    \end{equation*}
    then case $n\geq m$ follows.

    For \cref{eq:hol g+ n>m}, we use triangle inequality
    \[d(V_g^+,W)\geq d(gV,W)-d(V_g^+,gV).\]
    Therefore
    \begin{align*}
    &\nu^{*n}\{g\in G:\ d(V_g^+,W)\leq q^{-cm} \}\\
    \leq &\nu^{*n}\{g\in G:\ d(gV,W)\leq 2q^{-cm} \}+\nu^{*n}\{g\in G:\ d(V_g^+,gV)\geq q^{-cm} \}.
    \end{align*}
    We may assume $c<\chi_1/2$. Then $q^{-cm}\leq q^{-(\chi_1-c)n}$ and the second term above can be estimated using \cref{eq:close g+ gv}. The first term can be estimated using \cref{eq:gvw n>m}. Hence, we obtain \cref{eq:hol g+ n>m}.
    
    For \cref{eq:hol g- n>m},
    due to $H_g^-=(V_{(g^t)}^+)^\perp$, it is equivalent to  prove 
    \begin{equation}
    \label{eq:hol g- n> m equivalent}
    \nu^{*n}\{g\in G:\ d(V_{(g^t)}^+,W)\leq q^{-cm} \}\leq Cq^{-\epsilon m}, 
    \end{equation}
    with $W=V^\perp$, which is similar to \cref{eq:hol g+ n>m}. Let $\nu^t$ be the probability measure on $G$ defined by $\nu^t(g)=\nu(g^t)$. It is finitely generated and Zariski dense. \cref{eq:hol g- n> m equivalent} holds as we have \cref{eq:hol g+ n>m} for $\nu^t$. 
\end{proof}

Next we use LDP of $\SL_3(\R)$ to obtain LDP for the $\SL_2(\R)$-part in the $UL$-decomposition. It will allow us to reduce the $\SL_3(\R)$-action on $\P(\R^3)$ to the induced $\SL_2(\R)$-action on $\P(\R^2)$ and obtain the convergence of entropy for the induced action. 

Recall that we have the projection map $\pi_L:UL\to L$ (see \cref{lem:pi L V}). 
In the following lemma, it is important that it holds for all $n\gg \log C$, which will be very useful in proving porosity in \cref{sec:porosity}.
\begin{lem}

    There exists $\beta>0$ such that the following holds. For any $C>1$, there exists $N_C=O(\log C)>0$ such that we have for any $n>N_C$ and $V\in\P(\R^3)$
    \begin{align}
   \label{equ:LDP-g-1V}
     \nu^{*n}\{g\in G:\ d(g^{-1}V,V^\perp)\leq 1/C \}&\leq C^{-\beta}, \\
      \label{equ:LDP-g-V-perp}
          \nu^{*n}\{g\in G:\ |\chi_1(h)-\chi_1(g)|\geq \log C \}&\leq C^{-\beta},
    \end{align} 
    where $h\in \SL_2^\pm(\R)$ comes from the $UL$-decomposition of $g$: $\pi_L(g)=\begin{pmatrix}\det h& 0\\ n&h\end{pmatrix}$ with $n\in \R^2$.  
\end{lem}

\begin{proof}
    Applying \cref{eq:gvw n>m} to the measure $\nu^-$ and $W=V^\perp$, we obtain \cref{equ:LDP-g-1V}.
    
 For any $g\in UL$, write  $g=\begin{pmatrix}\lambda^{-2} & x & y \\ 0 & \lambda & 0 \\ 0 & 0 & \lambda \end{pmatrix}\begin{pmatrix} \det h & 0 \\ n & h\end{pmatrix}$. We define $\tilde{h}:=\Pi_{E_1^\perp}g|_{E_1^\perp}=\lambda h\in\GL(E_1^\perp)$. We have
    $ q^{\chi_1(h)}=\|\tilde{h}\|^2/|\det \tilde{h}|. $
     The determinant is given by 
    \begin{equation*}
    \begin{split}
        |\det \tilde{h}|=\lambda^2=|\langle g^{-1}e_1, e_1 \rangle|
        =d(g^{-1}E_1,E_1^\perp)\|g^{-1}e_1\|.
        \end{split}
    \end{equation*}
    The norm of $\|\tilde{h}v\|$ for $v\in E_1^\perp$ is given by
    \begin{equation}
     \label{equ:hv gv}\|\tilde{h}v\|=\|\Pi_{E_1^\perp}gv \|=\|(gv)\wedge e_1\|=d(g\R v,E_1)\|gv\|.
    \end{equation}
    Hence 
    \begin{equation}\label{equ: hv}
    \|g\|d(gE_2,E_1)d(E_2,H_g^-)\leq   d(gE_2,E_1)\|ge_2\|\leq \|\tilde{h}\|\leq \|g\|.
    \end{equation}
    Therefore, 
    \begin{equation*}     
   \frac{(d(g E_2, E_1)d(E_2,H_g^{-1}) \|g\|)^2}{d(g^{-1}E_1,E_1^\perp)\|g^{-1}e_1\|} \leq q^{\chi_1(h)}=\frac{\|\tilde{h}\|^2}{|\det\title{h}|}\leq \frac{\|g\|^2}{d(g^{-1}E_1,E_1^\perp)\|g^{-1}e_1\|}.
    \end{equation*}
    It follows from \cref{lem:gv d v g-} that
    \begin{equation}
    \label{eqn: chi 1 g}
    d(E_1,H^-_{g^{-1}})\leq \frac{\|g^{-1}e_1\|}{\|g^{-1}\|}\leq 1. 
    \end{equation}
    We also have $\|g\|^2/\|g^{-1}\|=\sigma_1(g)^2/\sigma_1(g^{-1})=\sigma_1(g)^2\sigma_3(g)=\sigma_1(g)/\sigma_2(g)=q^{\chi_1(g)}$. Therefore 
    \begin{equation}\label{equ: chi1 h}
        q^{\chi_1(g)}(d(g E_2, E_1)d(E_2,H_g^{-1}) )^2 \leq q^{\chi_1(h)}\leq  q^{\chi_1(g)}\frac{1}{d(g^{-1}E_1,E_1^\perp)d(E_1,H_{g^{-1}}^-) }.
    \end{equation}
    
    Consider the elements $g\in G$ satisfying all the following inequalities:
    \begin{enumerate}[(i)]

    \item $d(E_1,H^-_{g^{-1}})\geq q^{-C/8}$;

    \item $d(g^{-1}E_1,E_1^\perp)\geq q^{-3C/4}$;

        \item $d(E_2, H_g^-)\geq q^{-C/8}$;

    \item $d(gE_2,E_1)\geq q^{-3C/8}$.
    \end{enumerate}
    It follows from \cref{eq:gvw n>m} and \cref{eq:hol g- n>m}, that $\nu^{*n}$-measure of such $g$ is greater than $1-C^{-\beta}$ for some $\beta>0$. Using (i), (ii) to \cref{equ: chi1 h}, we have $q^{\chi_1(h)-\chi_1(g)}\leq q^C$; using (iii), (iv) to \cref{equ: chi1 h}, we have $q^{\chi_1(g)-\chi_1(h)}\leq q^C$. 
    This finishes the proof of \cref{equ:LDP-g-V-perp}.
\end{proof}

We need more LDP 
estimates of $h$.

\begin{lem}
    For any $c>0$, there exist $C>0$ and $\epsilon >0$ such that for any 
    $x\in \P(E_1^\perp)$ and $n\in \N$, we have
  \begin{align}
          \label{equ:LDP-h-V-g}
          \nu^{*n}\{g\in G:\ \frac{1}{n}|\chi_1(h)-\chi_1 n|\geq c \}&\leq  Cq^{-\epsilon n }, \\
   \label{equ:LDP-h-}
     \nu^{*n}\{g\in G:\ d(H_h^-,x)\leq q^{-cn} \}&\leq Cq^{-\epsilon n} , \\
     \label{lem:h+hx}
     \nu^{*n}\{g\in G:\ d(V^+_h,hx)\geq q^{(-\chi_1+c)n} \}&\leq  Cq^{-\epsilon n},
    \end{align}
    where $h\in \SL_2^\pm(\R)$ comes from the $UL$-decomposition of $g$: $\pi_L(g)=\begin{pmatrix}\det h & 0\\ n& h\end{pmatrix}$ with $n\in \R^2$.
\end{lem}

\begin{proof}
For \cref{equ:LDP-h-V-g},
     we compare $\chi_1(h)$ with $\chi_1(g)$ using \cref{equ:LDP-g-V-perp}. Then use \cref{eq:ldp sigma g} for the LDP of $\chi_1(g)$.

For any $g\in UL$, write  $g=\begin{pmatrix}\lambda^{-2} & x & y \\ 0 & \lambda & 0 \\ 0 & 0 & \lambda \end{pmatrix}\begin{pmatrix} \det h & 0 \\ n & h\end{pmatrix}$. We define $\tilde{h}:=\Pi_{E_1^\perp}g|_{E_1^\perp}=\lambda h\in\GL(E_1^\perp)$.  
Take a unit vector $v$ in $x$.
Due to \cref{lem:gv d v g-}, that is 
    \[d(H^-_h,x)
    \geq \frac{\|\tilde hv\|}{\|\tilde h\|}-q^{-\chi_1(h)} .\] 
   To obtain \cref{equ:LDP-h-}, it suffices to estimate the $\nu^{*n}$-measure of $g$ such that
   \begin{enumerate}[(i)]
   \item either its corresponding $\tilde{h}$ satisfies $\|\tilde{h}v\|\leq 4q^{-cn}\|\tilde{h}\|$,
   \item or $\chi_1(h)\leq 2cn$.
   \end{enumerate}
    We can bound the $\nu^{*n}$-measure of $g$ satisfying (ii)  by applying \cref{equ:LDP-h-V-g} to another sufficiently small constant $c$. For (i), by \cref{equ: hv}, 
    \[\|\tilde{h}\|
    \leq \sigma_1(g).    \]
    So due to \cref{equ:hv gv}, we have
    \begin{equation*}
   d(g\R v,E_1)\frac{\|gv\|}{\sigma_1(g)}\leq \frac{\|\tilde{h}v\|}{\|\tilde{h}\|}.
    \end{equation*}
   We use \cref{eq:hol gV} and \cref{eq:ldp cocycle g} to estimate the first term and the second term on the left, respectively.  This yields the estimate of the $\nu^{*n}$-measure of $g$ satisfying (i), and hence we obtain \cref{equ:LDP-h-}.
     
      Using \cref{lem:gv d v g-}, we have
    \[ d(V^+_h,hx)\leq \frac{q^{-\chi_1(h)}}{d(x,H^-_h)}. \]
    Then \cref{lem:h+hx} follows from \cref{equ:LDP-h-V-g} and \cref{equ:LDP-h-}.
\end{proof}

At the end of the subsection, we introduce two kinds of good subsets in $L$:
\begin{itemize}
\item for any $n\in \N$ and $c>0$, set 
\begin{equation}
\label{equ: L n eps}
 L( n,c):=\{\ell\in L:\ d(\ell^{-1}E_1,E_1^\perp)\geq q^{-c n},\ \chi_1(h)\in [\chi_1-c,\chi_1+c]n \};  
\end{equation}

\item for any $t\geq 1$ and $C_1>0$, set
\begin{equation}\label{equ: l t C1}
    L(t,C_1):=\{\ell\in L:\ d(\ell^{-1}E_1,E_1^\perp)>1/C_1,\ \|h\|\in q^{t/2}[1/C_p,C_p]\},
\end{equation}
\end{itemize}
where $C_p$ is defined as in \cref{eqn: partition size}, and $h$ is the $\SL_2^{\pm}(\R)$ matrix of $\ell$: we have $\ell=\begin{pmatrix} \det h&0\\ n& h\end{pmatrix}$ with $n\in \R^2$.
For $\ell\in L(t,C_1)$, we have $\|h\|\geq 2$. Moreover, it follows from \cref{equ:LDP-g-1V} and \cref{equ:LDP-h-V-g} that
\begin{equation}
\label{equ:ln epsilon}
 (\pi_{L}\nu^{*n})(L(n,c)^c)\leq Cq^{-\epsilon(c) n}.
\end{equation}

\subsection{Coding, partitions of symbolic space and random words}\label{sec:partition}

We consider the general coding scheme of the random walk. 
Consider the symbolic space $\Lambda^{\N}$, the set of one-sided infinite words over $\Lambda$. Let $\Lambda^\ast$ be the set of finite words over $\Lambda$ and for each $n\in \N$, $\Lambda^n$ be the set of words of length $n$. For a word $\mathbf{i} = (i_1,...,i_n) \in \Lambda^\ast$, we define
$g_{\mathbf{i}}:= g_{i_1}\circ \cdots \circ g_{i_n}$.
Similarly, define $p_{\mathbf{i}}:=p_{i_1} \cdots p_{i_n}$.

The space $\Lambda^\N$  comes with the natural partitions into level-$n$ cylinder sets. However, it will be more convenient to consider general partitions into cylinders of varying lengths.

\paragraph{The partition $\Psi_n^q$ and the definition of $q$.}

We consider a collection of words  $$\Psi_n^q=\{  (i_1,\cdots,i_m)\in \Lambda^\ast : \chi_1(g_{i_1\ldots i_m})\geq  n > \chi_1(g_{i_1\ldots i_{j}})\ \text{ for all }j<m \}.$$
Different from the setting in \cite{barany_hausdorff_2017}, the exponent $\chi_1$ here is not monotone with respect to the length of the word $m$. 
However, note that the measure $\nu$ has compact support, and the Cartan projection is subadditive (see for example \cite[Lemma 2.3]{kasselCorank08}). There exists a constant $c_0>0$, depending only on the support of $\nu$, such that for every $n\in \N$ and for every $\mathbf{i} \in \Psi_n^q$, 
\begin{equation}\label{equ:bounded residual time}
c_0 q^{-n} \leq q^{-\chi_1(g_{\bf i})} \leq  q^{-n}.
\end{equation}
If we take $q>1/c_0$, then for any $n\neq k$, we have
\begin{equation}\label{equ:psi q n k}
 \Psi^q_n\cap\Psi^q_k=\emptyset.
\end{equation}

By the definition of $\Psi_n^q$, no word in is a prefix of another word. And the law of large numbers holds for $\nu$(\cite{furstenberg_noncommuting_1963}): we have for $\nu^{\otimes \N}$-a.e. $\bi\in\Lambda^\N$, 
\begin{equation}
\label{eqn: chi1 LLN}
\frac{\chi_1(g_{i_1\ldots i_m})}{m}\rightarrow \chi_1>0.
\end{equation}
Therefore, the set of $\bi\in \Lambda^{\N}$ such that $\chi_1(g_{i_1\ldots i_m})<n $ for all $m$ is of measure zero. Consider the map $\Psi_n^q:\Lambda^\N\rightarrow \Lambda^*$ sending $\bi$ to its first $(m+1)$ coordinates with $m$ the least nonnegative integer such that $\chi_1(g_{i_1\ldots i_m})\geq n$. It is defined $\nu^{\otimes\N}$- almost everywhere, and it gives a measurable partition of $\Lambda^\N$.
\begin{defi}\label{defi: q}
From now on, we fix the choice of an integer $q$ with $q>\max\{1/c_0,C_L, 4C_p^2\}$ (see \cref{equ:C_L} and \cref{eqn: partition size}),  and the logarithm is in base $q$, $\log=\log_q$.
\end{defi}


 \begin{defi}[Random words $\mathbf{U}(n)$, $\mathbf{I}(n)$]
 \label{defn:random words}
For any given $n \in \N$, we introduce the random words $\bI(n)$ and $\bU(n)$ taking values from subsets coming from measurable partitions of $\Lambda^\ast$:
\begin{itemize}
    \item $\bI(n)$ is the random word taking values in $\Psi_n(=\Psi_n^q)$ according to the probability $\nu^{\otimes \N}.$
i.e. \[\P(\bI(n)=\bi)=\begin{cases}
    p_{\bf i}, \text{ if }\bi\in \Psi_n\\0, \text{ otherwise}.
\end{cases}\]
\item $\bU(n)$ is the random word taking values from $\Lambda^n$ according to the probability $\nu^{\otimes n}$.
\end{itemize}
\end{defi}
For each $n\in \N$ and for any measurable set $\cal{U}\subset G$, we have
\begin{equation}
\label{equ:switch}
\nu^{*n}\{g\in G: g\in \cal U\}=\P_{i=n}\{\bU (i): g_{\bU (i)}\in \cal U\}.
\end{equation}
Due to the stationarity of the measure $\mu$ and that $\bI(n)$ is defined by stopping time, using the martingale property, we have for any $n\in\N$
\begin{equation}\label{equ:mu ej}
    \mu=\E_{j=n}(g_{\bI(j)}\mu).
\end{equation}


The following lemma is the relation between $\bU (i)$ and $\bI (i)$.
\begin{lem}\label{lem: unify un in} 


There exist $C_I>1$ and $\beta>0$ such that the following holds. Let $\cal U\subset\Lambda^\ast$ be a set of words. Suppose that for some $\epsilon > 0$ and $n(\epsilon)\in \N$, we have for every $n \geq n(\epsilon)$,
$$\P_{1\leq i\leq n}\{\bU(i)\in \cal U\}>1-\epsilon.$$
Then for every $n\geq n(\epsilon)$,
$$\P_{1\leq i\leq n}\{\bI(i)\in \cal U\}>1-C_I\epsilon/2-O(q^{-\beta n}).$$
\end{lem}

\begin{proof}
The subtlety of the proof lies in that $\chi_1(g_{i_1\ldots i_m})$ is not monotone with respect to the length of the word $m$.    
    
    For any $n\geq n(\epsilon)$, we want to bound
    \[ \P_{1\leq i\leq n}\{\bI(i)\in \cal U^c\}. \]
    Set $C'>2(1-\log c_0)/\chi_1$ and $\cal B:=\{\bi\in \Psi_i:1\leq i\leq n\}$.  We have $\Psi_i\cap \Psi_j=\emptyset$ for any $1\leq i<j\leq C'n$ (\cref{equ:psi q n k}).

    We consider two cases. For any $\bi\in \cal B$ with the word length $l(\bi)\leq C'n$, by definition, 
     $\bi\in\bU(j)$ for some $j\leq C'n$. So we have
    \begin{equation}\label{equ:w less n}
     \P_{1\leq i\leq n}(\bI(i):\bI(i)\in \cal U^c, l(\bI(i))\leq C'n)\leq C' \P_{1\leq i\leq C'n }(\bU(i):\bU(i)\in \cal U^c)\leq C'\epsilon. 
     \end{equation}
    Consider any $\bi\in \cal B$ with $l(\bi)>C'n$. Take $\epsilon_1=\min\{\epsilon,\chi_1/2\}$. It follows from \cref{equ:bounded residual time} that
    \[ \chi_1(g_{\bi})\leq n-\log c_0\leq C'n(\chi_1-\epsilon_1)\leq  l(\bi)(\chi_1-\epsilon_1). \]
    As $\bi\in U(j)$ for some $j>C'n$, the large deviation estimate \cref{eq:ldp sigma g} yields
    \begin{align}\label{equ:w geq n}
    \P_{1\leq i\leq n }\{\bI(i): l(\bI(i))> C'n\}\leq & \frac{1}{n}\sum_{j>C'n}\P \{\bU(j): \chi_1(g_{\bU(j)})\leq l(\bU(j))(\chi_1-\epsilon_1)\}\nonumber\\
    \leq &\frac{1}{n}\sum_{j\geq C'n}Cq^{-\beta(\epsilon_1) j}. 
    \end{align}
     Combining \cref{equ:w less n,equ:w geq n}, we finish the proof of the lemma.
%
\end{proof}

\subsection{Entropy}
For a probability measure $\mu$ on a metric space $X$ and any countable measurable partitions $\cal E, \cal E'$ of $X$, define the entropy and conditional entropy by 
\begin{align*}
 H(\mu,\cal E):=&\sum_{E\in\cal E}-\mu(E)\log\mu(E),\\
 H(\mu,\cal E'|\cal E):=&H(\mu,\cal E'\vee \cal E)-H(\mu,\cal E),
\end{align*}
where $\cal E'\vee \cal E$ denotes the common refinement of $\cal E'$ and $\cal E$. 

We collect some basic properties of entropy. 
\begin{lem}\label{lem: entropy upper}
    Let $\cal E$ and $\cal F$ be two measurable partitions of $X$. If each atom in $\cal E$ intersects at most $k$ atoms in $\cal F$, then
    \[ H(\mu,\cal F)\leq H(\mu,\cal E)+O(\log k).   \]
\end{lem}

Suppose $X=\P(\R^2)$ or $L$ and $\calQ_n$ the $q$-adic partition on $X$.
This lemma implies that for any $m,n\in\N$,
    \begin{equation}\label{eq:entropy m n}
        |H(\mu,\calQ_n)-H(\mu,\calQ_m)|=O(|m-n|).
    \end{equation}

    Let $f$ and $g$ be two continuous maps on $X$, such that 
    $\|f-g\|_\infty\leq  q^{-n}$. 
    Denote by $f\mu$ and $g\mu$ the pushforward measures. Then we have 
    \begin{equation}\label{equ:2.14}
        |H(f\mu,\calQ_n)-H(g\mu,\calQ_n)|=O(1)
    \end{equation}
    (see \cite[(2.14)]{barany_hausdorff_2017}).

\begin{lem}[Lemma 4.3 in \cite{hochman_dimension_2017}]\label{lem: entpy prpty SL2R act}
Suppose $f$ scales $\supp \mu$ by $u>0$ with distortion $C>1$ (see \cref{def: scale distor} for definition). Then 
$$H(f\mu, \cal Q_{n-\log u})=H(\mu,\cal Q_n)+O(\log C).$$
\end{lem}

The following is about the concavity of entropy. Let $\delta\in [0,1]$, and $\mu_1,\mu_2$ be probability measures on $X$ and $\cal Q$ be a measurable partition. We have 
\begin{equation}\label{eqn: concav alm conv}
\begin{split}
\delta H(\mu_1,\cal Q)+(1-\delta)H(\mu_2,\cal Q)&\leq  H(\delta\mu_1 +(1-\delta)\mu_2,\cal Q)\\
&\leq \delta H(\mu_1,\cal Q)+(1-\delta)H(\mu_2,\cal Q)+H(\delta),
\end{split}
\end{equation}
where $H(\delta):=-\delta \log (\delta)-(1-\delta)\log(1-\delta)$. The quantity $H(\delta)$ is small if $\delta$ is sufficiently close to $0$ or $1$. 

\begin{lem}\label{lem:concave conditional}
    For any $k,n\in\N$, any $\epsilon_1,\epsilon_2\in(0,1)$ and any probability measures $\tau,\tau_1,\tau_2$ on $X$ that satisfy $\tau=(1-\epsilon_2)\tau_1+\epsilon_2\tau_2$ and $\diam(\supp\tau_1)=O(q^{-k+\epsilon_1n})$, we have
    \begin{align*}
\frac{1}{n}H\left(\tau,\calQ_{n+k}|\calQ_{k}\right)\geq & \frac{1}{n}H(\tau,\calQ_{n+k})-\left(\epsilon_2+\frac{H(\epsilon_2)}{n}+O(\epsilon_1)\right), \\
\frac{1}{n}H\left(\tau,\calQ_{n+k}|\calQ_{k}\right)\geq &\frac{1}{n}H(\tau_1,\calQ_{n+k})-(\epsilon_2+O(\epsilon_1)). 
\end{align*}
\end{lem}
\begin{proof} Note that the support of the measure $\tau_1$ intersects at most $O(q^{\epsilon_1 n})$ atoms of $\calQ_k$. By \cref{eq:entropy m n} and \cref{eqn: concav alm conv}, we have
    \begin{align*}
& \frac{1}{n}H\left(\tau,\calQ_{n+k}|\calQ_{k}\right)
\geq  (1-\epsilon_2)\frac{1}{n}H(\tau_1,\calQ_{n+k}|\calQ_{k}) \\
\geq  &(1-\epsilon_2)\frac{1}{n}H(\tau_1,\calQ_{n+k})-O(\epsilon_1)
\geq  \frac{1}{n}H(\tau,\calQ_{n+k})-(\epsilon_2+H(\epsilon_2)/n+O(\epsilon_1)).
\end{align*}

For the second one, we have
    \begin{align*}
& \frac{1}{n}H\left(\tau,\calQ_{n+k}|\calQ_{k}\right)
\geq  (1-\epsilon_2)\frac{1}{n}H(\tau_1,\calQ_{n+k}|\calQ_{k}) \\
\geq  &\frac{1}{n}H(\tau_1,\calQ_{n+k})-O(\epsilon_1)-\epsilon_2\frac{1}{n}H(\tau_1,\calQ_{n+k}|\calQ_{k})
\geq  \frac{1}{n}H(\tau_1,\calQ_{n+k})-(\epsilon_2+O(\epsilon_1)).\qedhere
\end{align*}
\end{proof}

\subsection{Entropy dimension of projection measure}\label{sec: def dim mesure}

A probability measure $\mu$ on a metric space $X$ is called \textit{exact dimensional} if for $\mu$-a.e. $x$, we have that
\[\lim_{r\rightarrow 0}\frac{\log\mu (B(x,r))}{\log r}  \]
exists and equals a constant independent of $x$. This constant is called the {\it{exact dimension of the measure $\mu$}}. Young \cite{young_dimension_1982} proved that if $\mu$ is exact dimensional, then its exact dimension coincides with $\dim\mu$, the Hausdorff dimension of $\mu$, which we defined in the Introduction. Hence, we also use $\dim \mu$ to denote the exact dimension of $\mu$.

In the same paper \cite{young_dimension_1982}, Young proved a fundamental property of exact dimensional measures (see also \cite{fan_relationships_nodate}).
\begin{lem}
    Let $\mu$ be an exact dimensional probability measure on $\R$. Then
    \[ \dim\mu=\lim_{n\rightarrow\infty}\frac{1}{n}H(\mu,\calQ_n). \]
\end{lem}
This lemma enables us to use the entropy to compute the dimension of the measure $\mu$. The limit on the right-hand side is usually called the {\it{entropy dimension}}.

Let $\mu$ be the $\nu$-stationary measure on $\P(\R^3)$ given as in \cref{sec:regularity of stationary measures}. In  \cite{ledrappier_exact_2021} and \cite{rapaport_exact_2021}, it is proved that for $\mu^-$-a.e. $V\in \P(\R^3)$ the projection measure $\pi_{V^\perp}\mu$ is exact dimensional.
Thus we have
\begin{lem}
\label{lem:entropy dimension}
    For $\mu^{-}$-a.e. $V\in\P(\R^3)$, we have
    \[ \dim\pi_{V^\perp}\mu=\lim_{n}\frac{1}{n}H(\pi_{V^\perp}\mu,\calQ_n), \]
    and they are of the same value.
\end{lem}
See \cite[Corollary 6.5]{ledrappier_exact_2021} and \cite[Theorem 1.3]{rapaport_exact_2021} for the proof.

The following lemma is an application of ``pseudo-continuity" property of fixed-scale entropy.
\begin{lem}\label{lem: Pseudo cont entpy}
For any $\epsilon>0$, any $m\geq M(\epsilon)$ large enough, any $n\geq N(\epsilon,m)$, we have 
\begin{equation}\label{eqn: BHR 3.15 ineq}
\inf_{V\in \P(\R^3)}\mathbb P\left \{\bU(n): \alpha-\epsilon\leq \frac{1}{m}H(\pi_{(g^{-1}_{\mathbf U(n)}V)^\perp}\mu, \mathcal Q_m)\leq \alpha+\epsilon\right\}>1-\epsilon.    
\end{equation}
\end{lem}
\begin{proof}
By \cref{lem:entropy dimension}, for $\mu^{-}$-almost every $V$, $\pi_{V^\perp}\mu$ has entropy dimension $\alpha$. We have the equidistribution of random walks (\cref{lem:equidistribution x}), i.e. $\mathbb{E}(\delta_{{A}_{\mathbf{U}(n)}^{-1}V})$ weakly converges to $\mu^{-}$ as $n$ goes to infinity and the convergence is uniform with respect to $V$. Then \eqref{eqn: BHR 3.15 ineq} can be shown immediately if the scale-$m$ entropy were continuous as a function of measure. Since it is not, we need the following argument. 

We can replace the entropy $\frac{1}{m}H(\cdot,\cal Q_m)$ by a continuous variant $F_m(\cdot)$ at the cost of $O(1/m)$, where $F_m$ is continuous in the weak$^*$ topology. For a possible choice of $F_m$, see the end of Section 2.8 in \cite{barany_hausdorff_2017}. Due to \cref{lem:entropy dimension}, for large $m$, we can obtain an open set $E(m)$ with $\mu^-(E(m))>1-\epsilon/2$ such that for $V\in E(m)$, we have $|F_m(\pi_{V^\perp}\mu)-\alpha|<\epsilon/2$. Applying \cref{lem:equidistribution x} to $E(m)$,  for large $n>N_2$, we have \cref{eqn: BHR 3.15 ineq} with $\frac{1}{m}H(\cdot,\cal Q_m)$ replaced $F_m(\cdot)$.
Then replace $F_m$ back to $\frac{1}{m}H(\cdot ,\cal Q_m)$. The resulting $O(1/m)$ error can be absorbed in $\epsilon$, which proves the claim. 
\end{proof} 

We recall the notion of Furstenberg entropy, which appears in the statement of our main results. 
\begin{defi}
 Let $\mu$ be a $\nu$-stationary measure on $\PP(\RR^3)$ or $\cF(\RR^3),$ (the flag variety whose precise definition is given in \cref{defi: flag}) the \textit{Furstenberg entropy} is given by
\begin{equation}\label{equ:f entropy}
                h_{\mathrm{F}}(\mu,\nu)=\int \log\frac{\dd g\mu}{\dd\mu}(\xi) \left(\frac{\dd g\mu}{\dd\mu}(\xi) \right)\dd\nu(g)\dd\mu(\xi). 
 \end{equation}
 \end{defi}
The Furstenberg entropy is mysterious and might be difficult to compute. To obtain a more calculable dimension formula, it is expected the Furstenberg entropy is equal to the random walk entropy, which holds in some concrete examples; see our second paper with Jiao \cite{JLPX}.

\section{Non-concentration on arithmetic sequences}\label{sec:non concentration}One of the main ingredients to show Theorem \ref{thm:lyapunov} is to prove the \textit{porosity} property of projections of the stationary measure, see \cref{sec:porosity}. To obtain the porosity of measures, a key step which we show below is to get the \textit{non-concentration at arithmetic sequences} property. We identify $\P(\R^2)$ with $\R/\Z$ (see \cref{eqn:identification}).
\begin{defi}\label{defi.continuity}
 A probability measure $\tau$ on $\P(\R^2)$ is called \textit{non-concentrated on arithmetic sequences} (across scales) 
 if for any $\delta>0$, there exist $l,k_0\in \mathbb{N}$ such that for all $k>k_0$, we have 
  \[ \tau\left(\bigcup_{0\leq n<q^{k}} B\left(\frac{n}{q^k},\frac{1}{q^{k+l}}\right)\right)\leq \delta. \]
\end{defi}

In \cite{barany_hausdorff_2017}, a property called \textit{uniformly continuous across scales} (UCAS) was defined for measures on metric spaces, which plays an important role in the proof of porosity in their setting (Lemma 3.5 in \cite{barany_hausdorff_2017}). This property is quite strong and can imply the uniform H\"older regularity of measures. However, it is not clear whether it holds for a (projection of a) general stationary measure. In our paper, we apply the knowledge of Fourier decay of Furstenberg measures in \cite{li_fourier_2018} to obtain the weaker substitution of UCAS, i.e. non-concentration at arithmetic sequences, which implies Lemma \ref{lem: BHR 3.5} that plays a role as that of Lemma 3.5 in \cite{barany_hausdorff_2017}.



Recall that the Fourier coefficients of $\tau\in \bf{P}(\P(\R^2))$ are defined as follows: for any $m\in\Z$,
\begin{equation*}
\hat{\tau}(m):=\int_{x\in \R/\Z} e^{2\pi imx}\dd \tau(x). 
\end{equation*}
We say $\tau$ is a Rajchman measure (or Fourier decay) if $\hat\tau(m)\rightarrow 0$ as $|m|\rightarrow\infty$.
\begin{lem}\label{lem: Rjch doubling}
    Let $\tau\in \bf{P}(\P(\R^2))$ be a Rajchman measure. Then $\tau$ satisfies Definition \ref{defi.continuity}. In particular, we can choose $l,k_0\in \N$ in \cref{defi.continuity} such that 
    \begin{equation*}
    \hat{\tau}(jq^{k_0})\leq \delta^3/(10^4q^2)\,\,\, \text{ for any}\,\,\,j\neq 0 \,\,\,\text{and}\,\,\, 10< q^{l}\delta\leq 10q. 
    \end{equation*}
\end{lem}
\begin{proof}
    Let $f$ be a smooth function supported on $B(0,2/q^{l})\subset \R/\Z$ such that it equals to $1$ on $B(0,1/q^l)$ and $\|f''\|_{L^1}\leq 4q^{2l}$. Define $F:\R/\Z\to \R$ by $F(x)=\sum_{0\leq n< q^k} f(xq^k-n) $. Then 
    \begin{align*}
        \tau\left(\bigcup_{0\leq n<q^k} B\left(nq^{-k},q^{-(k+l)}\right)\right)\leq \int F\dd\tau=\sum_{m\in \Z} \hat\tau(m)\hat{F}(-m),
    \end{align*}
    where $\hat{F}(m)=\int_{\R/\Z}F(x)e^{i2\pi mx}dx$.
    We have
    \begin{align*}
     \hat{F}(m)&=\int_{\R/\Z} \sum_{0\leq n< q^k} f(q^kx-n)e^{i2\pi  mx}\dd x=\frac{1}{q^k}\int_{\R/\Z} \sum_{1\leq n\leq q^k} f(y) e^{i2\pi m(y+n)/q^k}\dd y\\
     &=\frac{\sum_{1\leq n\leq q^k} e^{i2\pi mn/q^k}}{q^k}\int_{\R/\Z}  f(y) e^{i2\pi my/q^k}\dd y.
    \end{align*}
    It is non-zero only if $m=jq^k$. Therefore, we have $\hat{F}(jq^k)=\hat{f}(j)$ and
    \[\tau\left(\cup_{0\leq n<q^{k}} B\left(nq^{-k},q^{-(k+l)}\right)\right)\leq \sum_{j\in \Z} \hat{\tau}(q^kj)\hat{f}(-j). \]
    When $j=0$, we have $\hat{f}(0)=\int f\leq  4q^{-l}$. For the sum of the remaining terms, we have
    \begin{equation*}
    \left|\sum_{j\neq 0} \hat{\tau}(q^kj)\hat{f}(j)\right|\leq \left(\sup_{j\neq 0}|\hat\tau(q^kj)|\right)\sum_{j\neq 0} |\hat{f}(j)|. 
    \end{equation*}
    Since $f$ is smooth, we have 
    \[ \sum_{j\neq 0} |\hat{f}(j)|\leq \sum_{j\neq 0}\frac{1}{j^2}\|f''\|_{L^1}  \leq 4\|f''\|_{L^1}\leq 16q^{2l}<\infty. \]
   Since $\tau$ is a Rajchman measure, we have  for any $k$ large enough, $\sup_{j\neq 0}|\hat\tau(q^kj)|\leq \delta^{3}/(10^4q^2)$. Choose $l\in \N$ such that $10<q^l\delta\leq 10 q$. We have
    \[ \sum_{j\in \Z} \hat{\tau}(q^kj)\hat{f}(j)\leq \int f+\left(\sup_{j\neq 0}|\hat\tau(q^kj)|\right)\sum_{j\neq 0} |\hat{f}(j)| \leq 4/q^l+(\delta^3/10^4q^2)\times 16q^{2l}<\delta.  \]
    Therefore $\tau$ satisfies Definition \ref{defi.continuity}.
\end{proof}

%
For higher dimensional cases, the Fourier decay of Furstenberg measures on flag varieties and projective spaces is known in \cite{li_fourier_2018}, cf. \cref{sec: appd F decay} for $3$-dimensional case. For our purpose, we only consider the flag variety $\mathcal F(\R^3)$ here.
\begin{defi}\label{defi: flag}
\begin{enumerate}[1.]
    \item (Definitions and notations)   Let 
\[\cF(\RR^3):=\lb{(V_1,V_2):
V_1 \sbs V_{2}, 
V_i\text{ is a linear subspace of }\RR^3 \text{ of dimension }i}\] 
be the space of flags in $\RR^3$,  which admits a natural $\SL_3(\RR)$-action.  
We abbreviate $\cF(\RR^3)$ to $\cF$ in later discussion. 
\item (Metrics on $\cF$) For $\eta=(V_1(\eta),V_2(\eta)),\eta'=(V_1(\eta'),V_2(\eta'))\in \cF$, let 
\begin{equation*}
d_1(\eta,\eta')=d_{\P(\R^3)}(V_1(\eta),V_1(\eta'))\,\,\, \text{and}\,\,\, d_2(\eta,\eta')=d_{\P(\R^3)}(V_2(\eta),V_2(\eta')).
\end{equation*}
Here $d_{\P(\R^3)}(V_2(\eta),V_2(\eta'))$ is the Hausdorff distance between two projective lines in $(\P(\R^3),d)$. Then the metric $d$ on $\cF$ can be defined as the following:
 \begin{equation}\label{equ.d eta eta'}
 d(\eta,\eta'):=\max\{ d_1(\eta,\eta'),d_2(\eta,\eta')\}.
 \end{equation}
\item (Lifting functions from $\P(\R^3)$)A function $\varphi$ on $\cF$ is called \emph{lifted from $\P(\R^3)$} if $\varphi(\eta)$ only depends on the first coordinate $V_1(\eta)$.
\end{enumerate}
\end{defi}
Recall that $\mu$ is the $\nu$-stationary measure on $\P(\R^3)$. In the following, for any $V\in \P(\R^3)$, by identifying $\P(V^\perp)$ isometrically with $\P(\R^2)$, we view $\pi_{V^\perp}\mu$ as a measure on $\P(\R^2)$. 
\begin{prop}\label{prop: unfm dbling}For any $\delta>0$, there exist $l,k_0\in \N$ such that for any $V\in \P(\R^3)$ and any $k\geq k_0$
\begin{equation}\label{eqn: unifm dbling}
    \pi_{V^\perp}\mu(\bigcup_{0\leq n<q^k}B(nq^{-k},q^{-(k+l)}))\leq \delta.    
\end{equation} 
\end{prop}

\begin{proof}
By the proof of \cref{lem: Rjch doubling}, it is sufficient to prove $\widehat{\pi_{V^\perp}\mu}(m)\rightarrow 0$ uniformly for all $V$ as $|m|\rightarrow\infty$. To obtain it our strategy is to apply  \cref{thm:fourier}. Recall the constants $\epsilon_0,\epsilon_1>0$ from \cref{thm:fourier}. Let $\epsilon_2=\epsilon_0/8$.
Fix an arbitrary $V\in \P(\R^3)$ and an integer $m\in \Z$ such that  $|m|^{\epsilon_2}>10$. 
\paragraph{Step 1: Define, lift and localize  $\varphi$.}

Consider the map 
\begin{align}
\label{eqn:Fourier projection}
\varphi: \P(\R^3)-\{V\}&\to \P(V^\perp),\\ x &\mapsto \pi_{V^\perp} x.\nonumber
\end{align}
Then $\varphi$ can be lifted as a map from $\cF-\{\eta=(V_1(\eta),V_2(\eta)): V_1(\eta)\neq V\}$. From now to the end of the proof of Proposition \ref{prop: unfm dbling} we only consider the lifting of $\varphi$ and we still denote by $\varphi$. By the definition, $\varphi$ is neither a well-defined map on whole $\cF$ nor a well-defined real-valued function, to which we cannot apply Fourier decay (i.e. \cref{thm:fourier}) directly. To overcome this difficulty we localize $\varphi$ through partition of unity.

 Let $\rho_m(t)$ be a $C^1$-function on $[0,1]$ such that $\rho_m(t)=0$ if $t\leq 1/|m|^{\epsilon_2}$ and $\rho_m(x)=1$ if $t\geq 2/|m|^{\epsilon_2}$. We further assume that the Lipchitz constant of $\rho_m$ satisfies $\Lip(\rho_m)\leq 2|m|^{\epsilon_2}$. 
Let $r_m$ be a cutoff function on $\cal F$ given by 
\begin{equation}\label{equ.rm}
r_m(\eta)=\rho_m(d_{\P(\R^3)}(V_2(\eta),V)).
\end{equation}
Roughly speaking the support of $r_m$ consists of points $\eta=(V_1(\eta),V_2(\eta))$ such that $V$ is not too close to $V_2(\eta)$; in particular, $V$ is also not too close to $V_1(\eta)$ and we have
\begin{equation}\label{equ: x V}
    d(V_1(\eta),V)\geq 1/|m|^{\epsilon_2} \text{ if $r_m(\eta)>0$}.
\end{equation}
Our idea to localize $\varphi$ is to locally identify $\varphi$ with smooth real-valued functions $\varphi_j$ on $\calF$ such that $\varphi=\varphi_j\mod\Z$. Then the value of the oscillatory integral (i.e. \cref{eqn: Osl int}) remains the same.
Here is how we proceed: take smooth functions $p_1$ and $p_2$ on $\R/\Z$ as a partition of unity with respect to open intervals $d(z,\Z)< 1/3$ and $d(z,\frac{1}{2}+\Z)<1/3$. Let $r_{m,j}(\eta)=r_m(\eta)p_j(\pi_{V^\perp}(V_1(\eta)))$ with $\eta=(V_1(\eta),V_2(\eta))$, for $j=1,2$. 
Then on a neighbourhood of the support of $r_{m,j}$, the map $\varphi$ takes value in a proper subset of $\R/\Z$ which can be identified with a subset of $\R$ continuously. Therefore on the neighbourhood of the support of $r_{m,j}$, $\varphi$ can be identified with smooth real-valued lifted functions $\varphi_{j}$, without changing the integral. 
\paragraph{Step 2: Verify \cref{defi:c r good simple}.}  Our plan is to apply \cref{thm:fourier} to $\varphi_j$. 
Since both $\varphi_j$ are functions lifted from $\P(\R^3)$, we only need to verify the conditions in \cref{defi:c r good simple}. A remarkable fact for $\cF$ is that for any $\eta\in \cF$, its coordinate provides not only its projection $V_1(\eta)$ to $\P(\R^3)$, but also an important tangent direction $V_{\alpha_2}\subset T_\eta\cF$ (see \cref{sec: appd F decay}). 
\begin{enumerate}[1.]
\item (Verify  \cref{G1} and \cref{G4} with $C=64m^{4\epsilon_2}$). Using the spherical law of sines, we have
\begin{align*}
\partial_{1}\varphi_j(\eta)=\lim_{t\rightarrow 0}\frac{\varphi_j(V_1)-\varphi_j(y(t))}{t}=\lim_{t\rightarrow 0}\frac{\angle_V(V_1,y(t))}{d_R(V_1,y(t))}=\lim_{t\rightarrow 0}\frac{\sin(\angle_V(V_1,y(t)))}{d(V_1,y(t))}=\lim_{t\rightarrow 0}\frac{\sin(\angle_{V_1}(V,y(t)))}{d(y(t),V)},\end{align*}
where $y(t)$ is a smooth curve on $\P(\R^3)$ which is contained in $V_2$  with $d_R(V_1,y(t))=t$. The angle $\angle_{V_1}(V,y(t))$ is equal to angle between two hyperplane $V_2=V_1\oplus y(t)$ and $V_1\oplus V$, so
\begin{align*}
\partial_{1}\varphi_j(\eta)=\frac{d(V_2,(V_1\oplus V))}{d(V_1,V)}=\frac{d(V_2,V)}{d(V_1,V)^2},
\end{align*}
where we use again spherical law of sines to obtain $d(V_2,(V_1\oplus V))=d(V_2,V)/d(V_1,V)$\footnote{ The Hausdorff distance $d(V_2,(V_1\oplus V))$ equals $\sin(\angle_{V_1}(V,y(t)))$. Let $V'$ be the point in $V_2$ such that $d(V,V')=d(V,V_2)$. Then $\angle_{V'}(V,V_1)=\pi/2$. Hence by spherical law of sines, $\sin(\angle_{V_1}(V,y(t)))/d(V,V_2)=\sin(\angle_{V_1}(V,y(t)))/d(V,V')=\sin(\pi/2)/d(V,V_1)$. } for the last equality. By the choice of $r_{m,j}$, on its support, due to \cref{equ.rm} and \cref{equ: x V}, we have (See \cref{fig:alpha1})
\[\partial_{1}\varphi_j(\eta)\in[ 1/|m|^{\epsilon_2},|m|^{2\epsilon_2}], \]
which implies \cref{G1} and \cref{G4} with $C=64m^{4\epsilon_2}$.
\begin{figure}[!ht]
    \centering
    \includegraphics{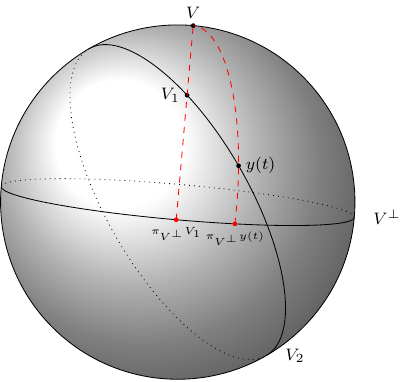}
    \caption{$\alpha_1$-direction}
    \label{fig:alpha1}
\end{figure}
\item (Verify \cref{G2}). \cref{G2} can be obtained directly from \cref{lem:projection}.
\item (Verify \cref{G3}). We have
\[\partial_1\varphi_j(\eta)-\partial_1\varphi_j(\eta')=\frac{d(V_2, V)}{d(V_1,V)^2}-\frac{d(V_2',V)}{d(V_1',V)^2}  =\frac{d(V_2, V)d(V_1',V)^2-d(V_2', V)d(V_1,V)^2}{d(V_1,V)^2d(V_1',V)^2}. \]
where $V_i':=V_i(\eta')$.
 Due to the triangle inequality, the numerator is less than
\[|(d(V_2,V)-d(V_2',V))d(V_1',V)^2+d(V_2',V)(d(V_1',V)^2-d(V_1,V)^2)|\leq 2(d(V_2,V_2')+d(V_1,V_1')). \]
Due to \cref{equ: x V}, for $\eta,\eta'$ in the $1/C$-neighborhood of the support of $r$, the denominator is greater than $1/(2|m|^{\epsilon_2})^4$. Hence, we have
\[ |\partial_1\varphi_j(\eta)-\partial_1\varphi_j(\eta')|\leq 64|m|^{4\epsilon_2}d(\eta,\eta') \]
which is \cref{G3}, where $d(\eta,\eta')$ is defined in \cref{equ.d eta eta'}.
\item (Estimate the Lipschitz constants of $r_{m,j}$).
Recall a basic estimate of the Lipschitz constant of a product of two Lipchitz continuous functions $f,g$: 
\[\Lip(fg)\leq \Lip(f)|g|_\infty+|f|_\infty\sup_{\eta\neq\eta',\eta,\eta'\in\supp f}\frac{|g(\eta)-g(\eta')|}{d(\eta,\eta')}. \]
By this estimate and (2) in \cref{lem:projection},
\begin{align*}
\Lip(r_{m,j})&\leq \Lip(r_m)+\sup_{\eta\neq \eta',\eta,\eta'\in\supp r_m}\frac{|p_j(\pi_{V^\perp}(V_1(\eta)))-p_j(\pi_{V^\perp}(V_1(\eta')))|}{d(\eta,\eta')}
\\ 
&\leq |m|^{\epsilon_2}+\Lip(p_j)|m|^{\epsilon_2}\leq 4|m|^{\epsilon_2}.    
\end{align*}
\end{enumerate}
\paragraph{Step 3: Estimate of $\widehat{\pi_{V^\perp}\mu}(m)$.}
From Step 2, each $\varphi_j$ is $(m^{\epsilon_0},r_{m,j})$-good (due to the choice of $\epsilon_2$). Then for each $j$ and all $m$ with $|m|$ large enough, by Theorem \ref{thm:fourier} we have
\[|\int e^{2\pi im\varphi_j(\eta)}r_{m,j}(\eta)\dd\mu_{\cal F}(\eta)|\leq |m|^{-\epsilon_1},  \]
where $\mu_{\cal F}$ is the $\nu$-stationary measure on $\cal F$, whose projection on $\P(\R^3)$ is $\mu$.
Therefore, combined with \cref{lem: Hold reg proj mes} and \cref{equ: x V}, we obtain
\begin{eqnarray*}
    |\widehat{\pi_{V^\perp}\mu}(m)|&=& |\int e^{2\pi im\varphi(\eta)}\dd\mu_{\cal F}(\eta)|\\ 
    &\leq &\left|\int e^{2\pi im\varphi(\eta)}(r_{m,1}(\eta)+r_{m,2}(\eta))\dd\mu_{\cal F}(\eta)\right|+ \left|\int e^{2\pi im\varphi(\eta)}(1-r_{m,1}(\eta)-r_{m,2}(\eta))\dd\mu_{\cal F}(\eta)\right|\\
    &\leq &\sum_j\left|\int e^{2\pi im\varphi_j(\eta)}r_{m,j}(\eta)\dd\mu_{\cal F}(\eta)\right|+\mu_{\cal F}(\{r_m\neq 1\})\\
    &\leq &\sum_j\left|\int e^{2\pi im\varphi_j(\eta)}r_{m,j}(\eta)\dd\mu_{\cal F}(\eta)\right|+\mu(\{V_2(\eta): d(V_2(\eta),V)<1/|m|^{\epsilon_2}\})\ll |m|^{-\beta},
\end{eqnarray*}
for some $\beta>0$ and uniform for all $V$. The proof is complete.
\end{proof}


\begin{prop}\label{prop: unfm dbling g}
For any $\delta>0$, there exist $l,k_0\in \N$ such that for any $V\in \P(\R^3)$, any $k\geq k_0$ and $g\in \SL_2(\R)$, we have
\begin{equation}\label{eqn: unifm dbling g}
    g\pi_{V^\perp}\mu\left(\bigcup_{0\leq n<q^k}B(nq^{-(k+[\chi_1(g)])},q^{-(k+[\chi_1(g)]+l)})\right)\leq \delta.    
\end{equation} 
\end{prop}
\begin{proof}

Abbreviate $\pi_{V^\perp}\mu$ by $\tau$. It follows from \cref{lem: Hold reg proj mes} that there exists $s\in (0,1)$ such that $\tau(b(g^-,s)^c)<\delta/10$ and $s$ only depends on $\mu$ and $\delta$. We decompose the measure $\tau$ as follows:
\begin{equation*}
\tau=\lambda\cdot \tau_{b(g^-,s)^c}+(1-\lambda)\cdot \tau_{b(g^-,s)},
\end{equation*}
where $\lambda=\tau(b(g^-,s)^c)$. We denote $\tau_{b(g^-,s)^c}$ by $\tau_I$ to simply to notation. 
Set 
$$
t=[
\chi_1(g)-2|\log s|-1]=[
\chi_1(g)+2\log s-1].$$ 
Let $S_t:\R \to \R, x\mapsto q^tx$ be the scaling map given as in \cref{defi:st}. \cref{lem:action g} gives that $gb(g^-,s)\subset B(g^+,q^{-\chi_1(g)}/s^2)$.
If $t$ is a non-negative integer, then $S_t$ induces a map $\R/\Z\to \R/\Z$, which we still denote by $S_t$ by abusing the notation. \cref{lem:action g} gives that $gb(g^-,s)\subset B(g^+,q^{-\chi_1(g)}/s^2)$. Hence the restriction of $S_t$ on $gb(g^-,s)$ is a diffeomorphism. If $t$ is negative, then $S_t$ is a contracting map. Identifying $\R/\Z$ (hence $\P(V^\perp)$) with $[-\frac{1}{2},\frac{1}{2})$ such that $V^+_g=0$, we have that $gb(g^-,s)\subset (-\frac{1}{2},\frac{1}{2})$ is an interval containing $0$. Hence $S_t$ restricted on $gb(g^-,s)$ is also a diffeomorphism to its image in $gb(g^-,s)$.  

We have
\begin{align*}
g\tau_\bI\left(\cup_{n}B\left(\frac{n}{q^{[\chi_1(g)]+k}},\frac{1}{q^{[\chi_1(g)]+k+\ell}}\right)\right)&=S_{t}(g\tau_\bI)\left(S_{t}(\cup_{n}B\left(\frac{n}{q^{[\chi_1(g)]+k}}, \frac{1}{q^{[\chi_1(g)]+k+\ell}})\right) \right)\\
&=S_{t}g\tau_\bI\left(\cup_{n}B\left(\frac{n}{q^{[\chi_1(g)]+k-t}}, \frac{1}{q^{[\chi_1(g)]+k-t+\ell}}\right) \right),
\end{align*}
where we only sum up $n$'s such that the interval $B\left(\frac{n}{q^{[\chi_1(g)]+k}},\frac{1}{q^{[\chi_1(g)]+k+\ell}}\right)$ intersects the support of $\tau_\bI$.

We use \cref{lem: sl2 bounded distortion} to obtain distortion estimates. 
For $y\in b(g^-,s)$, we have
\begin{equation}\label{equ:g'' bounded distortion}
  |(S_{t}g)'y|=q^{t}|g'y|\in [s^2/q,1/q], |(S_{t}g)''(y)|=q^{t}|g''(y)|\leq 10/(qs).
  \end{equation}
Let $p_1$ be a cuffoff function which equals $1$ on $b(g^-,s)$ and $0$ outside $b(g^-,s/2)$. Let $p_1\tau$ be the measure on $\P(V^\perp)$ defined by $p_1\tau (f)=\int p_1\cdot f d\tau$ for any measurable function $f$ on $\P(V^\perp)$. Consider the map $\psi: \P(\R^3)-\{V\}\to \P(V^\perp)$ defined by
\[\psi(x)=(S_{t}g)\pi_{V^\perp}(x)=S_tg\varphi(x) \]
with $\varphi$ given as in \cref{eqn:Fourier projection}. Then $\psi$ can be lifted to a function defined on $\mathcal{F}-\{\eta=(V_1(\eta),V_2(\eta)):V_1(\eta)\neq V\}$.

Take any integer $m$ with $|m|^{\epsilon_2}>4$, where $\epsilon_2=\epsilon_0/8$ with $\epsilon_0$ given as in \cref{thm:fourier}. Let $r_{m,1}(\eta)=r_m(\eta)p_1(\pi_{V^\perp}(\eta)) $ where $\eta=(V_1(\eta),V_2(\eta))$  and $r_m$ is defined as in \cref{equ.rm}.
Here is the key observation: 
since $S_tg$ on $b(g^-,s)$ is a bi-Lipschitz map to its image with bounded $C^2$-norm (\cref{equ:g'' bounded distortion}), if $\varphi$ is $(C,r_{m,1})$-good as in \cref{defi:c r good simple}, then $(S_tg)\varphi$ is $(C',r_{m,1})$-good with $C'=(10q/s)^2C $. 
As $s$ is fixed, we can verify as in \cref{prop: unfm dbling} that $\varphi$ is $(64m^{\epsilon_0/2},r_{m,1})$-good and hence $\psi=(S_tg)\varphi$ is $(m^{\epsilon_0},r_{m,1})$-good for all integer $m$ with $|m|$ sufficiently large. Then the same argument as in \cref{prop: unfm dbling} yields that for all integer $m$ with $|m|$ large, 
\[|\widehat{S_tg(p_1\tau) }(m)|\leq |\int e^{2\pi im\psi(\eta)}r_{m,1}(\eta)\dd\mu_\calF(\eta)|+\mu_{\cal F}(\{r_{m}=1\}^c)\ll m^{-\beta}.  \]
\cref{lem: Rjch doubling} implies that there exist $l,k_0\in \N$ (depending on $\delta$) such that for all $k\in N$ with $[\chi_1(g)]+k-t>k_0$,
\begin{align*}
&S_{t}g\tau_\bI\left(\cup_{n}B\left(\frac{n}{q^{[\chi_1(g)]+k-t}}, \frac{1}{q^{[\chi_1(g)]+k-t+\ell}}\right)\right)\\
\leq &\frac{1}{1-\lambda}S_{t}g(p_1\tau)\left(\cup_{n}B\left(\frac{n}{q^{[\chi_1(g)]+k-t}}, \frac{1}{q^{[\chi_1(g)]+k-t+\ell}}\right) \right) <\delta/2.
\end{align*}
Note that $g\tau=(1-\lambda)\cdot g\tau_\bI+\lambda \cdot g\tau_{
b(g^-,s)^c}$ with $\lambda<\delta/10$. The proof is complete.
\end{proof}

\section{Convolution and porosity}\label{sec:porosity}

\begin{defi}
Let $\alpha,\epsilon>0$ and $m,i_1, i_2\in \mathbb{N}$ with $i_1\leq i_2$. We call a probability measure $\tau$ on $\R$ or $\P(\R^2)$ is $(\alpha,\epsilon,m)$-\textit{entropy porous} from scale $i_1$ to $i_2$  if
\[\P_{i_1\leq j\leq i_2}\left\{\frac{1}{m}H(\tau_{x,j},\calQ_{m+j})<\alpha+\epsilon \right\}>1-\epsilon. \]
\end{defi}

The main result of this part is the entropy porosity of projections of stationary measures, which is a key ingredient in the entropy growth argument.

\begin{prop}\label{prop: BHR 3.19}For every $\epsilon > 0,\,\, m\geq  M(\epsilon), \,\,k \geq  K(\epsilon,m), \,\,n\geq  N(\epsilon,m,k)$,
\begin{equation}
 \inf_{V\in \P(\R^3)} \P_{1\leq i\leq n}\left\{
  \pi_{V^\perp} g_{\bI(i)}\mu \text{ is }(\alpha,\epsilon,m)\text{-entropy porous from scale $i$ to $i+k$}
 \right\}>1-\epsilon .
\end{equation}
\end{prop}

\subsection{General results on entropy porosity}
In this section, we list some general lemmas in \cite[Section 3.1]{barany_hausdorff_2017} which show that if a measure $\tau$ decomposes as a convex combination of measures $\tau_i$ whose support has a short length, then many properties of the $\tau_i$, and specifically their entropies, are inherited by the $q$-adic components of $\tau$. 

In the following lemma, one fixes a scale $k$, a short scale $k+l_0$ for the length of the support of the measures $\tau_i$ which make up $\tau$, and an even shorter scale $k+m$ at which the entropy appears. The dependence of the parameters is that $l_0$ is large depending on $\tau$ and $\delta$, $m\gg l_0$, and $k$ is arbitrary.
\begin{lem}\label{lem: BHR 3.5}  For every $\epsilon > 0$, there exists $\delta > 0$ with the following property. Let $\tau \in \bf P(\P(\R^2))$ be written as a convex combination $\tau = (\sum_{i=1}^{N} p_i \tau_i)+p_0\tau_0$, with $p_0<\delta$, and suppose that for $l_0=l_0(\tau,\delta)$, some $m,k\in \mathbb N$ and $\alpha>0$,
\begin{enumerate}[(1)]
    \item $\frac {1}{m} H(\tau_i,\mathcal Q_{k+m}) > \alpha$ for every $i = 1,..., N$.
    \item $\diam(\supp \tau_i) \leq q^{-(k+l_0)}$ for every $i=1,\ldots,N$.
    \item \[ \tau\left(\bigcup_n B\left(\frac{n}{q^k},\frac{1}{q^{k+l_0}}\right)\right)\leq \delta. \]
\end{enumerate}

Then
\[\mathbb P_{j=k}\left(\frac{1}{m}H(\tau_{x,j},\mathcal Q_{j+m})>\alpha-\epsilon \right)>1-\epsilon. \]
\end{lem}
We view $\tau_0$ and $\tau_i$ with $1\leq i\leq N$ as the \textit{bad} part and the \textit{good} part of $\tau$ respectively, see \cref{defi: gd bd decomp pure measure} in the next subsection. \cref{lem: BHR 3.5} can be proved as in \cite[Lemma 3.5]{barany_hausdorff_2017}. We replace UCAS condition  (\cite[Definition 3.11]{barany_hausdorff_2017}) by the current weaker condition $(3)$ (\cref{defi.continuity}), which is sufficient for the proof.

As in \cite[Lemma 3.7]{barany_hausdorff_2017}, we use the above lemma to deduce the entropy porosity.

\begin{lem}\label{lem: BHR 3.7'}
For every $\epsilon>0$, there exists $\delta>0$ with the following property. Let $m,l\in \mathbb{N}$ and $n>n(m,l)$ be given, and suppose that $\tau\in \mathbf{P}(\mathbb{P}(\mathbb{R}^2))$ is a measure such that for $(1-\delta)$-fraction of $1\leq k\leq n$, we can write $\tau=\sum p_i\tau_i$ so as to satisfy the conditions of the previous lemma for the given $\delta$, $l_0$, $m$ and $k$. Assume further that  
\begin{equation*}
|\frac{1}{n}H(\tau,\mathcal{Q}_n)-\alpha|<\delta.
\end{equation*}
Then $\tau$ is $(\alpha,\epsilon,m)$-entropy porous from scale $1$ to $n$.

\end{lem}

\begin{lem}\label{lem: BHR 3.9}For every $\epsilon > 0$, there exists $\delta > 0$ with the following property. Let $\tau \in \bf P(\P(\R^2))$ be written as a convex combination $\tau = (\sum_{1\leq i\leq N} p_i \tau_i)+p_0\tau_0$, with $p_0<\delta$, and suppose that for $l_1=l_1(\tau,\delta)$, some $k,p\in \mathbb N,\,\, m>m(\epsilon,p)$ and $\beta>0$,
\begin{enumerate}[(1)]
    \item $\frac {1}{m} H(\tau_i,\mathcal Q_{k+m}) < \beta$ for every $i = 1,\ldots, N$.
    \item $\diam\supp \tau_i\leq q^{-(k+l_1)}$ for every $i=1,\ldots, N$.
    \item \[ \tau\left(\bigcup_n B\left(\frac{n}{q^k},\frac{1}{q^{k+l_1}}\right)\right)\leq \delta. \]
    \item Every interval of length $q^{-k}$ intersects the support of at most $p$ of the measures $\tau_i$.
\end{enumerate}

Then
\[\mathbb P_{j=k}\left\{\frac{1}{m}H(\tau_{x,j},\mathcal Q_{j+m})<\beta+\epsilon \right\}>1-\epsilon. \]
\end{lem}
The proof is the same as that of \cite[Lemma 3.9]{barany_hausdorff_2017}, where they used almost convexity of entropy to replace the concavity of entropy in the proof of \cite[Lemma 3.5]{barany_hausdorff_2017}. Notice that condition $(3)$ is changed. See the discussion below \cref{lem: BHR 3.5}.

\subsection{Decomposition of a general measure}



\cref{lem: BHR 3.5} motivates the following decomposition of measures for the projective action on $\P(\R^d)$. This decomposition is essential because the projective action is more subtle: it only contracts outside a neighbourhood of a hyperplane (\cref{lem:action g}). This decomposition allows us to estimate the entropy of measures that are pushed forward by the projective action.
\begin{defi}\label{defi: gd bd decomp pure measure}Let $\tau$ be a probability measure on a metric space $X$. For $s,\delta>0$, a convex decomposition $\tau=\theta\tau_{\bf {II}}+(1-\theta)\tau_{\bf I}$ is called an \emph{$(s,\delta)$-decomposition} of $\tau$, if 
$\theta\leq\delta$ and $\rm{diam}(\supp\tau_{\bf I})\leq s$. We call $\tau_{\bf I}$ and $\tau_{\bf{II}}$ the \emph{good} part and the \emph{bad} part of $\tau$, respectively. 

Moreover, if the good part $\tau_{\bf I}$ satisfies $\frac{1}{m}H(\tau_{\bf I},\calQ_{m+l})>\alpha$, then we call the decomposition an \emph{$
(s,\delta)$-decomposition with $\alpha$-entropy concentration at scale $(m,l)$}. 
\end{defi}

We will be most interested in the following decomposition coming from the action of $\SL_d(\R)$ on $\P(\R^d)$. 

\begin{defi}\label{defi: gd bd decomp} Let $d\geq 2$. Let $h\in \GL_d(\R)$ be such that $\sigma_1(h)>\sigma_2(h)$ and $\tau$ be a probability measure on $\P(\R^d)$. The $r$-attracting decomposition of the pair $(h,\tau)$ is defined as follows:
\begin{equation}
     {h}\tau=\theta\cdot  {h}(\tau_{b(h^-,r)^c})+(1-\theta)\cdot  {h}(\tau_{b(h^-,r)}),
\end{equation}
where $\theta=\tau(b(h^-,r)^c)$. 
We call ${h}(\tau_{b(h^-,r)^c})$ (resp. ${h}(\tau_{b(h^-,r)})$ ) the \emph{$r$-repelling part} (resp. \emph{the $r$-attracting part}) of 
 $(h,\tau)$.
\end{defi}

We need a lemma of a more precise estimate of the support.
\begin{lem}
For any $h\in\SL_2^\pm(\R)$ with $\sigma_1(h)>\sigma_2(h)$ and  any $r\in [q^{-\chi_1(h)},1/10)$, we have
\begin{equation}\label{equ:supp attracting}
B(h^+,q^{-\chi_1(h)}/2r) \subset \supp\ h(\tau_{b(h^-,r)})\subset B(h^+,q^{-\chi_1(h)}/r^2) .
\end{equation}
\end{lem}
\begin{proof}
Due to \cref{lem:action g}, we obtain the second relation.

From some elementary computations, we can obtain a lower bound of the support. Without loss of generality, we may assume $h=\diag(q^{\chi_1(h)/2},q^{-\chi_1(h)/2})$. In view of \cref{two metrics}, by a straightforward computation, we have that the diameter of $hb(h^-,r)$ is bounded above by  $\sin\left(\arctan\left(\frac{q^{-\chi_1(h)}}{\tan (\arcsin(r))}\right)\right)$.
\end{proof}

The following lemma links the decomposition given in \cref{defi: gd bd decomp pure measure} and the one given in \cref{defi: gd bd decomp}.


\begin{lem}\label{lem: gd decomposition} Let $\tau$ be a probability measure on $\P(\R^d)$. Suppose there exist $C,\beta>0$ such that for any $R\geq 0$ and hyperplane $W$ in $\P(\R^d)$, we have
\begin{equation}\label{eqn: Hold reg mes}
\tau(B(W,R))\leq CR^\beta.
\end{equation}
Then for any $g\in \SL_d(\R)$ and $0<r\leq 1/2$, the $r$-attracting decomposition of $(g,\tau)$ is an $(q^{-\chi_1(g)}/r^2,Cr^\beta)$-decomposition of the pushforward measure $g\tau$.
\end{lem}
\begin{proof}
We have $\diam (B(g^+,q^{-\chi_1(g)}/r^2))\geq \diam (gb(g^-,r))$ (see \cref{lem:action g}). And by \cref{eqn: Hold reg mes}, we have $\tau (b(g^-,r)^c)\leq Cr^\beta$. 
\end{proof}

\begin{lem}\label{lem: Entpy gd pt}
Let $g\in \SL_2(\mathbb{R})$ and $\tau$ be a probability measure that satisfies the assumption of 
 \cref{lem: gd decomposition}.
For any $\epsilon>0,\,\, 0<r<r(\epsilon)<\frac{1}{2}$ and $m\geq M(\epsilon,r)$, if $\frac{1}{m}H(\tau,\cal Q_m)\geq \alpha$, then the $r$-attracting decomposition of $(g,\tau)$ is a $(q^{-\chi_1(g)}/r^2,Cr^\beta)$-decomposition of $g\tau$ with $(\alpha-\epsilon)$-entropy concentration at scale $(m,l)$ with $l=\chi_1(g)$.
\end{lem}

\begin{proof}
 
We use \eqref{eqn: concav alm conv} to obtain
\begin{align*}
\frac{1}{m}H\left(\tau_{ b(g^-, r)}, \mathcal Q_m\right)&\geq \frac{1}{1-\theta}\left(\frac{1}{m}H(\tau,\mathcal Q_m)-\frac{1}{m}H(\theta)-\frac{\theta}{m}H\left(\tau_{ b(g^-, r)^c}, \mathcal Q_m\right)\right)\\
&\geq \frac{1}{1-\theta}\left(\frac{1}{m}H(\tau,\mathcal Q_m)-\frac{1}{m}H(\theta)-\theta\right).
\end{align*}
with $\theta<Cr^{\beta}$ by the previous lemma. 
Due to \cref{lem: sl2 bounded distortion} and \cref{lem: entpy prpty SL2R act}, we have
\begin{eqnarray*}
\frac{1}{m}H\left(g(\tau_{b(g^-,r )}), \mathcal Q_{m+\chi_1(g)}\right)=\frac{1}{m}H\left(\tau_{b(g^-,r )},\mathcal Q_m\right)+\frac{1}{m}O(|\log r|).
\end{eqnarray*}
Combining the above two formulas, we finish the proof by taking $C r^\beta$ (hence $\theta$) sufficiently small with respect to $\epsilon$ and $m$ large with respect to $|\log r|$ and $\epsilon$.
\end{proof}

\subsection{Keeping porosity under projective transformation}
Let $\mu$ be the $\nu$-stationary measure on $\P(\R^3)$ given as in \cref{sec:regularity of stationary measures}. 
In this section, we show the entropy properties of $\pi_{V^\perp}\mu$ are inherited by the pushforward measure $g\pi_{V^\perp}\mu$ with $V\in \P(\R^3)$ and $g\in \SL_2^\pm(\R)$, which is a projective version of \cite[Lemma 3.10]{barany_hausdorff_2017}.

 \begin{prop}\label{prop: kep poro projct transform}
  For any $\epsilon>0$, there exists $\epsilon'>0$ such that for any $m\geq m(\epsilon),\,\,\,n\geq n(m,\epsilon)$ the following holds. For any $n_1\in \mathbb{N}$ and $V\in \mathbb{P}(\mathbb{R}^3)$,
  if the measure $\pi_{V^\perp}\mu$ is $(\alpha, \epsilon',m)$-entropy porous from scale $n_1$ to $n_2=n_1+n$, then for any $g\in \SL_2(\R)^\pm$, $g\pi_{V^\perp}\mu$ is $(\alpha, \epsilon,m)$-entropy porous from scale $n_1+\chi_1(g)$ to $n_2+\chi_1(g)$.
 \end{prop}
 
\begin{proof}
Abbreviate $\pi_{V^\perp}\mu$ by $\tau$.
Fix $\epsilon>0$ and $g\in \operatorname{SL}_2^\pm(\mathbb{R})$.
Without loss of generality, we can assume $\epsilon$ small such that $\epsilon^2<\frac{\epsilon}{100}$. Recall the uniform H\"older constant $C>0$ and H\"older exponent of $\beta>0$ given in \cref{lem: Hold reg proj mes}. 

We fix a few constants:
\begin{itemize}
\item let $\delta>0$ be the constant given in \cref{lem: BHR 3.9} for $\epsilon^2$;

\item let $l_0,\,k_0\in \mathbb{N}$ be the constants $l, k_0$ given in \cref{prop: unfm dbling g} with parameter $\delta/10$;

\item fix an $l\gg l_0$ such that $q^{-(l-l_0)}\leq r^2$ with $r=(\delta/10C)^{1/\beta}/4$.
\end{itemize}
These constants actually depend only on $\epsilon$. We will determine the constants $m,n\in \mathbb{N}$ later. 

We will prove the proposition for $\epsilon'=\min\{\epsilon^4, \frac{\delta^2}{10}\}$.
Assume that $\tau$ is $(\alpha, \epsilon', m)$-entropy porous from scale $n_1$ to $n_2$. By Markov inequality, there are at least $(1-\sqrt{\epsilon'})$ fraction of levels $n_1\leq j\leq n_2$ such that 
\begin{equation*}
\P\left\{\frac{1}{m}H(\tau_{j,x},\cal Q_{j+m})<\alpha+\epsilon'\right\}>1-\sqrt{\epsilon'}.
\end{equation*}
By taking $n$ sufficiently large, which depends on $k_0,l, \epsilon'$, and hence only depends on $\epsilon$, we have that for at least $(1-2\sqrt{\epsilon'})$ fraction of levels $n_1\leq j+l \leq n_2$ such that 
\begin{equation}
\label{eqn:level j+l}
j\geq k_0\,\,\,\text{and}\,\,\, \P\left\{\frac{1}{m}H(\tau_{j+l,x},\cal Q_{j+l+m})<\alpha+\epsilon'\right\}>1-\sqrt{\epsilon'}.
\end{equation}
 We call $j$ \textit{typical} if $j$ satisfies both inequalities in \cref{eqn:level j+l}. 
 
 Recall $r=(\delta/10C)^{1/\beta}/4$. Applying \cref{lem: Hold reg proj mes} to $\tau$, we have for any $x\in \P(\R^2)$, 
 \begin{equation*}
 \tau(B(x,4r))\leq \frac{\delta}{10}.
 \end{equation*}
 In particular, we have
 \begin{equation*}
 \tau (b(g^-,4r)^c)\leq \frac{\delta}{10}.
 \end{equation*}
We have $q^{-(j+l)}\leq q^{-(l-l_0)}\leq r$, hence for typical $j$, we have
\begin{equation}
\label{eqn:nice x}
\P\left\{
\begin{array}{l}x:
\frac{1}{m}H(\tau_{j+l,x},\cal Q_{m+j+l})<\alpha+\epsilon',\\
\operatorname{supp}(\tau_{j+l,x})\subset b(g^-,r) 
\end{array}
\right\}>1-\sqrt{\epsilon'}-\frac{\delta}{10}.
\end{equation}

We consider the action of $g$ on $b(g^-,r)$. By \cref{lem: sl2 bounded distortion}, it scales by $q^{-\chi_1(g)}$ with distortion $r^{-2}$. Therefore for typical $j$ and for $x$ in the event given in \cref{eqn:nice x}, \cref{lem: entpy prpty SL2R act} yields
\begin{equation}
\label{eqn:act by g}
\frac{1}{m}H(g(\tau_{j+l,x}),\cal Q_{m+j+l+\chi_1(g)})<\alpha+\epsilon'+\frac{O(|\log r|)}{m}\leq \alpha+2\epsilon'.
\end{equation}
The second inequality holds by taking $m$ sufficiently large, depending on $r,\epsilon'$ and hence only depending on $\epsilon$. Combining \cref{eqn:nice x,eqn:act by g}, we have
\begin{equation}
\label{eqn: nice atom push}
\P\left\{
\begin{array}{l}x:
\frac{1}{m}H(g(\tau_{j+l,x}),\cal Q_{m+j+l+\chi_1(g)})<\alpha+2\epsilon',\\
\operatorname{supp}(\tau_{j+l,x})\subset b(g^-,r) 
\end{array}
\right\}>1-\sqrt{\epsilon'}-\frac{\delta}{10}>1-\frac{\delta}{2},
\end{equation}
where the second inequality is due to $\epsilon'\leq \delta^2/10$.
We take out all $\tau_{j+l,x}$ in the event in \eqref{eqn: nice atom push}. There are finitely many different measures $\tau_{j+l,x}$, and we denote these measures by $\tau_1,\ldots,\tau_{N}$ and $\theta_i=\tau(\operatorname{supp}\tau_i)$ for $1\leq i\leq N$. We will apply Lemma \ref{lem: BHR 3.9} to $g\tau$ in the level $\chi_1(g)+j$. The convex combination is given by
\[g \tau=\E g\tau_{j+l,x}=\sum_{1\leq i\leq N} \theta_i\cdot g\tau_i+\theta_0\cdot g\tau_0,\]
where $\tau_0$ is the convex combinations of $\tau_{j+l,x}$ not in the event in \eqref{eqn: nice atom push}. We have $\theta_0<\delta/2$.  
We check the conditions of Lemma \ref{lem: BHR 3.9}:
\begin{enumerate}
\item $\frac{1}{m}H(g\tau_{i}, \cal Q_{m+j+\chi_1(g)})\leq \frac{1}{m}H(g\tau_{i}, \cal Q_{m+j+l+\chi_1(g)})<\alpha+2\epsilon'<\alpha+\epsilon^2$ for $1\leq i\leq N$;
\item $\rm{diam}(\supp(g(\tau_{i})))<q^{-\chi_1(g)-l_0-j}$ for $1\leq i\leq N$. Here we use the fact that the distortion of $g$ on $b(g^-,r)$ is $r^{-2}$ with $q^{-(l-l_0)}\leq r^2$;

\item $g\tau\left(\cup_{n}B\left(\frac{n}{q^{[\chi_1(g)]+j}},\frac{1}{q^{[\chi_1(g)]+j+l_0}}\right)\right)<\delta/10$. This is due to \cref{prop: unfm dbling g} and 
typical $j\geq k_0$; 
\item Each interval of length $q^{-[\chi_1(g)]-j}$ intersects at most $p=q^{\olive{l}}$ of the support of $g\tau_{i}$, thanks to $\diam(g\tau_i)\geq q^{-\chi_1(g)-j-l} $.
\end{enumerate}
Therefore by our choice of $\delta$, it follows from Lemma \ref{lem: BHR 3.9} that for typical $j$, for $m\geq m(\epsilon, q^{l})$ (hence only depends on $\epsilon$),
\[\P\{\frac{1}{m}H((g\tau)_{j+\chi_1(g),x}, \cal Q_{j+\chi_1(g)+m})\leq \alpha+2\epsilon'\}>1-\epsilon^2.\]
Taking account of non-typical $j$, we get that 
$$\P_{n_1\leq j\leq n_2}\{\frac{1}{m}H((g\tau)_{j+\chi_1(g),x}, \cal Q_{j+\chi_1(g)+m})\leq \alpha+2\epsilon'\}>(1-\epsilon^2)(1-2\sqrt{\epsilon'}).$$
By our choice of $\epsilon'$, we have that $\alpha+2\epsilon'\leq \alpha+\epsilon$, $(1-\epsilon^2)(1-2\sqrt{\epsilon'})>1-\epsilon$. Therefore, $g\tau$ is $(\alpha,\epsilon,m)$-entropy porous from scale $n_1+\chi_1(g)$ to $n_2+\chi_1(g)$. 
\end{proof}

\subsection{Decomposition of random projection measures}

Let's recall the decomposition given in \cref{eqn: dec pi V A}: for any $\spadesuit\in \{\bU(n), \bI(i)\}$ and $V\in \P(\R^3)$, if $g_{\spadesuit}^{-1}V\not\subset V^{\perp}$, then we have
\begin{align}
\label{eqn:decomposition projective measure}
&{}\pi_{V^{\perp}}\circ g_{\spadesuit}\nonumber\\
=&{}h_{V,g_{\spadesuit}}\circ \pi_{g^{-1}_{\spadesuit}V,V^\perp}\nonumber\\
=&{}h_{V,g_{\spadesuit}}\circ \pi(g^{-1}_{\spadesuit}V, V^{\perp},(g^{-1}_{\spadesuit}V)^{\perp})\circ \pi_{(g^{-1}_{\spadesuit}V)^{\perp}},
\end{align}
where $h_{V,g_{\spadesuit}}$ is a projective transformation of $\P(V^{\perp})$ and we regard it as an element in $\SL_2^\pm(\R)$ by identifying $\P(V^{\perp})$ with $\P(\R^2)$. To simply the notations, we will denote $h_{V,g_{\spadesuit}}$ by $h_{V,\spadesuit}$ and $\pi_{g^{-1}_{\spadesuit}V,V^{\perp}}$ by $\pi_{V,\spadesuit}$ in the followings.

The focus of this subsection is to study the decomposition of a random projection measure $\pi_{V,\spadesuit}\mu$ defined as in \cref{defi: gd bd decomp pure measure} and \cref{defi: gd bd decomp}. The following lemma is a corollary of the large deviation theorem of random walks on projective space.
\begin{lem}\label{lem: BHR 3.14}
For every $\epsilon>0$, there exist $C_1=C_1(\epsilon),N_1=N_1(\epsilon)\geq 1$ such that for $n\geq N_1$, $V\in \P(\R^3)$

\begin{equation}\label{eqn: 3.14' un}
\mathbb P
\left\{ 
\begin{array}{ll}
&\bU(n): d(g_{\mathbf U(n)}^{-1}V, V^\perp)>1/C_1, 
\\&
| \chi_1(h_{V,\bU(n)})- \chi_1(g_{\bU(n)})|\leq \log C_1
\end{array}
\right\}>1-\epsilon. 
\end{equation}
Similarly, we have a corresponding inequality for $\mathbf I(i)$ 
\begin{equation}\label{eqn: 3.14' in}
\inf_{V\in \P(\R^3)}\mathbb P_{1\leq i\leq n}\left\{ 
\begin{array}{ll}
&\bI(i): d(g_{\mathbf I(i)}^{-1}V, V^\perp)>1/C_1
\\
&
|\chi_1(h_{V,\bI(i)})-\chi_1(g_{\bI(i)})|\leq\log C_1
\end{array}
\right\}>1-\epsilon. 
\end{equation}
\end{lem}
\begin{proof}
Recall the relation between the convolution measure $\nu^{*n}$ and the random walk $\bU(n)$ \cref{equ:switch}. The first line for the statement of $\bU(n)$ is due to \cref{equ:LDP-g-1V}. The second line is obtained by applying \cref{equ:LDP-g-V-perp} to the random walk, which is the conjugation of $\nu$ by $k\in \SO(3)$ with $kV=E_1$.

Then the statement of $\bI(n)$ follows from \cref{lem: unify un in} and the statement \cref{eqn: 3.14' un} for $\bU(n)$ with $\epsilon=\epsilon/C_I$.    
\end{proof}

To simplify the notation we denote by $X_{\mathbf{U},n, V}=X_{\mathbf{U},n, V}(C_1,\epsilon)$ 
the set of $\bU(n)$ considered in \eqref{eqn: 3.14' un}. Then we have the following property for $X_{\mathbf{U},n, V}$.


\begin{lem}\label{lem: decmp prj mesr U}For any $\epsilon>0$ and $\delta>0$, let $C_1=C_1(\epsilon)\geq 1$ and $N_1=N_1(\epsilon)$ be the constants given in \cref{lem: BHR 3.14}.
Then there exist constants $r=r(\epsilon,C_1,\delta)$ and $C_2=C_2(C_1,r)$ such that for any $V\in \P(\R^3)$ and any $n\geq N_1$, if $\mathbf U(n)\in X_{\mathbf U,n,V}$, then the $r$-attracting decomposition of
$(h_{V,\bU(n)}, \pi_{V,\mathbf U(n)}\mu)$ enjoys the following properties:
\begin{enumerate}
\item it is a $(C_2\cdot q^{-\chi_1(g_{\bU(n)})}, \delta)$-decomposition of $\pi_{V^\perp}{g}_{\mathbf U(n)}\mu$. 

\item 
for any $m\geq M_1(\epsilon,C_1,r)$, if 
\begin{equation*}
\frac{1}{m}H(\pi_{(g^{-1}_{\mathbf U(n)}V)^\perp}\mu, \mathcal Q_m) \geq \alpha,  
\end{equation*}
then the  decomposition is 
with $(\alpha-\epsilon)$-entropy concentration at the scale $(m,\chi_1(g_{\bU(n)}))$.
\end{enumerate}
\end{lem}
\begin{proof}

Recall the decomposition of $\pi_{V^{\perp}}\circ g_{\bU(n)}$ in \cref{eqn:decomposition projective measure}. 

Due to \cref{eqn: 3.14' un} and \cref{lem: Hold regul general projc mesre}, the measure $\pi_{V,\bU(n)}\mu$ is uniform H\"older regular: for any $x\in \P(V^{\perp})$ and any $R>0$, we have
\begin{equation*}
(\pi_{V,\bU(n)}\mu)(B(x,R))\leq CC_1^{\beta}R^{\beta}.
\end{equation*}
Applying \cref{lem: gd decomposition} to the measure $\pi_{V,\bU(n)}\mu$, we have for any $0<r\leq \frac{1}{2}$, the $r$-attracting decomposition $(h_{V,\bU(n)}, \pi_{V,\mathbf U(n)}\mu)$ is a $(q^{-\chi_1(h_{V,\bU(n)})}/r^2, CC_1^{\beta}r^{\beta})$-decomposition of $\pi_{V^\perp}{g}_{\mathbf U(n)}\mu$. We choose $r$ such that $CC_1^{\beta}r^{\beta}\leq \delta$ and 
and replace $\chi_1(h_{V,\bU(n)})$ by $\chi_1(g_{\bU(n)})$ with bounded loss which is due to \cref{eqn: 3.14' un}. Therefore, the decomposition is a $(C_2\cdot q^{-\chi_1(g_{\bU(n)})}, \delta)$-decompostion of $\pi_{V,\bU(n)}\mu$ with $C_2=C_2(C_1,r)$. 
    
    For the second statement, due to \cref{eqn: 3.14' un} and  \cref{lem:g-1 V lip},  the map $\pi(g^{-1}_{\bU(n)}V, V^{\perp},(g^{-1}_{\bU(n)}V)^{\perp})$ has scale $1$ with distortion $C_1$. Hence, by \cref{lem: entpy prpty SL2R act}, we obtain
    \[\frac{1}{m}H(\pi_{V,\bU(n)}\mu, \mathcal Q_m) \geq \alpha-O(\frac{\log C_1}{m}). \]
    By taking $m$ large, we can use \cref{lem: Entpy gd pt} to obtain the entropy concentration at scale $(m,\chi_1(h_{V,\bU(n)}))$ and use $|\chi_1(h_{V,\bU(n)})-\chi_1(g_{\bU(n)})|\leq \log C_1$ to change the scale and conclude.
\end{proof}

Now we show that the good part of projections of typical cylinders has high entropy at smaller scales, i.e. the good part version of Lemma 3.15 of \cite{barany_hausdorff_2017}.
\begin{lem}\label{lem: BHR 3.15}
For any $\epsilon>0$ and $\delta>0$, there exist $r_0=r_0(\epsilon,\delta),\  C_0=C_0(\epsilon,\delta )\geq 1$ such that for  $m\geq M_0(\epsilon,C_0)$, $n\geq N_0(\epsilon, m)$, we have
\begin{equation*}
\inf_{V\in \P(\R^3)}\mathbb P
\left\{
\begin{array}{ll}
&\bU(n): \text{ the $r_0$-attracting decomposition of }(h_{V,\bU(n)}, \pi_{V,\bU(n)}\mu)\\
&\text{is a}\,\,(C_0\cdot q^{-\chi(g_{\bU(n)})},\delta)\text{-decomposition of }\pi_{V^\perp}  g_{\bU(n)}\mu\\
& \text{with } (\alpha-\epsilon)\text{-entropy concentration at the scale}\,\,(m, \chi_1(g_{\bU(n)}))
\end{array}
\right\}>1-\epsilon.    
\end{equation*}
\end{lem}

\begin{proof}
To prove \cref{lem: BHR 3.15},  it suffices to estimate the probability of $\bU(n)$ satisfying the conditions in \cref{lem: decmp prj mesr U}.
 We choose $\epsilon_0,\delta_0$ much smaller than given $\epsilon,\delta$. 
 By \cref{lem: BHR 3.14}, there exist $C_1(\epsilon_0)$ and $N_1(\epsilon_0)$ such that for any $n\geq N_1(\epsilon_0)$, the probability of  $\mathbf U(n)$ belonging to the set $X_{\mathbf U,n,V}=X_{\mathbf U,n,V}(C_1(\epsilon_0),\epsilon_0)$, which is the probability of $\bU(n)$ satisfying the conditions in \eqref{eqn: 3.14' un}, is greater than $1-\epsilon_0$.


Moreover, Lemma \ref{lem: Pseudo cont entpy} yields that for any $m\geq M_2(\epsilon_0)$ and any $n\geq N_2(\epsilon_0,m)$, the probability of $\bU(n)$ satisfying
\begin{equation}\label{eqn: def n m nice}\frac{1}{m}H(\pi_{(g^{-1}_{\mathbf U(n)}V)^\perp}\mu, \mathcal Q_m) \geq \alpha-\epsilon_0 \end{equation}
is greater than $1-\epsilon_0$. 
Consequently, letting $r=r(\epsilon_0,C_1,\delta_0)$ be the constant given in \cref{lem: decmp prj mesr U},
for any $m\geq \max\{M_1(\epsilon_0, C_1,r), M_2(\epsilon_0)\}$ and  any $n\geq \max \{N_1(\epsilon_0), N_2(\epsilon_0,m)\}$, the probability of  $\bU(n)$ satisfying the conditions of \cref{lem: decmp prj mesr U} is greater than $1-2\epsilon_0$. 
\end{proof}
    
We can make a similar statement for $\pi_{V^\perp}g_{\mathbf I(n)}\mu$, 
\begin{lem}\label{lem: 3.16' BHR} For any $\epsilon>0$ and $\delta>0$, there exist $r_1=r(\epsilon, \delta), C_1=C_1(\epsilon,\delta )\geq 1$ such that for  $m\geq M_1(\epsilon,C_1)$, $n\geq N_1(\epsilon, m)$, we have
\begin{equation}
\label{eqn: 3.16' BHR}
\inf_{V\in \P(\R^3)}\mathbb P_{1\leq i\leq n}
\left\{
\begin{array}{ll}
&\bI(i):\text{ the $r_1$-attracting decomposition of}\,\,(h_{V,\bI(i)}, \pi_{V,\bI(i)}\mu)\\
&\text{is a }(C_1\cdot q^{-i},\delta)\text{-decomposition of }\pi_{V^\perp}  g_{\mathbf I(i)}\mu\\
& \text{with } (\alpha-\epsilon)\text{-entropy concentration at the scale}\,\,(m, i)
\end{array}
\right\}>1-\epsilon. 
\end{equation}
\end{lem}
\begin{proof} 
Recall $C_I>1$ from \cref{lem: unify un in}. Given any $\epsilon>0$ and $\delta>0$, let $\epsilon_0=\epsilon/C_I$, $r_0=r_0(\epsilon_0/2,\delta)$, $C_0=C_0(\epsilon_0/2,\delta)$ and $M_0(\epsilon_0/2,C_0)$ be the constants given in \cref{lem: BHR 3.15}. For any $m\geq M_0$ and any $V\in \P(\R^3)$, we define the subset $\cal U_V$ of the set of words $\Lambda^\ast$ as follows: 
\begin{equation*}
    \cal U_V:=\left\{
    \begin{array}{ll}
    &\bf i\in\Lambda^\ast:\text{ the $r_0$-attracting decomposition of }(h_{V,\bf i}, \pi_{V,\bf i}\mu)\\
&\text{is a }(C_0\cdot q^{-\chi_1(g_{\bf i})},\delta)\text{-decomposition of }\pi_{V^\perp}  g_{\bf i}\mu\\
& \text{with $(\alpha-\epsilon_0/2)$-entropy concentration at the scale }\,\,(m, {\chi_1(g_{\bf i})})
\end{array}
\right\}.
\end{equation*}
By \cref{lem: BHR 3.15}, for any $n\geq N_0(\epsilon_0/2,m)$, we have
\begin{equation*}
\inf_{V\in \P(\R^3)}\P\{\bU(n): \bU(n)\in \cal U_V\}>1-\epsilon_0/2.
\end{equation*}
Hence, it follows that for any sufficiently large $n$ depending on $\epsilon_0$,
\begin{equation*}
\inf_{V\in \P(\R^3)} \P_{1\leq i\leq n} \{\bU (i): \bU (i)\in \cal U_V\}>1-\epsilon_0.
\end{equation*}
Lemma \ref{lem: unify un in} implies for any $n\geq N(\epsilon_0,m)$, 
\begin{equation*}
\inf_{V\in \P(\R^3)} \P_{1\leq i\leq n} \{\bI (i): \bI (i)\in \cal U_V\}>1-C_I\epsilon_0=1-\epsilon.
\end{equation*}
For any $V\in \P(\R^3)$ and any $1\leq i\leq n$, if $\bI (i)\in \cal U_V$, we actually have that the $r_0$-attracting decomposition of $(h_{V,\bI (i)},\pi_{V,\bI (i)}\mu)$ is a $(C_1\cdot q^{-i},\delta)$-decomposition of $\pi_{V^{\perp}}g_{\bI (i)}\mu$ with $(\alpha-\epsilon_0)$-entropy concentration at the scale $(m,i)$. Here is how we obtain the last statement: Due to \cref{equ:bounded residual time}, we replace 
$C_0\cdot q^{-\chi_1(g_{\bI(i)})}$
by $C_1\cdot q^{-i}$, and replace the scale 
$(m,\chi_1(g_{\bI(i)}))$ by $(m,i)$ by choosing $m$ large enough if necessary.
\end{proof}
In the later proof, we need an immediate corollary of \cref{lem: 3.16' BHR}, which relaxes the condition on the parameter $r$.
\begin{lem}\label{lem:entropy m2}
For any $\epsilon>0$ and $\delta>0$, there exist $r_1=r(\epsilon, \delta), C_1=C_1(\epsilon,\delta )\geq 1$ such that for  $m\geq M_1(\epsilon,C_1)$, $n\geq N_1(\epsilon, m)$ and $r<r_1$, we have
\begin{equation}\label{equ: mr}
\inf_{V\in \P(\R^3)}\mathbb P_{1\leq i\leq n}
\left\{
\begin{array}{ll}
&\bI(i): \text{ the $r$-attracting part $\frak m_r$ of}\,\,(h_{V,\bI(i)}, \pi_{V,\bI(i)}\mu)\\
&\text{satisfies }
\frac{1}{m}H(\frak m_r, \calQ_{m+i})\geq \alpha-\epsilon-2\delta
\end{array}
\right\}>1-\epsilon. 
\end{equation}
\end{lem}
\begin{proof}
We have the following general result.
    \begin{claim*}
    Let $\beta,\delta\in (0,1)$ and $m,i\in \mathbb{N}$. Let  $\tau_1,\tau_2\in \mathbf{P}(\P(\R^2))$ such that $\tau_1$ is a restriction of $\tau_2$ with $\tau_2((\supp\tau_1)^c)\leq \delta$. 
    If the measure $\tau_1$ satisfies
    \[\frac{1}{m}H(\tau_1,\calQ_{m+i})\geq \beta, \]
    then
    \[\frac{1}{m}H(\tau_2,\calQ_{m+i})\geq\beta-\delta. \]
    \end{claim*}
    Due to the hypothesis, we can write $\tau_2=(1-\eta)\tau_1+\eta\tau'$ with $\eta\leq\delta$. Then the claim follows directly from the concavity of entropy \cref{eqn: concav alm conv}.
    
    For each $\bI(i)$ belonging to the set in the formula \cref{eqn: 3.16' BHR},
    we consider $\tau_1=\frak m_{r_1}$ (the $r_1$-attracting part) and $\tau_2=\frak m_r$. Due to $r\leq r_1$, the measure $\tau_1$ is a restriction of $\tau_2$. Since the $r_1$-attracting decomposition of $\,\,(h_{V,\bI(i)}, \pi_{V,\bI(i)}\mu)\text{is an }(C_1\cdot q^{-i},\delta)\text{-decomposition of }\pi_{V^\perp}  g_{\mathbf I(i)}\mu$, we have $\tau_2((\supp\tau_1)^c)\leq \frac{\tau(b(h^-,r)^c)}{\tau(b(h^-,r_1))}\leq \frac{\delta}{(1-\delta)}\leq 2\delta$. As the $r_1$-decomposition has $(\alpha-\epsilon)$-entropy concentration, that is $\frac{1}{m}H(\tau_1,\calQ_{m+i})\geq \alpha-\epsilon$, we can apply the claim to $\tau_1$ and $\tau_2$ to finish the proof.
\end{proof}

\subsection{Proof of \cref{prop: BHR 3.19}}

We are ready to prove \cref{prop: BHR 3.19}. 
In the following proposition, we fix a point $V\in \P(\R^3)$ such that the projection measure $\pi_{V^\perp}\mu$ has entropy dimension $\alpha$.
\begin{prop} 
\label{prop: porous projection}
For any $\epsilon>0$, $m\geq M(\epsilon)$ and $n\geq N(\epsilon, m,V)$, the projection measure $\pi_{V^\perp}\mu$ is $(\alpha, \epsilon,m)$-entropy porous from scale $1$ to $n$.    
\end{prop}

\begin{proof} Write $\tau=\pi_{V^\perp}\mu$. Let $\epsilon>0$ be as given. We fix a few other constants:
\begin{itemize}
\item let $\delta'(\epsilon/2)$ be the constant given in \cref{lem: BHR 3.5} for $\epsilon/2$ and $\delta''(\epsilon)$ be the one given in \cref{lem: BHR 3.7'} for $\epsilon$. Fix $\delta<\min\{\delta'(\epsilon/2),\delta''(\epsilon)\}$. 
\item choose constants $\epsilon_0,\delta_0$ much smaller than $\epsilon,\delta$ respectively. 
\item let $C_1=C_1(\epsilon_0,\delta_0)$ be the constant given in \cref{lem: 3.16' BHR}. By Proposition \ref{prop: unfm dbling}, there exist $l_0=l_0(\delta_0), k_0$ which are independent of $V$ such that for any $k\geq k_0$ large we have 
\begin{equation}\label{eqn: raj dta0}
    \tau\left(\cup_nB\left(\frac{n}{q^k}, \frac{1}{q^{k+l_0}}\right)\right)< \delta_0.
\end{equation} 
We fix an $l$ such that $C_1\cdot q^{-l}\leq q^{-l_0}$.
\end{itemize}

We want to apply \cref{lem: BHR 3.7'} to $\tau$. We first construct convex combinations of $\tau$ which satisfy the conditions in \cref{lem: BHR 3.5}.

By \cref{lem: 3.16' BHR}, for $m\geq M_1(\epsilon_0,C_1)$, $n\geq N_1(\epsilon_0,m)$, we have
\begin{equation}
\label{eqn: 3.16' BHR change}
\mathbb P_{1\leq i\leq n}
\left\{
\begin{array}{ll}
&\bI (i+l): \pi_{V^\perp}  g_{\mathbf I(i+l)}\mu\,\, \text{has a}\,\, (C_1\cdot q^{-(i+l)},\delta_0)\text{-decomposition }\\
& \text{with $(\alpha-\epsilon_0)$-entropy concentration at the scale }(m, i+l)
\end{array}
\right\}>1-2\epsilon_0. 
\end{equation}
We replace the index $i$ in \cref{eqn: 3.16' BHR} by $i+l$ in \cref{eqn: 3.16' BHR change}. This is possible by taking $n$ large enough. 
Applying Markov's inequality, we find that for at least a $(1-\sqrt{2\epsilon_0})$-fraction of levels $1\leq k\leq n$, we have
\begin{equation}
\label{eqn: poro final cond}
\mathbb P_{i=k}
\left\{
\begin{array}{ll}
&\bI(i+l): \pi_{V^\perp}  g_{\mathbf I(i+l)}\mu\,\,\text{has a}\,\,(C_1\cdot q^{-(i+l)},\delta_0)\text{-decomposition}\\
&\text{with $(\alpha-\epsilon_0)$-entropy concentration at the scale }(m, i+l)
\end{array}
\right\}>1-\sqrt{2\epsilon_0}.
\end{equation}
Since we can take $n$ arbitrarily large, we can assume for at least a $(1-2\sqrt{\epsilon_0})$ fraction of levels $1\leq k\leq n$ satisfies both \eqref{eqn: poro final cond} and \eqref{eqn: raj dta0}. In the following discussion, we fix any such $k$.

Note that we have $\tau=\E_{i=k}(\pi_{V^\perp}g_{\bI(i+l)}\mu )$ due to \cref{equ:mu ej}. As a result, we can use \cref{eqn: poro final cond} to construct a convex combination of $\tau$:
$$\tau=\sum_{1\leq j\leq N}p_j\tau_j+p_0\tau_0,$$
where for $1\leq j\leq N$, $\tau_j$ comes from the good part of $\pi_{V^{\perp}}g_{\bI(k+l)}\mu$ with $\bI(k+l)$ belonging to the event in \cref{eqn: poro final cond}; $\tau_0$ is the sum of the bad parts of $\pi_{V^{\perp}}g_{\bI(k+l)}\mu$ with $\bI(k+l)$ belonging to the event in \cref{eqn: poro final cond} and those $\pi_{V^{\perp}}g_{\bI(k+l)}\mu$ with $\bI(k+l)$ not belonging to the event in \cref{eqn: poro final cond}. This construction gives $p_0\leq \sqrt{2\epsilon_0}+\delta_0$.
We continue to check that this convex combination satisfies the conditions of \cref{lem: BHR 3.5}:
\begin{enumerate}
\item for $1\leq j\leq N$, due to $\tau_j$ the good part of $\pi_{V^{\perp}}g_{\bI(k+l)}\mu$ with $\bI(k+l)$ in \cref{eqn: poro final cond}, we have $$\frac{1}{m}H(\tau_j, \mathcal Q_{k+m})\geq \frac{1}{m}H(\tau_j, \mathcal Q_{k+l+m})-O(\frac{l}{m})\geq \alpha-2\epsilon_0,$$ 
where the second inequality is possible by taking $m$ large enough depending on $l,\epsilon_0$.
\item for $1\leq j\leq N$, we have $\rm{diam}(\supp \tau_j)\leq C_1\cdot q^{-(k+l)}\leq q^{-(k+l_0)}$.
\item $\tau\left(\cup_n B\left(\frac{n}{q^k}, \frac{1}{q^{k+l_0}}\right)\right)<\delta_0$ by \cref{eqn: raj dta0}.
\end{enumerate}
We can choose $\epsilon_0, \delta_0$ such that $\sqrt{2\epsilon_0}+\delta_0<\delta$ and $\epsilon_0<\frac{\epsilon}{2}$, etc. Then \cref{lem: BHR 3.5} implies that
\begin{equation*}
\mathbb P_{i=k}\left\{\frac{1}{m}H(\tau_{x,i},\mathcal Q_{i+m})>\alpha-\epsilon \right\}>1-\epsilon.
\end{equation*}
Since $\tau$ has entropy dimension $\alpha$, we can take $n$ large such that $|\frac{1}{n}H(\tau, \mathcal Q_n)-\alpha|<\delta$ (in view of \cref{lem:entropy dimension}, here is the only place that $n$ depends on $V$). Therefore, we can apply \cref{lem: BHR 3.7'} to $\tau$ and conclude that for any large $n$, $\tau$ is $(\alpha,\epsilon,m)$-entropy porous from scale $1$ to $n$.
\end{proof}

Lemma \ref{lem:entropy dimension} states that there exists a $\mu^-$-full measure set $E_0$ such that for every $V\in E_0$, $\pi_{V^\perp}\mu$ has entropy dimension $\alpha$. Hence \cref{prop: porous projection} holds for every $V\in E_0$. 
In the following corollary, we get rid of the dependency of $n$ on $V\in \P(\R^3)$.
\begin{cor}\label{cor:porous projection}
For every $\epsilon>0, m>M_1(\epsilon)$ and $n > N_1(\epsilon,m)$, there exists an open set $E = E(\epsilon,m,n) \subset \P(\R^3)$ of measure $\mu^{-} (E)>1-\epsilon$ such that for every $V \in E$, the projection measure $\pi_{V^\perp}\mu$ is $(\alpha, \epsilon, m)$-entropy porous from scale $1$ to $n$.
\end{cor}

Recall that in \cref{lem: good continious V},
we let $\mathcal C=\P(\R^3)-\P(E_1^\perp)$. For each $V\in \mathcal{C}$, the identification between $\P(V^{\perp})$ and $\P(\mathbb{R}^2)$, denoted by $\mathcal{I}_V$, is to identify $\pi_{V^{\perp}}V_0$ with $\R(1,0)$.
The maps $\mathcal{I}_V$ are continuous with respect to $V\in \mathcal{C}$, which is part of the reason why we can obtain \cref{cor:porous projection}. 
The proof of \cref{cor:porous projection} relies on \cref{lem: pseudo cont entpy prf2} and \cref{lem: pseudo cont entpy prf 1}.

\begin{proof}[Proof of \cref{cor:porous projection}]
Let $m>M_1(\epsilon)=\max\{M(\epsilon),M_2(\epsilon)\}$ with $M(\epsilon)$ and $M_2(\epsilon)$ given as in \cref{prop: porous projection} and \cref{lem: pseudo cont entpy prf 1} respectively. For $n>N(\epsilon,m)$ sufficiently large,
the existence of measurable set $E_0(\epsilon,m,n)$ satisfying \cref{cor:porous projection} is a direct consequence of \cref{prop: porous projection}. It remains to show that we can find an open set $E(\epsilon,m,n)$ with large measure.

 Let $V\in E_0$. We find a small open neighborhood $B_{\epsilon,m,n}(V)$ of $V$  in $\P(\R^3)$ that has the following property: for any $W\in B_{\epsilon, m,n}(V)$, identifying $\P(W^{\perp})$ with $\P(\R^2)$ by $\cal I_W$, we have any scale $1\leq i\leq n$, any $x\in \P(\R^2)$ and any sub-interval $I$ of $\cal Q_{i+m}$, 
\begin{equation}\label{eqn: closeness V V'}
|(\pi_{V^\perp}\mu)_{x,i}(I)-(\pi_{W^\perp}\mu)_{x,i}(I)|<q^{-3m}
\end{equation}
The existence of $B_{\epsilon,m,n}(V)$ is guaranteed by \cref{lem: pseudo cont entpy prf2}. We continue to use Lemma \ref{lem: pseudo cont entpy prf 1} to obtain that for any component measure $(\pi_{V^\perp}\mu)_{x,i}$, if 
$\frac{1}{m}H((\pi_{V^\perp}\mu)_{x,i}, \cal Q_{i+m})<\alpha+\epsilon$, then $\frac{1}{m}H((\pi_{W^\perp}\mu)_{x,i}, \cal Q_{i+m})<\alpha+2\epsilon$ for any $W\in B_{\epsilon,m,n}(V)$. Based on these, as $\pi_{V^\perp}\mu$ is $(\alpha,\epsilon,m)$-entropy porous from scale $1$ to $n$, $\pi_{W^\perp}\mu$ is also $(\alpha,2\epsilon,m)$-entropy porous from scale $1$ to $n$ for any $W\in B_{\epsilon,m,n}(V)$. 

As a result, for each $V\in E_0 $, we can find an open neighbourhood of $V$ such that the entropy porosity still holds. Take the union as the set $E(\epsilon,m,n)$.
\end{proof}

We are ready to prove the main proposition of this section.

\begin{proof}[Proof of \cref{prop: BHR 3.19}]
Let $\epsilon>0$ be as given. 
Recall the decompositions given in \cref{eqn:decomposition projective measure} for any $\spadesuit \in \{\bU(n),\bI(i)\}$ and $V\in \P(\R^3)$ such that $g^{-1}_{\spadesuit}V\notin V^{\perp}$. 

Let $\epsilon_0=\min\{\epsilon'(\epsilon/10), \epsilon/(2C_I+1) \} $, where $\epsilon'(\epsilon/10)$ is the constant given in \cref{prop: kep poro projct transform} for $\epsilon/10$. For $m\geq M(\epsilon_0)$ and $k\geq K(m,\epsilon_0)$, let $E(\epsilon_0,m,k)$ be the open set in $\P(\R^3)$ given in \cref{cor:porous projection}. Using the equidistribution of $g_{\bU(n)}^{-1}V$ (\cref{lem:equidistribution x}), we have that for any large $n$ and every $V\in \P(\R^3)$, we have
\begin{equation*}
\P\{\bU(n):g_{\bU(n)}^{-1}V\in E(\epsilon_0,m,k)\}>1-2\epsilon_0.
\end{equation*}
Using \cref{cor:porous projection,lem: unify un in}, we have for any $n\geq N(\epsilon_0,m,k)$
\begin{equation}
\label{eqn:proof prop 4.2 1}
\inf_{V\in \P(\R^3)}\P_{1\leq i\leq n}\{\bI(i):\pi_{(g_{\bI(i)}^{-1}V)^{\perp}}\mu\,\,\text{is}\,\,(\alpha,\epsilon_0,m)\text{-entropy porous from scale}\,\,1\,\,\text{to}\,\,k\}>1-2C_I\epsilon_0.
\end{equation}
Meanwhile, \cref{lem: BHR 3.14} states that there exist $C_1(\epsilon_0)$ and $N_1(\epsilon_0)\geq 1$ such that for any $n\geq N_1(\epsilon_0)$, we have
\begin{equation}
\label{eqn:proof prop 4.2 2}
\inf_{V\in \P(\R^3)}\mathbb P_{1\leq i\leq n}\left\{ 
\begin{array}{ll}
&\bI(i): d(g_{\mathbf I(i)}^{-1}V, V^\perp)>1/C_1
\\
&
|\chi_1(h_{V,\bI(i)})-\chi_1(g_{\bI(i)})|\leq\log C_1
\end{array}
\right\}>1-\epsilon_0. 
\end{equation}
\cref{eqn:proof prop 4.2 1,eqn:proof prop 4.2 2} give that for any $n\geq \max\{N(\epsilon_0,m,k),N_1(\epsilon_0)\}$,
\begin{equation*}
\inf_{V\in \P(\R^3)}\P_{1\leq i\leq n}=\left\{\begin{array}{ll}
     & \bI(i): \bI(i)\,\,\text{satisfies both the conditions in}  \\
     & 
\,\,\cref{eqn:proof prop 4.2 1}\,\,\text{and}\,\,\cref{eqn:proof prop 4.2 2}\end{array}\right\}>1-(2C_I+1)\epsilon_0.
\end{equation*}

To continue, we will use the decomposition 
\begin{equation*}
\pi_{V^{\perp}}g_{\bI(i)}=h_{V,\bI(i)}\circ \pi(g^{-1}_{\bI(i)}V, V^{\perp},(g^{-1}_{\bI(i)}V)^{\perp})\circ \pi_{(g_{\bI(i)}^{-1}V)^{\perp}}.
\end{equation*} 
Due to $\epsilon_0\leq \epsilon'(\epsilon/10)$,
it follows from Proposition \ref{prop: kep poro projct transform} that for $m\geq m(\epsilon')$, $n\geq n(m, \epsilon')$, any $V\in \P(\R^3)$ and any $\tilde{g}\in \SL_2(\R)^\pm$, if $\pi_{(g^{-1}_{\bI(i)}V)^\perp}\mu$ is $(\alpha, \epsilon_0, m)$-entropy porous from scale $1$ to $k$, then $\tilde{g}\pi_{(g^{-1}_{\bI(i)}V)^\perp}\mu$ is 
 $(\alpha, \epsilon/10, m)$-entropy porous from scale $\chi_1(\tilde{g})+1$ to $\chi_1(\tilde{g})+k$. In particular, we consider the words $\bI(i)$ satisfying both the conditions in \cref{eqn:proof prop 4.2 1} and \cref{eqn:proof prop 4.2 2} and let $$\tilde{g}= h_{V,\bI(i)}\circ \pi(g^{-1}_{\bI(i)}V, V^{\perp},(g^{-1}_{\bI(i)}V)^{\perp}).$$ We get $\pi_{V^\perp} g_{\bI(i)}\mu$ is $(\alpha,\epsilon/10,m)$\text{-entropy porous}  from scale $\chi_1(\tilde{g})+1$ to $\chi_1(\tilde{g})+k$ from the decomposition of $\pi_{V^\perp} g_{\bI(i)}\mu $. 
 To finish, we estimate $\chi_1(\tilde{g})=\chi_1(h_{V,\bI(i)}\circ \pi(g^{-1}_{\bI(i)}V, V^{\perp},(g^{-1}_{\bI(i)}V)^{\perp}))$. Using \cref{equ:bounded residual time}, \cref{eqn:proof prop 4.2 2},  and \cref{lem:g-1 V lip}, we have that $\chi_1(\tilde{g})$ equals $i$, up to an additive constant $C_2=C_2(\epsilon_0)$. Therefore, $\pi_{V^\perp} g_{\bI(i)}\mu$ is
$(\alpha, \epsilon/10+O(C_2/k), m)$-entropy porous from scale $i$ to $i+ k$. By taking $k$ large enough and due to $(2C_I+1)\epsilon_0\leq\epsilon$, we obtain the statement of \cref{prop: BHR 3.19}.
\end{proof}

\section{Entropy growth}

The main result of this section is \cref{thm.entropy}, which is a projective version of \cite[Theorem 4.1]{barany_hausdorff_2017}. The main difficulty in the current setting comes from the non-uniform contraction and non-linearity of the projective action. To deal with the non-uniform contraction of the action, we rely on the decomposition introduced in \cref{defi: gd bd decomp pure measure}, which makes it possible to analyze the change of the entropy. To deal with the non-linearity of the action, we repeatedly use \cref{eqn: dec pi V A}.

\subsection{Entropy growth under convolutions: Euclidean case}
For $\alpha,\beta\in \bf P(\R)$, we denote by $\alpha*\beta$ the additive convolution of $\alpha, \beta$. For every bounded support $\alpha,\beta\in \bf P(\R)$, we have the following lower bound from the concavity of entropy (cf. \cite[Corollary 4.10]{hochman_self-similar_2014}):
\begin{equation}\label{equ:entropy growth trivial}
\frac{1}{m}H(\alpha*\beta,\calQ_m)\geq \frac{1}{m}H(\beta,\calQ_m)-O(\frac{1}{m}). 
\end{equation}
One may expect that $\frac{1}{m}H(\alpha*\beta, \cal Q_m)$ is close to $\frac{1}{m}H(\alpha, \cal Q_m)+\frac{1}{m}H(\beta, \cal Q_m)$%
, but in general this is not the case. The following modification of \cite[Theorem 4.1]{barany_hausdorff_2017} shows a sufficient condition which implies the presence of non-trivial entropy growths. 
See also Theorem 2.8 in \cite{hochman_self-similar_2014}. 
\begin{thm}\label{thm:entropy euclidean}
    For every $\epsilon\in (0,\frac{1}{10})$, there exists $\delta>0$ such that for $m\geq m(\epsilon,\delta)$, $C\geq 1$ and $n>N(\epsilon,\delta,m)\log C$, the following holds.

    Let $k\in\N$ and $\theta,\tau\in \bf P(\R)$, and suppose that
    \begin{itemize}
        \item $\theta$ and $\tau$ are supported on intervals of length less than $Cq^{-k}$,
        \item $\tau$ is $\left(1-\frac{5\epsilon}{2},\frac{\epsilon}{2},m\right)$-entropy porous from scale $k$ to $k+n$,
        \item 
        $ \frac{1}{n}H(\theta,\calQ_{k+n})>3\epsilon$.
    \end{itemize}
    Then we obtain the growth of the entropy of scale $(k,n)$:
    \[ \frac{1}{n}H(\theta*\tau,\calQ_{k+n})\geq\frac{1}{n}H(\tau,\calQ_{k+n})+\delta.  \]
\end{thm}
\begin{proof}
    \cite[Theorem 4.1]{barany_hausdorff_2017} was stated for the case when $C=1$. Set $k'=k-\log C$ and  $\epsilon'=2\epsilon$.  Choose $n$ large enough so that $\frac{2\log C}{n}<\epsilon$. Then $k'$ and $\theta,\tau$ satisfy the hypothesis \cite[Theorem 4.1]{barany_hausdorff_2017}. More precisely, we have
     \begin{itemize}
        \item $\theta$ and $\tau$ are supported on intervals of length less than $q^{-k'}$,
        \item $\tau$ is $\left(1-\epsilon',\frac{\epsilon'}{2},m\right)$-entropy porous from scale $k'$ to $k'+n$,
        \item 
        $ \frac{1}{n}H(\theta,\calQ_{k'+n})>\epsilon'$.
    \end{itemize}
    Hence, there exists $\delta=\delta(\epsilon')>0$ such that
    \begin{equation*}
    \frac{1}{n}H(\theta*\tau,\cal Q_{k'+n})\geq \frac{1}{n}H(\tau,\cal Q_{k'+n})+\delta.
    \end{equation*}
    It follows that
    \begin{equation*}
    \frac{1}{n}H(\theta*\tau,\calQ_{k+n})\geq\frac{1}{n}H(\theta*\tau,\calQ_{k'+n})\geq\frac{1}{n}H(\tau,\calQ_{k'+n})+\delta\geq \frac{1}{n}H(\tau,\calQ_{k+n})+\delta-\frac{\log C}{n}. 
    \end{equation*}
    By taking $n$ large enough, we obtain the result.
\end{proof}
\subsection{Convolution of measures on \texorpdfstring{$\P(\R^3)$}{P(R\^{}3)} and measures on \texorpdfstring{$L$}{L}: preparation}\label{sec:pre convolution}

Recall the subgroup $L=L_{E_1}\simeq \SL_2(\R)^\pm\ltimes \R^2$ in $\SL_3(\R)$ defined in the $UL$-decomposition for $E_1\in \P(\R^3)$ (see \cref{sec:decomposition}). We are interested in the maps $\pi_{E_1^{\perp}}\ell$ from $\P(\R^3)$ to $\P(\R^2)$ with $\ell\in L$. Note that for each $\ell\in L$, $\pi_{E_1^\perp}\ell$ is not defined at $\ell^{-1}E_1$. 
We define a bad locus of $L\times \P(\R^3)$ by 
\begin{equation}\label{equ:bad locus}
    \cal{B}:=\{(\ell, \ell^{-1}E_1):\ell\in L\},
\end{equation}
which is a closed subset of $L\times\P(\R^3)$. 

Let $\theta\in \bf P(L)$. Let $\tau\in \bf P(\P(\R^3))$ be atom-free (for example the stationary measure $\mu$ on $\P(\R^3)$ is atom-free). 
Let 
$[\theta.\tau]$ be the projection of the convolution measure on $\P(\R^2)$. More precisely, let $[\ell x]=\pi_{E_1^\perp}\ell x$, then the measure $[\theta.\tau]$ is defined as follows: for any continuous function $f$ on $\P(\R^2)$,
\[\int_{\P(\R^2)} f(y) \dd [\theta.\tau](y):=\int_{\P(\R^3)} \int_{L} f([\ell x]) \dd\theta(\ell) \dd\tau(x).\]
Since we suppose that $\tau$ is atom-free, we have $\theta\otimes \tau (\mathcal{B})=0$. The convolution is well-defined. 

\subsubsection{Linearization of projection of convolution measures}
We first need a linearization lemma. Due to non-linearity, for later application of the entropy growth argument, 
we need to locally replace convolution in the sense of $L$-action by the additive convolution on $\R$, without large cost on entropy.

Identifying $\P(\R^2)$ with $\R/\Z$, for $v\in \P(\R^2)$, let $T_v$ be the translation on $\P(\R^2)$ given by $T_v(x)=x-v$. 
For any $\theta\in \mathbf{P}(L)$ and any $x_0\in \P(\R^3)$, we set $[\theta.x_0]:=[\theta.\delta_{x_0}]$, which is the projection of the convolution of $\theta$ with Dirac measure $\delta_{x_0}$. 
\begin{lem}[Linearization of measures]\label{lem:linearization}
   For any $\epsilon>0$, $0<r<1/2$, $C_1>20$, $k>K(\epsilon)$, $t\geq 1$, we have the followings. Let $\ell_0\in L(t,C_1)$ ($\subset L$ defined in \cref{equ: l t C1}).
   For any $x_0\in b(f_{\ell_0},r/2)$, any $\rho<q^{-k} (r/C_1)^{8}$, any measure $\theta\in \mathbf{P}(L)$ with $\supp\theta \subset B(\ell_0,\rho)$ and $\tau\in \mathbf{P}(\P(\R^3))$ with $\supp \tau\subset B(x_0,\rho)\subset b(f_{\ell_0},r)$, we have 
   \begin{align}
        &\left|\frac{1}{k}H([\theta.\tau],\cal Q_{k+t-\log \rho})-\frac{1}{k}H\left((S_{t(\ell_0,x_0)}T_{[\ell_0(x_0)]}[\theta . x_0])*(\pi_{(\ell_0^{-1}E_1)^\perp}\tau),\cal Q_{k-\log \rho}\right) \right|\label{eqn:linearization of measures}\\
        <&\epsilon+\frac{4\log(C_1/r)+O(1)}{k},  \nonumber
       \end{align}
    with  $t(\ell_0,x_0)$ the number given as in \cref{lem:linearization inequality} and 
    \begin{equation*}
    |t(\ell_0,x_0)-t|\leq 8\log (C_1/r).
    \end{equation*}
\end{lem}
\begin{proof}
Let $C_2=1/r$ and $y=\pi_{(\ell_0^{-1}E_1)^\perp}x_0$. Write $\ell_0=\begin{pmatrix}\det h_0 & 0\\ n_0& h_0\end{pmatrix}\in L$.

Under the assumption of this lemma, the condition of \cref{lem:linearization inequality} is satisfied. Using \cref{lem:linearization inequality} and the definition of the set $L(t,C_1)$, we obtain 
\begin{equation*}
    |t(\ell_0,x_0)-t|\leq |t(\ell_0,x_0)-2\log \|h_0\||+|2\log \|h_0\|-t|\leq 8\log (C_1/r). 
\end{equation*}

\textbf{Step 1.} We first prove
\begin{equation*}
  |\frac{1}{k}H([\theta.\tau],\cal Q_{k+t-\log \rho})-\frac{1}{k}H([\theta . x_0]*(S_{-t(\ell_0,x_0)}T_{y}\pi_{(\ell_0^{-1}E_1)^\perp}\tau),\cal Q_{k+t-\log \rho}) |<\epsilon. 
 \end{equation*}

Let $f,h_{z_0}$ with $z_0=(\ell_0,x_0)$ be two maps from $L\times\P(\R^3)-\cal{B}$ to $\P(\R^2)$ given by 
\[ f(\ell,x)=[\ell(x)],\ h_{z_0}(\ell,x)=[\ell(x_0)]+S_{-t(\ell_0,x_0)}(\pi_{(\ell_0^{-1}E_1)^\perp}x-y).  \]
 
The push-forward measure $h_{z_0}(\ell,x)(\theta,\tau)$ on $\P(\R^2)$
is exactly the convolution measure $[\theta . x_0]*(S_{-t(\ell_0,x_0)}T_y\pi_{(\ell_0^{-1}E_1)^\perp}\tau) $. 

Due to \cref{lem:linearization inequality} and $\rho\leq q^{-k} (r/C_1)^8$, we have
\begin{equation*}
\left\|(f-h_{z_0})|_{B(\ell_0,\rho)\times B(x_0,\rho)}\right\|\leq 2C_L(C_1/r)^{8}\rho^2\|h_0\|^{-2} \leq 2C_LC_p^2\rho q^{-k-t}.
\end{equation*}
Here $C_L,C_p$ are large fixed constants and we take $K(\epsilon)\gg\log(C_LC_p^2)$. 
Using the continuity of entropy 
\cref{equ:2.14} and \cref{eq:entropy m n}, we conclude Step 1. 

\textbf{Step 2.}
   Consider the following measure on $\P(\R^2)$
    $$\nu=(T_{[\ell_0(x_0)]}[\theta.x_0])*S_{-t(\ell_0,x_0)}T_y(\pi_{(\ell_0^{-1}E_1)^\perp}\tau).$$
   Observe that the support of $[\theta.x_0] $ is near $[\ell_0(x_0)]$. So  the support of $T_{[\ell_0(x_0)]}[\theta.x_0] $ is near zero.  Originally, $S_t$ is a scaling map defined on $\R$. 
   Comparing $\nu$ with the target measure that appears in \cref{eqn:linearization of measures}, the strategy is to show that the support of $\nu$ is small enough. Then the restriction map $S_{t}:\supp\,\, \nu\to S_t(\supp \,\,\nu)$ is a diffeomorphism, and hence we can use the invariance of entropy under the scaling map to change the scale in the partition 
   \begin{equation}\label{equ:invariance scale}
H(\nu,\calQ_{t+k})=H(S_{t}\nu,\calQ_{k})+O(1).
   \end{equation}
   We estimate the support of $\nu$.  
   \begin{itemize}
   \item ($\supp [\theta.x_0]$) Observe that $\supp[\theta.x_0]\subset [\supp\theta.\{x_0\} ]$. Let $E=\supp\theta$ and $F=\{x_0\}$. Under the assumption of this lemma, we can 
    apply \cref{lem:diameter action} to $E=\supp\theta$ and $F=\{x_0\}$ and obtain
    \begin{equation}\label{equ:supp theta x0}
     \diam(\supp[\theta.x_0])\leq 16C_L\frac{C_1}{r^2}\diam(\supp\theta )/\|h_0\|^2 \leq 16C_LC_p^2 \frac{C_1}{r^2}\rho q^{-t}.
    \end{equation}
    \item ($\supp (\pi_{(\ell_0^{-1}E)^\perp}\tau)$)
   Note that $\supp\tau\subset b(f_{\ell_0},r) $. Then for any $y,z\in\supp\tau$, using \cref{equ: x psi-1 {E_1}}, we   obtain $d(y,\ell^{-1}_0E_1)>r/C_1,\,\,\, d(z,\ell^{-1}_0E_1)>r/C_1$. So  by \cref{lem:projection}(2), \[ d(\pi_{(\ell_0^{-1}E)^\perp} y, \pi_{(\ell_0^{-1}E)^\perp} z)\leq \frac{C_1}{r} d(y,z). \]
    Therefore
    \[ \ \diam(\supp(\pi_{(\ell_0^{-1}E)^\perp}\tau))\leq \diam(\supp \tau) C_1/r \leq C_1\rho/r. 
    \]
    \item 
   We want to deduce from the above two diameter estimates that
    \begin{equation*}
     \diam(\supp \nu)\leq q^{-t(\ell_0,x_0)}/2,
    \end{equation*}
    which is sufficient to have
    \begin{equation}\label{equ:diam nu}
     16C_LC_p^2 \frac{C_1}{r^2}\rho q^{-t}+ q^{-t(\ell_0,x_0)}C_1\rho/r\leq q^{-t(\ell_0,x_0)}/2.      
    \end{equation}
    This is possible because we have $|t(\ell_0,x_0)-t|\leq 4\log(C_1/r)$ and $\rho \leq q^{-k}(C_1/r)^{8}$. Using the assumption $k\geq K(\epsilon)\geq \log(64C_LC_p^2)$, we get \cref{equ:diam nu}.  
    \end{itemize}
    Therefore, we obtain the restriction of the scaling map $S_{t(\ell_0,x_0)}:\supp \nu\to S_{t(\ell_0,x_0)}(\supp \nu)$ is a diffeomorphism.
    
    Recall that we use the convention $\calQ_t=\calQ_{[t]}$ for the partition. If $t(\ell_0,x_0)$ is an integer, we can apply directly \cref{equ:invariance scale} and obtain
    \begin{equation}
    \label{eqn:entropy linearization of measures}
     H(S_{t(\ell_0,x_0)}\nu,\calQ_{k-\log\rho+t-t(\ell_0,x_0)})=H(\nu,\calQ_{k-\log\rho+t})+O(1). 
     \end{equation}
    Otherwise, we can use \cite[(4.1)]{hochman_dimension_2017} and still obtain \cref{eqn:entropy linearization of measures}. 
    
    Comparing $S_{t(\ell_0,x_0)}\nu$ with the target measure $S_{t(\ell_0,x_0)}T_{[\ell_0(x_0)]}[\theta . x_0]*(\pi_{(\ell_0^{-1}E_1)^\perp}\tau)$ in \cref{eqn:linearization of measures}, the only difference is the shift map $T_y$. But the shift map only changes the entropy by $O(1)$. So combining with Step 1, we obtain
    \begin{equation*}
         \left|\frac{1}{k}H([\theta.\tau],\cal Q_{k+t-\log \rho})-\frac{1}{k}H\left((S_{t(\ell_0,x_0)}T_{[\ell_0(x_0)]}[\theta . x_0])*(\pi_{(\ell_0^{-1}E_1)^\perp}\tau),\cal Q_{k-\log \rho+t-t(\ell_0,x_0)}\right) \right|<\epsilon+\frac{O(1)}{k}.  
        \end{equation*}
        To finish, we replace $ \cal Q_{k-\log \rho+t-t(\ell_0,x_0)}$ by $\cal Q_{k-\log \rho}$, and it brings the error term $4\log(C_1/r)/k$ by \cref{eq:entropy m n}.
\end{proof}

\subsubsection{Decompose projection of convolution measures}
\cref{lem: BHR 4.3} uses the concavity of entropy to decompose the projection of the convolution into measures with small support without losing too much entropy. It is called the multi-scale formula for entropy.

We recall and introduce a few notions.
\begin{itemize}
\item For each $j\in \N$, let $\mu_{\varphi(j,\bf i)}$ be the random measure ${g_{\mathbf{i}}}\mu$ with $\mathbf{i}$ in the symbolic space $ \mathbf{I}(j)$ given as in \cref{defn:random words}.

\item In \cref{defi: gd bd decomp}, we introduced the $r$-attracting decomposition of the pair $(g_{\bf i},\mu)$:
\begin{equation}
\label{equ:mu varphi r decomp}
 \mu_{\varphi(j,\bf i)}=g_{\bf i}\mu:=\epsilon\cdot  g_{\bf i}(\mu_{b(g_{\bf i}^-,r)^c})+(1-\epsilon)\cdot  g_{\bf i}(\mu_{b(g_{\bf i}^-,r)}), 
\end{equation}
where $\epsilon=\mu(b(g_{\bf i}^-,r)^c)$. Denote 
\begin{equation*}
(\mu_{\varphi(j,\bf i)})_{\bf I}:=g_{\bf i}(\mu_{b(g_{\bf i}^-,r)}),\,\,\,  (\mu_{\varphi(j,\bf i)})_{\bf {II}}:=g_{\bf i}(\mu_{b(g_{\bf i}^-,r)^c}). 
\end{equation*}
\cref{lem: gd decomposition} states that the decomposition of $\mu_{\varphi(j,\bf i)}$ in \cref{equ:mu varphi r decomp} gives a  decomposition satisfying \cref{defi: gd bd decomp pure measure}.
\item Let 
\begin{align}\label{equ:calE0}
    \cal E_0:=\{(\ell,\bf i)\in L\times \Lambda^*:\ V_{g_{\bf i}}^+\in b(f_{\ell},r/4) \},
\end{align}
where $b(f_{\ell},r)\subset \P(\R^3)$ is given as in \cref{defi:psi good region}.
\end{itemize}
The constant $r$ is the parameter for both $\cal E_0$ where $b(f_{\ell},r)$ is used to define $\cal E_0$, and $r$-attracting decomposition in \cref{equ:mu varphi r decomp}. 
\begin{prop}\label{lem: BHR 4.3}
    For any $C_1>1$ and $t\geq 1$, let $\theta\in\bf P(L)$ be such that $\supp(\theta)\subset L(t,C_1)$. 
    Then any pair of natural numbers $n\geq k$, we consider the $r$-attracting decomposition of $(g_{\bf i},\mu)$ with $r=q^{-\sqrt{k}/10}$. We have
    \[ \frac{1}{n}H([\theta.\mu],\calQ_{t+n})\geq \E_{1\leq j\leq n,\ell\sim\theta}\left(\frac{1}{k}H([\theta_{\ell,j}.(\mu_{\varphi(j,\bf i)})_{\bf I}],\calQ_{t+j+k}),\ \cal E_0\right)-O\left(r^\beta+\frac{k}{n}+\frac{\log(C_1/r^4)}{k}\right), \]
    where $\E(\ ,\cal E_0)$ is the restriction to the subset $\cal E_0$ and $\beta>0$ is from \cref{lem: Hold reg proj mes}.
\end{prop}



\begin{proof}
The proof consists of two steps. Firstly, we use the identity
\begin{equation*}
[\theta.\mu]=\mathbb{E}_{j=i}\left([\theta_{\ell,j}.\mu_{\varphi(j,\bf i)}]\right),
\end{equation*}
which is valid for any $i\in \N$.  Let $s$ be the integer part of $n/k$. Similar to the proof of Lemma 4.3 in \cite{barany_hausdorff_2017}, we have for 
for every residue $0\leq p < k$,
    \begin{equation}\label{eq:4}
\begin{split}
  H([\theta.\mu], \calQ_{t+n})&=\sum_{m=0}^{s-2}H\left([\theta.\mu], \calQ_{t+(m+1)k+p}\mid \calQ_{t+mk+p}\right)\\
  &\;\;\;\;+H([\theta.\mu], \calQ_{t+p})+H([\theta.\mu], \calQ_{t+n}| \calQ_{t+(s-1)k+p})\\
  & \geq  \sum_{m=0}^{s-2}H\left([\theta.\mu], \calQ_{t+(m+1)k+p}\mid \calQ_{t+mk+p}\right)\\
&\geq \sum_{m=0}^{s-2}\mathbb{E}_{j=mk+p}\left(H\left([\theta_{\ell,j}.\mu_{\varphi(j,\bf i)}], \calQ_{t+k+j}\mid \calQ_{t+j}\right)\right).
\end{split}
\end{equation}
where in the last line we used concavity of entropy. 
Now by averaging over $0\le p\le k-1$ and dividing by $n$, and recalling that $s/n\leq1/k$, we get
\begin{align}
\nonumber
\frac{1}{n}H([\theta.\mu], \calQ_{t+n})\geq\frac{1}{n}\sum_{p=0}^{k-1}\sum_{m=0}^{s-2}\mathbb{E}_{j=mk+p}\left(\frac{1}{k}H\left([\theta_{\ell,j}.\mu_{\varphi(j,\bf i)}], \calQ_{t+k+j}\mid \calQ_{t+j}\right)\right)-O(\frac{s}{n})\\ \label{equ:kn 1k}
=\mathbb{E}_{1\leq j\leq n}\left(\frac{1}{k}H\left([\theta_{\ell,j}.\mu_{\varphi(j,\bf i)}], \calQ_{t+k+j}\mid \calQ_{t+j}\right)\right)-O(\frac{k}{n}+\frac{1}{k}),
\end{align}

   Secondly, we drop the conditional part $\calQ_{t+j}$ using the support of measures $\theta_{\ell,j}.\mu_{\varphi(j,\bf i)}$, which is similar to \cite[Lemma 5.3]{hochman_dimension_2017}. We restrict to $\cal E_0$ and have
    \begin{eqnarray*}
    \E_{1\leq j\leq n}\left(\frac{1}{k}H([\theta_{\ell,j}.\mu_{\varphi(j,\bf i)}],\calQ_{t+j+k}|\calQ_{t+j})\right)  \geq \E_{1\leq j\leq n}\left(\frac{1}{k}H([\theta_{\ell,j}.\mu_{\varphi(j,\bf i)}],\calQ_{t+j+k}|\calQ_{t+j}),\ \cal E_0\right)
    \end{eqnarray*}
    
    Recall from \cref{equ:mu varphi r decomp} the $r$-attracting decomposition of $\mu_{\varphi(j,\bf i)}$:
    \begin{align}
 \mu_{\varphi(j,\bf i)}&=\epsilon\cdot  g_{\bf i}(\mu_{b(g_{\bf i}^-,r)^c}) +(1-\epsilon)\cdot  g_{\bf i}(\mu_{b(g_{\bf i}^-,r)})\nonumber\\
 \label{equ:mu j i}
&=\epsilon\cdot  (\mu_{\varphi(j,\bf i)})_{\bf {II}}+(1-\epsilon)\cdot  (\mu_{\varphi(j,\bf i)})_{\bf I}.
\end{align}
We have $\epsilon=\mu(b(g_{\bf i}^-,r)^c)\leq Cr^\beta$ with 
the inequality due to \cref{lem: Hold reg proj mes}. 

     By the definition of $\mathbf{I}(j)$ (\cref{equ:bounded residual time}), we have $\chi_1(g_{\bf i})\geq j$ for $\bf i\in \bI(j)$. 
     Combining with \cref{lem:action g}, we obtain the following lemma.
    \begin{lem}\label{equ:supp mu j i}
    For any $j\in \N$ and for pair $(\ell,\bf i)\in \cal E_0\cap (L\times\bI(j))$, we have 
    \begin{equation}\label{equ:diameter mu1}
    \diam(\supp((\mu_{\varphi(j,\bf i)})_{\bf I}))\leq q^{-j}/r^2.
    \end{equation}
    Moreover, if $j\geq \log (4/r^3)$, then
    \begin{equation}\label{EQ:supp mu j i}
    \supp((\mu_{\varphi(j,\bf i)})_{\bf I})\subset B(V_{g_{\bf i}}^+,q^{-j}/r^2) \subset b(f_{\ell},r/2).
    \end{equation}
    \end{lem}
   To continue, suppose $j\geq \log (2C_LC_1/r)$. 
    We can apply \cref{lem:diameter action} to $C_2=1/r$,
    $E:=\supp\theta_{\ell,j}$ and $F:=\supp(\mu_{\varphi(j,\bf i)})_{\bf I}$. This is because we have $\diam E=\diam \calQ_j(\ell)\leq q^{-j}\leq r/(C_LC_1)$ and \cref{equ:supp mu j i}. For $\ell\in \supp \theta\subset L(t,C_1)$, write $\ell=\begin{pmatrix}\det h & 0\\ n& h\end{pmatrix}$. We have $\|h\|\geq q^{t/2}/C_p$. 
    Hence, we obtain
    \begin{align*}
    \diam(\supp[\theta_{\ell,j}. (\mu_{\varphi(j,\bf i)})_{\bf I}])&\leq 16C_LC_1(\diam(\supp\theta_{\ell,j})+\diam(\supp(\mu_{\varphi(j,\bf i)})_{\bf I} ))/(r^2\|h\|^2)\\
    &\leq 32C_LC_p^2C_1q^{-t-j}/r^4. 
    \end{align*}
    Therefore, the support of $\supp[\theta_{\ell,j}. (\mu_{\varphi(j,\bf i)})_{\bf I}]$ is of size $O(q^{-(t+j)+\log(C_1/r^4)})$. Applying \cref{lem:concave conditional} to $[\theta_{\ell,j}.\mu_{\varphi(j\bf i)}]$ at the scale $(n,k)=(k,t+j)$, we obtain
    \begin{equation}
    \begin{split}
    \label{equ:lower entropy}
    \frac{1}{k}H([\theta_{\ell,j}.\mu_{\varphi(j,\bf i)}],\calQ_{t+j+k}|\calQ_{t+j})
    \geq  \frac{1}{k} H([\theta_{\ell,j}.(\mu_{\varphi(j,\bf i)})_{\bf I}],\calQ_{t+j+k})-\frac{1}{k}O(\log(C_1/r^3))-\epsilon.
    \end{split}
    \end{equation}
    Consequently, we have
    \begin{align}\nonumber
        &\E_{1\leq j\leq n}\left(\frac{1}{k}H([\theta_{\ell,j}.(\mu_{\varphi(j,\bf i)})],\calQ_{t+j+k}|\calQ_{t+j}),\ \cal E_0\right)\\ 
        \geq &\E_{1\leq j\leq n}\left(\frac{1}{k}H([\theta_{\ell,j}.(\mu_{\varphi(j,\bf i)})_{\bf I}],\calQ_{t+j+k}),\ \cal E_0\right) \label{equ:r beta 1k} -Cr^\beta-(\frac{1}{k}+\frac{1}{n})O(\log(C_1/r^4)).
    \end{align}
   Here is how we get the inequality: when $j\geq \max\{\log(2/r^2),\log(2C_LC_1/r) \}$, we use \cref{equ:lower entropy} to obtain a lower bound with an error term $  O(\log(C_1/r^3))/k+\epsilon$, which is less than $O(\log(C_1/r^4))/k+Cr^\beta$; otherwise, we use the trivial bound of entropy $H([\theta_{\ell,j}.(\mu_{\varphi(j,\bf i)})],\calQ_{t+j+k}|\calQ_{t+j})$
   and the definition of $\E_{1\leq j\leq n}$ to obtain an error term $O(\log(C_1/r^4))/n $.
    \end{proof}

By the same proof as in Lemma 4.4 in \cite{barany_hausdorff_2017}, we get 
\begin{lem}\label{lemma: BHR 4.4}
For any $\theta\in \mathbf P(L)$ and  any pair of natural numbers $n\geq k$, we have 
    \[\frac{1}{n}H(\theta,\calQ_n)=\E_{1\leq j\leq n}\left(\frac{1}{k}H(\theta_{\ell,j},\calQ_{j+k})\right)+O\left(\frac{k}{n}+\frac{\log (1+R)}{n}\right), \]
    where $R=\diam (\supp (\theta))$.
\end{lem}

\subsubsection{General projection of convolution of measures}
Let $\theta\in \bf{P}(L)$ and $x\in \P(\R^3)$. We define the measure $\theta.x$ on $\P(\R^3)$ by $\theta.x:=\int_{\SL_3(\R)} \delta_{gx}\dd\theta(g)$. Previously, we treated the measure $[\theta.x]$ which is $\pi_{E_1^\perp}(\theta.x)$.
More generally, for any $V\in \P(\R^3)$, we consider 
the measure $\pi_{V^\perp}(\theta.x)$. 
\begin{lem}\label{Lemma: BHR 4.5}
For every $C_1>1$, there exists a compact neighborhood $Z$ of the identity in $L$ such that for any $V\in \P(\R^3)$ with $d(V, E_1^\perp)>1/C_1$, any $\theta\in \bf{P}(L)$ such that $\supp (\theta)\subset Z$, and any $k,i\in \N$, we have \[ \mu\left\{x\in \P(\R^3):\ \frac{1}{k}H(\pi_{V^\perp}(\theta.x),\calQ_{k+i})\geq \frac{1}{Ck}H(\theta,\calQ_{k+i})-\frac{C}{k}\right\}\geq\frac{1}{C}, \]
where $C>1$ is a constant depending on $Z$ and $C_1$.
\end{lem}
The proof is given in \cref{sec:lemma 4.5}.

\subsection{Entropy growth under convolutions: projective case}
We are ready to prove a projective version of Theorem 4.6 in \cite{barany_hausdorff_2017}. Recall that $\alpha$ is the value of $\dim\pi_{V^\perp}\mu$ for $\mu^-$-a.e. $V$ defined in \cref{lem:entropy dimension}.

\begin{thm}\label{thm.entropy}
    Suppose $0<\alpha<1$. For every $\epsilon>0$ and $C_1>1$, there exists $\delta(\alpha,\epsilon,C_1)>0$ such that the following holds for any $n\geq N(\alpha,\epsilon,C_1)$. 
    
    Let $t\geq 1$ and let $\theta$ be a probability measure on $L$ satisfies
    \begin{itemize}
    \item $\diam(\supp(\theta))<C_p$, 
    \item $\supp(\theta)\subset L(t,C_1)$ (\cref{equ: l t C1}),  
    \item $\frac{1}{n}H(\theta,\calQ_n)>\epsilon$.
    \end{itemize}
    Then
    \[\frac{1}{n}H([\theta.\mu],\calQ_{t+n})>\alpha+\delta.  \]
\end{thm}

We introduce a parameter $k$, a large integer less than $n$ that will be defined at the end of the proof of Theorem \ref{thm.entropy}. We take $r$ to be $\exp(-\sqrt{k}/10)$.
Recall $\cal E_0$ from \cref{equ:calE0} with parameter $r$.  

The idea of the proof is similar to \cite[Section 4]{barany_hausdorff_2017} and \cite{hochman_dimension_2017}. In \cref{sec:entropy growth step 1}, we find measures with small support and positive entropy (\cref{lem:theta psi x positive}). In \cref{sec:entropy growth step 2}, we apply the linearization technique to the convolution to obtain euclidean convolution (\cref{equ:con m1}). In \cref{sec:entropy growth step 3}, we continue to modify the small measures, which is to replace $\frak m_1$ \cref{equ: m1 definition} by $\frak m_2$ \cref{equ:m2}, upgrading \cref{equ:con m1} to \cref{equ:before entropy growth}. Finally, in \cref{sec:entropy growth step 4}, we use \cref{prop: BHR 3.19} to verify that $\frak m_2$ is entropy porous, hence \cref{thm:entropy euclidean} implies entropy growth. 

\subsubsection
{Find measures with small support that have positive entropy} 
\label{sec:entropy growth step 1}
By \cref{lemma: BHR 4.4}, we have
    \begin{equation*}
    \E_{1\leq j\leq n}\left(\frac{1}{k}H(\theta_{\ell,j},\calQ_{j+k})\right)= \frac{1}{n}H(\theta,\calQ_n)-O(\frac{k}{n}+\frac{\log (1+C_p)}{n})>\epsilon/2 ,
        \end{equation*}
    where the second inequality holds under the assumption that $k/n$ and $\log(1+C_p)/n$ are small compare  to $\epsilon$. As a direct consequence, we have
    \begin{equation}\label{equ:exp entropy theta}
    \P_{1\leq j\leq n}\left(\frac{1}{k}H(\theta_{\ell,j},\calQ_{j+k})\geq \frac{\epsilon}{4}\right)\geq\frac{\epsilon}{4}.  
    \end{equation}
    \begin{lem}\label{lem:theta psi x positive}
     There exists $C_5=C_5(C_1)>0$ such that for $k$ large compared to $C_1$ and $1/\epsilon$,
    \[ \E_{1\leq j\leq n 
    }\left(\mu\left(x:\frac{1}{k}H([\theta_{\ell,j}.x],\calQ_{j+k+t})\geq \frac{\epsilon}{C_5}\right)\right)\geq \frac{\epsilon}{C_5}. \]
    \end{lem}
    \begin{proof} 
    Let $C_6=C_6(C_1)>1$ be the constant from \cref{Lemma: BHR 4.5}.
        
        Since \cref{Lemma: BHR 4.5} only works for measures supported in a neighbourhood of identity, we first translate the measures. Due to left-invariance of the metric on $L$ and \cref{eqn: partition size}, we have 
        \[ \frac{1}{k}H(\theta_{\ell,j},\calQ_{j+k})=\frac{1}{k}H(\ell^{-1}\theta_{\ell,j},\calQ_{j+k})+O(\frac{1}{k}).\]
        Then $\ell^{-1}\theta_{\ell,j}$ is a measure supported on a neighbourhood of the identity of radius $O(q^{-j})$. 
        
        Suppose $j$ is large enough. Then $\ell^{-1}\theta_{\ell,j}$ is supported on $Z$, the neighborhood of the identity given in \cref{Lemma: BHR 4.5}. As $\supp \theta\subset L(t,C_1)$, we get $d(\ell^{-1}E_1,E_1^\perp)\geq 1/C_1$. Hence we can apply \cref{Lemma: BHR 4.5} to $V=\ell^{-1}E_1$ and obtain
        \begin{equation}\label{equ:psi-1thetaj} 
          \mu\left(x:\ \frac{1}{k}H(\pi_{(\ell^{-1}E_1)^\perp}(\ell^{-1}\theta_{\ell,j}.x),\calQ_{j+k})\geq \frac{1}{C_6k}H(\theta_{\ell,j},\calQ_{j+k})-\frac{C_6}{k}\right)\geq \frac{1}{C_6}.
        \end{equation}
        Using \cref{lem: dec pi V A}, we rewrite
        \begin{equation*}       
        [\ell(\ell^{-1}\theta_{\ell,j}.x)]=F(\pi_{(\ell^{-1}E_1)^\perp}(\ell^{-1}\theta_{\ell,j}.x))  ,
        \end{equation*}
        where we write $\ell=\begin{pmatrix}\det h & 0\\ n& h\end{pmatrix}$ and $F=h\pi({\ell^{-1}(E_1)},E_1^\perp,(\ell^{-1}E_1)^\perp)$ is a diffeomorphism from $\P((\ell^{-1}E_1)^\perp)$ to $\P(E_1^\perp)$.
       The map $F$ satisfies 
\begin{equation}\label{equ:f contraction}
 \frac{d(F(y),F(y'))}{d(y,y')}\geq q^{-t}/(C_1C_p)^2 
\end{equation}
for $y\neq y'$ in $\P((\ell^{-1}E_1)^\perp)$, 
because the part $\pi({\ell^{-1}(E_1)},E_1^\perp,(\ell^{-1}E_1)^\perp)$ gives a contraction at most $C_1$ (\cref{lem:g-1 V lip}) and the part $h$ at most $\|h\|^{-2}$.
For $\cal F =\calQ_{j+k}$ and $\cal E=F^{-1}\calQ_{j+k+t}$, we have each atom of $\cal E$ intersects at most $O((C_1C_p)^2) $ atoms of $\cal F$ due to \cref{equ:f contraction}. We can apply \cref{lem: entropy upper}. 
        Therefore for $x$ satisfying \eqref{equ:psi-1thetaj}
        ,  we have 
        \begin{equation}\label{equ:psi j x}
        \begin{split}
            \frac{1}{k}H([\theta_{\ell,j}.x],\calQ_{j+k+t})&=
            \frac{1}{k}H(F\pi_{(\ell^{-1}E_1)^\perp}(\ell^{-1}\theta_{\ell,j}.x),\calQ_{j+k+t})\\
            &\geq \frac{1}{k}H(\pi_{(\ell^{-1}E_1)^\perp}(\ell^{-1}\theta_{\ell,j}.x),\calQ_{j+k})-\frac{O(\log (C_1C_p))}{k}\\
            &\geq \frac{1}{C_6k}H(\theta_{\ell,j},\calQ_{j+k})-\frac{C_6+O(\log (C_1C_p))}{k}.
        \end{split}
        \end{equation}

    
    For each $\theta_{\ell,j}$ in the set of \cref{equ:exp entropy theta} and $x$ satisfying \cref{equ:psi j x}, we have
    \[ \frac{1}{k}H([\theta_{\ell,j}.x],\calQ_{j+k+t})\geq \frac{1}{C_6k}H(\theta_{\ell,j},\calQ_{j+k})-\frac{C_6+O(\log C_1)}{k}\geq\frac{\epsilon}{4C_6}-\frac{C_6+O(\log C_1)}{k}\geq \frac{\epsilon}{8C_6}, \]
    if $k$ is large enough. Hence setting $C_5=32C_6$, we have 
    \[\E_{ j 
    }\left(\mu\left(x:\frac{1}{k}H([\theta_{\ell,j}.x],\calQ_{j+k+t})\geq \frac{\epsilon}{C_5}\right)\right)\geq \frac{1}{C_6}\P_{ j }\left(\frac{1}{k}H(\theta_{\ell,j},\calQ_{j+k})\geq \frac{\epsilon}{4}\right).  \]
    
   The proportion of $j$ between $1$ and $n$ such that the ball $B(id,q^{-j}) $ is not contained in $Z$ goes to zero as $n\rightarrow\infty$. The proof is complete by \cref{equ:exp entropy theta}.
    \end{proof}
    
\subsubsection{Apply linearization}
\label{sec:entropy growth step 2}
    We take a small constant $\epsilon_1>0$, which will be chosen small compared to fixed constants $\epsilon'$ and $\delta(\epsilon')$ defined later. 

    Let $K(\epsilon_1)$ be the constant given in \cref{lem:linearization}. We may assume $k>K(\epsilon_1)$. Consider any $j\in \N$ such that 
    \begin{equation*}
     q^{-j}\leq  q^{-k}(r/C_1)^{8}
     \end{equation*}
     and set $\rho=q^{-j}$.
    We apply \cref{lem:linearization} to
    \begin{itemize}
    \item the constants $\epsilon_1$, $k$ and $r=q^{-\sqrt{k}/10}$;

    \item the convolution measure $[\theta_{\ell,j}.(\mu_{\varphi(j,\bf j)})_{\bf I}]$ at any point $x\in \supp (\mu_{\varphi(j,\bf j)})_{\bf I}$ with $(\ell,\bf i)\in \cal E_0$.
    \end{itemize}
   We check the conditions of \cref{lem:linearization}:
   \begin{itemize}
   \item  $\supp(\theta_{\ell,j})\subset B(\ell,q^{-j})= B(\ell,\rho)$;

   \item due to $j\geq \log(4/r^3)$, by \cref{equ:supp mu j i},  we have $\supp (\mu_{\varphi(j,\bf i)})_{\bf I}\subset B(V_{g_{\bf i}}^+,\rho)\subset b(f_{\ell},r/2)$.  Hence, for any $x\in \supp (\mu_{\varphi(j,\bf i)})_{\bf I}$, we have $ \supp (\mu_{\varphi(j,\bf i)})_{\bf I}\subset B(x,2\rho)\subset b(f_{\ell},r)$.
   \end{itemize}
    Therefore, \cref{lem:linearization} yields for any $x\in \supp (\mu_{\varphi(j,\bf i)})_{\bf I}$
    \begin{align}\label{equ:con m1 before integral}
        &\frac{1}{k}H([\theta_{\ell,j} . (\mu_{\varphi(j,\bf i)})_{\bf I}],\cal Q_{k+t+j})
        \\
        \geq & \frac{1}{k}H(S_{t(\ell,x)}T_{[\ell(x)]}[\theta_{\ell,j} . x]*\pi_{(\ell^{-1}E_1)^\perp}(\mu_{\varphi(j,\bf i)})_{\bf I},\cal Q_{k+j})-O\left(\frac{\log(C_1/r)}{k}\right)-\epsilon_1,\nonumber
    \end{align}
    with
    \begin{equation}\label{equ:t psi x}
        |t(\ell,x)-t |\leq 8\log(C_1/r).
    \end{equation}

    Hence, combining with \cref{lem: BHR 4.3}, we have
    \begin{equation}\label{equ:con m1}
    \begin{split}
    &\frac{1}{n}H([\theta.\mu],\calQ_{t+n})\geq \E_{1\leq j\leq n,\ell\sim\theta}\left(\frac{1}{k}H([\theta_{\ell,j}.(\mu_{\varphi(j,\bf i)})_{\bf I}],\calQ_{t+j+k}),\ \cal E_0\right)-\textit{error term}\\
    &\geq\E_{1\leq j\leq n,\ell\sim\theta}\left(\left(\int \frac{1}{k}H((S_{t(\ell,x)}T_{[\ell(x)]}[\theta_{\ell,j}.x])* \pi_{(\ell^{-1}E_1)^\perp}(\mu_{\varphi(j,\bf i)})_{\bf I},\calQ_{j+k})\dd(\mu_{\varphi(j,\bi)})_\bI(x)\right),\cal E_0\right) \\
    &\ \ -\textit{error term}\ \ \text{(integrating \cref{equ:con m1 before integral} for $x$ with the measure $(\mu_{\varphi(j,\bi)})_\bI $)}.
    \end{split}
    \end{equation}

\subsubsection{Replace measures by their good parts}
\label{sec:entropy growth step 3}
For any $\bi\in \bI(j)$ and $V\in\P(\R^3)$, \cref{eqn:decomposition projective measure} states that  
\begin{align}
\label{equ: central equation}
    {}\pi_{V^{\perp}}\circ g_{\bi}={}h_{V,g_{\bi}}\circ \pi_{g^{-1}_{\bi}V,V^\perp}.
\end{align}
Recall we denote $h_{V,g_{\bf i}}$ by $h_{V,\bi}$ and $\pi_{g^{-1}_{\bi}V,V^\perp} $ by $\pi_{V,\bi}$ for simplification.

Recall that $\mu_{\varphi(j,\bf i)}=g_{\bf i}\mu$. So far, we have considered the measure 
\begin{equation}\label{equ: m1 definition}
\frak m_1:=\pi_{(\ell^{-1}E_1)^\perp}((\mu_{\varphi(j,\bf i)})_{\bf I}),
\end{equation}
which is the pushforward of the $r$-attracting part of $(g_{\bf i},\mu)$ by the projection $\pi_{(\ell^{-1}E_1)^\perp}$. 
To continue, we will switch the order and consider the following measure
    \begin{equation}\label{equ:m2}
    \frak m_2(\ell,j,\bi):=
    \text{$r^5$-attracting part of the pair $(h_{\ell^{-1}E_1,\bf i},\pi_{\ell^{-1}E_1,\bf i} \mu)$}
    \end{equation}
    (see \cref{defi: gd bd decomp} for the construction of the attracting part).
     \textbf{The measure $\frak m_2(\ell,j,\bi)$ is  a random measure with respect to the random variables $\ell$($\sim \theta$) and $\bi$($\in\bI(j)$).} We will frequently write $\frak m_2(\ell,j,\bi)$ as $\frak m_2$ to simplify the notation.
     The advantage of considering $\frak m_2$ is that is $\frak m_2$ is a restriction of  $\pi_{(\ell^{-1}E_1)^\perp}\mu_{\varphi(j,\bf i)}$ and  we can apply to it the porosity result  \cref{prop: BHR 3.19}.
       
   
  
  To relate $\frak m_1$ and $\frak m_2$, we introduce the following subset:
  \begin{equation}\label{equ:cal E0'}
  \cal E_0'(j)=\left\{\begin{array}{ll}
  &(\ell,\bi)\in\cal E_0\cap (L\times \bI(j)):\ d(\pi_{(\ell^{-1}E_1)^\perp}V_{g_{\bf i}}^+,{h^+_{\ell^{-1}E_1,\bf i}} )\leq q^{-j}/r^2,\\
  &\ |\chi_1(h_{\ell^{-1}E_1,\bf i})-\chi_1(g_{\bi})|\leq |\log r|,\ d(g_\bi^{-1}\ell^{-1}E_1,(\ell^{-1}E_1)^\perp)\geq r 
  \end{array}
  \right\}.
  \end{equation}
We will upgrade \cref{equ:con m1} in Step 2 to the following version. 
    \begin{prop}\label{lem:before entropy growth}
    For any $n\geq k\geq K(1)$, we have 
    \begin{equation}\label{equ:before entropy growth}
    \begin{split}
        &\frac{1}{n}H([\theta.\mu],\calQ_{t+n})\\
        &\geq\E_{1\leq j\leq n,\ell\sim\theta}\left(\int \frac{1}{k}H((S_{t(\ell,x)}T_{[\ell(x)]}[\theta_{\ell,j}.x])*\frak m_2(\ell,j,\bi),\calQ_{j+k})\dd(\mu_{\varphi(j,\bi)})_\bI(x),\cal E_0'(j)\right)\\
        &-\epsilon_1-O(\frac{k}{n}+r^\beta+ \frac{\log(C_1/r)}{k}),
        \end{split}
    \end{equation}
    where $r=q^{-\sqrt{k}/10}$.
\end{prop}
    We start the proof of \cref{lem:before entropy growth} by comparing the support of $\frak m_1$ and $\frak m_2$.

       \begin{lem}\label{lem:property m2}
       For any $k\geq K(1)\geq 1$ and 
       $(\ell,\bf i)\in\cal E_0'(j)$, we have:
       \begin{enumerate}
           \item  If $j\geq 6|\log r|$, then
       \begin{equation}\label{equ:supp m2}
       B({h^+_{\ell^{-1}E_1,\bi}}, c_0q^{-j}/2r^4) \subset \supp\frak m_2\subset B({h^+_{\ell^{-1}E_1,\bi}}, q^{-j}/r^{11}),
   \end{equation}
   where $0<c_0<1$ is the constant given as in \cref{equ:bounded residual time}.
       \item If $j\geq  \log(4/r^3)$, we can write
    \begin{equation}\label{equ:supp relation}
    \frak m_2 =(1-\delta)\frak m_1+\delta \frak m',
    \end{equation}
    with  $\frak m'$  a probability measure on $\P((\ell^{-1}E_1)^\perp)$ and $\delta<Cr^\beta$.

        \item  We can write
        \begin{equation*}
        \pi_{(\ell^{-1}E_1)^\perp}\mu_{\varphi(j,\bi)}=\lambda \frak{m}''+(1-\lambda)\frak m_2(\ell,j,\bi),
        \end{equation*}
        where $\frak{m}''$ is a probability measure on $\P((\ell^{-1}E_1)^\perp)$. It is a $(q^{-j}/r^{11},Cr^{\beta})$-decomposition of $\pi_{(\ell^{-1}E_1)^\perp}\mu_{\varphi(j,\bi)}$ (see \cref{defi: gd bd decomp pure measure}). 

        \item 
For any pair $j',j\in\N$ with $j'-j\geq 4|\log r|$, the equality of the component measures
     \begin{equation}\label{lem:support mu2}
(\pi_{(\ell^{-1}E_1)^\perp}\mu_{\varphi(j,\bf i)})_{x,j'}=(\frak m_2(\ell,j,\bi))_{x,j'} 
     \end{equation}
     holds at a set of $x$ with $\pi_{(\ell^{-1}E_1)^\perp}\mu_{\varphi(j,\bf i)}$-measure greater than $1-O(r^\beta)$.        
        \end{enumerate}
       \end{lem}
       \begin{proof}
       
  \begin{enumerate}[1.]
  \item
   Using the definition of the random variable $\bI(j)$ (\cref{equ:bounded residual time}), we have $\chi_1(g_{\bf i})\in[j,j+|\log c_0|]$. Combined with \cref{equ:cal E0'} and $k\geq K(1)$, we have $\chi_1(h_{\ell^{-1}E_1,\bi})\in j+[-|\log r|,|\log r|+|\log c_0|] $.
   For $\frak m_2$, it is the $\text{$r^5$-attracting part of the pair $(h_{\ell^{-1}E_1,\bf i},\pi_{\ell^{-1}E_1,\bf i} \mu)$}$ and we also have $q^{-\chi_1(h_{\ell^{-1}E_1,\bi})}\leq r^5$ as $j\geq 6|\log r|$.  Hence by
   \cref{equ:supp attracting}, we have 
   \begin{equation*}
       B({h^+_{\ell^{-1}E_1,\bi}}, c_0q^{-j}/2r^{4}) \subset \supp\frak m_2\subset B({h^+_{\ell^{-1}E_1,\bi}}, q^{-j}/r^{11}).
   \end{equation*}

\item As $j\geq \log (4/r^3)$, we can apply 
    \cref{EQ:supp mu j i} to $(\mu_{\varphi(j,\bf i)})_{\bI}$ and have 
   \[\supp\frak m_1= \pi_{(\ell^{-1}E_1)^\perp}\supp (\mu_{\varphi(j,\bf i)})_{\bf I}\subset \pi_{(\ell^{-1}E_1)^\perp}B(V_{g_{\bf i}}^+, q^{-j}/r^2)\subset \pi_{(\ell^{-1}E_1)^\perp}b(f_{\ell},r) .  \]
   As $(\ell,\bf i)\in \cal E_0$, by the definition of $\cal E_0$ (\cref{equ:calE0}), we have $V^+_{g_{\bf i}}\in b(f_{\ell},r)$. Using \cref{equ: x psi-1 {E_1}}, we have   $d(V_{g_{\bf i}}^+,\ell^{-1}E_1)>r/C_1$. Applying (2) in \cref{lem:projection} to $B(V_{g_{\bf i}}^+, q^{-j}/r^2)$, we obtain
   \[ \supp\frak m_1\subset B(\pi_{(\ell^{-1}E_1)^\perp}V_{g_{\bf i}}^+, C_1q^{-j}/r^3). \]
 It follows from the definition of $\cal E_0'(j)$ that $d({h^+_{\ell^{-1}E_1,\bi}},\pi_{(\ell^{-1}E_1)^\perp}V_{g_{\bf i}}^+ )\leq q^{-j}/r^2$. Therefore, we have $\supp \frak m_1\subset \supp\frak m_2$.

   From the definitions of $\frak m_2$ and \cref{equ: central equation}, we know that $\frak m_2$ is a restriction of $\pi_{(\ell^{-1}E_1)^\perp}(\mu_{\varphi(j,\bf i)})$. So we can write $\frak{m}_2=\pi_{(\ell^{-1}E_1)^\perp}((g_{\bf i}\mu)_S)$ with $S=\pi_{(\ell^{-1}E_1)^\perp}^{-1}\supp\frak m_2$. So the fact that $\supp \frak m_1\subset \supp\frak m_2$ yields the inclusion $S\supset S'$ with $S'=\supp((\mu_{\varphi(j,\bf i)})_\bI)$. Since the measure $ (\mu_{\varphi(j,\bf i)})_\bI$ is the restriction of $\mu_{\varphi(j,\bf i)}$ on $S'$, we can write
   \begin{equation}\label{equ:m2 before projection}
(\mu_{\varphi(j,\bf i)})_S=(1-\delta)(\mu_{\varphi(j,\bf i)})_\bI+\delta \mu', 
   \end{equation}
   with $\mu'$ a probability measure on $\P(\R^3)$, and
   \[\delta=1-(\mu_{\varphi(j,\bf i)})_S(S')\leq 1- \mu_{\varphi(j,\bf i)}(S')\leq Cr^\beta,\]
   where the last inequality is due to \cref{equ:mu j i}.
   Applying $\pi_{(\ell^{-1}E_1)^\perp}$ to \cref{equ:m2 before projection}, we obtain \cref{equ:supp m2}.

    \item 
    By the same argument as in \cref{lem: decmp prj mesr U}, using \cref{lem: gd decomposition} and \cref{lem: Hold regul general projc mesre}, we obtain the third statement.

    \item 
    With the third statement available, it remains to show that the two atoms of $\calQ_{j'}$ containing the endpoints of the support of $\frak m_2$ have small measure. As we denote $g_{\bf i}\mu$ by $\mu_{\varphi(j,\bf i)}$, we have
    \begin{align*}
     &\pi_{(\ell^{-1}E_1)^\perp}g_{\bf i}\mu(B(x,q^{-j'}))=h_{\ell^{-1}E_1,\bi}\pi_{\ell^{-1}E_1,\bi}\mu(B(x,q^{-j'}))\\
    & \leq \sup_{y}\pi_{\ell^{-1}E_1,\bi}\mu(B(y,q^{-(j'-j)}/r))\leq Cr^\beta, 
      \end{align*}
      where the first inequality is due to the definition in \cref{equ:cal E0'} the second inequality is due to \cref{lem: Hold regul general projc mesre}.
 \end{enumerate}
       \end{proof}
    We will prove \cref{lem:before entropy growth} in the remaining Step 3. Then in Step 4, we will estimate the integrand in the right-hand side of \cref{equ:before entropy growth}.
    
    

  \begin{proof}[Proof of \cref{lem:before entropy growth}]
   We only consider $(\ell,\bi)\in\cal E_0'(j)$, $x\in\supp(\mu_{\varphi(j,\bi)})_\bI$ and $j\geq 6|\log r|$. The part of expectation with $j\leq 6|\log r|$ is bounded by $\frac{6|\log r|}{n}\times\frac{k+6|\log r|}{k}=O(\frac{|\log r|}{n}) $.
   
   Due to \cref{equ:supp relation} and \cref{eqn: concav alm conv} (inverse of the concavity of entropy), we obtain
    \begin{align*}
     \frac{1}{k}H(S_{t(\ell,x)}T_{[\ell(x)]}[\theta_{\ell,j}.x]*\frak m_1,\calQ_{j+k})\geq &\frac{1}{1-\delta}(\frac{1}{k}H(S_{t(\ell,x)}T_{[\ell(x)]}[\theta_{\ell,j}.x]*\frak m_2 ,\calQ_{j+k})\\
     &-2\delta\frac{1}{k}H(S_{t(\ell,x)}T_{[\ell(x)]}[\theta_{\ell,j}.x]*\frak m',\calQ_{j+k}) )-\frac{2}{k}H(\delta).
    \end{align*}
    We give an upper bound to the second term on the right-hand side. Note that by \cref{lem: entropy upper}, for any $\tau \in \P(\P(\R^2))$, we have
    \[H(\tau,\calQ_{j+k})= H(\tau,\calQ_{j+k}|\calQ_j)+H(\tau,\calQ_{j})\leq k+O(\log (\diam(\supp\tau)q^j)). \]
     To apply this to the convolution measure, we estimate the diameter of the convolution measure.
     In Step 2, we already verified that we can apply \cref{lem:linearization} to $(\theta_{\ell,j},x)$ for any $j\in \N$ satisfying $q^{-j}\leq q^{-k}(r/C_1)^{8}$. In particular, 
    \cref{equ:supp theta x0} in \cref{lem:linearization} holds for $[\theta_{\ell,j}.x]$.  Combining with \cref{equ:t psi x}, we have
    \begin{equation}\label{equ:supp theta x}
    \diam(\supp S_{t(\ell,x)}T_{[\ell(x)]}[\theta_{\ell,j}.x])\leq 16C_LC_p^2C_1^9q^{-j}/r^{10}.
    \end{equation}
    The support of $\frak m'$ is controlled by \cref{equ:supp m2}, that is $O(q^{-j}/r^7)$. So we have
    \[ \frac{1}{k}H(S_{t(\ell,x)}T_{[\ell(x)]}[\theta_{\ell,j}.x]*\frak m',\calQ_{j+k}) )\leq 1+ \frac{O(\log(C_1/r))}{k}. \]
    
    Hence, we have
    \[\frac{1}{k}H(S_{t(\ell,x)}T_{[\ell(x)]}[\theta_{\ell,j}.x]*\frak m_1,\calQ_{j+k})\geq \frac{1}{k}H(S_{t(\ell,x)}T_{[\ell(x)]}[\theta_{\ell,j}.x]*\frak m_2 ,\calQ_{j+k})-O(\delta\log(C_1/r))/k-2H(\delta)/k, \]
    which holds for every $(\ell,\bi)\in\cal E_0'(j)$.
    
    Notice that $\frak m_1=\pi_{(\ell^{-1}E_1)^\perp}(\mu_{\varphi(j,\bi)})_\bI$. We replace $\frak m_1$ by $\frak m_2$ in \cref{equ:con m1} and obtain
    \begin{align*}
    &\frac{1}{n}H([\theta.\mu],\calQ_{t+n}) 
    \geq \E_{1\leq j\leq n,\ell\sim\theta}\left(\int \frac{1}{k}H((S_{t(\ell,x)}T_{[\ell(x)]}[\theta_{\ell,j}.x])*\frak m_2(\ell,j,\bi),\calQ_{j+k})\dd(\mu_{\varphi(j,\bi)})_\bI(x),\cal E_0'(j)\right) \\
    &
    \ \ \ -\textit{error term}.  \qedhere
    \end{align*}
    \end{proof}

     Let $\P_{1\leq j\leq n,\ell\sim\theta}$ be the product measure $\P_{1\leq j\leq n}\otimes\theta$ for the random variable $(\bi, \ell)$. For any $\ell\in L$ and any $j\in \N$, let $\cal E_0'(j)(\ell)=\{\bi\in \bI:\ (\ell,\bi)\in \cal E_0'(j) \}$.
   

     \begin{lem}\label{lem: change support}
        For every $\epsilon>0$ and $n\geq k\geq K(\epsilon)$, we have for any $\ell\in \supp \theta$,
        \begin{equation*}
        \P_{1\leq j\leq n}(\cal E_0'(j)(\ell))\geq 1-\epsilon
        \end{equation*}
        and hence
        \begin{equation*}
         \P_{1\leq j\leq n,\ell\sim\theta}(\cal E_0'(j))\geq 1-\epsilon.
        \end{equation*}
    \end{lem}
  

    \begin{proof}
   Fix any $\ell\in \supp \theta$ and let $V=\ell^{-1}E_1\in \P(\R^3)$. 
    
   By \cref{lem:vh va}, we have
   \begin{align*}
   &\P_{1\leq j\leq n}(\cal E_0'(j)(\ell))\\
   \geq &\P_{1\leq j\leq n}\left\{\bf i\in \bI(j): V_{g_{\bf i}}^+\in b(f_{\ell},2r),\,\,\, d(V^\perp,H_{g_{\bf i}}^-)>2r,\,\,\, |\chi_1(h_{V,\bf i})-\chi_1(g_{\bi})|\leq |\log r|/2, d(g_\bi^{-1}V,V^\perp)>r\right\}. 
   \end{align*}
   Denote the set in the right-hand side by $\cal U$. In view of \cref{lem: unify un in}, to estimate 
   the probability $\P_{1\leq j\leq n}$ of the above set for $\bI$, it suffices to prove for $\bU$.
   

   We apply the large deviation estimate \cref{eq:hol g+ n>m} to $V_{g_{\bf i}}^+$ and the kernel of the linear form $f_{\ell}$ (\cref{defi:psi good region}); apply \cref{eq:hol g- n>m} to $H_{g_{\bf i}}^-$ and any point $V_1\in V^{\perp}$ (which is stronger than the Hausdorff distance $d(V^\perp, H^-_{g_{\bf i}})$); for the third and fourth condition in $\cal U$, apply \cref{equ:LDP-g-V-perp} (which is stated for $V=E_1$, but due to equivalence, it is also true for general $V$) and \cref{equ:LDP-g-1V}.  Therefore, we obtain for $n\geq k$,
    \begin{align*}
     \P_{1\leq j\leq n}(\bf U(j)\in \cal U^c)&=\frac{|\log r|/c}{n}\P_{1\leq j< |\log r|/c}(\bf U(j)\in\cal U^c)+\frac{n-|\log r|/c}{n}\P_{ |\log r|/c\leq j\leq n }(\bf U(j)\in\cal U^c) \\
     &<\frac{|\log r|/c}{n} +C r^\beta \leq C(r^\beta+1/\sqrt{k}), 
    \end{align*}
     where the constant $c>0$ comes from \cref{eq:hol g+ n>m}  and the last inequality is due to $r=\exp(-\sqrt{k}/10)$, $k\leq n$.
    Using \cref{lem: unify un in}, we have for $n\geq k$
    \begin{equation}\label{equ:Ij U}
     \P_{1\leq j\leq n}(\bf I(j)\in\cal U^c)\ll  r^\beta +1/\sqrt{k}+e^{-\beta n}. 
     \end{equation}
    The proof is complete by taking $k$ large.
    \end{proof}
   

\subsubsection{Apply entropy growth of the Euclidean case}
\label{sec:entropy growth step 4}
 Let
    \begin{equation}\label{equ:epsilon'}
        \epsilon'=\min\{ \frac{\epsilon}{10C_5},\frac{1-\alpha}{10} \},
    \end{equation}
    where $C_5$ comes from \cref{lem:theta psi x positive}.
    Let $\delta_1=\delta(\epsilon'/2)$ from \cref{thm:entropy euclidean} and let $m=M(\epsilon'/4)$ from \cref{prop: BHR 3.19}. Take $\epsilon_1>0$ to be a constant 
    independent of $\epsilon', \delta_1, m$, which will be determined at the end of the proof.
We introduce the following sets for $j,k\in\N$ 
   \begin{align*}
   \cal E_1(j)&=\left\{ (\ell,\bi,x)\in\cal E_0'(j)\times\P(\R^3):\,\,\,x\in \supp (\mu_{\varphi(j,\bf i)})_{\bI},\,\,\,  \frac{1}{k}H(S_{t(\ell,x)}T_{[\ell(x)]}[\theta_{\ell,j}.x],\calQ_{j+k})\geq \frac{\epsilon}{C_5}  \right\},  \\
   \cal E_2(j)&=\left\{\begin{array}{ll}
      &(\ell,\bi,x)\in\cal E_0'(j)\times\P(\R^3): \frak m_2(\ell,j,\bi) \text{ is $(1-\epsilon',\epsilon'/2,m)$-entropy porous}\\ 
      &\text{from scale $j$ to $j+k$ }
   \end{array}   \right\}, \\
   \cal E_3(j)&=\left\{ \begin{array}{ll}
   &(\ell,\bi,x)\in\cal E_0'(j)\times\P(\R^3):\,\,\,x\in \supp (\mu_{\varphi(j,\bf i)})_{\bI},\\ 
   &\frac{1}{k}H(S_{t(\ell,x)}T_{[\ell(x)]}[\theta_{\ell,j}.x]*\frak m_2(\ell,j,\bi),\calQ_{j+k})
   \geq  \frac{1}{k}H(\frak m_2(\ell,j,\bi),\calQ_{j+k})+\delta_1/2
   \end{array}
   \right \},  \\
   \cal E_4(j)&=\left\{ (\ell,\bi,x)\in\cal E_0'(j)\times\P(\R^3):\  \frac{1}{k}H(\frak m_2(\ell,j,\bi),\calQ_{j+k})\geq \alpha-\epsilon_1 \right\}.
      \end{align*}
    We denote by $\pi_2$ the projections from $L\times \Lambda^*\times\P(\R^3)$ to $L\times\Lambda^*$.
   We want to estimate
\[ \E_{1\leq j\leq n,\ell\sim\theta}\left(\int \frac{1}{k}H((S_{t(\ell,x)}T_{[\ell(x)]}[\theta_{\ell,j}.x])*\frak m_2(\ell,j,\bi),\calQ_{j+k})\dd(\mu_{\varphi(j,\bi)})_\bI(x),\cal E_0'(j)\right)\]
in \cref{equ:before entropy growth}. The rough idea is to estimate the measure of each set above and it is achieved in \cref{cla:e1} to \cref{cla:e4}.
   
   \begin{claim} \label{cla:e1}
   For $k>K(\epsilon',\delta_1)$, we have
   \[\cal E_1(j)\cap\cal E_2(j)\subset \cal E_3(j). \]
    \end{claim}
   \begin{proof}

  For any $(\ell,\bf i,x)\in \cal E_1(j)\cap \cal E_2(j)$, we want to apply \cref{thm:entropy euclidean} to $S_{t(\ell,x)}T_{[\ell(x)]}[\theta_{\ell,j}.x]$ and $\frak m_2(\ell,j,\bi)$
   to obtain growth of entropy at the scale $(k,n)=(j,k)$. In view of the definitions of $\cal E_1(j)$ and $\cal E_2(j)$, 
   it remains to estimate the size of the support of these two measures.

   It follows from \cref{equ:supp m2} that
   \begin{equation*}
   \supp \frak m_2(\ell,j,\bi) \subset B({h^+_{\ell^{-1}E_1,\bi}}, q^{-j}/r^{11}).
   \end{equation*}
    And \cref{equ:supp theta x} yields that 
    \begin{equation*}
    \diam \left(\supp S_{t(\ell,x)}T_{[\ell(x)]}[\theta_{\ell,j}.x]\right)\leq 16C_LC_p^2C_1^9q^{-j}/r^{10}.
    \end{equation*}
  Take the constant $C$ in \cref{thm:entropy euclidean} to be $16C_LC_p^2C_1^9/r^{12}$. Then 
   \[\frac{\log C}{k}=\frac{\log(16C_LC_p^2)+9\log C_1-12\log r}{k}\leq \frac{\log(16C_LC_p^2)+9\log C_1+\sqrt{k}}{k},\] which goes to $0$ as $k$ goes to infinity. So we have $k>N(\epsilon',\delta_1,m)\log C$. Therefore, \cref{thm:entropy euclidean} yields    
    \[\frac{1}{k}H(S_{t(\ell,x)}T_{[\ell(x)]}[\theta_{\ell,j}.x]*\frak m_2(\ell,j,\bi),\calQ_{j+k})\geq \frac{1}{k}H(\frak m_2(\ell,j,\bi),\calQ_{j+k})+\delta_1.\qedhere \]
   \end{proof}
    \begin{claim} \label{cla:e2}
    For $k\geq K(C_1,\epsilon)$, we have 
   \[\E_{1\leq j\leq n,\ell\sim\theta}\left(\int 1_{\cal E_1(j)}(\ell,\bi,x)  \dd(\mu_{\varphi(j,\bi)})_\bI(x)\right)\geq \frac{\epsilon}{4C_5}.\]
   \end{claim}
   \begin{proof}
      For any $j\in \N$, we have the equality $\mu=\E_{j}(\mu_{\varphi(j,\bf i)})$. Hence, using \eqref{equ:mu j i}, we have
       \[ \sup_{A \text{ Borel }}|\mu(A)-\E_{j}((\mu_{\varphi(j,\bf i)})_{\bf I})(A)|\leq \E_j(|\mu_{\varphi(j,\bi)}(A)-(\mu_{\varphi(j,\bf i)})_{\bf I}(A)| )\ll r^\beta.  \]
       Let $A(\ell,j)=\{x:\ \frac{1}{k}H([\theta_{\ell,j}.x],\calQ_{j+k+t})\geq \frac{\epsilon}{C_5} \}$. Then for any $1\leq j\leq n$,  
       \[\E_{j,\ell\sim\theta}\left((\mu_{\varphi(j,\bf i)})_{\bf I}(A(\ell,j)) \right)=\E_{\ell\sim\theta}\E_{j}\left((\mu_{\varphi(j,\bf i)})_{\bf I}(A(\ell,j)) \right)\geq \E_{\ell\sim\theta}\left(\mu(A(\ell,j))\right)-O(r^\beta), \]
       where the inequality is due to the above inequality. Summing over $1\leq j\leq n$ and combining with \cref{lem:theta psi x positive}, we obtain 
       \begin{align}
       \label{equ:claim 2}
       &\E_{1\leq j\leq n,\ell\sim\theta}\left( (\mu_{\varphi(j,\bf i)})_{\bf I}\left\{x:  \frac{1}{k}H([\theta_{\ell,j}.x],\calQ_{j+k+t})\geq \frac{\epsilon}{C_5}\right\}\right)\\
       &\geq \E_{1\leq j\leq n,\ell\sim \theta}\left(\mu\left\{x: \frac{1}{k}H([\theta_{\ell,j}.x],\calQ_{j+k+t})\geq \frac{\epsilon}{C_5}\right\}\right)-O(r^\beta)\nonumber\\
       &\geq\frac{\epsilon}{C_5}-O(r^\beta)\geq \frac{\epsilon}{2C_5}\nonumber. 
       \end{align}
       Then we can use \cref{equ:t psi x}, $|t(\ell,x)-t|\leq 8\log C_1/r$ (\cref{equ:t psi x}) to replace $\frac{1}{k}H([\theta_{\ell,j}.x],\calQ_{j+k+t})$ by $\frac{1}{k}H(S_{t(\ell,x)}T_{[\ell(x)]}[\theta_{\ell,j}.x],\calQ_{j+k})$, and it produces 
       an error $ O(\log C_1/r)\leq \epsilon'/2$. 

Comparing \cref{equ:claim 2} with Claim 2, we 
   use \cref{lem: change support} to add the restriction of $\cal E_0'(j)$. The proof is complete.
   \end{proof}
   \begin{claim}\label{cla:e3}  
   
   For $k>K(\epsilon')$ and $n>N(\epsilon',k)$, we have
   \[\P_{1\leq j\leq n,\ell\sim\theta}(\pi_2(\cal E_2(j)))\geq 1-\epsilon'.\]
   \end{claim}
   \begin{proof}
   We fix any $\ell\in \supp\theta$, we will show that
   \begin{equation*}
   \P_{1\leq j\leq n}\left\{\bf i\in\cal E_0'(j)(\ell):\ \frak m_2(\ell,j,\bi) \text{ is }(1-\epsilon',\epsilon'/2,m)\text{-entropy porous}  \text{ from scale $j$ to $j+k$}\right\}>1-\epsilon'.
   \end{equation*}
   Then the claim follows by integrating $\ell$ using the measure $\theta$.
   
   It is good to recall that 
   that 
   $\frak m_2(\ell,j,\bi)$ is a restriction of $\pi_{(\ell^{-1}E_1)^\perp}g_{\bf i}\mu$ from \cref{equ:m2} and \cref{lem:property m2}.
    \cref{prop: BHR 3.19} yields
     \begin{equation}\label{equ:i porous}
     \P_{1\leq j\leq n}\left\{\bf i\in \bI(j):\pi_{(\ell^{-1}E_1)^\perp}g_{\bf i}\mu \text{ is }(\alpha,\epsilon'/4,m)\text{-entropy porous}  \text{ from scale $j$ to $j+k$}\right\}>1-\epsilon'/4. 
     \end{equation}
     
     We take $k$ large enough such that for $r=\exp(-\sqrt{k}/10)$, we have $r^\beta<\epsilon'/4$ and $|\log r|/k=1/\sqrt{k}\leq \epsilon'/4$. 
     
     For $\bi\in\cal E_0'(j)(\ell)$, if $\bi$ belongs to the set defined in \cref{equ:i porous}, then   \cref{lem:support mu2} implies that $\frak{m}_2(\ell,j,\bf i)$ is 
     $(\alpha,\epsilon'/2,m)$-entropy porous from scale $j$ to $j+k$. (Actually, \cref{lem:support mu2} only allows us to obtain the entropy porosity of $\frak{m}_2(\ell,j,\bf i)$ from scale $j+4|\log r|$ to $j+k$. Since $|\log r|/k$ is sufficiently small, we obtain the result.)

        Combined with \cref{lem: change support} for the measure of $\cal E_0'(j)(\ell)$, we obtain the first inequality in the proof and finish the proof of the claim.
     %
     %
     \end{proof}
     
\begin{claim}\label{cla:e4} For $k>K(\epsilon_1)$ and $n>N(\epsilon_1,k)$
\[\P_{1\leq j\leq n,\ell\sim\theta}(\pi_2(\cal E_4(j)))\geq 1-\epsilon_1.\]
\end{claim}

\begin{proof}
Fix any $\ell \in \supp \theta$. Recall from \cref{equ:m2} that $\frak m_2(\ell,j,\bi)$ is the $r^5$-attracting part of the pair $(h_{\ell^{-1}E_1,\bf i},\pi_{\ell^{-1}E_1,\bf i} \mu)$.
 Let $\epsilon=\delta_1=\epsilon_1/8$. If $r^5\leq r_1(\epsilon_1)$, then we can apply \cref{lem:entropy m2} to $\frak m_2$ with $V=\ell^{-1}E_1$ and $(m,i)=(k,j)$. This gives
    \begin{equation}\label{equ:pi 2 e4}
     \P_{1\leq j\leq n}\left(\frac{1}{k}H(\frak m_2(\ell,j,\bi),\calQ_{k+j})\geq\alpha-\epsilon_1/2 \right)\geq 1-\epsilon_1/4.
    \end{equation}
    Observed that $\pi_2(\cal E_4(j))(\ell)=\{\bi:\ (\ell,\bi)\in  \pi_2(\cal E_4(j))\}$ is the intersection of the set of \cref{equ:pi 2 e4} with $\cal E_0'(j)(\ell)$. Combined with \cref{lem: change support}, we finishes the proof of the claim.
    \end{proof}

    \begin{proof}[Proof of \cref{thm.entropy}] With \cref{lem:before entropy growth} available, it remains to estimate 
    \[ \E_{1\leq j\leq n,\ell\sim\theta}\left(\int \frac{1}{k}H((S_{t(\ell,x)}T_{[\ell(x)]}[\theta_{\ell,j}.x])*\frak m_2(\ell,j,\bi),\calQ_{j+k})\dd(\mu_{\varphi(j,\bi)})_\bI(x),\cal E_0'(j)\right).\]
    Here we abuse the notation, for a set $F$ in $L\times \Lambda^*\times \P(\R^3)$, we use $\P_{1\leq j\leq n}(F)$ to denote
    \[\P_{1\leq j\leq n}(F)=\E_{1\leq j\leq n,\ell\sim\theta}\left(\int 1_F(\ell,\bi,x)\dd(\mu_{\varphi(j,\bi)})_\bI( x)\right). \]
    Under this notation, due to the sets $\cal E_2(j)$ and $\cal E_4(j)$ independent of $x$, we have $\P_{1\leq j\leq n}(\cal E_2(j))=\P_{1\leq j\leq n,\ell\sim\theta}(\pi_2(\cal E_2(j)))$, similarly for $\cal E_4(j)$. \cref{cla:e2}, \cref{cla:e3} and \cref{cla:e4} are the estimates of $\P_{1\leq j\leq n}(\cal E_1(j))$, $\P_{1\leq j\leq n}(\cal E_2(j))$ and $\P_{1\leq j\leq n}(\cal E_4(j)) $.
    
    Therefore, we have
    \begin{align*}
        &\E_{1\leq j\leq n,\ell\sim\theta}\left(\int \frac{1}{k}H(S_{t(\ell,x)}T_{[\ell(x)]}[\theta_{\ell,j}.x]*\frak m_2(\ell,j,\bi),\calQ_{j+k})\dd(\mu_{\varphi(j,\bi)})_\bI( x),\cal E_0'(j)\right)\\
        &\geq (\delta_1/2+\alpha-\epsilon_1)\P_{1\leq j\leq n}(\cal E_3(j)\cap \cal E_4(j))+(\alpha-\epsilon_1-O(1/k))\P_{1\leq j\leq n}(\cal E_4(j)-\cal E_3(j))\\ &\geq
        (\delta_1/2)\P_{1\leq j\leq n}(\cal E_3(j)\cap \cal E_4(j))+(\alpha-\epsilon_1)\P_{1\leq j\leq n}(\cal E_4(j))-O(1/k)
    \end{align*}
    where on $\cal E_3(j)\cap \cal E_4(j)$, we use their definition; on $\cal E_4(j)-\cal E_3(j)$, we estimate the integrand using the trivial lower bound  \cref{equ:entropy growth trivial}.
    
    For $\P_{1\leq j\leq n}(\cal E_4(j))$, we use \cref{cla:e4}. For $\P_{1\leq j\leq n}(\cal E_3(j)\cap \cal E_4(j))$, we use \cref{cla:e1}, \cref{cla:e2}, \cref{cla:e3} and \cref{cla:e4} to obtain
    \[\P_{1\leq j\leq n}(\cal E_3(j)\cap \cal E_4(j))\geq \P_{1\leq j\leq n}((\cal E_1(j)\cap \cal E_2(j))\cap \cal E_4(j))\geq  \frac{\epsilon}{4C_5}-\epsilon'-\epsilon_1\geq \frac{\epsilon}{10 C_5}-\epsilon_1, \]
    where $\epsilon'<\epsilon/10C_5$ due to \cref{equ:epsilon'}.
    Therefore
    \begin{align*}
   & \E_{1\leq j\leq n,\ell\sim\theta}\left(\int \frac{1}{k}H(S_{t(\ell,x)}T_{[\ell(x)]}[\theta_{\ell,j}.x]*\frak m_2(\ell,j,\bi),\calQ_{j+k})\dd(\mu_{\varphi(j,\bi)})_\bI( x),\cal E_0'(j)\right)\\
    &\geq (\delta_1/2)(\frac{\epsilon}{10 C_5}-\epsilon_1)+(\alpha-\epsilon_1)(1-\epsilon_1)-O(1/k)=\alpha+\frac{\delta_1 \epsilon}{20C_5}-\epsilon_1(\alpha+\delta_1)-O(1/k).
    \end{align*}
    
    Then from \cref{equ:before entropy growth}, we obtain
    \[ \frac{1}{n}H([\theta.\mu],\calQ_{t+n})\geq  \alpha+\delta_1\epsilon/5C-\epsilon_1-O(\frac{k}{n}+r^\beta+ \frac{\log(C_1/r)}{k}), \]
    where $\delta_1\epsilon/C$ is constant and other terms can be made arbitrarily small by taking $\epsilon_1$ small, then $k$ and $n/k$ large and $1/r=\exp(\sqrt{k}/10)$.
\end{proof}

\section{Exponential separation}
Recall the exponential separation condition for $\nu$ supported on $\SL_3(\R)$: there exist $C>0$ and $n_C\in\N$ such that for any $n\geq n_C$ any $g\neq g'$ in $\supp\nu^{*n}$, we have
\[ d(g,g')>1/C^n, \]
where $d$ is the left-invariant Riemannian metric on $\SL_3(\R)$. 
Notice that from \cref{lem: U V propty}, for any $g\in \SL_3(\R)$ such that $g^{-1}V\notin V^\perp$, the action of $\pi_{V^{\perp}}g$ from $\P(\R^3)$ to $\P(V^\perp)$ can be identified by an element $\ell$ in $L_V$.

\begin{lem}\label{lem:g V V perp}For a fixed $g\in\SL_3(\R)$, the set
\[ S(g):=\{V\in\P(\R^3):\ g^{-1}V\in V^\perp \}  \]
is a subvariety of $\P(\R^3)$ whose image under $p_2$
is contained in a hyperplane in $\P(Sym^2\R^3)$, where $p_2$ is defined in \cref{lem:subvariety}. 
\end{lem} 
\begin{proof}The set of such $V$ is given by the zero locus of the equation\[\langle g^{-1}v,v\rangle=0, \]
where $v$ is a non-zero vector in $V$. For $v,w\in\R^3$ let $f(v,w)=\langle g^{-1}v,w\rangle+\langle g^{-1}w,v\rangle$, which is non-trivial. Then due to $f(w,v)=f(v,w)$, the linear form $f$ on $\R^3\otimes\R^3$ induces a linear form on $Sym^2\R^3$. Since $\langle g^{-1}v,v\rangle=0$ is equivalent to $f(v,v)=0$, we obtain the lemma.\end{proof}

For any $V\in\P(\R^3)$, recall that we have a left-invariant Riemannian metric on $L_V$. The action of $\SL_3(\R)$ on row vectors $W$ is by right multiplication. We say 
 $W\in\P(V^\perp)$ if $W\cdot V=0$, where $W\cdot V$ is the product of a unit row vector in $W$ with a unit column vector in $V$. Recall from \cref{lem:pi L V} that $\pi_{L_V}$ is a map from a dense open set of $\SL_3(\R)$ to $L_V$.

\begin{lem}\label{lem:pi V bad}There exists $c'>0$ such that for any $C>0$ large enough, any $g\neq g'$ in $\SL_3(\R)$ with $d(g,g')>1/C$, the set 
\begin{equation} \label{set: exp sep g V} 
    \{V\in\P(\R^3): \ d(\pi_{L_V}g,\pi_{L_V}g')<1/C^4\} 
\end{equation}
has diameter less than $c'\|g\|^6/C$.
\end{lem}
The following three lemmas are preparations for the proof of Lemma \ref{lem:pi V bad}. 
\begin{lem}\label{lem:distance equivalent}
    \begin{enumerate}
        \item 
    
    There exists $\epsilon>0$ such that for any $V$, we have for any $u\in U_V$ and $\ell\in Z_V:=B(id, \epsilon)\subset L_V$,
    \[d(\pi_{L_V}u\ell,\pi_{L_V}id)=d(\pi_{L_V}\ell,\pi_{L_V}id)\simeq\sup_{W\in \P(V^\perp)}d(W\ell,W)=\sup_{W\in \P(V^\perp)}d(Wu\ell,W). \]
    Here the first two $d$ are the left-invariant distance in $L_V$, and the last two $d$ are the spherical distance on the projective space. 
    
   \item  There exists $\epsilon_1>0$ independent of $V$ such that for any $g\in B(\id,\epsilon_1)\subset \SL_3(\R)$,
   \[ d(\pi_{L_V}g,\pi_{L_V}id)\simeq\sup_{W\in \P(V^\perp)}d(Wg,W). \]

   \item There exists $\epsilon_2>0$ independent of $V$ such that for any $g\in \SL_3(\R)$ if
   \[\sup_{W\in \P(V^\perp)}d(Wg,W)\leq \epsilon_2 \]
   then
   \[ d(\pi_{L_V}g,\pi_{L_V}id)\ll\sup_{W\in \P(V^\perp)}d(Wg,W). \]
   \end{enumerate}
\end{lem}
\begin{proof}
    Without loss of generality, we can assume $V=E_1$. Recall that the element $\ell\in L$ can be written as $\ell=\begin{pmatrix} \det h & 0\\ n & h
    \end{pmatrix}$. For $\ell$ close to identity, we have $h\in \SL(2,\R)$. For $(0,w_1,w_2)$ a unit row vector in $W\in V^\perp$, we denote $(w_1,w_2)$ by $\mathbf{w}$. Then we have
    \begin{equation}\label{equ:w ell w}
 d(W\ell,W)=\frac{\|(0,\mathbf{w})\ell\wedge (0,\mathbf{w}) \|}{\|(0,\mathbf{w})\ell\|\| (0,\mathbf{w}) \|  }.  
    \end{equation}
     Since $\ell$ is close to the identity, the vector $(0,\mathbf{w})\ell$ has length close to 1.  So we have
    \begin{align*}
    d(W\ell,W)&\simeq \|(0,\mathbf{w})\ell\wedge(0,\mathbf{w}) \|\simeq \|\mathbf{w}\cdot n \|+\|\mathbf{w}h\wedge\mathbf{w} \|\\
    &\simeq \|\mathbf{w}\cdot n \|+d(\R(\mathbf{w}h),\R(\mathbf{w})).           \end{align*}
    Notice that the metric induced by the norm and the Riemannian distance are bi-Lipschitz equivalent in a neighborhood of the identity in $L$.
    Therefore we have $d(W\ell,W)\leq \|n\|+\|h-id\|\simeq d(\ell,id)$. 
    
    On the other hand, we have
    \[\sup_{\mathbf{w}}\|\mathbf{w}\cdot n \|=\|n\|. \]
    For any triple of unit vectors $\mathbf{w}_i, i=1,2,3$ which are $1/10$ separated,
    \[ d_{\SL_2(\R)}(h,id)\ll\sup_{i=1,2,3}(d(\R(\mathbf{w}_ih),\R(\mathbf{w}_i)))\leq \sup_{\mathbf{w}}d(\R(\mathbf{w}h),\R(\mathbf{w}))\leq \sup_{W\in \P(V^\perp)} d(W\ell,W). \]
   Here we use the fact that the map $h\mapsto (\mathbf{w}_1h,\mathbf{w}_2h,\mathbf{w}_3h), \PSL_2(\R)\to \P(\R^2)^3$ 
   is a smooth injection and bi-Lipschitz to its image on a small neighborhood of identity (see, for example Page 827 in \cite{hochman_dimension_2017}). Therefore 
    \begin{equation}\label{equ:w ell}
    \sup_{W\in \P(V^\perp)} d(W\ell,W)\gg d_{\SL_2(\R)}(h,id)+\|n\|\gg d(\ell,id) .
    \end{equation}
   Since the group $U_V$ acts trivially on the row vector on $V^\perp$, the first statement follows easily from \cref{equ:w ell}.
     
     The second statement is due to the fact that the image of the product map from $U_V\times Z_V$ to $\SL_3(\R)$ contains a neighborhood of the identity.

    For the last statement, we write $g=u\ell$.
    The condition and \cref{equ:w ell w} imply that for any unit vector $\mathbf{w}$ in $V^\perp$ 
    \begin{equation}
    \label{g close to id}
   \|\mathbf{w}h\wedge \mathbf{w}\|+\|\mathbf{w}\cdot n\|\leq \epsilon_2\|(0,\mathbf{w})\ell\|\leq \epsilon_2(\|\mathbf{w}h\|+\|\mathbf{w}\cdot n\|).  
   \end{equation}
   Hence
   \begin{equation}
   \label{h close to id}
    \sup_{W\in V^\perp}d(W,Wh)\leq\epsilon_2
    \end{equation}
    for the $h$ part in $\pi_{L_{V}}g$. We can obtain the $h$ close to $id$ as follows: for any $\R v_1,\R v_2\in V^\perp$ with $d(\R v_1,\R v_2)\geq 4\epsilon_2$, \cref{h close to id} gives ``$\simeq$'' in the following equation
    \begin{equation*}
    1\simeq \frac{d(\R v_1h,\R v_2h)}{d(\R v_1,\R v_2)}=\frac{\|v_1\|\|v_2\|}{\|v_1h\|\|v_2h\|};
    \end{equation*} by taking unit vectors $v_1$ such that $\|v_1h\|=\|h\|$, we obtain that $\|h\|$ close to $1$; lastly, use Cartan decomposition to obtain $h$ close to $id$. Using \cref{g close to id}, we obtain that $\|n\|$ is close to zero. If $\epsilon_2$ is sufficiently small such that $\ell\in B(\id,\epsilon)\subset L$, then we can use the first statement to conclude about $g$.
\end{proof}

\begin{lem}\label{lem:bad V}There exists a constant $c>0$ such that the following holds. For any $V\in\P(\R^3)$, $g\in \SL_3(\R)$ and $\epsilon>0$ which is sufficiently small, if  $d(id,g)>\epsilon^{1/3}$ and $d(\pi_{L_V}id,\pi_{L_V}g)\leq \epsilon$, then for any $V'$ with $d(V',V)\geq c\epsilon^{1/3}$, we have $d(\pi_{L_{V'}}id,\pi_{L_{V'}}g)> \epsilon$. 
   \end{lem} 
    \begin{proof}
    Suppose that our lemma does not hold, i.e. $d(\pi_{L_{V'}}id, \pi_{L_{V'}}g)\leq \epsilon$ for some $V'$ with $d(V',V)\geq c\epsilon^{1/3}$, where the constant $c$ will be defined at the end of the proof of the lemma. For simplicity, we can assume $V=E_1$ and $V'=\R(e_1+\delta e_2)$ where $|\delta|<1$ and $\delta\geq d(V,V')\geq c\epsilon^{1/3}$.
    
    Consider the $UL$ decomposition $g=u\ell$ of $g$. We write $u=\begin{pmatrix}
       \lambda^2 & x & y\\ 0 & \lambda^{-1} & 0 \\ 0 & 0 & \lambda^{-1} 
    \end{pmatrix}.$ \textbf{The idea is to use the condition  $d(\pi_{L_{V'}}id, \pi_{L_{V'}}g)\leq \epsilon$ to obtain that the $u$ part is also small, which contradicts the hypothesis $d(id,g)>\epsilon^{1/3}$.}
    
    We apply Lemma \ref{lem:distance equivalent} (2) to $V'=\R(e_1+\delta e_2)$ and $W=\R(\delta e_1^t- e_2^t),\ \R(\delta e_1^t- e_2^t+e_3^t)\in (V')^\perp$. Combined with $d(\pi_{L_{V'}}id, \pi_{L_{V'}}g)\leq \epsilon$, we obtain
    \begin{equation}\label{equ:delta e1t}
    \begin{split}
    d(\delta e_1^t- e_2^t,(\delta e_1^t- e_2^t)g)&\leq c_2\epsilon,\\
    d(\delta e_1^t- e_2^t+e_3^t,(\delta e_1^t- e_2^t+e_3^t)g)&\leq c_2\epsilon.
    \end{split}
    \end{equation}
 Due to $d_L(\ell,id)\leq \epsilon$, we have $d_{\SL_3(\R)}(\ell,id)\leq O(\epsilon)$. Hence $\ell$ moves any vector in $\P(\R^3)$ of distance at most  $c_3\epsilon$ for some $c_3>0$ if $\epsilon$ is small enough. Therefore by triangle inequality and \cref{equ:delta e1t} we obtain a similar inequalities as \cref{equ:delta e1t} for $u$:
\begin{eqnarray}
&&d(\delta e_1^t- e_2^t, \delta e_1^t+(- \lambda^{-3}+\delta x\lambda^{-2})e_2^t+\delta y\lambda^{-2}e_3^t)\nonumber\\
&=&d(\delta e_1^t-e_2^t, (\delta e_1^t-e_2^t)u)\nonumber\\
&\leq &d(\delta e_1^t-e_2^t, (\delta e_1^t-e_2^t)g)+d((\delta e_1^t-e_2^t)u\ell, (\delta e_1^t-e_2^t)u)\nonumber\\
&\leq &c_2\epsilon+c_3\epsilon\leq 
c_4\epsilon   \label{eqn: triangle 1}
\end{eqnarray}
for some constant $c_4>0$. Similarly, we get 
\begin{equation}\label{eqn: triangle 2}
    d(\delta e_1^t- e_2^t+e_3^t,  \delta e_1^t+(- \lambda^{-3}+\delta x\lambda^{-2})e_2^t+(\delta y\lambda^{-2}+\lambda^{-3})e_3^t)\leq c_4\epsilon.  
\end{equation}
We consider the following elementary geometric lemma. 
\begin{lem}\label{lem: Spherical estimate exp section}There exists $c_5>0$ such that for any $\delta\geq 6c_4\epsilon$ with $\epsilon$ small enough, if $w,w'\in \mathrm{Span}(e_2^t,e_3^t)$ such that $1\leq \|w\|\leq 2$ and $d_{\P(\R^3)}(\delta e_1^t+w, \delta e_1^t+w')<c_4\epsilon$, then \[\|w-w'\|\leq c_5\cdot  \frac{\epsilon}{\delta}.\] 
\end{lem}
\begin{proof}Consider the triangle $A_1A_2A_3$ contained in $\R^3$, where $A_1=0\in \R^3$, $A_2=\delta e_1^t+w$, $A_3=\delta e_1^t+w'$, then by the sine rule we have 
\begin{equation}\label{equ:sine rule}
\frac{\|A_2 A_3\|}{\|A_1 A_2\|}= \frac{|\sin \angle A_2 A_1 A_3|}{|\sin \angle A_1 A_3 A_2|}.
\end{equation}
By our assumption of $\epsilon,\delta$, \[1\leq \|A_1A_2\|\leq 2,~~\|A_2A_3\|=\|w-w'\|,~~ |\sin\angle A_2A_1A_3|\leq c_4\epsilon.\]
From the third one and $w,w'\in \mathrm{Span}(e_2^t,e_3^t)$ we have
\[ \delta\|w-w'\|\leq \|(\delta e_1^t+w)\wedge(\delta e_1^t+w')\|\leq c_4\epsilon\|\delta e_1^t+w\|\|\delta e_1^t+w'\|\leq c_4\epsilon \cdot 3(1+\|w'\|). \]
Therefore
\[ (1-\frac{3c_4\epsilon}{\delta})\|w'\|\leq\|w-w'\|- \frac{3c_4\epsilon}{\delta}\|w'\|+\|w\| \leq \frac{3c_4\epsilon}{\delta}+2.  \]
This implies $\|w'\|\leq 6,\text{ and } $
\[|\sin\angle A_1A_3A_2|=\frac{\|(\delta e_1^t+w')\wedge(w-w')\|}{\|\delta e_1^t+w'\|\|w-w'\|}\geq \frac{\delta\|w-w'\|}{\|\delta e_1^t+w'\|\|w-w'\|}\geq \delta/7. \] 
Therefore by \cref{equ:sine rule} 
\[\|w-w'\|\leq c_5\cdot  \frac{\epsilon}{\delta}\] for some $c_5>0$. \end{proof}
Letting $(w,w')$ be $(-e_2^t, (- \lambda^{-3}+\delta x\lambda^{-2})e_2^t+\delta y\lambda^{-2}e_3^t)$ in \eqref{eqn: triangle 1} and $(-e_2^t+e_3^t, (- \lambda^{-3}+\delta x\lambda^{-2})e_2^t+(\delta y\lambda^{-2}+\lambda^{-3})e_3^t)$ in \eqref{eqn: triangle 2}, applying Lemma \ref{lem: Spherical estimate exp section}, we get 
\begin{align}\label{equ:w e3}
 \|(1-\lambda^{-3}+\delta x\lambda^{-2})e_2^t+\delta y\lambda^{-2}e_3^t\|&\leq c_5\frac{\epsilon}{\delta},\\
 \label{equ:w' e3}
 \|(1-\lambda^{-3}+\delta x\lambda^{-2})e_2^t+(-1+\delta y\lambda^{-2}+\lambda^{-3})e_3^t\|&\leq c_5\frac{\epsilon}{\delta}.
\end{align}
Due to $\delta\geq  c\epsilon^{1/3}$, by taking $c$ large enough (compared to $100c_i,i=2,\cdots,5$) and comparing  
the coefficients of $e_3^t$ in \cref{equ:w e3} and \cref{equ:w' e3}, we obtain $|\delta y\lambda^{-2}|,|1-\lambda^{-3}|\leq 2c_5\frac{\epsilon}{\delta}$, which implies $|\lambda-1|\leq \epsilon^{2/3}/10$ and $|\delta y|\leq \epsilon^{2/3}/10$. From coefficients of $e_2^t$ in \cref{equ:w e3} and \cref{equ:w' e3}, we obtain $|\delta x|\leq \epsilon^{2/3}/10$. Therefore we obtain $|x|,|y|\leq \epsilon^{1/3}/10$ and $|\lambda-1|\leq \epsilon^{1/3}/10$. Then we obtain a contradiction from $d(id,g)>\epsilon^{1/3}$ if $\epsilon$ is small enough.
\end{proof}
 
\begin{proof}[Proof of \cref{lem:pi V bad}]
By left-invariance of $d$ on $\SL_3(\R)$, we have $d(g,g')=d(id,g^{-1}g')$. Fix any $V$ in the set defined in \eqref{set: exp sep g V}, we plan to obtain the information of $\pi_{L_{(g^{-1}V)}}g^{-1}g'$ and hence obtain \cref{lem:pi V bad}.
We consider the $U_VL_V$ decomposition of $g,g'$ respectively: $g=u\ell$ and $g'=u'\ell'$. Then by \eqref{set: exp sep g V} we have 
    \[ d(\pi_{L_V}g,\pi_{L_V}g')=d_{L_V}(\ell,\ell')=d_{L_V}(id,\ell^{-1}\ell')\leq 1/C^4. \]
    This means the element $\ell^{-1}\ell'$ is close to $id$ in both $L_V$ and $\SL_3(\R)$. Since we assume that $C$ is large enough, we can assume $\ell^{-1}\ell'$ moves point on $\P(\R^3)$ with distance not greater than $C^{-7/2}$. 
    For any row vector $W\in (\ell^{-1}V)^\perp$, which means $0=W\cdot (\ell^{-1}V)=(W\ell^{-1})\cdot V $, we have $W\ell^{-1}\in V^\perp$. Due to $g^{-1}g'=\ell^{-1}u^{-1}u'\ell'$ and the left action of $U_V$ on column vectors in $\P(V^\perp)$ being trivial, so
    \begin{equation}\label{equ:w gg'}
     \sup_{W\in (\ell^{-1}V)^\perp}d(W,Wg^{-1}g')=\sup_{W\in (\ell^{-1}V)^\perp}d(W,W\ell^{-1}\ell'    )< 1/C^{7/2}.        
    \end{equation}
    (This is the main advantage of considering row vectors.) Due to the previous estimate (\cref{equ:w gg'}), 
    we can apply \cref{lem:distance equivalent} (3) to $g^{-1}g'$ and $\ell^{-1}V$ and obtain 
    \begin{equation}
    \label{projection small}
    d(\pi_{L_{\ell^{-1}V}}id,\pi_{L_{\ell^{-1}V}}g^{-1}g')\leq 1/C^3. 
    \end{equation}
 Then we can apply \cref{lem:bad V} to obtain that the set of $\ell^{-1}V=g^{-1}V$ such that \cref{projection small} holds has radius less than $c'/C$, where $c'=2c$ and $c$ is the constant we get in \cref{lem:bad V}. So the set of $V$ such that \cref{projection small} holds for $g^{-1}V$ has radius less than $c'\|g\|^6/C$, since the Lipschitz constant of $g$ on projective space is bounded by $\|g\|^6$ (See for example (13.2) in \cite{benoist_random_2016}). 
\end{proof}
Once \cref{lem:pi V bad} is available, we use an argument similar to that in \cite{barany_hausdorff_2017}, \cite{hochman_dimension_2017} and \cite{hochman_self-similar_2014} to obtain the following main proposition of this section.
\begin{prop}\label{prop.separation}
    Suppose that $ \nu$ satisfies the exponential separation condition. Then there exist $C>0$ and a subset $Y\subset \P(\R^{3})$ with $\mu^-$-measure zero such that for every $V\in \P(\R^{3})-Y$, there exists $N_V\in \mathbb{N}$ such that for any $n\geq N_V$, for any $g\neq g'\in \supp\nu^{*n}$
\[d(\pi_{L_V} g,\pi_{L_V} g')>C^{-n}. \]
\end{prop}
\begin{proof}
    For $\bi,\bj$ in $\Lambda^n$ with $g_\bi\neq g_\bj$, let 
    \[ E_{\bi,\bj}(C):=\{V\in\P(\R^3):\ d(\pi_{L_V}g_\bi,\pi_{L_V}g_\bj)<1/C^n\}. \]
We consider the set 
  \[ E_1(C):=\bigcap_{m\in \N}\bigcup_{n\geq m}\bigcup_{\bi,\bj\in\Lambda^n}E_{\bi,\bj}(C). \]
  Let $S$ be the countable union of subvarieties $S(g)$ given by \cref{lem:g V V perp} for $g\in \cup_{n\geq 1}\supp\nu^{*n}$. Then $S$ has $\mu^-$-measure $0$.
For any $V$ not in $E_1(C)\cup S$, there exists $N(V)\geq 1$ such that for any $n\geq N(V)$ and $g\neq g'\in \supp(\nu^{*n})$, the elements $\pi_{L_V} g, \pi_{L_V}g'$ are well-defined, and $d(\pi_{L_V} g, \pi_{L_V}g')\geq C^{-n}$. Hence we only need to show that $E_1(C)$ has $\mu^-$ measure $0$ when $C$ is large enough. 

Guivarc'h theorem (\cref{lem: Hold reg proj mes}) gives H\"older regularity of $\mu^{-}$. Hence by the Frostman lemma, there exists $s(\mu^-)>0$ such that any subset of $\P(\R^3)$ with positive $\mu^-$-measure set has Hausdorff dimension at least $s(\mu^-)$. Then it suffices to show that for any $s>0$, we can choose $C$ large enough such that $\dim 
   (E_1(C))\leq s$. We apply \cref{lem:pi V bad} to estimate the Hausdorff dimension of $E_1(C)$.

Fix an $s>0$ small enough and write $\cal H^s_\infty$ for the $s$-dimensional Hausdorff content on $\P(\R^3)$. By exponential separation condition of $\nu$, for $n$ large, we have $C'$ such that for any $g\neq g'\in\supp \nu^{*n}$, $d(g,g')>(C')^{-n}$. Then for $C\gg C'$ we get 
\begin{eqnarray*}
\cal H^s_\infty(E_1(C))&\leq& \lim_{N\to \infty} \sum_{n=N}^\infty \sum_{\bi, \bj\in \Lambda^n}(\rm{diam}\{E_{\bi,\bj}(C)\})^s\\
&\leq& \lim_{N\to \infty} \sum_{n=N}^\infty \sum_{\bi, \bj\in \Lambda^n}\left(\frac{\|g_\bi\|^{6}c'}{C^{n/4}}\right)^s\\ &&(\text{ apply \cref{lem:pi V bad} with $C$ large such that }d(g_\bi, g_\bj)>C^{-n/4})\\
&\leq &\lim_{N\to \infty} \sum_{n=N}^\infty \sum_{\bi, \bj\in \Lambda^n}(\frac{b}{C^{1/4}})^{sn}\text{ (here $b:=\sup_{g\in \supp \nu}\|g\|^6
\cdot c'$)}\\
&\leq & \lim_{N\to \infty}\sum_{n=N}^\infty D^{2n} (\frac{b}{C^{1/4}})^{sn} \text{ (here $D:=\#\{\supp \nu\}$)}\\
&\leq &0 \text{ (by taking $C$ large enough compare to $b,D$)}.
\end{eqnarray*}
Therefore the Hausdorff dimension of $E_1(C)$ is not greater than $s$ for $C$ sufficiently large.
\end{proof}

\section{Proofs of Theorem \ref{thm:lyapunov} and Theorem \ref{thm:projection}}\label{sec:main argument}
In this section, after preparations (\cref{lem:nu random walk}, \cref{lem: BHR 6.2} and \cref{lem: BHR 6.4}), we first apply Theorem \ref{thm.entropy} to show \cref{thm:projection}, then combine it with a Lerappier-Young formula recently shown in \cite{ledrappier_exact_2023} and \cite{rapaport_exact_2021} to conclude Theorem \ref{thm:lyapunov}.
\subsection{Preparations: estimates for the entropy}
Recall $\cal Q^{L_V}_n$ is the $q-$adic partition of the group $L_V$. In the case that there is no ambiguity, we may write it as $\calQ_n$. Consider the set
$$R:=\{V\in \P(\R^3):g^{-1}V\notin V^\perp, \forall g\in \cup_{n\geq 0} \supp \nu^{\ast n}\}.$$
By \cref{lem:subvariety} and \cref{lem:g V V perp}, $R$ is a $\mu^-$-full measure set. Recall that $\pi_{L_V}$ is a projection from $U_VL_V$ to $L_V$ defined in \cref{lem:pi L V}. Then for $\nu^{*n}, n\geq 1$ on $\SL_3(\R)$ and $V\in R$, the measure $\pi_{L_V}\nu^{*n}$ is a well-defined measure on $L_V$. From the exponential separation condition, we have
\begin{lem}\label{lem:nu random walk}
There exists $C>1$ such that for $\mu^{-}$-almost every $V\in \P(\R^3)$,
\[\lim_{n\rightarrow\infty}\frac{1}{n}H(\pi_{L_V}\nu^{*n},\calQ^{L_V}_{Cn})=h_{\mathrm{RW}}(\nu). \]
\end{lem}
\begin{proof}    
By Proposition \ref{prop.separation}, there exists $C>0$ such that for $V$ in a full $\mu^-$-measure set, distinct atoms of $\pi_{L_V}\nu^{*n}$ 
have distance at least $q^{-Cn/2}$ from each other for any $n$ large enough. On the other hand, the size of atoms of the partition $\calQ_{Cn}$ is about $q^{-Cn}$. So for $n$ large, the atoms of $\pi_{L_V}\nu^{*n}$ are located in different atoms in $\calQ_{Cn}$. Therefore 
    \[\lim_{n\rightarrow\infty}\frac{1}{n}H(\pi_{L_V}\nu^{*n},\calQ_{Cn})=\lim_{n\rightarrow\infty}\frac{1}{n}H(\pi_{L_V}\nu^{*n})=\lim_{n\rightarrow\infty}\frac{1}{n}H(\nu^{*n})=h_{\mathrm{RW}}(\nu), \]
    where the second equality is due to the injectivity of the map $\pi_{L_V}$ on $\supp\nu^{*n}$ given in \cref{prop.separation}.
\end{proof}        

\begin{lem}\label{lem: BHR 6.2}
 Let $C>1$ 
 and $\tau>0$. Then for $\mu^-$-a.e. $V\in\P(\R^3)$, 
 \[\limsup_{n\to \infty}\pi_{L_V}\nu^{*n}\{\ell\in L_V: \frac{1}{Cn}H(\pi_{V^\perp}\ell\mu, \cal Q_{(C+\chi_1)n})>\alpha-\tau\}=1.\]
\end{lem}
\begin{proof}
Recall that for $g\in U_VL_V$, $\pi_{V^{\perp}}((\pi_{L_V}g)\mu)=\pi_{V^\perp}g\mu$, and the measure $\nu^{*n}$ is the law of $\bU(n)$. So it suffices to show that for any $\epsilon>0$, for $\mu^{-}$-almost every $V$,
\begin{equation}\label{eqn: BHR 6.2 1st step}
    \limsup_{n\to \infty}\P\left\{\bU(n):\frac{1}{Cn}H(\pi_{V^\perp}g_{\bU(n)}\mu, \cal Q_{(C+\chi_1)n})>\alpha-\tau\right\}\geq 1-\epsilon.
\end{equation}
Fix a small $\epsilon> 0$ in \eqref{eqn: BHR 6.2 1st step}. 
By \cref{lem: BHR 3.14}, \cref{lem: decmp prj mesr U} (let $\delta$ in \cref{lem: decmp prj mesr U} to be $\tau/4$) and \cref{eq:ldp sigma g}, there exists some small $r=r(\epsilon, \tau)>0$ such that for any $n$ large, we can construct a set $X_{n,\epsilon}$ of $\bU(n)$ such that
\begin{enumerate}
    \item $\P(\bU(n)\in X_{n,\epsilon})>1-\epsilon/2$;
    \item For any $\bU(n)\in X_{n,\epsilon}$, $|\chi_1(g_{\bU(n)})-\chi_1n|\leq \tau n/10$;
    \item For any $\bU(n)\in X_{n,\epsilon}, n\geq N(r)$,  the $r$-attracting part $(\pi_{V^\perp}g_{\bU(n)}\mu)_{\bI}$ of the pair $(h_{V,{\bU(n)}},  \pi_{V,\bU(n)}\mu)$
satisfies
\begin{equation}\label{equ:un entropy concen}
\frac{1}{Cn}H((\pi_{V^\perp}g_{\bU(n)}\mu)_{\bI}, \cal Q_{Cn+\chi_1(g_{\bU(n)})}) \geq \frac{1}{Cn}H(\pi_{(g_{\bU(n)}^{-1}V)^\perp}\mu, \cal Q_{Cn})-\frac{\tau}{10}, 
\end{equation}
and $\pi_{V^\perp}g_{\bU(n)}\mu=(1-\tilde\delta)(\pi_{V^\perp}g_{\bU(n)}\mu)_\bI+ \tilde\delta(\pi_{V^\perp}g_{\bU(n)}\mu)_{\bI\bI} $ with $\tilde\delta\leq \tau/4$.
\end{enumerate}
By the 2nd condition of $X_{n,\epsilon}$ and \cref{eq:entropy m n}, for $\bU(n)\in X_{n,\epsilon}$,
replacing $\chi_1(g_{\bU(n)})$ by $\chi_1n$ in LFS of \cref{equ:un entropy concen} changes the value of LFS less than $\tau/5$. So we obtain
\[
\frac{1}{Cn}H((\pi_{V^\perp}g_{\bU(n)}\mu)_{\bI}, \cal Q_{(C+\chi_1)n}) \geq \frac{1}{Cn}H(\pi_{(g_{\bU(n)}^{-1}V)^\perp}\mu, \cal Q_{Cn})-\frac{\tau}{2}. \]
Then by the concavity of entropy, for $\bU(n)\in X_{n,\epsilon}$ 
\begin{eqnarray}\label{eqn: gd pt lg etrpy sec 6}
\frac{1}{Cn}H(\pi_{V^\perp}g_{\bU(n)}\mu, \cal Q_{(C+\chi_1)n})
&\geq& (1-\tilde\delta)\frac{1}{Cn}H((\pi_{V^\perp}g_{\bU(n)}\mu)_{\bI}, \cal Q_{(C+\chi_1)n}) \\
&\geq& (1-\tilde\delta)(\frac{1}{Cn}H(\pi_{(g_{\bU(n)}^{-1}V)^\perp}\mu, \cal Q_{Cn})-\frac{\tau}{2})\nonumber\\
&\geq& \frac{1}{Cn}H(\pi_{(g_{\bU(n)}^{-1}V)^\perp}\mu, \cal Q_{Cn})-\frac{3\tau}{4}, \nonumber
\end{eqnarray}
where the last inequality is due to $\tilde \delta\leq\tau/4$.

On the other hand, we can show the following lemma, which is similar to Lemma 6.2 in \cite{barany_hausdorff_2017}.
\begin{lem}\label{lem: Maker's in BHR 6.2}
Let $\tau>0$, $\epsilon>0$.  For $\mu^-$-a.e. $V\in \P(\R^3)$, there exists a sequence $n_k\to \infty$ such that 
$$\P\left\{\bU(n_k):\frac{1}{Cn_k}H(\pi_{(g_{\bU(n_k)}^{-1}V)^\perp}\mu, \cal Q_{Cn_k})>\alpha-\tau/10\right\}>1-\epsilon/10.$$    
\end{lem}
\begin{proof}
 Let $f_n(\omega,V)$ be a function on $\Lambda^\N\times\P(\R^3) $ defined by 
 \[ f_n(\omega,V)=\begin{cases}
 1,\ \text{ if }\frac{1}{Cn}H(\pi_{V^\perp}\mu, \cal Q_{Cn})>\alpha-\tau/10 \\
 0,\ \text{ otherwise.}
 \end{cases}   \]
 By exact dimensionality of $\pi_{V^\perp}\mu$, 
 $f_n$ converges $\nu^{\otimes\N}\otimes \mu^{-}$-a.e. to $1$ as $n\to \infty$ (\cref{lem:entropy dimension}). 
 
 We define a skew-product transformation $T$ on $\Lambda^{\N}\times \P(\R^3)$ by 
 \[T(\omega,V)=(T\omega,g_{\omega_1}^{-1}V),\]
 where $T\omega$ is the shift map given by $(T\omega)_j=\omega_{j+1}$. Due to stationarity of $\mu^-$ under $\nu^-$, the measure $\nu^{\otimes\N}\otimes\mu^-$ is $T$-invariant. Moreover, because $\mu^-$ is the unique $\nu^-$-stationary measure, $\nu^{\otimes\N}\otimes\mu^-$ is also $T$-ergodic, see for example \cite[Proposition 2.23]{benoist_random_2016}.
 
 Therefore by Maker's ergodic theorem, for $\nu^{\otimes\N}\otimes \mu^{-}$-a.e. $(\omega,V)$, we have
 $$\frac{1}{N}\sum_{n=1}^Nf_n(T^n(\omega,V))\to 1\,\,  \text{as} \,\,N\to \infty. $$ 
 Integrating the above equation over the $\omega$-component, by dominated convergence theorem, we get for $\mu^-$-a.e. $V$, 
 $$\frac{1}{N}\sum_{n=1}^N\P\left\{\bU(n):\frac{1}{Cn}H(\pi_{(g_{\bU(n)}^{-1}V)^\perp}\mu, \cal Q_{Cn})>\alpha-\tau/10\right\}\to 1\,\,\text{as}\,\, N\to \infty.$$
 Therefore, we can find a sequence $n_k\to \infty$ such that
 \[\P\left\{\bU(n_k):\frac{1}{Cn_k}H(\pi_{(g_{\bU(n_k)}^{-1}V)^\perp}\mu, \cal Q_{Cn_k})>\alpha-\tau/10\right\}>1-\epsilon/10.\qedhere \]
 \end{proof}

Combine Lemma \ref{lem: Maker's in BHR 6.2} with \eqref{eqn: gd pt lg etrpy sec 6}.
Let $\{n_k\}$ be the infinite sequence given in \cref{lem: Maker's in BHR 6.2}.  For all large $n_k$, $\frac{1}{Cn}H(\pi_{V^\perp}g_{\bU(n)}\mu, \cal Q_{(C+\chi_1)n})>\alpha-\tau$ holds for a set of $\bU(n_k)$ with measure greater than $1-\epsilon$. Hence
\[\P\left\{\bU(n_k):\frac{1}{Cn_k}H(\pi_{V^\perp}g_{\bU(n_k)}\mu, \cal Q_{(C+\chi_1)n_k})\geq \alpha-\tau\right\}>1-\epsilon. \qedhere\]
\end{proof}

\begin{lem}\label{lem: BHR 6.4}
    For $\mu^{-}$-a.e. $V\in \P(\R^3)$, we have
    \[\lim_{n\rightarrow\infty}\frac{1}{n}H(\pi_{L_V}\nu^{*n},\calQ_1^{L_V})\leq\alpha\chi_1(\nu).  \]
\end{lem}
\begin{proof}
By \cref{lem:entropy dimension}, for $\mu^-$- a.e. $V$ we have
\begin{equation}
\label{entropy limit}
\lim_{n\to \infty}\frac{1}{\chi_1n}H(\pi_{V^\perp}\mu, \cal Q_{\chi_1n})\rightarrow \alpha.
\end{equation}
Fix any $V$ for which \cref{entropy limit} holds. By conjugating with an element in $\SO_3(\R)$, we may assume $V=E_1$.

Let $X_i$ be matrix-valued i.i.d. random variables with the law $\nu$. Denote by  $Z_n(\omega):=X_1(\omega)\cdots X_n(\omega), \omega \in \Lambda^{\times \N}$. By Furstenberg's theorem (see for example \cite[Proposition 4.7]{benoist_random_2016}), the following Furstenberg boundary map is well-defined for $\nu^{\otimes \N}$-a.e. $\omega$ $$\xi: \Lambda^{\times \N}
\to \P(\R^3),\,\, \xi(\omega) \text{ is the unique point such that }\delta_{\xi(\omega)}=\lim_{n\to \infty}Z_n(\omega)_\ast\mu, $$ 
and satisfies that $\mu=\E(\delta_{\xi(\omega)})$ and $\xi(\omega)=X_1(\omega)\xi(T\omega), \nu^{\otimes \N}$-a.e.  
 $\omega$, where $T:\Lambda^{\times \N}\to \Lambda^{\times \N}$ is the shift map.


For each element $g\in \SL_3(\R)$, we consider its $UL$-decomposition: $g=u\ell=u\begin{pmatrix} \det h & 0\\ n& h\end{pmatrix}$ with $u\in U$ and $\ell\in L$.
 Recall $h^+$ is the contracting point of the $h$ part of $\ell$ in $\P({E_1}^\perp)$. Let $v_n(\omega)=h(Z_n(\omega))^+$. Then we can state the following exponential convergence lemma.
\begin{lem}\label{lem: HS 3.3} For any $\epsilon>0$, there exists $c(\epsilon)>0$ such that 
\begin{equation}\label{eqn: HS eqn 4.7}
\P\{\omega\in \Lambda^{\times \N}:d_{\P(\R^2)}(\pi_{{E_1}^\perp}\xi(\omega), v_n)<q^{-(\chi_1-\epsilon)n}\}= 1-O(e^{-c(\epsilon)n}). 
\end{equation}
\end{lem}
\begin{proof}
Fix any $x\in\P(E_1^\perp)$ and let $w_n(\omega)=Z_n(\omega)x$ in $\P(\R^3)$.
\begin{claim*}
\begin{equation*}
\P\{\omega\in \Lambda^{\times \N}:d_{\P(\R^3)}(\xi(\omega), w_n)<q^{-(\chi_1-\epsilon/10)n}\}=1-O(e^{-c(\epsilon)n}) .    
\end{equation*}
\end{claim*}
\begin{proof}
Due to the equivariance of $\xi$, we have
\[d(\xi(\omega),w_n)=d(Z_n(\omega)\xi(T^n\omega),Z_n(\omega)x). \]
By \cref{lem:action g}, if we have
\begin{equation}\label{equ: zn tnw}
\chi_1(Z_n(\omega))>(\chi_1-\epsilon/20)n,\  d(\xi(T^n\omega),H_{Z_n(\omega)}^-)>q^{-\epsilon n/40}\,\,\text{and}\,\, d(x,H_{Z_n(\omega)}^-)>q^{-\epsilon n/40},
\end{equation}
then we have \[d(Z_n(\omega)\xi(T^n\omega),Z_n(\omega)x)<q^{-\chi_1(Z_n(\omega))+\epsilon n/20}\leq q^{-(\chi_1-\epsilon/10)n}.\] 
Notice that the random variable $Z_n(\omega)$ has the law $\nu^{*n}$. Then for fixed $\xi(T^nw)$, by \cref{eq:ldp sigma g,eq:hol g-} we know \cref{equ: zn tnw} holds with $\nu^{*n}$-probability greater than $1-O(e^{-cn})$. Since $T^n\omega$ and $Z_n(\omega)$ are independent, we obtain the claim by applying a Fubini argument.
%
\end{proof}

We return to the proof of \cref{eqn: HS eqn 4.7}. Observe that 
\begin{equation}\label{equ:pi e1 xi omega}
d_{\P(\R^2)}(\pi_{{E_1}^\perp}\xi(\omega), v_n)\leq d_{\P(\R^2)}(\pi_{{E_1}^\perp}\xi(\omega), \pi_{E_1^\perp}w_n)+d_{\P(\R^2)}(\pi_{E_1^\perp}w_n, v_n).
\end{equation}
It suffices to show for each term on RHS of \eqref{equ:pi e1 xi omega}, the probability that it is less than $\frac{1}{2}q^{-(\chi_1-\epsilon)n}$ is at least $(1-O(e^{-c(\epsilon)n}))$.

For the term $d_{\P(\R^2)}(\pi_{{E_1}^\perp}\xi(\omega), \pi_{E_1^\perp}w_n)$, note that the distribution of $\xi(\omega)$ is $\mu$. It follows from \cref{lem: Hold reg proj mes} that  we have $d(\xi(\omega),E_1)>q^{-\epsilon n/10}$ with probability at least $(1-O(e^{-c(\epsilon)n}))$. Combining the claim with \cref{lem:projection}, we get that with probability greater than $(1-O(e^{-c(\epsilon)n}))$, $$d_{\P(\R^2)}(\pi_{{E_1}^\perp}\xi(\omega), \pi_{E_1^\perp}w_n)<q^{-(\chi_1-\epsilon/10)n}\cdot q^{\epsilon n/10}<\frac{1}{2}q^{-(\chi_1-\epsilon)n}, \text{ for $n$ large enough.}$$

%
For the term $d_{\P(\R^2)}(\pi_{E_1^\perp}w_n, v_n)$, notice that $v_n=h(Z_n(\omega))^+$ and $\pi_{{E_1}^\perp}w_n=\pi_{E_1^\perp}Z_n(\omega)x=h(Z_n(\omega))x $ (use the $UL$-decomposition and the assumption $x\in\P(E_1^\perp)$). Then by \cref{lem:h+hx}, with probability greater than $(1-O(e^{-c(\epsilon)n}))$,
$$d_{\P(\R^2)}(v_n , \pi_{{E_1}^\perp}w_n)<q^{-(\chi_1-\epsilon/10)n}<\frac{1}{2}q^{-(\chi_1-\epsilon)n}, \text{ for $n$ large enough,}$$
which completes the proof of Lemma \ref{lem: HS 3.3}.
\end{proof}
Now we go back to the proof of Lemma \ref{lem: BHR 6.4}. By \eqref{entropy limit},  
Lemma 4.4 of \cite{hochman_dimension_2017} and \cref{eqn: HS eqn 4.7}, we have
 \begin{equation}
 \label{equ:v_n}
 |\frac{1}{\chi_1n}H(\tau_n, \cal Q_{\chi_1n})-\alpha|\to 0,
 \end{equation}
here the measure $\tau_n$ is  the distribution of $v_n$ on $\P(E_1^\perp)$.

We define a map $\Phi$ from $L$ to $\P(E_1^\perp)\cong \P(\R^2)$ by mapping $\ell=\begin{pmatrix}\det h &0\\n & h\end{pmatrix}$ to $h^+$. (If $h\in \SO_2(\R)$, we take $h^+=E_2$.) Then $\Phi(\pi_L\nu^{*n})$ is the distribution of $v_n$ and $\tau_n=\Phi(\pi_L\nu^{*n})$.
Recall that the set $L(n,\epsilon)\subset L$ is defined in \cref{equ: L n eps}. Then we have 
\begin{lem}\label{lem:bounded atom}
For any atom $E$ in the partition $\calQ_{\chi_1n}$ of $\P(\R^2)$, $\Phi^{-1}(E)\cap L(n,\epsilon)$ intersects at most $O(q^{10\epsilon n})$ number of atoms in $\calQ_1^{L}$.
\end{lem}
\begin{proof}
From the definition of $L(n,\epsilon)$, for an atom $E$ in $\calQ_{\chi_1n}$, either $L(n,\epsilon)\cap \Phi^{-1}E=\emptyset$ or we could pick a block diagonal element $\ell=\begin{pmatrix}
        1 & 0 \\ 0 & h
    \end{pmatrix}\in L(n,\epsilon)\cap \Phi^{-1}E$. 
Our lemma is trivial if $L(n,\epsilon)\cap \Phi^{-1}E=\emptyset$. Suppose $L(n,\epsilon)\cap \Phi^{-1}E\neq \emptyset$, then for any \footnote{For $\ell'=\begin{pmatrix}
        -1 & 0 \\ n' & h'
    \end{pmatrix}$, we can find some element $\ell=\begin{pmatrix}
        -1 & 0 \\ 0 & h
    \end{pmatrix}$ and do the same argument.}$\ell'=\begin{pmatrix}
        1 & 0 \\ n' & h'
    \end{pmatrix}\in L(n,\epsilon)\cap \Phi^{-1}E$    
 we have
    \[ \ell^{-1}\ell'=\begin{pmatrix}
        1 & 0 \\ h^{-1}n' & h^{-1}h'
    \end{pmatrix}. \]
Take the Cartan decomposition of $h, h'$, $h=k_ha_hl_h$ and $h'=k_{h'}a_{h'}l_{h'}$. Since $h^+,(h')^+\in E$, we have
    \[  d(k_h^{-1}k_{h'},id)=d(h^+,(h')^+)\leq q^{-\chi_1n}. \]
By definition of $L(n,\epsilon)$, $\chi_1(h),\chi_1(h')\leq(\chi_1+\epsilon)n$, 
    \[ \|h^{-1}h'\|=\| l_h^{-1}a_h^{-1}(k_h^{-1}k_{h'})a_{h'}l_{h'}\|=\|a_h^{-1}(k_h^{-1}k_{h'})a_{h'} \|\leq q^{\frac{1}{2}(\chi_1+\epsilon)n}\cdot O(q^{-\chi_1n})\cdot q^{\frac{1}{2}(\chi_1+\epsilon)n}\leq O(q^{2\epsilon n}). \]
Recall a basic property of left-invariant and right $\SO(2)$-invariant Riemannian metric on $\SL_2(\R)$ i.e. for diagonal matrix we have $d(\diag(a,a^{-1}),id)= 2\log|a|$ \cite{bryant}. Combining it with the Cartan decomposition, we get
    \[ d_L(h^{-1}h',id)\leq 1+2\log\|h^{-1}h'\|\leq 4\epsilon n. \]
  We view $h$ as the matrix $\begin{pmatrix} 1 &0\\ 0& h\end{pmatrix}$ in $L$ and similarly for $h'$.  We then have
    \begin{align*}
    &d_L(\ell^{-1}\ell',id)\leq d_L(\ell^{-1}\ell',h^{-1}h')+d_L(h^{-1}h',id)= d_L((h')^{-1}h\ell^{-1}\ell',id)+d_L(h^{-1}h',id)\\
    &= d_L(\begin{pmatrix}
        1 & 0 \\ (h')^{-1}n' & id_2
    \end{pmatrix},id)+d_L(h^{-1}h',id)\ll  \|(h')^{-1}n'\|+d(h^{-1}h',id)\leq O(q^{2\epsilon n}),         
    \end{align*}
    where we used $\|(h')^{-1}n'\|\leq 1/d(\ell^{-1}E_1,E_1^\perp)$ (\cref{coordinate and distance psi}) and the definition of $L(n,\epsilon)$.
    Therefore by our choice of $\ell'$, $\ell^{-1}(L(n,\epsilon)\cap\Phi^{-1}E)$ is contained in a 
    neighborhood of $id$ in $L$ with radius $O(q^{2\epsilon n})$. By left invariance of $d_L$ and $\dim L=5$, we get $\ell^{-1}(L(n,\epsilon)\cap\Phi^{-1}E)$ intersects at most $O(q^{10\epsilon n})$ number of atoms in $\calQ_1^L$.
\end{proof}
We are ready to prove \cref{lem: BHR 6.4}. To simplify the notation, let $\nu_1=\pi_{L}\nu^{*n}|_{L(n,\epsilon)}, \nu_2=\pi_{L}\nu^{*n}|_{L(n,\epsilon)^c}$. We take the decomposition  $\pi_{L}\nu^{*n}=(1-\delta)\nu_1+\delta\nu_2$, with $\delta=\pi_{L}\nu^{*n}(L(n,\epsilon)^c)$.  Denote $\Phi\nu_i$ by $m_i$, $i=1,2$, then $\tau_n=\Phi(\pi_L\nu^{*n})=(1-\delta)m_1+\delta m_2$. Since $\supp\nu_1\subset L(n,\epsilon)$, by \cref{lem:bounded atom} and \cref{lem: entropy upper} we obtain
\begin{align*}
    \frac{1}{n}H(\nu_1,\calQ_1^L)&\leq \frac{1}{n}H(\nu_1,\Phi^{-1}\calQ_{\chi_1n})+O(\epsilon)\\
   &=\frac{1}{n}H(\Phi\nu_1,\calQ_{\chi_1n})+O(\epsilon)=\frac{1}{n}H(m_1,\calQ_{\chi_1n})+O(\epsilon).
\end{align*}
Combined with \cref{eqn: concav alm conv},
\begin{align*}
    H(\pi_{L}\nu^{*n},\calQ_1^L)&\leq (1-\delta)H(\nu_1,\calQ_1^L)+\delta H(\nu_2,\calQ_1^L)+H(\delta)\\
    &\leq (1-\delta)H(m_1,\calQ_{\chi_1n})+O(n\epsilon)+\delta H(\nu_2,\calQ_1^L)+H(\delta)\\
   &\leq H(\tau_n,\calQ_{\chi_1 n})+O(n\epsilon)+\delta H(\nu_2,\calQ_1^L)+H(\delta).
\end{align*}
The measure $\nu_2$ is supported on a ball of radius $q^{C_\nu n}$ due to compact support of $\nu$, so $H(\nu_2,\calQ_1^L)\leq O(C_\nu n)$. Due to $\delta=O(q^{-c(\epsilon)n})$ (\cref{equ:ln epsilon}), we obtain
\[ \frac{1}{n} H(\pi_{L}\nu^{*n},\calQ_1^L)\leq \frac{1}{n}H(\tau_n,\calQ_{\chi_1 n})+O(\epsilon+\delta)\leq\frac{1}{n}H(\tau_n,\calQ_{\chi_1 n})+O(\epsilon+q^{-c(\epsilon)n}). \]
The proof can be completed by \cref{equ:v_n} and using the fact
 that $\epsilon$ can be arbitrarily small.
\end{proof}

\subsection{Proof of Theorem \ref{thm:projection}}
We start with a contradiction hypothesis that for $\mu^{-}$-a.e. $V$ we have 
\begin{equation}
\label{equ:contradiction}
\alpha=\dim\pi_{V^\perp}\mu <\min\left\{1,\frac{h_{\mathrm{RW}}(\nu)}{\chi_1(\nu)} \right\}. 
\end{equation}
and $\pi_{V^\perp}\mu$ is of exact dimension (\cref{lem:entropy dimension}).
We fix any $V\in \P(\R^3)$ satisfying \cref{equ:contradiction}, \cref{lem:nu random walk}, \cref{lem: BHR 6.2} and \cref{lem: BHR 6.4}. Without loss of generality, using conjugation with some element $k$ in $SO(3)$, that is replacing $\nu$ by $k\nu k^{-1}$, we can assume $V=E_1$.

\paragraph{Lower bound of entropy: application of entropy growth \cref{thm.entropy} and contradiction hypothesis:} Due to \cref{lem:nu random walk} and \cref{lem: BHR 6.4}, we have that for any large $n$,
\begin{align*}
\frac{1}{n}H(\pi_{L}\nu^{*n},\calQ_{Cn}|\calQ_1)
=&\frac{1}{n}(H(\pi_{L}\nu^{*n},\calQ_{Cn})-H(\pi_{L}\nu^{*n},\calQ_1))\\
\geq &(h_{\mathrm{RW}}(\nu)-\alpha\chi_1(\nu))/2.
\end{align*}
Observe that
\[\frac{1}{n}H(\pi_{L}\nu^{*n},\calQ_{Cn}|\calQ_1)=\E_{i=1}\left(\frac{1}{n}H((\pi_{L}\nu^{*n})_{\ell,i},\calQ_{Cn})\right). \]
By Chebyshev's inequality, there exists $\epsilon>0$ such that for any $n$ large enough
\begin{equation}\label{eq:good psi}
\P_{i=1}\left(\frac{1}{n}H((\pi_{L}\nu^{*n})_{\ell,i},\calQ_{Cn})>\epsilon\right)>\epsilon.
\end{equation}

We fix such $\epsilon$ in the proof of Theorem \ref{thm:projection}. We will see later \cref{eq:good psi} is actually the source of positive entropy to apply the entropy growth argument. Denote by $\theta=\pi_{L}\nu^{*n}$ and $\cal U_0= \cal U_0(n):=\{\ell\in L: \frac{1}{n}H(\theta_{\ell, 1},\calQ_{Cn})>\epsilon\}$, then \cref{eq:good psi} is equivalent to $\pi_L\nu^{\ast n}(\cal U_0(n))>\epsilon$ for $n$ large. 

Consider the subset $\cal U_0'=\cal U_0'(n,C_1')$ of $\cal U_0$ defined by $$\cal U_0'(n,C_1'):=\{\ell\in \cal U_0(n): d(\ell^{-1}E_1,E_1^\perp)>1/C_1'\}$$
The constant $C_1'$ is chosen such that $C_1'^{-\beta}<\epsilon/8$, where $\beta$ is given in \cref{equ:LDP-g-1V}. 
Then by \cref{equ:LDP-g-1V} and $\pi_L\nu^{\ast n}(\cal U_0(n))>\epsilon$, for any $n$ large enough we have 
\begin{equation}\label{eqn: U0' large}
\pi_L\nu^{\ast n}(\cal U_0'(n,C_1'))>7\epsilon/8 
\end{equation}

Now we try to apply \cref{thm.entropy} using $\cal U_0'$. For any $\ell\in \cal U_0'(n,C_1')$, 
for the component measure $\theta_{\ell,1}$  
we have
\begin{itemize}
\item $\diam(\theta_{\ell,1})\leq \diam( \calQ_1(\ell))\leq 1< C_p$;
 \item for any $\ell'\in \supp\theta_{\ell,1}\subset\calQ_1(\ell)$, $|\chi_1(h(\ell'))-\chi_1(h(\ell))|\leq 1 $ by our choice of $q$ (and the choice of constants $C_p, C_L$), see \cref{defi: q} and \cref{lem:distance psi psi'}.
\item for any $\ell'\in \supp\theta_{\ell,1}\subset\calQ_1(\ell)$, 
    \begin{equation}\label{eqn: half dist to E1 perp}
       d(\ell'^{-1}E_1,E_1^\perp)\geq 1/2d(\ell^{-1}E_1,E_1^\perp)\geq 1/2C_1'
    \end{equation}
\item  $\frac{1}{n}H(\theta_{\ell,1},\calQ_{Cn})>\epsilon $ since $\ell\in \cal U_0$.
\end{itemize}
Therefore for $\ell\in \cal U_0'$, we can apply \cref{thm.entropy} to $\theta_{\ell,1}$ and $\mu$ with $C_1=2C_1'$ and contracting rate $t=\chi_1(h(\ell))$: for $n$ large enough, 
\[\frac{1}{Cn}H([\theta_{\ell,1}.\mu],\calQ_{Cn+\chi_1(h(\ell))})>\alpha+\delta,\]
where the constant $\delta=\delta(\alpha,\epsilon,C_1)$ is given as in  \cref{thm.entropy}.

By \cref{equ:LDP-h-V-g} and \eqref{eqn: U0' large}, for any $\epsilon_1>0$, for $n$ large enough, the subset $\cal U_0''(n,C_1',\epsilon)\subset \cal U_0'$ formed by those those $\ell\in \cal U_0'$ such that $|\chi_1(h(\ell))-\chi_1n|\leq \epsilon_1n$ satisfies 
$  \pi_L\nu^{\ast n}(\cal U_0''(n,C_1',\epsilon_1))>3\epsilon/4 $. Fix a positive number $\epsilon_1$ much smaller than $\delta$, then by \cref{eq:entropy m n}, for any $\ell \in \cal U_0''$ we have 
\begin{equation}\label{equ:nu n growth}
\begin{split}
 \frac{1}{Cn}H([\theta_{\ell,1}.\mu],\calQ_{Cn+\chi_1n})
 \geq \frac{1}{Cn}H([\theta_{\ell,1}.\mu],\calQ_{Cn+\chi_1(h(\ell))})-O(\frac{\epsilon_1}{C})>\alpha+\delta-O(\frac{\epsilon_1}{C})>\alpha+\delta/2.
\end{split}
\end{equation}
We set $\cal U_1(n)=\cal U_1:=\{\ell\in L:\,\,\ell\,\, \text{satisfies}\,\, \cref{equ:nu n growth}\}$. Then by the discussion above about $\cal U_0''$, for $n$ large enough, the $\theta$-measure of $\cal U_1$ greater than $\epsilon/2$, and $\cal U_1$ consists of atoms of $\calQ_1$.


Due to \cref{lem: BHR 6.2}, for any $\tau>0$, there exists an infinite sequence $\{n_k\}$ and a sequence of subset $\cal U_2=\cal U_2(n_k)\subset L$ such that for each $k$, $\cal U_2(n_k)$ has  $\nu^{n_k}$-probability greater than $1-\tau$, and for any $\ell\in \cal U_2$,
\begin{equation}\label{equ:psi mu geq}
 \frac{1}{Cn}H([\ell\mu],\calQ_{(C+\chi_1)n})>\alpha-\tau. 
\end{equation}
Take $n$ large enough to be some $n_k$ defined above. Consider \begin{align*}
&\E_{i=1}\left(\frac{1}{Cn}H\left([\theta_{\ell,i}.\mu],\calQ_{(C+\chi_1)n}\right)\right)
=\int \frac{1}{Cn}H\left([\theta_{\ell,1}.\mu],\calQ_{(C+\chi_1)n}\right)\dd\theta(\ell)\\
=& \int_{\cal U_1}\frac{1}{Cn}H\left([\theta_{\ell,1}.\mu],\calQ_{(C+\chi_1)n}\right)\dd\theta(\ell)+\int _{\cal U_1^c}\frac{1}{Cn}H\left([\theta_{\ell,1}.\mu],\calQ_{(C+\chi_1)n}\right)\dd\theta(\ell).
\end{align*}
For the first term on RHS, we apply \cref{equ:nu n growth} to estimate it. For the second term, since $\cal U_1$ consists of atoms of $\cal Q_1$, by the concavity of entropy we obtain
\[\int _{\cal U_1^c}\frac{1}{Cn}H\left([\theta_{\ell,1}.\mu],\calQ_{(C+\chi_1)n}\right)\dd\theta(\ell)\geq \int _{\cal U_1^c}\frac{1}{Cn}H\left([\ell\mu],\calQ_{(C+\chi_1)n}\right)\dd\theta(\ell)\geq (\alpha-\tau)\theta(\cal U_2-\cal U_1), \]
where the last inequality is due to \cref{equ:psi mu geq}. Therefore, we have
\begin{align*}
\E_{i=1}\left(\frac{1}{Cn}H\left([\theta_{\ell,i}.\mu],\calQ_{(C+\chi_1)n}\right)\right)\geq &  \frac{\epsilon}{2}(\alpha+\frac{\delta}{2})+(1-\tau-\epsilon/2)(\alpha-\tau)
>\frac{\epsilon\delta}{4}+(\alpha-\tau)(1-\tau). 
\end{align*}
By taking $\tau$ small, we obtain
\begin{equation}\label{eq.growth}
\limsup_{n\rightarrow\infty}\E_{i=1}\left(\frac{1}{Cn}H\left([\theta_{\ell,i}.\mu],\calQ_{(C+\chi_1)n}\right)\right)>\alpha+\epsilon\delta/8. 
\end{equation}
\\
\textbf{Upper bound of entropy: application of exact dimension.}
Another consequence from the exact dimension of $\pi_{E_1^\perp}\mu$ is the following 
\[\lim_{n\rightarrow\infty}\frac{1}{Cn}H(\pi_{E_1^\perp}\mu,\calQ_{(C+\chi_1)n}|\calQ_{\chi_1n})=\alpha. \]
Note that $\pi_{E_1^\perp}\mu=\pi_{E_1^\perp}(\nu^{*n}*\mu)=[(\pi_{L}\nu^{*n}).\mu]=[\theta.\mu]$, so $\pi_{E_1^\perp}\mu=\E_{i=1}([\theta_{\ell,i}.\mu])$. By the concavity of entropy we have 
\begin{equation}\label{equ:main upper}
\limsup_{n\rightarrow\infty}\E_{i=1}\left(\frac{1}{Cn}H([\theta_{\ell,i}.\mu],\calQ_{(C+\chi_1)n}|\calQ_{\chi_1n})\right)\leq \alpha.
\end{equation}
We would like to eliminate the conditional entropy in \cref{equ:main upper}. This will enable us to compare \cref{equ:main upper} with \cref{eq.growth} and obtain a contradiction. 
The idea is that for most of $\ell$ in $\supp\theta$, we can decompose $[\theta_{\ell,1}.\mu]$ similar to the $r$-attracting decomposition (\cref{defi: gd bd decomp}) as the following claim.

\begin{claim}
\label{another good bad decomposition}
There exists $\epsilon_0>0$ such that for any positive $\epsilon_1<\epsilon_0$, for $\ell\in L(n,\epsilon_1)$ defined in \cref{equ: L n eps}, the component $[\theta_{\ell,1}.\mu]$ can be decomposed as the sum $[\theta_{\ell,1}.\mu]=(1-\epsilon_2)\tau_1+\epsilon_2 \tau_2$ such that 
\begin{equation*}
 \epsilon_2\ll q^{-\beta\epsilon_1 n}\,\,\,\text{and}\,\,\, \diam(\supp \tau_1)\leq  q^{-(\chi_1-4\epsilon_1)n} \text{ for $n$ large enough}. 
\end{equation*}
\end{claim}
\begin{proof}[Proof of \cref{another good bad decomposition}]
Consider $\tau=\tau(\ell)=\ell^{-1}\theta_{\ell,1}$ on $L$ and ${\tau}*\mu$ on $\P(\R^3)$. Then
\[ [\theta_{\ell,1}.\mu]=[\ell( {\tau}*\mu)]. \]
And $\tau(\ell)$ is supported on a uniformly bounded neighbourhood of the identity in $L$. 


Let $m:={\tau}*\mu$,
$m_1:=m_{b(f_{\ell},q^{-\epsilon_1n})}$ and $m_2:=m_{b(f_{\ell},q^{-\epsilon_1n})^c}$, where the good region $b(f_{\ell},q^{-\epsilon_1n})$ is introduced in \cref{defi:psi good region}, whose complement is the $q^{-\epsilon_1n}$ neighborhood of a hyperplane.
Since $\tau$ is uniformly bounded supported, there exists $C'>0$ such that for any $g\in \supp \tau$, we can apply the bound in \cref{lem: Hold reg proj mes} to obtain
\[ g\mu(b(f_{\ell},q^{-\epsilon_1n})^c)\leq \mu(b(W,C'q^{-\epsilon_1n} )^c)\ll (q^{-\epsilon_1n})^\beta, \]
for some hyperplane $W$. Hence for $\epsilon_2$ from the decomposition of $m=(1-\epsilon_2)m_1+\epsilon_2 m_2$
\begin{equation}\label{equ: epsilon 2}
\epsilon_2=\tau*\mu(b(f_{\ell},q^{-\epsilon_1n})^c)=\int g\mu(b(f_{\ell},q^{-\epsilon_1n})^c)\dd\tau(g)\ll q^{-\beta \epsilon_1n}.
\end{equation}
Apply \cref{lem:bad psi region} we obtain
\[ \diam([\ell m_1])\leq  q^{-(\chi_1-4\epsilon_1)n}, \text{ for $n$ large enough.}\]
By letting $\tau_1=[\ell m_1]$ and $\tau_2=[\ell m_2]$, we finish the proof of the claim.
\end{proof}
 
Due to \cref{another good bad decomposition}, we can apply \cref{lem:concave conditional} to $\ell\in L(n,\epsilon_1)$ and obtain
\begin{align*}
& \E_{i=1}\left(\frac{1}{Cn}H\left([\theta_{\ell,i}.\mu],\calQ_{(C+\chi_1)n}|\calQ_{\chi_1n}\right)\right)\\
\geq & \int_{L(n,\epsilon_1)}\frac{1}{Cn}H([\theta_{\ell,1}.\mu],\calQ_{(C+\chi_1)n}|\calQ_{\chi_1n})\dd\theta(\ell) \\
\geq  &\int_{L(n,\epsilon_1)}\frac{1}{Cn}H([\theta_{\ell,1}.\mu],\calQ_{(C+\chi_1)n})-\frac{1}{Cn}(O(\epsilon_2n)+H(\epsilon_2)+O(\epsilon_1n))\dd\theta (\ell)\\
\geq &\E_{i=1}\left(\frac{1}{Cn}H\left([\theta_{\ell,i}.\mu],\calQ_{(C+\chi_1)n}\right)\right)-\frac{1}{Cn}(O(\epsilon_2n)+H(\epsilon_2)+O(\epsilon_1n)+O(q^{-c(\epsilon_1)n})),
\end{align*}
where the last inequality is due to \cref{equ:ln epsilon}. By taking $\epsilon_1$ sufficiently small (and hence $\epsilon_2$ small by \cref{equ: epsilon 2}), as $n$ going to infinity, combining with \cref{equ:main upper}, we obtain a contradiction to \eqref{eq.growth}. 

Hence we conclude $\dim\pi_{E_1^\perp}\mu=\min\{1,h_{\mathrm{RW}}(\nu)/\chi_1(\nu) \}$.

 
\subsection{A Ledrappier-Young formula and the proof of \cref{thm:lyapunov} from \cref{thm:projection}.}\label{sec:ledrappier-young}
We first recall the Ledrappier-Young type formulas of Furstenberg measures recently shown in \cite[Corollary 6.5]{ledrappier_exact_2021} and  \cite{rapaport_exact_2021}. (See \cite{ledrappier_exact_2023} for a more general case.)
\begin{prop}\label{prop.ly.formula} 
    Let $\nu$ be a finitely supported probability measure such that $\supp\nu$ generates a Zariski dense subgroup in $\SL_3(\R)$. Then there exist $\gamma_1,\gamma_2\geq 0$ such that for $\mu^-$ a.e. $V$ we have
    \begin{enumerate}
        \item the projection measure $\pi_{V^\perp}\mu$ is of exact dimension $\gamma_1$;
        \item consider the disintegration of the measure $\mu$ under the projection $\pi_{V^\perp}$,
        \[\mu=\int_{\P(V^\perp)} \mu^x_{V}\ \dd( \pi_{V^\perp}\mu)(x). \]
        For $\pi_{V^\perp}\mu$ a.e. $x\in \P(V^\perp)$,
        the measure $\mu^x_{V}$  is of exact dimension $\gamma_2$;
        \item we have \begin{equation}\label{equ:hf fiber}
     h_{\mathrm{F}}(\mu,\nu)=\gamma_1\chi_1+\gamma_2\chi_2,\ \dim\mu=\gamma_1+\gamma_2, 
    \end{equation} where $h_{\mathrm{F}}(\mu,\nu)$ is the Furstenberg entropy defined by \eqref{equ:f entropy}.
     \end{enumerate}
    \end{prop}
\begin{proof}\cref{prop.ly.formula} is essentially in Corollary 6.5 of \cite{ledrappier_exact_2021} which is presented in a different way. To explain it, we recall some notions from \cite{ledrappier_exact_2021}. For the set $\{1,2,3\}$, we consider the topopogies $T^2=(\{1\},\{2,3\},\{3\})\prec T^1=(\{1,3\},\{2,3 \},\{3\})\prec T^0=(\{1,2,3\},\{2,3\},\{3\})$, see Section 1.4 in \cite{ledrappier_exact_2021}.  Let $f=(\{0\}=U_0\subset U_1\subset U_2\subset U_3=\R^3)$ and $f'=(\{0\}=U_0'\subset U_1'\subset U_2'\subset U_3'=\R^3)$ be two flags  in general position, that is to say, $U_1\oplus U_2'=U_1'\oplus U_2=\R^3$. For each topology $T=T^i$, we define the map $[F_T(f,f')]_I$ with $I\in T$ by
\[[F_T(f,f')]_I=\oplus_{i\in I}(U_i\cap U_{4-i}'). \]
For example, let $T$ be $T^0=(\{1,2,3\},\{2,3\},\{3\})$ and consider the map $[F_{T^0}(f,f')]_I$ with $I\in T^0$
\begin{itemize}
    \item for  $I=\{1,2,3\}$, due to flags in general position, we obtain $\R^3$;
    \item for $I=\{2,3\}$, we obtain $(U_2\cap U_2')\oplus U_1'=U_2'$;
    \item for $I=\{3\}$, we obtain $U_1'$.
\end{itemize}
In summary, we get $[F_{T^0}(f,f')]=(\R^3,U_2',U_1')$. Similarly we obtain $[F_{T^2}(f,f')]=(U_1,U_2',U_1')$ and $[F_{T^1}(f,f')]=(U_1\oplus U_1',U_2',U_1')$. Therefore we have
projection mappings (in the sense of Page 8 \cite{ledrappier_exact_2021})
\[ (U_1,U_2',U_1')\xrightarrow{\pi_{T^2,T^1}} (U_1\oplus U_1',U_2',U_1')\xrightarrow{\pi_{T^1,T^0}}(\R^3,U_2',U_1').  \]
Let $\mu_\calF,\mu^-_\calF$ be the stationary measures of $\nu,\nu^-$ on $\calF$, respectively.
For a topology $T$, consider the pushforward measure $(F_T)_*(\mu_\calF\otimes\mu_\calF^-)$ and the disintegration
\[ (F_T)_*(\mu_\calF\otimes\mu_\calF^-)=\int (\mu_\calF)_T^{f'}\dd\mu_\calF^-(f'), \]
where the fiber measure $(\mu_\calF)_T^{f'}$ is supported on $X_T^{f'}$, the set of $F_T(f,f')$ for all $f$ such that $f,f'$ are in general position.
Since $T^1$ is finer than $T^0$, for $\mu_\calF^-$ a.e. $f'$, we have
\[ (\pi_{T^1,T^0})_*(\mu_\calF)_{T^1}^{f'}=(\mu_\calF)_{T^0}^{f'}. \]
Here $(\mu_\calF)_{T^0}^{f'}$ is supported on $X_{T^0}^{f'}$ which only consists of one point, and $(\mu_\calF)_{T^1}^{f'}$ is supported on $X_{T^1}^{f'}\simeq  \{U_1\oplus U_1',\ (f,f') \text{ are in general position} \}$.
Hence, for the disintegration
\[ (\mu_\calF)_{T^1}^{f'}=\int \mu_{T^1,T^0}^{x'} \dd(\mu_\calF)_{T^0}^{f'}(x'), \]
the fiber measure $\mu_{T^1,T^0}^{x'}=(\mu_\calF)_{T^1}^{f'}$. 
We have the isomorphism $\{U_1\oplus U_1', U_1\oplus U_2'=\R^3 \}\simeq \P((U_1')^\perp)-\{\text{point}\}$.
Hence the fiber measure $\mu_{T^1,T^0}^{x'}$ can be identified with $\pi_{V^\perp}\mu$ for $V=U_1'$ following the law $\mu^-$. Similarly, the fiber measure $\mu_{T^2,T^1}^{y'}=\mu_V^{x}$ for $y'=(U_1\oplus U_1',U_2',U_1')\in X_{T_1}^{f'}$, with $V=U_1'$ and $x=\pi_{V^\perp}(U_1\oplus U_1')$.


With all these identifications above, \cite[Corollary 6.5]{ledrappier_exact_2021} can be viewed as \cref{equ:hf fiber}.
\end{proof}


\begin{proof}[Proof of \cref{thm:lyapunov} from \cref{thm:projection}] By the exponential separation assumption in \cref{thm:lyapunov}, \cref{thm:projection} implies that $\gamma_1=\min\{1,\frac{h_{\mr{RW}}(\nu)}{\chi_1} \}$.  

If $\gamma_1=1$, \cref{thm:lyapunov} directly follows the definition of $\dim_{\rm{LY}}(\nu)$ and \cref{equ:hf fiber}. For the case that $\gamma_1=\frac{h_{\mr{RW}}(\nu)}{\chi_1}<1$, we use the general fact that $h_{\mathrm{F}}(\mu,\nu)\leq h_{\mr{RW}}(\nu)$, cf.  \cite{kaimanovich_random_1983} and also \cite[Theorem 2.31]{furman_random_2002}. These results are stated for the discrete groups. If $G_\nu=\langle\supp\nu \rangle$ is not discrete in $\SL_3(\R)$, we can give a discrete topology on $G_\nu$ as an abstract group, denoted by $G_\nu'$ and corresponding measure by $\nu'$. Then the space $(\P(\R^3),\mu)$ is still a $(G_\nu',\nu')$-space. So $h_{\mathrm{F}}(\mu,\nu)\leq h_{\mr{RW}}(\nu')=h_{\mr{RW}}(\nu)$. 
Then we obtain $\dim_{\rm{LY}}\mu\leq\dim\mu$. Since $\dim_{\rm{LY}}\mu$ is always an upper bound of $\dim\mu$, we finish the proof.
\end{proof}

\section{Proof of \cref{thm:dimension_jump} and \cref{thm:hausdorff} }\label{sec:proofs}

\paragraph{Approximation of affinity exponents}
To complete the proof of Theorem \ref{thm:hausdorff}, it remains to obtain $\dim L(\rho(\Gamma))$ from the dimension formula of stationary measures. We need a variational principle type result, i.e. show that the affinity exponent $s_{\mathrm{A}}(\rho)$ of $\rho(\Gamma)$ can be approximated by the Lyapunov dimension of stationary measures.
\begin{thm}
\label{thm:variational principle}
    Let $\Gamma$ be a hyperbolic group and $\rho\in\HA(\Gamma,\SL_3(\RR))$ whose image is Zariski dense. For every $\ve>0,$ there exists a Zariski dense probability measure $\nu$ supported on $\rho(\Gamma)$ such that its Furstenberg measure $\mu$ satisfies $\dim_{\mr{LY}}\mu\geqslant s_{\mathrm{A}}(\rho)-\ve.$
\end{thm}

The proof is given in the second paper \cite{JLPX}.

\subsection{Proof of Theorem \ref{thm:hausdorff}}\label{sec:thm hausdorff}
A direct application of the dimension variation is to compute the dimension of the minimal sets. Once we establish the identity between the Lyapunov dimension and the exact dimension of stationary measures, we can obtain the identity between the affinity exponent and the Hausdorff dimension of minimal sets.
\begin{proof}[Proof of \cref{thm:hausdorff}]
    Recall that $\bG$ is the Zariski closure of $\rho(\Gamma).$ We first consider the case $\bG=\SL_3(\RR).$ Note that $\rho(\Gamma)$ is a discrete subgroup of $\SL_3(\RR).$ Then for every $\nu\in \sP_{\mr{f.s.}}^\bG(\rho(\Gamma)),$ $\nu$ is a Zariski dense finitely supported probability measure with exponential separation. Hence the unique $\nu$-stationary measure $\mu$ on $\PP(\RR^3)$ satisfying $\dim_{\mr{LY}}\mu=\dim \mu$ by \cref{thm:lyapunov}. Moreover, $\mu$ is supported on the limit set $L(\rho(\Gamma))\sbs\PP(\RR^3).$ By \cite{young_dimension_1982}, $\dim\mu=\dim_{\mr{H}}\mu\leqslant\dim_{\mr{H}}L(\rho(\Gamma)).$ Combining with the upper bound estimate of Hausdorff dimension in \cite{pozzetti_conformality_2019}, we obtain
    \begin{align*}
        \dim_{\mr H}L(\rho(\Gamma))&\leqslant s_{\mathrm{A}}(\rho)\leqslant \sup\lb{\dim_{\mr{LY}}\mu:\nu\in\sP_{\mr{f.s.}}^\bG(\rho)}\\
        &=\sup\lb{\dim\mu:\nu\in\sP_{\mr{f.s.}}^\bG(\rho)}\leqslant \dim_{\mr H}L(\rho(\Gamma)).
    \end{align*}
    Then all unequal signs are equal. In particular, $\dim_{\mr H}L(\rho(\Gamma))=s_{\mathrm{A}}(\rho).$

Secondly, we consider the case $\bG\ne \SL_3(\R)$.
Since $\rho$ is an irreducible Anosov representation, due to \cite[Corollary 2.20]{bridgeman_pressure_2015}\footnote{The statement of Corollary is for projective Anosov representation. In $\SL_3(\R)$, projective Anosov is equivalent to Anosov, so we can freely use this Corollary.}, the Zariski closure of $\rho(\Gamma)$ is semi-simple without compact factor, and its centralizer is contained in $\{\pm\Id \}$. So $\rho(\Gamma)$ is the image of an irreducible representation of $\SL_2(\R)$. We can write $\rho$ as the composition of $\rho_0:\Gamma\rightarrow \SL_2(\R)$ and an irreducible representation $\iota_1:\SL_2(\R)\rightarrow \SL_3(\R)$. The representation $\iota_1$ induces an algebraic map from $\P(\R^2)$ to $\P(\R^3)$. We only need to compute the Hausdorff dimension of $L(\rho_0(\Gamma))$. Due to the classical result of Patterson, Bowen and Sullivan, the Hausdorff dimension of the limit set is also equal to the critical exponent, which is equal to the affinity exponent in this case. 

The proof of the continuity of $s_{\mathrm{A}}(\rho)$ is postponed to the next subsection.
\end{proof}

\subsection{Continuity and discontinuity of dimensions: proof of Theorem \ref{thm:dimension_jump} and the continuity part of Theorem \ref{thm:hausdorff}.}\label{sec:jump}

Recall that $\HA(\Gamma,\SL_3(\R))$ is the space of Anosov represesntations of $\Gamma$ inside $\Hom(\Gamma,\SL_3(\R))$, which is an open subset.
In the following of this section, we always assume that $\rho$ is a representation in $\HA(\Gamma,\SL_3(\RR)).$
Recall that the Lie algebra $\frak a=\{\lambda=\diag(\lambda_1,\lambda_2,\lambda_3):\lambda_i\in\R,\  \sum\lambda_i=0  \}$ and the positive Weyl chamber $\frak a^+=\{\lambda\in\frak a: \lambda_1\geq \lambda_2\geq\lambda_3 \}$. For an element $g\in \SL_3(\R)$ let $\lambda(g)\in\frak a^+$ be its Jordan projection. Let $\alpha_1$ and $\alpha_2$ be two simple roots on $\frak a$, given by $\alpha_i(\lambda)=\lambda_i-\lambda_{i+1}$, which are nonnegative on $\frak a^+$. For any nonzero $\psi\in \{a\alpha_1+b\alpha_2:\ a,b\in\R_{\geq 0}  \}\sbs\frak a^*$, we define 
\[ h_\rho^\psi=\limsup_{T\rightarrow\infty}\frac{1}{T}\log \#\{[\gamma]\in[\Gamma]: \ \text{non torsion},\ \psi(\lambda(\rho(\gamma)))\leq T  \}, \]
where $[\Gamma]$ is the set of conjugacy classes and $[\gamma]$ is the conjugacy class of $\gamma$ in $\Gamma$. For such $\psi$, we have $h_\rho^\psi\in(0,\infty)$ since $\rho\in \HA(\Gamma,\SL_3(\R))$. By \cite[Cor. 4.9]{potrie_eigenvalues_2017} or \cite{bridgeman_pressure_2015}, we know that $h_\rho^{\psi}$ is analytic with respect to $\rho$ in $\HA(\Gamma,\SL_3(\R))$.

In order to study the affinity exponent, we consider the following function $\psi_s$ on $\frak a$ as
\begin{equation}\label{eqn: psi s}
    \psi_s=\case{
    &s\alpha_1,&0<s\leq 1; \\
    &\alpha_1+(s-1)(\alpha_1+\alpha_2),&1<s\leq 2.}
\end{equation}
Since $\rho$ is an Anosov representation, we obtain that $h_\rho^{\psi_s}\in(0,\infty)$ for every $0<s\leq 2$. 

\begin{lem}
    \begin{enumerate}[(1)]
        \item For any $0<t<1$, we have $h_\rho^{\psi_{ts}}>th_\rho^{\psi_{ts}}\geq h_\rho^{\psi_s} ;$
        \item For $0<t<1, 0<s\leq 1$, we have $th_\rho^{\psi_{ts}}= h_\rho^{\psi_s}$;
        \item For $0<t<1, s>1$ and $ts>1$, we have
        \[ h_\rho^{\psi_s}\geq  \frac{ts-1}{s-1}h_\rho^{\psi_{ts}}. \]
    \end{enumerate}
\end{lem}
\begin{proof}
    The first two items can be checked by definition. The third one follows from the inequality $\psi_s\leq \frac{s-1}{ts-1}\psi_{ts} $ when $0<t<1, s>1$ and $ts>1$.
\end{proof}
As a corollary, we have

\begin{lem}\label{lem: psi s}
$h_\rho^{\psi_s}$ is a strictly decreasing continuous function on $s$.
\end{lem}
\begin{lem}\label{lem: aff exp}
The affinity exponent $s_{\mathrm{A}}(\rho)$ is the unique $s$ such that $h_\rho^{\psi_s}=1$.
\end{lem}
\begin{proof}

In \cite[Corollary 4.4]{sambarino_hyperconvex_2014}, it is proved that for Anosov representation $\rho$, the exponent $h_\rho^{\psi_s}$ is also equal to
\[ \limsup_{T\rightarrow\infty}
\frac{1}{T}\log \#\{\gamma\in\Gamma:\ \psi_s(\kappa(\rho\gamma))\leq T  \}, \]
which equals the critical exponent of the series
\[ \sum_{\gamma\in\Gamma}q^{-t\psi_s\kappa(\rho\gamma)} \]
with respect to $t\in \R$. Here $q$ is a large integer, which is the base of the logarithm in this article.
Hence, if $h_\rho^{\psi_s}>1$, then the series $\sum_{\gamma\in\Gamma}q^{-\psi_s\kappa(\rho\gamma)}$ diverges; if $h_\rho^{\psi_s}<1$, then the series $\sum_{\gamma\in\Gamma}q^{-\psi_s\kappa(\rho\gamma)}$ converges. By the definition of affinity exponent, this completes the proof.
\end{proof}

We are able to prove the continuity part in \cref{thm:hausdorff}.
\begin{prop}\label{lem: continuity aff exp}
The affinity exponent $s_{\mathrm{A}}(\rho)$ is continuous on $\HA(\Gamma,\SL_3(\R))$.
\end{prop}
\begin{proof}
 Consider the $\R^+$-value map $h_\rho^{\psi_s}$ that maps $(\rho,s)$ to $h_\rho^{\psi_s}$. For a fixed $s$, the aforementioned result (\cite[Cor. 4.9]{potrie_eigenvalues_2017} and \cite{bridgeman_pressure_2015}) implies that $h_\rho^{\psi_s}$ is an analytic function on $\rho\in\HA(\Gamma,\SL_3(\RR))$. For a fixed $\rho$, $h_\rho^{\psi_s}$ is a  strictly decreasing continuous function (\cref{lem: psi s}). 
A basic analysis lemma says that a two-variable function, which is continuous on $\rho$ and strictly monotonic continuous on $s$, satisfies the implicit function theorem. We obtain that $s_{\mathrm{A}}(\rho)$, the solution of $h_\rho^{\psi_s}=1$ (\cref{lem: aff exp}), is a continuous function on $\HA(\Gamma,\SL_3(\R))$.
\end{proof}


We prove \cref{thm:dimension_jump} in a more general setting, that is starting from any convex cocompact representation $\rho_0:\Gamma\rightarrow\SL_2(\R)$. 
Recall that $\iota$ and $\rho_1=\iota\circ\rho_0$ are defined in the introduction.

\begin{thm}\label{thm:general_jump}
    For every $\epsilon>0$, there exists a small neighborhood $O$ of $\rho_1$ in $\mathrm{Hom}(\Gamma,\SL_3(\R))$ such that for any $\rho$ in $O$ we have
    \begin{itemize}
    \item either $\rho(\Gamma)$ acts reducibly on $\R^3$, i.e. fixing a line or a plane in $\R^3$; 
    \item or $\rho(\Gamma)$ is irreducible and 
    \[ |\dim L(\rho(\Gamma))- (\dim L(\rho_1(\Gamma))+\min\{\dim L(\rho_1(\Gamma)), 1/2\})|\leq \epsilon. \]
    \end{itemize}
\end{thm}
If we start with a surface group $\Gamma$, then the limit set $L(\rho_1(\Gamma))$ is the full circle of dimension one. Then we obtain $3/2$ in \cref{thm:dimension_jump}. 
\begin{proof}

 By the definition of the embedding $\iota$, for any $\gamma\in \Gamma$ we have $\sigma_1(\rho_1(\gamma))=1/\sigma_3(\rho_1(\gamma))=\mu_1(\rho_0(\gamma))=1/\mu_2(\rho_0(\gamma))$ and $\sigma_2(\rho_1(\gamma))=1$ (In order to distinguish the notations, we use $\mu_1(\rho_0(\gamma))\geq \mu_2(\rho_0(\gamma) )$ to denote the singular values of $\rho_0(\gamma)\in\SL(2,\RR)$). Then
\begin{equation}\label{eqn: dim jump1}
\sum_{\gamma\in\Gamma}\left(\frac{\mu_2(\rho_0\gamma)}{\mu_1(\rho_0\gamma)}\right)^s=\sum_{\gamma\in\Gamma}\left(\frac{\sigma_2(\rho_1\gamma)}{\sigma_1(\rho_1\gamma)}\right)^{2s}=\sum_{\gamma\in\Gamma}\frac{\sigma_2(\rho_1\gamma)}{\sigma_1(\rho_1\gamma)}\left(\frac{\sigma_3(\rho_1\gamma)}{\sigma_1(\rho_1\gamma)}\right)^{s-1/2} .
\end{equation}
The map $\iota$ induces an isometric embedding of the limit set $L(\rho_0(\Gamma))$ to $L(\rho_1(\Gamma))$, in particular, they have the same Hausdorff dimension. A classic result in the theory of Fuchsian groups tells us that $\dim L(\rho_0(\Gamma))$ is the critical exponent of the first series in \eqref{eqn: dim jump1}. Hence 
\begin{equation*}
\dim L(\rho_1(\Gamma))=\max\{s_{\mathrm{A}}(\rho_1)/2, s_{\mathrm{A}}(\rho_1)-1/2 \}.
\end{equation*}
Therefore
\begin{equation}\label{equ:L rho Gamma'1}
s_{\mathrm{A}}(\rho_1)=\min\{2\dim L(\rho_1(\Gamma)), \dim L(\rho_1(\Gamma))+1/2\}.
\end{equation}
Due to \cref{lem: continuity aff exp}, for any irreducible $\rho$ in the neighbourhood $O$ of $\rho_1$, by \cref{thm:hausdorff}. 
and \cref{equ:L rho Gamma'1} we have
\[\dim L(\rho(\Gamma))=s_{\mathrm{A}}(\rho)\geq \min\{2\dim L(\rho_1(\Gamma)), \dim L(\rho_1(\Gamma))+1/2\}-\epsilon. \qedhere\] 
\end{proof}

\appendix
\section{Auxiliary results}

\paragraph{Identification between \texorpdfstring{$\P(V^\perp)$}{P(V⊥)} with \texorpdfstring{$\P(\R^2)$}{P(R\^{}2)}}We will identify $\P(V^\perp)$ with $\P(\R^2)$ through rotation invariant metrics for general $V\in\P(\R^3)$. 
A prior, there is no canonical choice of such identification, but different identifications only differ by rotations and orientation. In order to apply certain continuity arguments, we fix a one-dimensional subspace $V_0=E_3$ and let $\mathcal C=\P(\R^3)-\P(E_1^\perp)\subset \P(\R^3)-\{V_0 \} $. Since most of our results are on measure-theoretical properties, we can freely delete this single \textit{bad circle} $\P(E_1^\perp)\subset \P(\R^3)$. 

\begin{lem}\label{lem: good continious V}
For each point $V\in \mathcal C$, we have an identification between $\P(V^\perp)$ and $\P(\R^2)$ and the identification is continuous with respect to $V$. 
\end{lem}
\begin{proof}
    For any $V\neq V_0$, the projection $\pi_{V^\perp}$ is well-defined at $V_0$. We take the isometric identification between $\P(V^\perp)$ and $\P(\R^2)$ such that $\pi_{V^\perp}V_0$ is identified with $\R(1,0)$. 
    
    For any $V$ not in $\P(E_1^\perp)$, we take a representative $v\in \S^2\subset \R^3$ of $V$ with positive first coordinate. Take $k\in \mathrm{SO}(3)$ such that $k v=(1,0,0).$
    The orientation of $\P(V^\perp)$ comes from the identification between $V^\perp$ and $E_2\oplus E_3$ by $k$, which is independent of the choice of $k$.
\end{proof}

Under the assumption of this lemma, by identifying $\P(V^\perp)$ with $\P(\R^2)$, we can regard the linear map $h_{V,g}$ in \eqref{defn:projections} as an element in $\PGL_2(\R)$. 

\paragraph{Weakly convergence of projections of a measure}
\begin{lem}
\label{lem: pseudo cont entpy prf2}
For any $V\in \mathcal C$ and any interval $I\subset \P(\R^2)$, we have
\begin{enumerate}
    \item $\pi_{W^\perp}(\mu)(I)$ converges to $\pi_{V^\perp}(\mu)(I)$ if $W$ converges to $V$.
    \item $(\pi_{W^{\perp}}\mu)_I\to (\pi_{V^\perp} \mu)_I$ weakly if  $W$ converges to $V$. 
\end{enumerate}
\end{lem}

\begin{proof}
Let $I\subset \P(\R^2)$ be any interval. For any $W\in \P(\R^3)$, denote the region  $(\pi_{W^\perp})^{-1}(\cal I_{W}^{-1}(I))$ by $S(I,W)$. So we have $$\pi_{W^\perp}(\mu)(I)=\mu(S(I,W)).$$ 

For the first statement, 
it suffices to show $\mu(S(I,W))\to \mu (S(I,V))$ when $W\to V$. Due to the definition of $\cal I_V$, for any $\rho>0$, there exists $\delta>0$ such that if $d(W,V)<\delta$, then for any $x\in\P(\R^2)$, 
\[d(\cal I_V^{-1}(x),\cal I_W^{-1}(x))<\rho.\]
Therefore, the symmetric difference $S(I,W)\Delta S(I,V)$ of $S(I,W)$ and $S(I,V)$ can be covered by the $\rho$ neighborhood of $\partial S(I,V)$, that is $\pi_{V^\perp}^{-1}\cal I_V^{-1}(\partial I) $. We this denote by $(\partial S(I,V))^{(\rho)}$. Notice that $\partial S(I,V)$ is the union of two great circles. Therefore, by Lemma \ref{lem: Hold reg proj mes}, we know that $\mu(S(I,W)\Delta S(I,V))$ can be arbitrarily small when $W\to V$.

We prove the second statement. For any sub-interval $J\subset I$, $\pi_{W^\perp}(\mu)(J)$ converges to $\pi_{V^\perp}(\mu)(J)$. Therefore for any open set $U\subset I$, say $U=\cup_{i=1}^\infty J_i$ which is a countable disjoint union of open intervals, we have that for any $n$, $\pi_{W^\perp}(\mu)(\cup_{i=1}^n J_i)\to \pi_{V^\perp}(\mu)(\cup_{i=1}^n J_i)$. Hence $\liminf_{W\to V}\pi_{W^\perp}(\mu)(U)\geq \pi_{V^\perp}(\mu)(U)$, then by (1) and equivalent condition of weak convergence of measures we get (2).
\end{proof}

\paragraph{Sequences of numbers}
\begin{lem}
\label{lem: pseudo cont entpy prf 1}
For any $\epsilon>0$ and any $m\geq M(\epsilon)$, the following holds. Given any two finite sequence of numbers $\{a_i\}_{i=1}^{q^m}$ and $\{b_i\}_{i=1}^{q^m}$ in $[0,1]$ that satisfy
\begin{equation*}
\sum a_i=\sum b_i=1\,\,\text{and}\,\,|a_i-b_i|\leq q^{-3m}\,\,\text{for each}\,\,1\leq i\leq q^{m},
\end{equation*}
we have 
\begin{equation*}
\frac{1}{m}\left|\sum a_i\log a_i-\sum b_i\log b_i\right|\leq \epsilon.
\end{equation*}
\end{lem}
\begin{proof}
We estimate $|a_i\log a_i-b_i\log b_i|$ in two ways:
\begin{enumerate}
    \item If $a_i>q^{-1.5m}$, since $|a_i-b_i|\leq q^{-3m}$, we can assume that $b_i>q^{-1.6m}$. Then by Lagrange mean value theorem we have $|\log a_i-\log b_i|\leq q^{1.6m}\cdot q^{-3m}=q^{-1.4m}$. Therefore \begin{eqnarray*}
|a_i\log a_i-b_i\log b_i|&\leq &|a_i-b_i|\cdot |\log a_i|+  |b_i|\cdot |\log a_i-\log b_i|\\
&\leq & q^{-3m}\cdot (1.5m)+q^{-1.4m}\leq q^{-1.3m}
\end{eqnarray*}
if $m$ is large enough.
\item If $a_i\leq q^{-1.5m}$, since $|a_i-b_i|\leq q^{-3m}$, we can assume that $b_i<q^{-1.4m}$ and $m$ large enough such that $q^{-m}$ less than the unique extreme point of $x\log x$ in $[0,1]$. Therefore we have 
$|a_i\log a_i-b_i\log b_i|\leq 2q^{-1.4m}\cdot (1.4m)\leq q^{-1.3m}$ if $m$ large enough.
\end{enumerate}
Therefore combining two estimates above we get $\frac{1}{m}\left|\sum a_i\log a_i-\sum b_i\log b_i\right|\leq \frac{1}{m}q^{-0.3m}\leq \epsilon$
for any $m$ large enough.
\end{proof}

\section{Spherical geometry}

\begin{lem}\label{lem:g-1 V lip}Let $C>1$.
If $d(V,(E_1)^\perp)\geq 1/C$, then the map $\pi(V,E_1^\perp,V^{\perp} )$ has scale $u=1$ and distortion bounded by $C$. 

Moreover, by the identification in Lemma \ref{lem: good continious V}, for $V_1\notin V_2^\perp$, the action $\pi(V_1,V_2^{\perp},V_1^\perp)$ can be identified with some element in $\PSL_2(\R)$ of norm $|1/d(V_1,V_2^\perp)|^{1/2}$. 
\end{lem}
\begin{proof}
Without loss of generality, we can assume that $V=\R v, v=\frac{e_1+\lambda e_3}{\sqrt{1+\lambda^2}}$ with $\sqrt{1+\lambda^2}=1/d(V,(E_1)^\perp)$. Let $V'=\R v', v'=\frac{-\lambda e_1+e_3}{\sqrt{1+\lambda^2}}$, then $(v,v', e_2)$ forms an orthonormal basis of $\R^3$. The map $\pi(V,E_1^\perp, V^\perp )$ projects a vector $a e_2+bv'$ to $E_1^\perp$, along kernel $V$. So the image is 
    \[(a e_2+b v')+b\lambda v=a e_2+b\sqrt{\lambda^2+1}e_3. \]
    Therefore with basis $e_2,v'$ and $e_2,e_3$, the transformation $\Pi(V,E_1^\perp, V^\perp)$ can be written as a matrix $\begin{pmatrix}
        1 & 0 \\ 0 & \sqrt{1+\lambda^2}        
    \end{pmatrix}$. Then the distortion estimate of $ \pi(V,E_1^\perp, V^\perp)$ follows from basic estimates of $\PSL_2(\R)$ (\cref{lem: sl2 basic}). 
    
    We can find $k\in\SO(3)$ such that $V_2=kE_1$, then
    \[\pi(V_1,V_2^\perp,V_1^\perp)=k\pi(k^{-1}V_1, E_1^\perp,(k^{-1}V_1)^\perp)k^{-1}. \]
    Then the statement follows from the computation of the first part by noticing $d(k^{-1}V_1,E_1^\perp)=d(V_1,V_2^\perp)$.
\end{proof}
We will use the following classical result in spherical geometry later. For completeness, we provide a proof. 
\begin{lem}\label{lem:projection}
Let $x\neq y,V\in \P(\R^3)$ and $C>0$. If we denote by $(xy)$ the projective line in $\P(\R^3)$ passing through $x,y$, then we have
\begin{enumerate}
    \item If $d(V,(xy))>1/C$, then $d(\pi_{V^\perp}x,\pi_{V^\perp}y)\geq d(x,y)/C$;
\item If either $d(V,x)\geq 1/C$ or $\ d(V,y)\geq 1/C$, then $d(\pi_{V^\perp}x,\pi_{V^\perp}y)\leq C d(x,y)$. 
\end{enumerate}
\end{lem}

\begin{proof}
Let $z=V$. Then $d(\pi_{V^\perp}x,\pi_{V^\perp}y)=\sin(\angle_z(x,y))$ (the angle between two projective lines $(xz),(yz)$ at vertex $z$). By the spherical law of sines, the definition of $d$ on $\P(\R^3)$ and the assumption of the lemma, we have
\begin{equation*}
    \frac{ d(x,y)}{\sin \angle_z(x,y)}=\frac{ d(y,z)}{\sin \angle_x(y,z)}
    \text{ and }
    \frac{d(z,(xy))}{\sin \angle_x(y,z)}=\frac{d(x,z)}{\sin (\pi/2)}.
\end{equation*}
For the first statment, if $d(z,(xy))\geq 1/C$, we have 
\[ \frac{ d(x,y)}{\sin \angle_z(x,y)}=\frac{d(y,z)d(x,z)}{d(z,(xy))}\leq C,  \]
which implies $d(\pi_{V^\perp}x,\pi_{V^\perp}y)=\sin\angle_z(x,y)\geq d(x,y)/C$.

For the second statement, without loss of generality, we assume that $d(V,y)=d(z,y)\geq 1/C$. By the spherical law of sines again, we have 
\[\frac{ d(x,y)}{\sin\angle_z(x,y)}=\frac{d(y,z)}{\sin\angle_x(y,z)}\geq 1/C,\]
which implies $d(\pi_{V^\perp}x,\pi_{V^\perp}y)=\sin\angle_z(x,y)\leq Cd(x,y)$.
\end{proof}

\section{Proof of Lemma \ref{Lemma: BHR 4.5}}\label{sec:lemma 4.5}

We consider the following definition, which is a quadratic analogue of $\sigma$-independence in \cite{barany_hausdorff_2017} (also occured in \cite{hochman_self-similar_2015}).

\begin{defi}
Recall the map $p_2$ from \cref{lem:subvariety}. For $\rho>0$ and $V\in \P(\R^3)$, we call five points $V_1,\cdots, V_5$ in $\P(\R^3)$ are $(\rho,V)$-quadratic independent if for every $1\leq i\leq 5$, $p_2(V_i)$ has distance at least $\rho$ to the subspace generated by $p_2(V)$ and $p_2(V_j),\ j\neq i$ in $\P(Sym^2\R^3)$.
\end{defi}
By the following lemma, we can always find $(\rho,V)$-quadratic independent sets in a set with a certain $\mu$ measure.

\begin{lem}\label{lem: random choosing points}For any $\epsilon>0,$ there exists $\rho>0$ such that for any $V\in \P(\R^3)$, any Borel set
$A\subset \P(\R^3), \mu(A)>(100\epsilon)^{1/5}$, there exists $\{V_1,\dots, V_5\}\subset A$ such that $V_1,\cdots, V_5$ are $(\rho,V)$-quadratic independent.     
\end{lem}
\begin{proof}
The proof is similar to that of \cite{hochman_self-similar_2015}. We first need a lemma saying that $\mu$ cannot have a big weight on a small neighbourhood of a subvariety. For a subset $W$ of a metric space write $W^{(\epsilon)}$ for its $\epsilon-$neighborhood. Then \cref{lem:subvariety} reads as
\begin{lem}\label{lem: mu not 1d}
For any $\epsilon>0$, there exists $\rho(\epsilon)>0$ such that for any hyperplan $W\subset \P(Sym^2\R^3)$, $(p_2)_*\mu(W^{(\rho)})<\epsilon$.
\end{lem}
We go back to the proof of Lemma \ref{lem: random choosing points}. For any $V_1,\dots,V_5\in\P(\R^5)$, Llet $\cal C(V_1,V_2,\dots V_5)$ be the subspace generated by $p_2(V_1),\cdots, p_2(V_5)$ in $\P(Sym^2\R^3)$.
Now we fix $V\in \P(\R^3)$, $\epsilon>0$, and $\rho(\epsilon)$ we get from last lemma. Let $X_1,\dots, X_5$ be independent $\P(\R^3)$-valued random variables, each distributed according to $\mu$. 
Therefore by \cref{lem: mu not 1d}, if $\mu(A)>(100\epsilon)^{1/5}$,
\begin{eqnarray*}
 && \P(X_i\in A, X_i\notin  \cal C(V, X_j, j\neq i)^{(\rho)} )
\geq    \P(X_i\in A, \forall i)- 100\epsilon\\
 &>&\mu(A)^5-100\epsilon>0.
\end{eqnarray*}
Notice that any realization $X_1,\dots, X_5$ from the event above are $(\rho, V)$-quadratic independent. The proof is complete.
\end{proof}

Recall that $L\simeq \SL_2^\pm(\R)\ltimes\R^2$ is the group formed by $\{\begin{pmatrix}
 \det h&\\
 n&h
\end{pmatrix}, n\in \R^2, h\in \SL_2^\pm(\R)\}$. 
Fix $\epsilon=10^{-7}$ and $\rho=\rho(\epsilon)$ is decided by Lemma \ref{lem: random choosing points}. Denote by 
\begin{eqnarray*}
    I(\rho, V)&:=&\{(V_1,\dots, V_5)\in \P(\R^3)^5:\  V_1,\dots, V_5 \text{ are $(\rho,V)$-quadratic independent}\}\\
    I(\rho)&:=&\{(V_1,\cdots,V_5,V)\in\P(\R^3)^6:\  \  (V_1,\dots, V_5)\in I(\rho,V),\  d(V, E_1^{\perp})\geq 1/2C_1\},
\end{eqnarray*}then $I(\rho,V)$ and $I(\rho$) are closed and hence compact.

For any $(\rho,V)$-quadratic independent five points $\{V_1,\dots, V_5\}$, we consider the evaluation map from $L$ to $\P(\R^3)^5$,  \[E_{V_1,\dots V_5}: L\to \P(\R^3), \quad E_{V_1,\dots, V_5}(g):=(g(V_1),\dots, g(V_5)).\]
Define the map 
$$\Phi^{V}_{V_1,\dots, V_5}:L\to \P(\R^2)^5,\quad \Phi^{V}_{V_1,\dots, V_5}(g):= \pi^5_{V^\perp}\circ E_{V_1,\dots,V_5}(g),$$ here we denoting by $\pi_{V^\perp}^5$ the map $(\pi_{V^\perp}, \dots, \pi_{V^\perp}):\P(\R^3)^5\to \P(\R^2)^5$. 

Notice that even though $\pi_{V^\perp}^5$ is not well-defined everywhere, the restriction of $\pi_{V^\perp}^5$ on $(\P(\R^3)-B(V, \rho/2))^5$ is $C^2$, uniformly on $V$. Then by invariance of $\SO(3)$, let $$C'_1=C'_1(\rho,V)=\|D\pi_{V^\perp}|_{\P(\R^3)-B(V, \rho/2)}\|<\infty,$$ 
which is finite and independent of $V$.

By continuity, there exists a compact neighborhood $Z_1$ of $\rm{id}\in L$ such that for any $g\in Z_1$ and $V_0\in \P(\R^3)$, we have $d_{\P(\R^3)}(gV_0,V_0)\leq \rho/2$. Since the action $E_{V_1,\dots, V_5}$ is $C^\infty$ on $L$ and all the derivatives continuously depend on $\{V_1,\dots, V_5\}$, hence by compactness of $I(\rho,V)$, $$C'_2:=\sup_{(V_1,\dots,V_5)\in I(\rho,V), g_0\in Z
_1}\|D_gE_{V_1,\dots, V_5}|_{g=g_0}\|$$
is finite. The definition of $\rho$-quadratic independence also implies $d(V,V_i)\geq \rho$. Therefore by our choice of $Z_1$ we know that 
\begin{equation}\label{eqn: C1 norm contrl DPhi}
   C'_3:=\sup_{(V_1,\dots,V_5,V)\in I(\rho), g_0\in Z_1}\|D_g\Phi^V_{V_1,\dots, V_5}|_{g=g_0}\|\leq C'_1C'_2 
\end{equation}
is finite.
The following lemma explains the reason why we consider quadratic independence.
\begin{lem}\label{lem: indep imply nondeg}
There exist a compact neighborhood $Z\subset Z_1$ of $\rm{id}\in L$ and a constant $C_4'>0$ such that
\begin{equation}\label{eqn: det contrl DPhi}
\inf_{(V_1,\dots, V_5,V)\in I(\rho), g_0\in Z}|\rm{Jac}(D_g\Phi^V_{V_1,\dots, V_5})|_{g=g_0}|\geq C_4'.
\end{equation}    
\end{lem}
\begin{proof}Notice that the function $|\rm{Jac}(D_g\Phi^V_{V_1,\dots, V_5})|$ continuously depends on $g_0\in Z_1$ and $(V_1,\dots,V_5,V)\in I(\rho)$. Therefore by compacticity $I(\rho)$, to find $Z$ and $C_4'$ we only need to show that for all $(V_1,\dots,V_5,V)\in I(\rho)$
\begin{equation}\label{eqn: jacb non deg}
    |\rm{Jac}(D_g\Phi^V_{V_1,\dots, V_5})|_{g=\rm{id}}|>0.
\end{equation}
We show \eqref{eqn: jacb non deg} quite straightforwardly. Apparently any vector in the Lie algebra $\rm{Lie}(L)$ can be written as $X_{\frak n,\frak h}=\begin{pmatrix}
 &&\\
n_1&h_{11}&h_{12}\\
n_2&h_{21}&-h_{11}
\end{pmatrix}$
with real coefficients.
We write that $V_i=\R\cdot (a_i, b_i,c_i)^t$ and $V=\R\cdot (v_1, v_2, v_3)^t$. Then if \eqref{eqn: jacb non deg} does not holds then there exist $(V_1,\dots, V_5,V)\in I(\rho)$ and $X_{\frak n, \frak h}\in \rm{Lie}(L)-\{0\}$ such that 
\begin{equation}\label{eqn: degene case contrad}
  (D_g\Phi^V_{V_1,\dots, V_5})|_{g=\rm{id}}(X_{\frak n,\frak h})= 0.
\end{equation}

Notice that for any $V_i\neq V$, the kernel $K_{V_i}\subset T_{V_i}\P(\R^3)$ of the tangent map $D_x\pi_{V^\perp}|_{x=V_i}:\P(\R^3)\to \P(\R^2)$ is the tangent space at $V_i$ of the project line $(V V_i)$ passing through $V, V_i$. 
Therefore if \eqref{eqn: degene case contrad} holds, then 
\begin{equation*}
D_gE_{V_1,\dots,V_5}|_{g=\rm{id}}(X_{\frak n, \frak h})\in  K_{V_1}\oplus \cdots \oplus K_{V_5}   .
\end{equation*}
By the separable form of the evaluation map and the canonical correspondence between Lie algebras and the one-parameter subgroup of Lie groups, we get that the last equation is equivalent to 
\begin{equation}\label{eqn: sepa jac eqn}
\frac{d}{dt}|_{t=0}(e^{tX_{\frak n, \frak h}}\cdot V_i)\in K_{V_i}=T_{V_i}(V V_i),    
\end{equation}
for $i=1,\cdots,5$.
If we consider the corresponding equation of \eqref{eqn: sepa jac eqn} in linear space rather than the equation in the tangent space of projective space, we get (here the action is linear action): for $i=1,\cdots,5$
\begin{equation*}
    \frac{d}{dt}|_{t=0}(e^{tX_{\frak n, \frak h}}\cdot \begin{pmatrix}
        a_i\\b_i\\c_i
    \end{pmatrix})\in  \left\{l\cdot\begin{pmatrix}
        a_i\\  b_i\\ c_i
    \end{pmatrix}+k\cdot \begin{pmatrix}
        v_1\\  v_2\\ v_3 
    \end{pmatrix}, (k,l)\in \R^2 \right\}.
\end{equation*}
By direct calculation, we have 
\begin{equation*}
    \frac{d}{dt}|_{t=0}(e^{tX_{\frak n, \frak h}}\cdot \begin{pmatrix}
        a_i\\b_i\\c_i
    \end{pmatrix})=\begin{pmatrix}
    &&\\
    n_1&h_{11}&h_{12}\\
n_2&h_{21}&-h_{11}
    \end{pmatrix}\cdot \begin{pmatrix}
        a_i\\b_i\\c_i
    \end{pmatrix}=\begin{pmatrix}
        \\
        a_in_1+b_ih_{11}+c_ih_{12}\\
        a_in_2+b_ih_{21}-c_ih_{11}
    \end{pmatrix}.
\end{equation*}
Therefore 
\begin{equation}\label{eqn: image maps in kernel of pi v perp}
   \det  \begin{pmatrix}0 & a_i & v_1\\
        a_in_1+b_ih_{11}+c_ih_{12}& b_i& v_2\\
        a_in_2+b_ih_{21}-c_ih_{11}   &  c_i & v_3 
    \end{pmatrix}=0.
\end{equation}

Recall that $\{e_1,e_2,e_3\}$ is a basis of $\R^3$ and $\{e_ie_j:\ i\leq j\}$ is a basis of $Sym^2\R^3$. Let $\{(e_ie_j)^*:\ i\leq j\}$ be the the dual basis of $\{e_ie_j:\ i\leq j\} $.
Consider the linear form $\varphi$ on $Sym^2\R^3$
\begin{eqnarray*}
&&(n_2v_2-n_1v_3)(e_1e_1)^*+\frac{1}{2}(h_{21}v_2-h_{11}v_3-n_2v_1)(e_1e_2)^*\\&+&\frac{1}{2}(n_1v_1-h_{12}v_3-h_{11}v_2)(e_1e_3)^*-(h_{21}v_1)(e_2e_2)^*+(h_{11}v_1)(e_2e_3)^*+(h_{12}v_1)(e_3e_3)^*.
\end{eqnarray*}
Then \cref{eqn: image maps in kernel of pi v perp} is just for $i=1,\cdots,5$, $u_i=(a_i,b_i,c_i)^t$
\begin{eqnarray}\label{eqn: conic curve pass xi}
\varphi(u_iu_i)=\varphi(a_1^2(e_1e_1)+2a_1b_1(e_1e_2)+2a_1c_1(e_1e_3)+b_1^2(e_2e_2)+2b_1c_1(e_2e_3)+c_1^2(e_3e_3))=0.
\end{eqnarray}
Due to $v_1\neq 0$ from the hypothesis of $V$, we see that if $X_{\frak n,\frak h}\neq 0$, then at least one coefficient of the equation is non-zero. Then $W=\P(\ker\varphi) $ is a hyperplane in $\P(Sym^2\R^3)$,

By \eqref{eqn: conic curve pass xi}, $W$ passes through $p_2(V_i)$ for $i=1,\cdots,5$. Observe that in \eqref{eqn: image maps in kernel of pi v perp} if let $(a_i,b_i,c_i)=(v_1,v_2,v_3)$ then the determinant is  $0$. 
Therefore $W$ also passes through $p_2(V)$ for $V= \R\cdot (v_1,v_2, v_3)^t$, which contradicts with the assumption that $V_1,\dots,V_5$ are $(\rho,V)$-quadratic independent.
\end{proof}
\begin{rem*}
    In \cref{lem: indep imply nondeg} and the next lemma, the condition $v_1\neq 0$, i.e. $d(V,E_1^\perp)>0$ is a crucial and necessary condition. Otherwise, the group $U_V$ may intersect $L$ non-trivially, which means there is a non-trivial subgroup $S$ of $L$ such that $\pi_{V^\perp}S=\pi_{V^\perp}id$.
    In this case, we cannot get a bi-Lipschitz map.
\end{rem*}
As a corollary of Lemma \ref{lem: indep imply nondeg}, we have 
 \begin{lem}\label{lem: bi-Lip Phi on Z}
 There exist a compact neighborhood $Z$ of identity of $L$ and $C'_6>0$ such that for any $(V_1,\dots,V_5,V)\in I(\rho)$, $\Phi^V_{V_1,\dots,V_5}|_Z$ is bi-Lipschitz with its image, with Lipschitz constant $C'_6$.  
\end{lem}
\begin{proof}
By \cref{lem: indep imply nondeg}, we get a compact neighborhood $Z_2\subset Z_1$ of identity in $L$ and constants $C'_3, C'_4>0$ such that 
\begin{equation}\label{eqn: appdix C'_4}
\sup_{(V_1,\dots,V_5,V)\in I(\rho), g_0\in Z}\|D_g\Phi^V_{V_1,\dots, V_5}|_{g=g_0}\|\leq C'_3, \inf_{(V_1,\dots, V_5,V)\in I(\rho), g_0\in Z}|\rm{Jac}(D_g\Phi^V_{V_1,\dots, V_5})|_{g=g_0}|\geq C'_4 
\end{equation}
which implies that there exists $C'_5>0$
$$
  \sup_{(V_1,\dots, V_5,V)\in I(\rho), g_0\in Z}\|(D_g\Phi^V_{V_1,\dots, V_5}|_{g=g_0})^{-1}\|\leq C'_5.  
$$
%

\paragraph{Claim 1:}There is a compact neighborhood $Z\subset Z_2$ of $id$ such that $\Phi^V_{V_1,\dots,V_5}$ is injective restricted on $Z$.
\begin{proof}[Proof of the Claim 1]
As we showed in \cref{eqn: appdix C'_4}, by taking $g_0=id$ and applying the inverse function theorem we know for every $(V_1,\dots,V_5,V)\in I(\rho)$, there is a neighborhood $Z_{V_1,\dots, V_5,V}$ of $id$ such that $\Phi^V_{V_1,\dots,V_5}|_{Z_{V_1,\dots,V_5,V}}$ is a diffeomorphism to its image. It suffices to show $Z_{V_1,\dots, V_5,V}$ has a uniform size or an effective estimate of the size of the neighbourhood we got from the inverse function theorem. 

The key point is that by shrinking $Z$ if necessary we could assume $\Phi^V_{V_1,\dots,V_5}|_Z$ is uniformly $C^2$ for $(V_1,\dots, V_5,V)\in I(\rho)$. Recall that $\Phi^V_{V_1,\dots,V_5}=\pi^5_{V^\perp}\circ E_{V_1,\dots,V_5}$. Notice that $E_{V_1,\dots, V_5}$ is uniformly $C^2$ on any compact set if $(V_1,\dots, V_5,V)$ take values in the compact set $I(\rho)$, and $\pi^5_{V^\perp}$ is uniformly $C^2$ on $((\P(\R^3)-B(V, \rho/2))^5$. By definition of quadratic independence, $d(V_i, V)\geq \rho$. Hence if $Z$ is small enough then $d(gV_i,V)>\rho/2$ for $i=1,\dots 5$. Consequently by taking $Z$ small enough we can assume $\Phi^V_{V_1,\dots,V_5}|_Z$ is uniformly $C^2$ for $(V_1,\dots, V_5,V)\in I(\rho)$. Then by the classical effective estimate (which needs a uniform $C^2$ estimate) of the size neighbourhood in inverse function theorem (for example, see \cite{lang2012real} Lemma 1.3 Section XIV), we get the proof of the claim. 
\end{proof}

Back to the proof of the Lemma.
 On the one hand, by \cref{eqn: appdix C'_4} and the arguments in the second paragraph of Claim 1, we get that there exists $s>0$ such that the family of maps $\Phi^V_{V_1,\dots,V_5}, (V_1,\dots, V_5,V)\in I(\rho)$ is uniformly local bi-Lipschitz on any small ball with radius $s$ in $Z$, with a uniform local Lipschitz constant $C_6''$. 
 
 On the other hand, by compactness of $I(\rho)$ and $Z\times Z-\{(g,g'): d
(g,g')<s\}$ and the injectivity we shown in Claim 1, there exists some constant $C_6'''>1$ such that $$\frac{d_L(g,g')}{d_{\P(\R^2)}(\Phi^V_{V_1,\dots,V_5}(g),\Phi^V_{V_1,\dots,V_5}(g') )}\in [1/C_6''', C_6],$$
for any $(V_1,\dots,V_5,V)\in I(\rho)$ and $g,g'\in Z, d(g,g')\geq s$. Let $C_6'=\max(C_6'', C_6''')$. Then for each $(V_1,\dots, V_5,V)\in I(\rho)$, $\Phi^V_{V_1,\dots,V_5}$ is uniformly bi-Lipschitz restricted on $Z$ to its image with bi-Lipschitz constant $C_6'$.
\end{proof}

Notice that $C'_6$ depends on $Z, I(\rho)$, and $I(\rho)$ depends on $C_1$ by definition, so we emphasize that $C'_6=C'_6(Z, C_1)$.
As a corollary, we have 
\begin{lem}\label{lem: BHR last claim in 4.5}
There exists a constant $C'=C'(Z,C_1)>0$ such that if $(V_1,\dots, V_5,V)\in I(\rho)$ and $\theta$ supported on $Z$, then for any $n\in \N$
\begin{equation}
  H(\theta, \cal Q_{n}^{L})\leq \sum_{j=1}^5 H(\pi_{V^\perp}(\theta. V_j), \cal Q_{n})+C' .
\end{equation}
\end{lem}
\begin{proof}Let $\pi_j : \P(\R^2)^5\to \P(\R^2)$ be the coordinate projections. Consider  $\Phi^V_{V_1,\dots,V_5}$ restricted on $Z$ in Lemma \ref{lem: bi-Lip Phi on Z}, then $\Phi^V_{V_1,\dots,V_5}$ is a diffeomorphism on $Z$ and be bi-Lipschitz to its image $\Phi^V_{V_1,\dots, V_5}(Z)$ with constant $C_6'$. Thus by \cref{lem: entpy prpty SL2R act} there is a constant $C'$ depending on $Z$ , $I(\rho)$ (hence $C_1$) such that
\[|H(\theta, \cal Q_{n}^{L})-H(\Phi^V_{V_1,\dots,V_5}\theta, \cal Q^{\P(\R^2)^5}_{n})|\leq C'.\]
Since $\cal Q^{\P(\R^2)^5}_{n}=\bigvee_{j=1}^5\pi_j^{-1}\cal Q^{\P(\R^2)}_{n}$, this is the same as 
\[|H(\theta, \cal Q_{n}^{L})-H(\Phi^V_{V_1,\dots,V_5}\theta,  \bigvee_{j=1}^5\pi_j^{-1}\cal Q^{\P(\R^2)}_{n})|\leq C'.\]
The statement now follows by 
\[H(\Phi^V_{V_1,\dots,V_5}\theta,  \bigvee_{j=1}^5\pi_j^{-1}\cal Q^{\P(\R^2)}_{n})\leq \sum_{j=1}^5 H(\Phi^V_{V_1,\dots,V_5}\theta,  \pi_j^{-1}\cal Q^{\P(\R^2)}_{n} )=\sum_{i=1}^5H(\pi_{V^\perp}(\theta. V_j), \cal Q_{n}).\]
\end{proof}

\begin{proof}[Proof of \cref{Lemma: BHR 4.5}]
Let $C'$ be as in the last lemma and set $$B=\{x\in \P(\R ^3): \frac{1}{k}H(\pi_{V^\perp}(\theta_\cdot x), \cal Q_{i+k})>\frac{1}{5k}H(\theta, \cal Q_{i+k})-\frac{C'}{k}\}.$$
We claim that $\mu(B)\geq 1-(100\epsilon)^{1/5}=0.9$. Otherwise by Lemma \ref{lem: random choosing points}, for any $V$ such that $d(V, E_1^{\perp})\geq \frac{1}{2C_1}$, there exists $(V_1,\dots,V_5)$ such that $V_i\notin B$, and $(V_1,\dots,V_5,V)\in I(\rho)$. Hence by applying Lemma \ref{lem: BHR last claim in 4.5}, we get 
$$H(\theta, \cal Q^L_{i+k})\leq \sum_{i=1}^5H(\pi_{V^\perp}(\theta_\cdot V_i),\cal Q_{i+k})+C'\leq H(\theta, \cal Q_{i+k}^L)-4C',$$
which is a contradiction. 

Therefore the lemma follows by letting $C=\max\{5,C'(Z,C_1)\}$.
\end{proof}

\section{Results in linear algebra}

\begin{lem}\label{lem:vh va}
    For any $1/4>r>0$, for $g\in\SL_3(\R)$ and $V\in\P(\R^3)$ with $d(V_g^+,V)>2r$, $d(V^\perp,H_{g}^-)>2r$ (in the sense of Hausdorff distance) and $|\chi_1(h_{V,g})-\chi_1(g)|\leq |\log r|/2$, then
    \[d(\pi_{V^\perp}V_g^+,{h^+_{V,g}})\leq q^{-\chi_1(g)}/r^2. \]
    
    \end{lem}
    \begin{proof}
        From the condition $d(V^\perp,H_{g}^-)>2r$, we obtain
        \[ \exists x\in\P(V^\perp), d(x,H_{g}^-),\ d(x,h_{V,g}^-)>r. \]
     Due to \cref{lem:gv d v g-}, combined with $\ \chi_1(h_{V,g})\geq \chi_1(g)- |\log r|/2 $, we have
    \begin{equation}\label{eqn.v g+ gx}
     d(V_{g}^+,gx) \leq \frac{q^{-\chi_1(g)}}{d(x,H_{g}^-)}\leq q^{-\chi_1(g)}/r,\ d({h_{V,g}^+},h_{V,g}x) \leq \frac{q^{-\chi_1(h_{V,g})}}{d(x,h_{V,g}^-)}\leq q^{-\chi_1(g)}/r^{3/2}.  
    \end{equation}
    Due to \cref{lem:projection} and $d(V_{g}^+,V)>2r $ from the hypothesis, we obtain from the first inequality of \cref{eqn.v g+ gx} that
    \[ d(\pi_{V^\perp}V_{g}^+,\pi_{V^\perp} gx) \leq q^{-\chi_1(g)}/2r^2.  \]
    Combined with $\pi_{V^\perp}gx=h_{V,g}x$ and the second inequality of \cref{eqn.v g+ gx}, we obtain the claim if $r<1/4$. 
    \end{proof}
   
\section{Fourier decay property of stationary measures on flag variety}\label{sec: appd F decay}
Recall that the flag variety $\calF=\calF(\R^3)$ is the space of flags $$\cF(\RR^3):=\lb{(V_1,V_2):  
V_1 \sbs V_{2},
V_i\text{ is a linear subspace of }\RR^3 \text{ of dimension }i}.$$ 
In this section, we recall and explain the result in \cite{li_fourier_2018} on the Fourier decay property of stationary measures on $\mathcal{F}$. 
Roughly speaking, the Fourier decay property of the stationary measure $\mu$ (see Theorem \ref{thm:fourier} for details)
 is the decay of the oscillatory integral
\begin{equation}\label{equ:fourier decay appd}
 \int_\calF e^{i\xi\varphi(\eta)}r(\eta)\dd\mu(\eta),  
\end{equation}
as the real parameter $\xi$ tends to infinity, where $\varphi$ is a $C^2$-function and $r$ is a $C^1$-function.

The classical van der Corput lemma is stated for  $([0,1],\Leb)$ instead of $(\calF,\mu)$ in \cref{equ:fourier decay appd}, 
and it requires a non-degenerate condition of $\varphi'$. In our higher dimensional case, a similar non-degeneracy assumption (\cref{defi:c r good}, \cref{defi:c r good simple}) is also introduced to obtain Fourier decay. 

Due to the group action of $\SL_3(\R)$ and the structure of $\calF$, not all the tangent directions on $\calF$ play the same role. For this reason, we introduce the \textit{simple root directions} of the tangent space. 
\paragraph{Simple root directions of $T\calF$.}
 Recall $\alpha_1,\alpha_2$ are the simple roots of the root system of $\SL_3(\R)$, that is for $\lambda$ in the Lie algebra $\frak a=\{\lambda=(\lambda_1,\lambda_2, \lambda_3):\lambda_i\in\R,\ \sum\lambda_i=0\} $, we have $\alpha_i(\lambda)=\lambda_i-\lambda_{i+1}$ for $i=1,2$. 
 
We define the line bundle $V_{\alpha_1}$ as the sub-bundle of the tangent bundle $T\calF$, whose fiber at any  $\eta=(V_1(\eta),V_2(\eta))\in \calF$ is the tangent space of the circle $\{(V_1', V_2(\eta)): V_1'\subset V_2(\eta), \dim V_1'=1\} \subset \calF$. Similarly, $V_{\alpha_2}$ is defined as the sub-line bundle of $T\calF$ over $\calF$, whose fiber at any $\eta\in T_\eta\calF$ is the tangent space of the circle  $\{(V_1(\eta), V_2'):V_1(\eta)\subset V_2', \dim V_2'=2\} \subset \calF$. From the definition, both $V_{\alpha_i}, i=1,2$ are $\SL_3(\R)$-invariant sub-bundles of $T\calF$. 

For each $\eta\in\calF$, let $V_{\alpha_1+\alpha_2}(\eta)$ be the one-dimensional subspace of $T_\eta\calF$ orthogonal to $V_{\alpha_1}(\eta)$ and $V_{\alpha_2}(\eta)$. 
For any $g\in \SL_3(\R)$, the tangent map $Dg:T\calF\to T\calF$ of the action $g:\calF\to \calF$ contracts $V_{\alpha_i}$ with a speed can be computed using $\alpha_i$ for $i=1,2$. For a non-zero vector $v$ in the direction $V_{\alpha_1+\alpha_2}(\eta)$, for most of $g$ the image vector $Dg(v)$ will have non-trivial component in $V_{\alpha_1}(g\eta)\oplus V_{\alpha_2}(g\eta)$ dominating the whole vector.
See Section 2.4 in \cite{li_fourier_2018} for more details of these directions and computations of the contraction rates.

 Let $Y_{\alpha_i}(\eta)$ be a unit vector in $V_{\alpha_i}(\eta)$ at $\eta$, for $i=1,2$, where the metric is from the $\SO(3)$-invariant Riemannian metric on $\calF$. We denote by $\partial_{i}\varphi(\eta)=\partial_{Y_{\alpha_i}}\varphi(\eta)$ the directional derivative at a point $\eta$.\footnote{The unit vector $Y_{\alpha_1}(\eta)$ cannot be continuously extended to the full $\calF$, due to the fact that $\cal F$ is non-orientable. But we can always extend it locally and will be used in \cref{G3} } 
 
 \paragraph{$(C,r)$-good functions and Fourier decay of stationary measures.}
We introduce the $(C,r)$-good condition of $\varphi$ for $C\geq 1, r\in C^1(\cF), \varphi\in C^2(\cF)$. Recall the metrics $d,d_i,i=1,2$ on $\cF$ defined in \cref{defi: flag}. 
 
\begin{defi}[Definition 4.1 in \cite{li_fourier_2018}]\label{defi:c r good}
     Let $C\geq 1$, $\varphi\in C^2(\calF)$ and $r \in C^1(\calF)$. Denote by $v_{i}:=\sup_{\eta\in\supp r}|\partial_{i}\varphi|$. Let $J$ be the $1/C$-neighborhood of $\supp r$. The function $\varphi$ is called $(C,r)$-good if:
     \begin{enumerate}

        \item For $\eta\in\supp r$ and $i=1,2$
         \begin{equation*}
         |\partial_{i}\varphi(\eta)|\geq v_{i}/C.   
         \end{equation*}

        \item  For $\eta,\eta'\in J$ and $d(\eta,\eta')<1/C$,
         \begin{equation*}
         |\varphi(\eta)-\varphi(\eta')|\leq C\sum_{1\leq i\leq 2} d_i(\eta,\eta')v_i.
          \end{equation*}

        \item For $\eta,\eta'\in J$ and $d(\eta,\eta')<1/C$, $i=1,2$
        \begin{equation*}
        |\partial_i\varphi(\eta)-\partial_i\varphi(\eta')|\leq Cd(\eta,\eta')v_i.
        \end{equation*}

        \item 
        \begin{equation*}
           \sup_{1\leq i\leq 2} v_i\in[1/C,C]. 
        \end{equation*}
         
     \end{enumerate}
 \end{defi}
The third condition serves as a similar role to the second derivative. The second condition is due to the fact that we only use two directions $V_{\alpha_1}$ and $V_{\alpha_2}$, and we need to use the derivations on these two directions to control the Lipschitz property.
 
For Lipschitz-continuous function $r$ on a metric space $(X,d_X)$, its Lipschitz constant is defined by
\begin{equation}\label{equ.d lip}
  \Lip(r)=\sup_{x\neq x'}\frac{|r(x)-r(x')|}{d_X(x,x')}.  
\end{equation}
The main result of Fourier decay is as follows.
\begin{thm}[Theorem 1.7 in \cite{li_fourier_2018}]\label{thm:fourier}
Consider the action of $\SL_3(\R)$ on $\cF$. Let $\nu$ be a Zariski dense finitely supported probability measure on $\SL_3(\R)$,  and $\mu_\calF$ be the $\nu$-stationary measure on $\calF$. Then there exist $\epsilon_0, \epsilon_1>0$ only depending on $\nu$ such that the following holds. 

Let $\xi>1$ be large enough. For any pair of real functions $\varphi\in C^2(\calF)$, $r\in C^1(\calF)$ if $\varphi$ is $(\xi^{\epsilon_0},r)$-good, $\|r\|_{\infty}\leq 1$ and $ \Lip(r)\leq \xi^{\epsilon_0}$, then
\begin{equation}\label{eqn: Osl int}
 \left|\int_\calF e^{i\xi\varphi(\eta)}r(\eta)\dd\mu_\calF(\eta)\right|\leq \xi^{-\epsilon_1}.
\end{equation}
\end{thm}

For its application in \cref{sec:non concentration}, it suffices to consider $\varphi\in C^2(\cF)$ lifted from $\P(\R^3)$ (cf. Definition \ref{defi: flag}). For such a function $\varphi$, $(C,r)$-good condition can be simplified as follows. 
\begin{defi}\label{defi:c r good simple}
         Let $C\geq 1$, $r\in C^1(\calF)$ and $\varphi\in C^2(\calF)$ be a function lifted from $\P(\R^3)$. Let $v_{1}=\sup_{\eta\in\supp r}|\partial_{1}\varphi|$. Let $J$ be the $1/C$-neighborhood of $\supp r$. We say that $\varphi$ is $(C,r)$-good if:
     \begin{enumerate}

        \item For $\eta\in\supp r$,
         \begin{equation}\label{G1}
         |\partial_{1}\varphi(\eta)|\geq v_{1}/C.   
         \end{equation}

        \item  For $\eta,\eta'\in J$ and $d(\eta,\eta')<1/C$,
         \begin{equation}\label{G2}
         |\varphi(\eta)-\varphi(\eta')|\leq C d(V_1(\eta),V_1(\eta'))v_1.
          \end{equation}

        \item For $\eta,\eta'\in J$ and $d(\eta,\eta')<1/C$
        \begin{equation}\label{G3}
        |\partial_1\varphi(\eta)-\partial_1\varphi(\eta')|\leq Cd(\eta,\eta')v_1.
        \end{equation}

        \item 
        \begin{equation}\label{G4}
            v_1\in[1/C,C]. 
        \end{equation}
         
     \end{enumerate}
 \end{defi}
 \paragraph{Claim:}If a lifted function $\varphi\in C^2(\cF)$ is $(C,r)$-good in the sense of Definition \ref{defi:c r good simple}, then it is $(C,r)$-good in the sense of Definition \ref{defi:c r good}.
\begin{proof}[Proof of the claim]
Since $\varphi$ is lifted from $\P(\R^3)$, the function $\varphi(\eta)=\varphi(V_1(\eta),V_2(\eta))$ 
is independent of $V_2(\eta)$.
Note that $Y_{\alpha_2}(\eta)$ is a tangent vector to the circle $\{(V_1(\eta),V_2'): V_1(\eta)\subset V_2', \dim V_2'=2 \} $ at $\eta$, and $\varphi$ is constant on this circle. Therefore $\partial_2\varphi(\eta)=\partial_{Y_{\alpha_2}}\varphi(\eta)=0$. Then \cref{defi:c r good} can be simplified to \cref{defi:c r good simple}.
\end{proof}


\bibliography{bibfile}

\bigskip

\noindent 
	Jialun Li. {\it CNRS-Centre de Math\'ematiques Laurent Schwartz, \'Ecole Polytechnique, Palaiseau, France.}  \\
	email: {\tt jialun.li@cnrs.fr} 
		
		\bigskip   

   \noindent 
   Wenyu Pan.
	{\textit{University of Toronto, 40 St. George St., Toronto, ON, M5S 2E4, Canada.}}  \\
	email: {\tt wenyup.pan@utoronto.ca} 
		
		\bigskip   

   \noindent
   Disheng Xu. 
	{\textit{Great Bay University,  Songshanhu International Community, Dongguan, Guangdong, 523000, China.}}  \\
	email: {\tt xudisheng@gbu.edu.cn} 
		
		\bigskip

\end{document}